\documentclass{elsarticle}
\usepackage{amsmath}

\usepackage{geometry}
\usepackage{subfigure}
\usepackage{setspace}
\usepackage{amssymb}
\usepackage{amsthm}
\usepackage{mathtools}
\usepackage{mathrsfs}
\usepackage{stmaryrd}
\usepackage{amsfonts}
\usepackage{booktabs}
\usepackage{longtable}
\usepackage{multirow}
\usepackage{mathtools}
\usepackage{float}
\usepackage{subfigure}
\usepackage{indentfirst}
\usepackage{epstopdf}
\usepackage{epsfig,amssymb,amsmath,color,graphicx,epsf,pifont,amsmath,CJK}
\usepackage{pgfplots}
\usepackage{tikz}
\usetikzlibrary{patterns,arrows,positioning,calc,fadings,shapes,decorations.markings}
\usetikzlibrary{shapes.geometric, arrows}
\tikzstyle{format}=[rectangle,draw,thin,fill=white] 
\tikzstyle{test}=[diamond,aspect=2,draw,thin] 
\tikzstyle{point}=[coordinate,on grid,] 

\usepackage{graphicx} 

\usepackage{listings}
\usepackage{algorithm}
\usepackage{algorithmic}
\algsetup{
  indent=4em,
  linenosize=\small,
  linenodelimiter=.
}
\usepackage{tabularx}
\newcolumntype{L}[1]{>{\raggedright\arraybackslash}p{#1}}
\newcolumntype{C}[1]{>{\centering\arraybackslash}p{#1}}
\newcolumntype{R}[1]{>{\raggedleft\arraybackslash}p{#1}}

\graphicspath{{pic/}}
\allowdisplaybreaks

\usepackage{graphics}
\usepackage{tikz}
\usepackage{bm}
\usepackage[colorlinks,
            linkcolor=red,
            anchorcolor=red,
            citecolor=blue
            ]{hyperref}

\newtheorem{theorem}{Theorem}[section] 
\newtheorem{lemma}[theorem]{Lemma}

\newtheorem{remark}[theorem]{Remark}

\newtheorem{example}[theorem]{Example}

\numberwithin{equation}{section}

\begin{document}
\begin{frontmatter}
\title{{\bf A second-order direct Eulerian GRP scheme for ten-moment Gaussian closure equations with source terms}}


\author[address1]{Jiangfu Wang}\ead{2101110034@stu.pku.edu.cn}
\author[address2,address1]{Huazhong Tang\corref{cor1}}\ead{hztang@math.pku.edu.cn}
%
\cortext[cor1]{Corresponding author.}
\address[address1]{Center for Applied Physics and Technology, HEDPS and LMAM, School of Mathematical Sciences, Peking University, Beijing, 100871, PR China.}
\address[address2]{Nanchang Hangkong University, Nanchang, 330000, Jiangxi Province, PR China.}

\begin{abstract}

This paper proposes a second-order accurate direct Eulerian generalized Riemann problem (GRP) scheme for the ten-moment Gaussian closure equations with source terms. The generalized Riemann invariants associated with the rarefaction waves, the contact discontinuity and the shear waves are given, and the 1D  exact Riemann solver is obtained.
After that, the generalized Riemann invariants and the Rankine-Hugoniot jump conditions are directly
used to resolve the left and right nonlinear waves (rarefaction wave and shock wave)  of the local GRP in Eulerian formulation,  and then the 1D direct Eulerian GRP scheme is derived. They are much more complicated, technical and nontrivial due to more physical variables and elementary waves.
 Some 1D and 2D numerical experiments are presented to check the accuracy and high resolution  of  the proposed GRP schemes, where the  2D   direct Eulerian GRP scheme is given  by using the Strang splitting method for simplicity.  It should be emphasized that several examples of 2D Riemann problems are constructed for the first time.

\vspace{2mm}
\end{abstract}

\begin{keyword}
Ten-moment equations; exact Riemann solver; generalized Riemann problem; generalized Riemann invariants; Rankine-Hugoniot jump conditions.
\end{keyword}

\end{frontmatter}
\section{Introduction}

It is well-known that the compressible Euler equations of gas dynamics can be derived from the Boltzmann equation \cite{levermore1996moment} by assuming local thermodynamic equilibrium. However, for many problems, such as collisionless plasma  \cite{dubroca2004magnetic,morreeuw2006electron,johnson2012ten,wang2018electron,dong2019global} and the non-equilibrium gas dynamics \cite{brown1995numerical}, the local thermodynamic equilibrium assumption does not hold and anisotropic effects are often present, so that the Euler equations are less suitable. The ten-moment Gaussian closure equations \cite{levermore1998gaussian} provide an alternative for such applications, where the pressure is described by an anisotropic and symmetric tensor.

Some numerical schemes have been proposed for solving the ten-moment equations in the past decades. A second-order upwind finite volume scheme was introduced in \cite{brown1995numerical}. First-order relaxation numerical schemes were employed to approximate the weak solutions of those equations, see \cite{berthon2006numerical, berthon2015entropy}. A Harten--Lax--van Leer-contact (HLLC) approximate Riemann solver was applied in \cite{sangam2008hllc} to solve the ten-moment equations coupled with magnetic field. In recent years, some high-order numerical methods were also developed, including high-order positivity-preserving discontinuous Galerkin (DG) methods \cite{meena2017positivity}, finite difference weighted essentially non-oscillatory (WENO) schemes \cite{meena2020positivity}, and high-order entropy stable finite difference methods   \cite{sen2018entropy} as well as DG methods \cite{biswas2021entropy}. Besides, a second-order robust monotone upwind scheme was formulated in \cite{meena2017robust}, and a second-order well-balanced (WB) scheme to handle equilibrium states was constructed in \cite{meena2018well}. Additionally, a robust finite volume scheme for the two-fluid ten-moment plasma flow equations was proposed in \cite{meena2019robust}. The high-order accurate positivity-preserving and well-balanced
discontinuous Galerkin schemes for ten-moment equations were also developed in \cite{wang2024high}, where a special modification to the numerical HLLC flux was imposed to enforce the well-balancedness and the geometric quasilinearization approach \cite{wu2021geometric} was applied to simplify the positivity analysis.

The generalized Riemann problem (GRP) scheme, which serves as an analytic second-order extension of the Godunov method, was initially developed for the study of compressible fluid dynamics \cite{ben1984second}. This approach employs a piecewise linear function to approximate the ``initial" data. Subsequently, it resolves a local GRP analytically at each cell interface to determine the numerical flux. For a detailed description, the readers are referred to \cite{ben2003generalized}. There are two versions of the original GRP scheme: the Lagrangian and the Eulerian. Typically, the Eulerian version is derived through a transformation based on the Lagrangian framework, which can be particularly intricate, especially when dealing with sonic and multi-dimensional scenarios. To circumvent those difficulties, second-order accurate direct Eulerian GRP schemes were respectively developed for the shallow water equations \cite{li2006generalized}, the Euler equations \cite{ben2006direct}, and a more general weakly coupled system \cite{ben2007hyperbolic}. Those schemes directly resolve the local GRPs in the Eulerian formulation  by using generalized Riemann invariants and Rankine-Hugoniot jump conditions.

Up to now, the GRP methodology has been widely implemented in many physically interesting cases, including the reactive flows \cite{ben1989generalized}, the motion of
elastic strings \cite{ta1992generalized}, the shallow water equations \cite{li2006generalized}, the radially symmetric compressible flows \cite{li2009implementation}, the one-dimensional (1D) and two-dimensional (2D) special relativistic hydrodynamics (RHD) \cite{yang2011direct,yang2012direct}, the gas-liquid two-phase flow in high-temperature and high-pressure transient wells \cite{xu2013direct}, the spherically symmetric general RHD \cite{wu2016direct}, the radiation
hydrodynamical equations \cite{kuang2019second}, the Baser-Nunziato two-phase model \cite{lei2021staggered}, the blood flow model in arteries \cite{sheng2021direct}, the compressible flows of real materials \cite{wang2023stiffened}, the Kapila
model of compressible multiphase flows \cite{du2023generalized}, the laminar
two-phase flow model with two-velocities \cite{zhang2024generalized} and so on. Besides, a comparison between the GRP scheme and the gas-kinetic scheme revealed the good performance of the GRP solver in simulating some inviscid flows \cite{li2011comparison}. By integrating the GRP scheme with the moving mesh method \cite{tang2003adaptive}, the adaptive direct Eulerian GRP scheme was effectively developed in \cite{han2010adaptive}, resulting in enhanced resolution and accuracy. The adaptive GRP scheme was further studied in simulating 2D complex wave configurations formulated with the 2D Riemann problems
of Euler equations \cite{han2011accuracy} and was also extended to unstructured triangular meshes \cite{li2013adaptive}. There are also many works on the high-order GRP schemes, such as the third-order GRP schemes for the Euler equations \cite{wu2014third2}, 1D RHD \cite{wu2014third1} and the general hyperbolic balance law \cite{qian2014generalized,qian2023high}. Moreover, a two-stage fourth order time-accurate
GRP scheme was also proposed for hyperbolic conservation laws \cite{li2016two} and for the special RHD \cite{yuan2020two}. Arbitrary high-order DG schemes based on the GRP solver were developed \cite{wang2015arbitrary},  where the reconstruction steps were halved compared with the existing high-order Runge-Kutta DG (RKDG) schemes.  The two-stage fourth-order DG method based on the GRP solver was also developed in \cite{cheng2019two}, and the computational cost can be considerably reduced by more than $50\%$ compared with the same order multi-stage strong-stability-preserving (SSP) RKDG method. The GRP-based high resolution schemes for the multi-medium flows \cite{huo2023grp,huo2024grp} and the axisymmetric hydrodynamics \cite{zhu2023high} were recently developed.

This paper develops the second-order direct Eulerian GRP scheme for the ten-moment Gaussian closure equations with source terms. The generalized Riemann invariants and the exact solutions of the 1D Riemann problem are given. Compared to  the Euler equations \cite{ben2006direct}, the shallow water equations \cite{li2006generalized} and the blood flow in arteries \cite{sheng2021direct} etc.,  the ten-moment equations  have more physical variables and   more elementary waves (e.g.  the left and the right shear waves which are linearly degenerate and separated by the contact discontinuity). More physical variables means that more linear algebra equations should be derived to compute the instantaneous time derivatives. Two more elementary waves means that four different solution states between the left and right nonlinear waves should be resolved technically. All these features make the derivation of the direct Eulerian GRP scheme for the ten-moment equations much more complicated and nontrivial.

This paper is organized as follows. Section \ref{preliminaries and notations} introduces the 1D ten-moment equations and gives corresponding eigenvalues, eigenvectors and generalized Riemann invariants. Section \ref{numerical schemes} introduces the outline of the 1D direct GRP scheme and its extension to 2D case.
 Section \ref{exact RP} presents the exact solver for the classical  Riemann problem, and  Section \ref{GRP solver} resolves the generalized Riemann problems. Some 1D and 2D numerical experiments are conducted in Section \ref{numerical experiments} to demonstrate the accuracy and performance of the proposed GRP scheme.
 It should be emphasized that several examples of 2D Riemann problems are constructed for the first time.
  Section \ref{conclusion} concludes this paper.

\section{Preliminaries and notations}\label{preliminaries and notations}
This section introduces the 1D ten-moment equations,  and corresponding eigenvalues as well as
eigenvectors,  discusses the characteristic fields, and
gives generalized Riemann invariants.

\subsection{1D ten-moment equations}\label{sect:2.1}
The 1D   ten-moment Gaussian closure equations with source term may be written into
the form of balance laws as
\begin{equation}\label{ten-moment2}
\frac{\partial\mathbf{U}}{\partial t}+\frac{\partial\mathbf{F(U)}}{\partial x}=\mathbf{S}^x(\mathbf{U}),
\end{equation}
where the solution vector $\mathbf{U}=(\rho,m_1,m_2,E_{11},E_{12},E_{22})^\top$,
 $\rho$, $\mathbf{m}=(m_1,m_2)^\top$, and $E_{ij}$
 denote the density, the momentum vector,  and  the component of the symmetric energy tensor
 $\mathbf{E}=(E_{ij})$,
 respectively, $1\le i,j \le 2$.
 The flux $\mathbf{F(U)}$ and the source term $\mathbf{S}^x(\mathbf{U})$ are respectively given by
\begin{align}
&\mathbf{F(U)}=\left(\rho u_1,\rho u_1^2+p_{11},\rho u_1u_2+p_{12},(E_{11}+p_{11})u_1,E_{12}u_1+\frac{1}{2}(p_{11}u_2+p_{12}u_1),E_{22}u_1+p_{12}u_2\right)^\top,
\label{FU} \\
&\mathbf{S}^x(\mathbf{U})=\left(0,-\frac{1}{2}\rho\partial_xW,0,-\frac{1}{2}\rho u_1\partial_xW,-\frac{1}{4}\rho u_2\partial_xW,0\right)^\top, \notag
\end{align}
where  $\mathbf{u}=(u_1,u_2)^\top = {\bf m}/\rho$ denotes the velocity vector,
and the function $W(x)$ is a given potential, which may denote the electron quiver energy in laser light (see e.g. \cite{morreeuw2006electron,sangam2007anisotropic}).
 The system \eqref{ten-moment2} is closed by the following ``equation of state''
\begin{equation}\label{EOS}
\mathbf{p}=(p_{ij})_{2\times 2}=2\mathbf{E}-\rho\mathbf{u}\otimes\mathbf{u},
\end{equation}
where  the symbol $\otimes$  denotes  the tensor product,
 and its solution  $\mathbf{U}$ should stays in the physically admissible state set
\begin{equation}
\mathcal{G}:=\left\{\mathbf{U}\in\mathbb{R}^6: \rho>0,~ \mathbf{x}^\top\mathbf{p}\mathbf{x}>0 \quad \forall\mathbf{x}\in\mathbb{R}^2 \setminus \{\mathbf{0}\} \right\},
\end{equation}
which means that $\rho>0$, $p_{11}>0$ and $\det(\mathbf{p}):=p_{11}p_{22}-p_{12}^2>0$.

Rewrite \eqref{ten-moment2} into the following quasi-linear form
\begin{subequations}
\begin{align}
  &\frac{\partial\rho}{\partial t}+u_1\frac{\partial\rho}{\partial x}+\rho\frac{\partial u_1}{\partial x}=0, \label{eq:rho}\\
  &\frac{\partial u_1}{\partial t}+u_1\frac{\partial u_1}{\partial x}+\frac{1}{\rho}\frac{\partial p_{11}}{\partial x}=-\frac{1}{2}W_x, \label{eq:u1}\\
  &\frac{\partial u_2}{\partial t}+u_1\frac{\partial u_2}{\partial x}+\frac{1}{\rho}\frac{\partial p_{12}}{\partial x}=0, \label{eq:u2}\\
  &\frac{\partial p_{11}}{\partial t}+3p_{11}\frac{\partial u_1}{\partial x}+u_1\frac{\partial p_{11}}{\partial x}=0, \label{eq:p11}\\
  &\frac{\partial p_{12}}{\partial t}+2p_{12}\frac{\partial u_1}{\partial x}+p_{11}\frac{\partial u_2}{\partial x}+u_1\frac{\partial p_{12}}{\partial x}=0, \label{eq:p12}\\
  &\frac{\partial p_{22}}{\partial t}+p_{22}\frac{\partial u_1}{\partial x}+2p_{12}\frac{\partial u_2}{\partial x}+u_1\frac{\partial p_{22}}{\partial x}=0. \label{eq:p22}
\end{align}
\label{quasilinear_form}
\end{subequations}
It is obvious that \eqref{eq:p11} and \eqref{eq:u1} imply that
\begin{align}
 &\frac{\partial u_1}{\partial x}=-\frac{1}{3p_{11}}\frac{\mathcal{D}p_{11}}{\mathcal{D}t}, \label{u1_x} \\
 &\frac{\partial p_{11}}{\partial x}=-\rho\left(\frac{\mathcal{D}u_1}{\mathcal{D}t}+\frac{1}{2}W_x\right), \label{p11_x}
\end{align}
respectively, where $\frac{\mathcal{D}}{\mathcal{D}t}:=\frac{\partial}{\partial t}+u_1\frac{\partial}{\partial x}$ is the material derivative. Then it follows that
\begin{equation}
\begin{split}
&\frac{\partial u_1}{\partial t}=\frac{u_1}{3p_{11}}\frac{\mathcal{D}p_{11}}{\mathcal{D}t}+\frac{\mathcal{D}u_1}{\mathcal{D}t}, \\
&\frac{\partial p_{11}}{\partial t}=\frac{\mathcal{D}p_{11}}{\mathcal{D}t}+\rho u_1\frac{\mathcal{D}u_1}{\mathcal{D}t}+\frac{1}{2}\rho u_1W_x.
\end{split}
\label{u1_p11_t}
\end{equation}


Besides, \eqref{eq:p12} and \eqref{eq:u2} imply that
\begin{align}
  &\frac{\partial u_2}{\partial x}=-\frac{1}{p_{11}}\left(\frac{\mathcal{D}p_{12}}{\mathcal{D}t}-\frac{2p_{12}}{3p_{11}}\frac{\mathcal{D}p_{11}}{\mathcal{D}t}\right), \label{u2_x} \\
  &\frac{\partial p_{12}}{\partial x}=-\rho\frac{\mathcal{D}u_2}{\mathcal{D}t}, \label{p12_x}
\end{align}
respectively, where \eqref{u1_x} has been used in \eqref{u2_x}. It follows that
\begin{equation}
\begin{split}
&\frac{\partial u_2}{\partial t}=\frac{\mathcal{D} u_2}{\mathcal{D}t}+\frac{u_1}{p_{11}}\left(\frac{\mathcal{D}p_{12}}{\mathcal{D}t}-\frac{2p_{12}}{3p_{11}}\frac{\mathcal{D}p_{11}}{\mathcal{D}t}\right), \\
&\frac{\partial p_{12}}{\partial t}=\frac{\mathcal{D}p_{12}}{\mathcal{D}t}+\rho u_1\frac{\mathcal{D}u_2}{\mathcal{D}t}.
\end{split}
\label{u2_p12_t}
\end{equation}
Moreover, \eqref{eq:p22} implies that
\begin{equation}
\frac{\mathcal{D}p_{22}}{\mathcal{D}t}=\frac{p_{11}p_{22}-4p_{12}^2}{3p_{11}^2}\frac{\mathcal{D}p_{11}}{\mathcal{D}t}+\frac{2p_{12}}{p_{11}}\frac{\mathcal{D}p_{12}}{\mathcal{D}t}, \label{Dp22_t}
\end{equation}
where \eqref{u1_x} and \eqref{u2_x} have been used.


\subsection{Eigenvalues,
eigenvectors and generalized Riemann invariants}
The eigenvalues of the system \eqref{ten-moment2} or \eqref{quasilinear_form}
are given (see e.g. \cite{brown1995numerical}) by
\[
\lambda_1=u_1-c, ~ \lambda_2=u_1-\frac{c}{\sqrt{3}}, ~ \lambda_3=\lambda_4=u_1, ~ \lambda_5=u_1+\frac{c}{\sqrt{3}}, ~ \lambda_6=u_1+c,
\]
where $c:=\sqrt{ {3p_{11}}/{\rho}}$ denotes the sound speed, and the  linearly independent six associated right eigenvectors for  \eqref{quasilinear_form} 
 may be chosen  (see e.g. \cite{sen2018entropy})  as follows
\begin{align*}
 &
 \mathbf{\tilde{r}}_1=\left(\rho p_{11},-cp_{11},-cp_{12},3p_{11}^2,3p_{11}p_{12},p_{11}p_{22}+2p_{12}^2\right)^\top,
\\
&
\mathbf{\tilde{r}}_2=\left(0,0,-\frac{c}{\sqrt{3}},0,p_{11},2p_{12}\right)^\top,
\\
&
\mathbf{\tilde{r}}_3=(0,0,0,0,0,1)^\top,
\ \
 \mathbf{\tilde{r}}_4=(1,0,0,0,0,0)^\top,
\\
&
\mathbf{\tilde{r}}_5=\left(0,0,\frac{c}{\sqrt{3}},0,p_{11},2p_{12}\right)^\top,
\\
&
\mathbf{\tilde{r}}_6=\left(\rho p_{11},cp_{11},cp_{12},3p_{11}^2,3p_{11}p_{12},p_{11}p_{22}+2p_{12}^2\right)^\top.
\end{align*}
Denote the primitive variables vector by $\mathbf{V}:=(\rho,u_1,u_2,p_{11},p_{12},p_{22})^\top$. The associated right eigenvector matrix $\mathbf{R(U)}=\frac{\partial{\mathbf U}}{\partial {\mathbf V}}
\cdot(\mathbf{\tilde{r}}_1,\cdots,\mathbf{\tilde{r}}_6)$ for \eqref{ten-moment2} is given in \ref{R(U)}.

 By simple calculations, one can find that for any physically admissible $\mathbf{V}$,  $\nabla_{\mathbf{V}}\lambda_i(\mathbf{V})\cdot\mathbf{\tilde{r}}_i(\mathbf{V})\neq0$ for $i=1,6$  and $\nabla_{\mathbf{V}}\lambda_i(\mathbf{V})\cdot\mathbf{\tilde{r}}_i(\mathbf{V})=0$ for $i=2,3,4,5$. Hence, only the first and the sixth characteristic fields are genuinely nonlinear, and the other characteristic fields are linearly degenerate.
Possible wave patterns associated with those characteristic fields are
\begin{align*}
 &  \text{for }  \lambda_1=u_1-c : \quad \text{1-shock or 1-rarefaction wave},
 \\
&
 \text{for }
 \lambda_2=u_1-\frac{c}{\sqrt{3}} : \quad \text{2-shear wave},
 \\
&
 \text{for }
 \lambda_3=\lambda_4=u_1 : \quad \text{3, 4-contact wave},
 \\
&
 \text{for }
 \lambda_5=u_1+\frac{c}{\sqrt{3}} : \quad \text{5-shear wave},
 \\
&
 \text{for }
\lambda_6=u_1+c : \quad \text{6-shock or 6-rarefaction wave}.
\end{align*}
Note that the characteristic fields of the 1D ten-moment equations are similar to those of the 1D shear shallow water model \cite{nkonga2022exact}. According to the following relations
\[
\frac{\text{d}\rho}{\tilde{r}_i^{(1)}}=\frac{\text{d}u_1}{\tilde{r}_i^{(2)}}=\frac{\text{d}u_2}{\tilde{r}_i^{(3)}}=\frac{\text{d}p_{11}}{\tilde{r}_i^{(4)}}
=\frac{\text{d}p_{12}}{\tilde{r}_i^{(5)}}=\frac{\text{d}p_{22}}{\tilde{r}_i^{(6)}}, ~ i=1,2,\cdots,6,
\]
 where $\tilde{r}_i^{(k)}$ denotes the $k$-th component of $\mathbf{\tilde{r}}_i$, $k=1,2\cdots,6$, one can derive
 the following generalized Riemann invariants for six characteristic fields
\begin{align*}
 &
 \text{for } \lambda_1=u_1-c:
 \quad \frac{p_{11}}{\rho^3}, \frac{p_{12}}{\rho^3}, u_1+c, u_2+\frac{\sqrt{3}p_{12}}{\sqrt{\rho p_{11}}}, \frac{\det(\mathbf{p})}{\rho^4},
\\
&
 \text{for }  \lambda_2=u_1-\frac{c}{\sqrt{3}}:
 \quad \rho, u_1, p_{11}, u_2+\frac{p_{12}}{\sqrt{\rho p_{11}}}, \det(\mathbf{p}),
 \\
&
 \text{for }
 \lambda_3=u_1: \quad \rho, u_1, u_2, p_{11}, p_{12},
 \\
&
 \text{for } \lambda_4=u_1:\quad  u_1, u_2, p_{11}, p_{12}, p_{22},
 \\
&
 \text{for } \lambda_5=u_1+\frac{c}{\sqrt{3}}:
 \quad \rho, u_1, p_{11}, u_2-\frac{p_{12}}{\sqrt{\rho p_{11}}}, \det(\mathbf{p}),
 \\
&
 \text{for } \lambda_6=u_1+c:\quad \frac{p_{11}}{\rho^3}, \frac{p_{12}}{\rho^3}, u_1-c, u_2-\frac{\sqrt{3}p_{12}}{\sqrt{\rho p_{11}}}, \frac{\det(\mathbf{p})}{\rho^4},
\end{align*}
which are constant across corresponding linearly degenerate wave or rarefaction wave.

\section{Numerical schemes}\label{numerical schemes}
This section introduces the 1D direct Eulerian GRP scheme and its extension to the 2D case by the Strang splitting method.

\subsection{Outline of   1D GRP scheme}\label{sect:3.1}
For the sake of simplicity, the space domain is divided into a uniform mesh $\{x_{j+\frac{1}2}, j\in\mathbb Z\}$ with  $\Delta x:=x_{j+\frac{1}{2}}-x_{j-\frac{1}{2}}$,  denote the $j$th cell by $I_j:=[x_{j-\frac{1}{2}},x_{j+\frac{1}{2}}]$.
Assume that  $\Delta t:=t_{n+1}-t_n$, and at the time $t=t_n$, the approximate solution
is piecewisely linear with the slope $\bm{\sigma}_j^n$, $n\geq 0$, and reconstructed as
\begin{equation}\label{piecewise linear initial data}
\mathbf{V}(x,t_n)=\mathbf{V}_j^n+\bm{\sigma}_j^n(x-x_j) \quad \forall x\in I_j,
\end{equation}
where $\mathbf{V}_j^n$ approximates the cell average of $\mathbf{V}(x,t_n)$ over the cell $I_j$. Hence the second-order Godunov-type scheme to solve \eqref{ten-moment2} is
\begin{equation}\label{1d_scheme}
\mathbf{U}_j^{n+1}=\mathbf{U}_j^{n}-\frac{\Delta t}{\Delta x}\left(\mathbf{F}_{j+\frac{1}{2}}^{n+\frac{1}{2}}-\mathbf{F}_{j-\frac{1}{2}}^{n+\frac{1}{2}}\right)+\frac{\Delta t}{2}\left(\mathbf{S}_{j+\frac{1}{2}}^{x,n+\frac{1}{2}}+\mathbf{S}_{j-\frac{1}{2}}^{x,n+\frac{1}{2}}\right),
\end{equation}
where $\mathbf{F}_{j+\frac{1}{2}}^{n+\frac{1}{2}}=\mathbf{F}\big(\mathbf{U}(\mathbf{V}_{j+\frac{1}{2}}^{n+\frac{1}{2}})\big)$ and $\mathbf{S}_{j+\frac{1}{2}}^{x,n+\frac{1}{2}}=\mathbf{S}^x(\mathbf{U}\big(\mathbf{V}_{j+\frac{1}{2}}^{n+\frac{1}{2}})\big)$,
 $\mathbf{V}_{j+\frac{1}{2}}^{n+\frac{1}{2}}$ denotes a second-order accurate approximation of $\mathbf{V}(x_{j+\frac{1}{2}},t_n+\frac{1}{2}\Delta t)$  and is analytically derived by resolving the local GRP at each point $(x_{j+\frac{1}{2}},t_n)$, i.e.,
\begin{equation}
\begin{cases}
\eqref{ten-moment2}, \quad t-t_n>0, \\
\mathbf{V}(x,t_n)=\begin{cases}
                    \mathbf{V}_j^n+\bm{\sigma}_j^n(x-x_j), \quad x<x_{j+\frac{1}{2}}, \\
                    \mathbf{V}_{j+1}^n+\bm{\sigma}_{j+1}^n(x-x_{j+1}), \quad x>x_{j+\frac{1}{2}}.
                    \end{cases}
\end{cases}
\label{GRP1}
\end{equation}
More specifically, $\mathbf{V}_{j+\frac{1}{2}}^{n+\frac{1}{2}}$ is usually calculated by
\begin{equation}\label{U_j+1/2^n+1/2}
\mathbf{V}_{j+\frac{1}{2}}^{n+\frac{1}{2}}=\mathbf{V}_{j+\frac{1}{2}}^{n}+\frac{\Delta t}{2}\left(\frac{\partial\mathbf{V}}{\partial t}\right)_{j+\frac{1}{2}}^n,
\end{equation}
where $\left(\frac{\partial\mathbf{V}}{\partial t}\right)_{j+\frac{1}{2}}^n$ is obtained by resolving the problem \eqref{GRP1}, according to those theorems in Section \ref{GRP solver}, and $\mathbf{V}_{j+\frac{1}{2}}^{n}:=\bm{\omega}\left(0;\mathbf{V}_{j+\frac{1}{2},L}^{n},\mathbf{V}_{j+\frac{1}{2},R}^{n}\right)$, here $\bm{\omega}\left(\frac{x-x_j}{t-t_n};\mathbf{V}_{j+\frac{1}{2},L}^{n},\mathbf{V}_{j+\frac{1}{2},R}^{n}\right)$ is the exact solution of the associated Riemann problem for \eqref{ten-moment2} centered at $(x_{j+\frac{1}{2}},t_n)$, i.e., the Cauchy problem
\begin{equation}\label{local RP1}
\begin{cases}
\displaystyle
\frac{\partial\mathbf{U(V)}}{\partial t}+\frac{\partial\mathbf{F(U(V))}}{\partial x}=0, \quad t-t_n>0 \\
\mathbf{V}(x,t_n)=\begin{cases}
                    \mathbf{V}_{j+\frac{1}{2},L}^{n}, \quad x<x_{j+\frac{1}{2}}, \\
                    \mathbf{V}_{j+\frac{1}{2},R}^{n}, \quad x>x_{j+\frac{1}{2}},
                  \end{cases}
\end{cases}
\end{equation}
where $\mathbf{V}_{j+\frac{1}{2},L}^{n}=\mathbf{V}_{j}^{n}+\frac{\Delta x}{2}\bm{\sigma}_j^n$ and $\mathbf{V}_{j+\frac{1}{2},R}^{n}=\mathbf{V}_{j+1}^{n}-\frac{\Delta x}{2}\bm{\sigma}_{j+1}^n$ are the left and right limiting values of $\mathbf{V}(x,t_n)$ at $x_{j+\frac{1}{2}}$. The exact solution of Riemann problem \eqref{local RP1} is given in Section \ref{exact RP}, while the approximate slope $\bm{\sigma}_j^{n}$ is evolved by some slope limiter.

After obtaining $\mathbf{U}_j^{n+1}$ by the scheme \eqref{1d_scheme},
 applying some slope limiter to get the slope of $\mathbf{U}$,
 denoted by $\widehat{\bm{\sigma}}_j^{n+1}$,
 and then the slope $\bm{\sigma}_j^{n+1}$ of  {$\mathbf{V}$} can be obtained by the chain rule.
The minmod type and van Leer limiters are utilized here.
With the minmod type limiter, $\widehat{\bm{\sigma}}_j^{n+1}$ is calculated componentwisely by
\begin{equation}\label{minmod limiter}
\widehat{\bm{\sigma}}_j^{n+1}=\mathbf{R}_j\cdot\text{minmod}\left(\frac{\theta}{\Delta x}\mathbf{R}_j^{-1}\left(\mathbf{U}_j^{n+1}-\mathbf{U}_{j-1}^{n+1}\right),\frac{1}{\Delta x}\mathbf{R}_j^{-1}\left(\mathbf{U}_{j+\frac{1}{2}}^{n+1,-}-\mathbf{U}_{j-\frac{1}{2}}^{n+1,-}\right),
\frac{\theta}{\Delta x}\mathbf{R}_j^{-1}\left(\mathbf{U}_{j+1}^{n+1}-\mathbf{U}_{j}^{n+1}\right)\right),
\end{equation}
where $\mathbf{R}_j:=\mathbf{R}(\mathbf{U}_j^{n+1})$, the parameter $\theta\in[1,2)$, and $\mathbf{U}_{j+\frac{1}{2}}^{n+1,-}:=\mathbf{U}(\mathbf{V}_{j+\frac{1}{2}}^{n+1,-})$ with $\mathbf{V}_{j+\frac{1}{2}}^{n+1,-}:=\mathbf{V}_{j+\frac{1}{2}}^n+\Delta t\left(\frac{\partial\mathbf{V}}{\partial t}\right)_{j+\frac{1}{2}}^n$. $\mathbf{R(U)}$ is the right eigenvector matrix of the Jacobian matrix $\frac{\partial\mathbf{F(U)}}{\partial\mathbf{U}}$ and given in \ref{R(U)}.
With the van Leer limiter, $\widehat{\bm{\sigma}}_j^{n+1}$ is calculated componentwise by
\begin{equation}\label{vanLeer limiter}
\widehat{\bm{\sigma}}_j^{n+1}=\mathbf{R}_j\cdot\text{vanLeer}\left(\mathbf{R}_j^{-1}\mathbf{U}_{j-1}^{n+1},\mathbf{R}_j^{-1}\mathbf{U}_{j}^{n+1},\mathbf{R}_j^{-1}\mathbf{U}_{j+1}^{n+1}\right),
\end{equation}
where the van Leer limiter function is defined as \cite{van1974towards,zou2021understand}
\[
\text{vanLeer}(a_L,a_M,a_R)=\begin{cases}
                        0,  &\text{if} ~ a_L=a_R ~ \text{or} ~ f\leq0 ~ \text{or} ~ f\geq1, \\
                        4f(1-f)s, & \text{if} ~ a_L\neq a_R ~ \text{and} ~ 0<f<1,
                      \end{cases}
\]
where $f=\frac{\Delta_{ML}}{\Delta_{RL}}$, $s=\frac{\Delta_{RL}}{2\Delta x}$, $\Delta_{ML}=a_M-a_L$ and $\Delta_{RL}=a_R-a_L$.

In summary, the  second-order direct Eulerian GRP scheme for the 1D ten-moment equations \eqref{ten-moment2} is implemented in the following four steps:

\textbf{Step 1:} Give the piecewise initial data \eqref{piecewise linear initial data}, and solve the local Riemann problem \eqref{local RP1} to get the solution $\mathbf{V}_{j+\frac{1}{2}}^n$ by using the exact Riemann solver stated in Section \ref{exact RP}.

\textbf{Step 2:} Calculate $\left(\frac{\partial\mathbf{V}}{\partial t}\right)_{j+\frac{1}{2}}^n$ by those theorems in Section \ref{GRP solver}, and $\mathbf{V}_{j+\frac{1}{2}}^{n+\frac{1}{2}}$ in \eqref{U_j+1/2^n+1/2}, then the numerical flux $\mathbf{F}_{j+\frac{1}{2}}^{n+\frac{1}{2}}$ and the source term $\mathbf{S}_{j+\frac{1}{2}}^{n+\frac{1}{2}}$ can be obtained.

\textbf{Step 3:} Evaluate the cell averages $\mathbf{U}_j^{n+1}$ by the scheme \eqref{1d_scheme}.

\textbf{Step 4:} Update the slope $\widehat{\bm{\sigma}}_j^{n+1}$ of the conservative variables by the slope limiter \eqref{minmod limiter} or \eqref{vanLeer limiter}, and calculate the slope $\bm{\sigma}_j^{n+1}$ of the primitive variables by the chain rule, then go to Step 1 by replacing $n$ with $n+1$.

\subsection{2D extension of GRP scheme}
For simplicity, this work is limited to a 2D uniform Cartesian grid and uses the Strang splitting method
to extend the above GRP scheme to the 2D case.
Consider the 2D ten-moment equations
\begin{equation}\label{2d ten moment equations}
\frac{\partial\mathbf{U}}{\partial t}+\frac{\partial\mathbf{F(U)}}{\partial x}+\frac{\partial\mathbf{G(U)}}{\partial y}=\mathbf{S}^x(\mathbf{U})+\mathbf{S}^y(\mathbf{U}),
\end{equation}
where $\mathbf{U}$,   $\mathbf{F(U)}$, and $\mathbf{S}^x(\mathbf{U})$ are the same as those in Section \ref{sect:2.1}, and
the  flux $\mathbf{G(U)}$ and  the source term $\mathbf{S}^y(\mathbf{U})$ are defined by
\begin{align*}
&\mathbf{G(U)}=\left(\rho u_2,\rho u_1u_2+p_{12},\rho u_2^2+p_{22},E_{11}u_2+p_{12}u_1,E_{12}u_2+\frac{1}{2}(p_{12}u_2+p_{22}u_1),(E_{22}+p_{22})u_2\right)^\top,
\\
&\mathbf{S}^y(\mathbf{U})=\left(0,0,-\frac{1}{2}\rho\partial_yW,0,-\frac{1}{4}\rho u_1\partial_yW,-\frac{1}{2}\rho u_2\partial_yW\right)^\top.
\end{align*}
The 2D ten-moment equations \eqref{2d ten moment equations} are split into two 1D subsystems
\[
\frac{\partial\mathbf{U}}{\partial t}+\frac{\partial\mathbf{F(U)}}{\partial x}=\mathbf{S}^x(\mathbf{U}) \quad \text{and} \quad
\frac{\partial\mathbf{U}}{\partial t}+\frac{\partial\mathbf{G(U)}}{\partial x}=\mathbf{S}^y(\mathbf{U}).
\]
Hence it suffices to consider the derivation of the GRP scheme for
the subsystem in the $x$-direction  in Section \ref{sect:3.1}.
If denoting the 1D GRP evolution operators  for the above two subsystems by $\mathcal{L}_x(\Delta t)$ and $\mathcal{L}_y(\Delta t)$, respectively,  for one time step $\Delta t$, then
by using the Strang splitting method, 
the simple 2D  GRP scheme can be given by
\[
\mathbf{U}^{n+1}=\mathcal{L}_x\left(\frac{\Delta t}{2}\right)\mathcal{L}_y(\Delta t)\mathcal{L}_x\left(\frac{\Delta t}{2}\right)\mathbf{U}^n.
\]

\section{Exact Riemann solver}\label{exact RP}
This section gives the exact Riemann solver for the  associated Riemann problem
\begin{equation}\label{local RP2}
\begin{cases}
\displaystyle
\frac{\partial\mathbf{U(V)}}{\partial t}+\frac{\partial\mathbf{F(U(V))}}{\partial x}=0, \quad t>0, \\
\mathbf{V}(x,0)=\begin{cases}
                    \mathbf{V}_L, \quad x<0, \\
                    \mathbf{V}_R, \quad x>0.
                \end{cases}
\end{cases}
\end{equation}
As an example, a special  local wave configuration shown in  Figure \ref{localwave_RP} for the associated Riemann problem is only considered here, where  there is a 1-rarefaction wave, a 2-shear wave, a 3,4-contact discontinuity, a 5-shear wave and a 6-shock wave. Other  local wave configurations may be similarly treated.
One can see from Figure \ref{localwave_RP}  that there are five solution states between the given left and right states  $\mathbf{V}_L$ and  $\mathbf{V}_R$ to be determined, that is, the state inside the 1-rarefaction wave denoted by $\mathbf{V}_{\text{Lfan}}$ and  four  solution states between the $i$-wave and the $(i+1)$-wave ($1\leq i<6$) denoted by
\begin{equation*}
\begin{cases}
\mathbf{V}_{\ast L}=(\rho_{\ast L},u_{1,\ast},u_{2,\ast L},p_{11,\ast},p_{12,\ast L},p_{22,\ast L})^\top, \\
\mathbf{V}_{\ast\ast L}=(\rho_{\ast L},u_{1,\ast},u_{2,\ast\ast},p_{11,\ast},p_{12,\ast\ast},p_{22,\ast\ast L})^\top, \\
\mathbf{V}_{\ast\ast R}=(\rho_{\ast R},u_{1,\ast},u_{2,\ast\ast},p_{11,\ast},p_{12,\ast\ast},p_{22,\ast\ast R})^\top, \\
\mathbf{V}_{\ast R}=(\rho_{\ast R},u_{1,\ast},u_{2,\ast R},p_{11,\ast},p_{12,\ast R},p_{22,\ast R})^\top.
\end{cases}
\end{equation*}
To determine them,  the generalized Riemann invariants  will be  used for the rarefaction wave, the shear waves and the contact discontinuity, while the  Rankine-Hugoniot jump conditions will be used for the shock wave.
 For the sake of convenience, the main steps of   the exact Riemann solver for
\eqref{local RP2}  is shown in Figure \ref{main results of exact RP}.

\begin{figure}[!htbp]
\centering
\includegraphics[width=8cm]{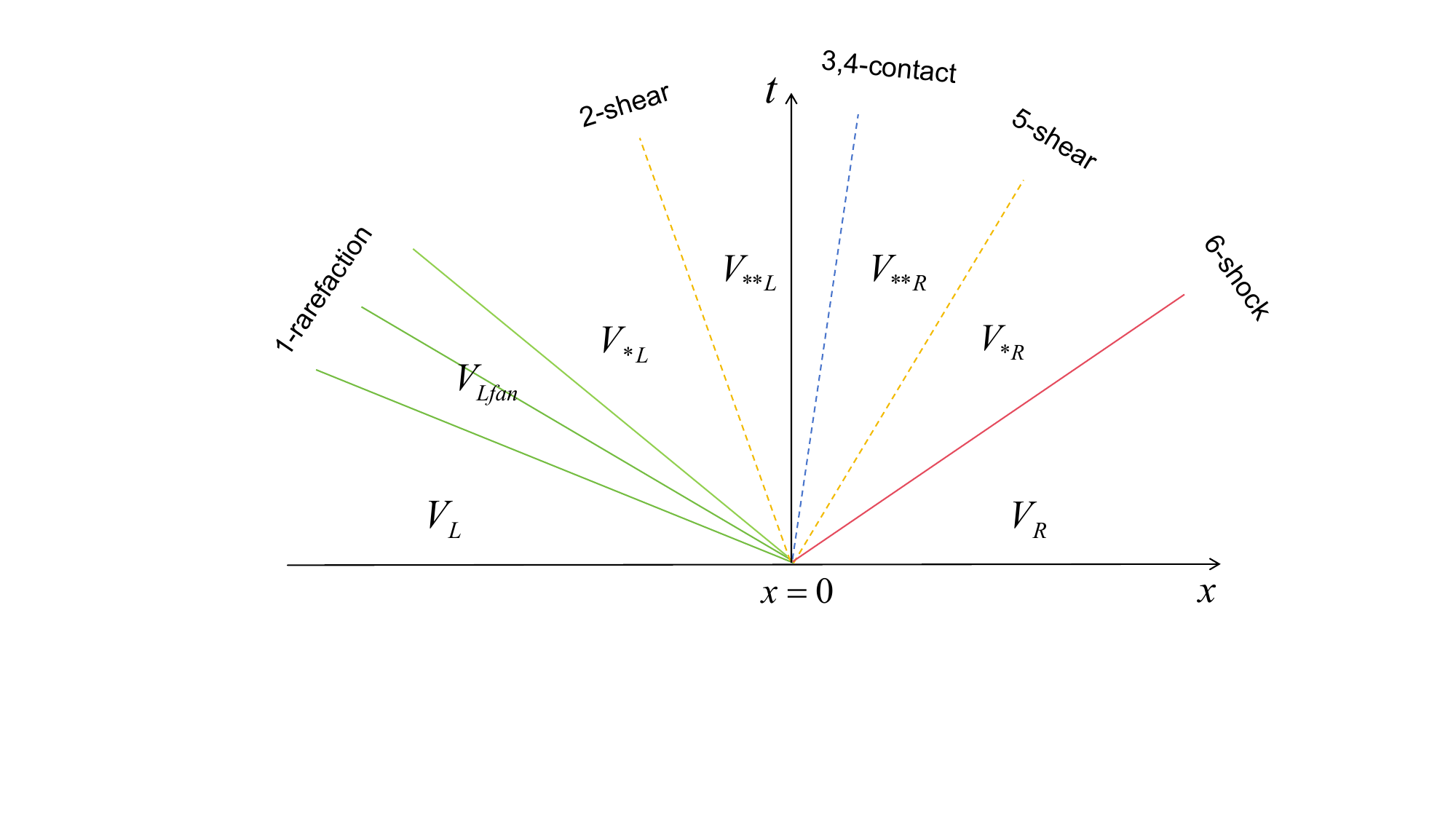}
\caption{The schematic description of a local wave configuration of the associated Riemann problem \eqref{local RP2}.}
\label{localwave_RP}
\end{figure}

\tikzstyle{format}=[rectangle,draw,very thick,fill=white] 
\tikzstyle{arrow}=[thin,->,>=stealth]

\begin{figure}[!htbp]
\begin{tikzpicture}[node distance=5cm, text width=5.3cm]
  \node[format] (start) {Compute $u_{1,\ast}$ and $p_{11,\ast}$, see Theorem \ref{theorem:u1s_p11s}
   and $\mathbf{V}_{\text{Lfan}}$, see \eqref{VLfan}.};


  \node[format] (input) [right of=start, yshift=1.5cm] {Compute $\mathbf{V}_{\ast L}$, see \eqref{V_astL}.};

  \node[format] (process)[right of=start, yshift=-1.5cm] {Compute  $\mathbf{V}_{\ast R}$, see \eqref{rho_astR}, \eqref{eq:u2_p12_astR} and \eqref{p22_astR}.};

  \node[format] (end) [right of=process, yshift=1.5cm] {Compute  $\mathbf{V}_{\ast\ast K}$ ($K=L,R$), see \eqref{u2_p12_2ast} and \eqref{p22_2astK}.};

  \draw [arrow] (start) -- (input);
  \draw [arrow] (start) -- (process);
  \draw [arrow] (input) -- (end);
  \draw [arrow] (process) -- (end);
\end{tikzpicture}
\centering
\caption{The main steps of the exact Riemann solver in Section \ref{exact RP}.}
\label{main results of exact RP}
\end{figure}
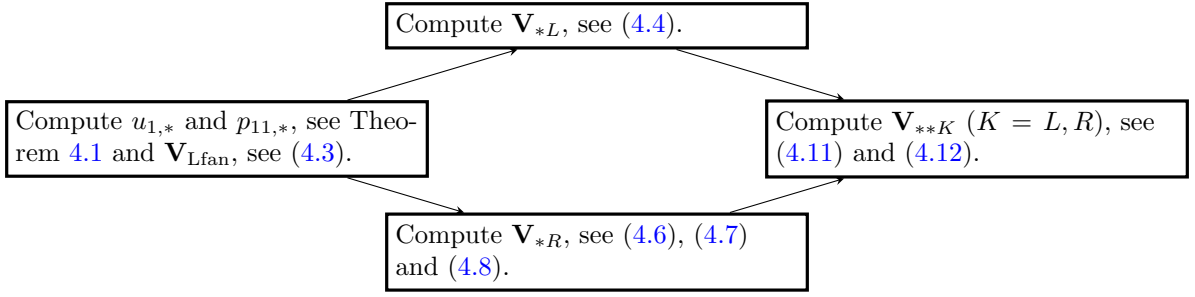

\subsection{Computations of $u_{1,\ast}$ and $p_{11,\ast}$}

Analogous to the derivation for the Euler equations  in \cite[Proposition 4.1]{toro2013riemann}, one can prove the following results for the 1D ten-moment equations.

\begin{theorem}[Computing  $p_{11,\ast}$ and $u_{1,\ast}$]\label{theorem:u1s_p11s}
The  pressure component $p_{11,\ast}$ of the Riemann problem \eqref{local RP2} is given by the root of the algebraic equation
\begin{equation}\label{p11_algebraic_equation}
f(p_{11},\mathbf{V}_L,\mathbf{V}_R):=f_L(p_{11},\mathbf{V}_L)+f_R(p_{11},\mathbf{V}_R)+u_{1,R}-u_{1,L}=0,
\end{equation}
where the function $f_L$ is given by
\begin{equation*}
f_L(p_{11},\mathbf{V}_L)=\begin{cases}
                           (p_{11}-p_{11,L})\left[\frac{1}{\rho_L(2p_{11}+p_{11,L})}\right]^{\frac{1}{2}}, & \mbox{if } p_{11}>p_{11,L} ~ \text{(1-shock)}, \\
                           c_L\left[\left(\frac{p_{11}}{p_{11,L}}\right)^{\frac{1}{3}}-1\right], & \mbox{if } p_{11}\leq p_{11,L} ~ \text{(1-rarefaction)},
                         \end{cases}
\end{equation*}
and the function $f_R$ is given by
\begin{equation*}
f_R(p_{11},\mathbf{V}_R)=\begin{cases}
                           (p_{11}-p_{11,R})\left[\frac{1}{\rho_R(2p_{11}+p_{11,R})}\right]^{\frac{1}{2}}, & \mbox{if } p_{11}>p_{11,R} ~ \text{(6-shock)}, \\
                           c_R\left[\left(\frac{p_{11}}{p_{11,R}}\right)^{\frac{1}{3}}-1\right], & \mbox{if } p_{11}\leq p_{11,R} ~ \text{(6-rarefaction)}.
                         \end{cases}
\end{equation*}
Moreover,  the velocity $u_{1,\ast}$ may be given by
\begin{equation*}
u_{1,\ast}=\frac{1}{2}(u_{1,L}+u_{1,R})+\frac{1}{2}[f_R(p_{11,\ast},\mathbf{V}_R)-f_L(p_{11,\ast},\mathbf{V}_L)].
\end{equation*}
\end{theorem}

\begin{proof}
The proof is similar to that of the Euler equations, see  \cite[Subsection 4.2]{toro2013riemann},  with replacing the adiabatic index $\gamma$ with 3.
\end{proof}

To find the root of the nonlinear equation \eqref{p11_algebraic_equation},
 the Newton-Raphson iterative method is applied, where the initial guess value may be  obtained in a adaptive way, similar to that for the Euler equations  \cite[Section 9.5 and Figure 9.4]{toro2013riemann}.

\subsection{Computations of $\mathbf{V}_{\text{Lfan}}$}

The 1-rarefaction wave, which is identified by the condition $p_{11,\ast}\leq p_{11,L}$, is enclosed by the head and the tail,  whose characteristic speeds are given respectively by
\[
S_{\text{HL}}=u_{1,L}-c_L, \quad S_{\text{TL}}=u_{1,\ast}-c_{\ast L},
\]
where $c_{\ast L}=c_L\left(\frac{p_{11,\ast}}{p_{11,L}}\right)^{\frac{1}{3}}$. The solution $\mathbf{V}_{\text{Lfan}}$ inside the 1-rarefaction fan is easily obtained by considering the characteristic ray through the origin $(0,0)$ and a general point $(x,t)$ inside the fan. The slope of the characteristic line is
\[
\frac{\text{d}x}{\text{d}t}=\frac{x}{t}=u_1-c.
\]
Using the generalized Riemann invariants associated with the 1-rarefaction wave yields
\begin{equation}
\mathbf{V}_{\text{Lfan}}\begin{cases}
\rho=\frac{\rho_L}{2}\left[1+\frac{1}{c_L}\left(u_{1,L}-\frac{x}{t}\right)\right], \\
u_1=\frac{1}{2}\left(c_L+u_{1,L}+\frac{x}{t}\right), \\
u_2=u_{2,L}+\frac{\sqrt{3}p_{12,L}}{\sqrt{\rho_Lp_{11,L}}}-\frac{\sqrt{3}p_{12}}{\sqrt{\rho p_{11}}}, \\
p_{11}=\frac{p_{11,L}}{8}\left[1+\frac{1}{c_L}\left(u_{1,L}-\frac{x}{t}\right)\right]^3, \\
p_{12}=p_{12,L}\frac{\rho^3}{\rho_L^3}, \\
p_{22}=\left[\frac{\det(\mathbf{p}_L)}{\rho_L^4}\rho^4+p_{12}^2\right]/p_{11}.
\end{cases}
\label{VLfan}
\end{equation}

\begin{remark}
For the 6-rarefaction wave, which is identified by the condition $p_{11,\ast}\leq p_{11,R}$ and enclosed by the characteristic  speeds given respectively by $S_{\text{HR}}=u_{1,R}+c_R$ and $S_{\text{TR}}=u_{1,\ast}+c_{\ast R}$ with $c_{\ast R}=c_R\left(\frac{p_{11,\ast}}{p_{11,R}}\right)^{\frac{1}{3}}$, the solution  $\mathbf{V}_{\text{Rfan}}$ is given by
\begin{equation*}
\mathbf{V}_{\text{Rfan}}\begin{cases}
\rho=\frac{\rho_R}{2}\left[1+\frac{1}{c_R}\left(\frac{x}{t}-u_{1,R}\right)\right], \\
u_1=\frac{1}{2}\left(u_{1,R}-c_R+\frac{x}{t}\right), \\
u_2=u_{2,R}-\frac{\sqrt{3}p_{12,R}}{\sqrt{\rho_Rp_{11,R}}}+\frac{\sqrt{3}p_{12}}{\sqrt{\rho p_{11}}}, \\
p_{11}=\frac{p_{11,R}}{8}\left[1+\frac{1}{c_R}\left(\frac{x}{t}-u_{1,R}\right)\right]^3, \\
p_{12}=p_{12,R}\frac{\rho^3}{\rho_R^3}, \\
p_{22}=\left[\frac{\det(\mathbf{p}_R)}{\rho_R^4}\rho^4+p_{12}^2\right]/p_{11}.
\end{cases}
\end{equation*}
\end{remark}

\subsection{Computing $\mathbf{V}_{\ast L}$}

For the state $\mathbf{V}_{\ast L}$, which is between the 1-rarefaction wave and the 2-shear wave, utilizing the generalized Riemann invariants across the 1-rarefaction wave can  obtain
\begin{align}
&\rho_{\ast L}=\rho_L\left(\frac{p_{11,\ast}}{p_{11,L}}\right)^{\frac{1}{3}}, \quad u_{2,\ast L}=u_{2,L}+\frac{\sqrt{3}p_{12,L}}{\sqrt{\rho_Lp_{11,L}}}-\frac{\sqrt{3}p_{12,\ast L}}{\sqrt{\rho_{\ast L}p_{11,\ast}}}, \notag \\
&p_{12,\ast L}=p_{12,L}\frac{\rho_{\ast L}^3}{\rho_L^3}, \quad p_{22,\ast L}=\left[\frac{\det(\mathbf{p}_L)}{\rho_L^4}\rho_{\ast L}^4+p_{12,\ast L}^2\right]/p_{11,\ast}. \label{V_astL}
\end{align}

\begin{remark}
For the state $\mathbf{V}_{\ast R}$, which is between the 6-rarefaction wave and the 5-shear wave, use of the generalized Riemann invariants across the 6-rarefaction wave may give
\begin{align*}
&\rho_{\ast R}=\rho_R\left(\frac{p_{11,\ast}}{p_{11,R}}\right)^{\frac{1}{3}}, \quad u_{2,\ast R}=u_{2,R}-\frac{\sqrt{3}p_{12,R}}{\sqrt{\rho_Rp_{11,R}}}+\frac{\sqrt{3}p_{12,\ast R}}{\sqrt{\rho_{\ast R}p_{11,\ast}}}, \\
&p_{12,\ast R}=p_{12,R}\frac{\rho_{\ast R}^3}{\rho_R^3}, \quad p_{22,\ast R}=\left[\frac{\det(\mathbf{p}_R)}{\rho_R^4}\rho_{\ast R}^4+p_{12,\ast R}^2\right]/p_{11,\ast}.
\end{align*}
\end{remark}

\subsection{Computing  $\mathbf{V}_{\ast R}$}

The solution state $\mathbf{V}_{\ast R}$  is between the 5-shear wave and the 6-shock wave, while
the 6-shock wave with the following speed
\begin{equation}\label{sigma_R}
\sigma_R=u_{1,R}+c_R\sqrt{\frac{2p_{11,\ast}}{3p_{11,R}}+\frac{1}{3}} {=u_{1,\ast}+\frac{1}{\rho_{\ast R}}\sqrt{(2p_{11,\ast}+p_{11,R})\rho_R}}
\end{equation}
is identified by the condition $p_{11,\ast}>p_{11,R}$.  {The derivation of \eqref{sigma_R} is similar to that of the shock speed for the Euler equations, see Section 3.1.3 in \cite{toro2013riemann}. Moreover, \eqref{sigma_R} implies that $u_{1,\ast}<\sigma_R$}.
By applying  the  Rankine-Hugoniot jump conditions for the variables $\rho$, $m_1$ and $E_{11}$, similar to the discussion for the Euler equations in \cite{toro2013riemann}, one can obtain
\begin{equation}\label{rho_astR}
\rho_{\ast R}=\rho_R\frac{2p_{11,\ast}+p_{11,R}}{p_{11,\ast}+2p_{11,R}}.
\end{equation}
The condition $p_{11,\ast}>p_{11,R}$ implies that $\rho_{\ast R}\in(\rho_R,2\rho_R)$.

Using the    Rankine-Hugoniot jump conditions for the variables $m_2$ and $E_{12}$ yields
\begin{equation}\label{eq:u2_p12_astR}
\mathbf{A}^R\left(\begin{array}{c}
              u_{2,\ast R} \\
              p_{12,\ast R}
            \end{array}\right)
=\left(\begin{array}{c}
   a_1^R \\
   a_2^R
 \end{array}\right),
\end{equation}
where
\begin{align*}
&\mathbf{A}^R:=\begin{pmatrix}
                \rho_{\ast R}(u_{1,\ast}-\sigma_R) & 1 \\
                E_{11,\ast R}-\frac{1}{2}\rho_{\ast R}u_{1,\ast}\sigma_R & u_{1,\ast}-\frac{1}{2}\sigma_R
              \end{pmatrix}, \\
&a_1^R:=\rho_R(u_{1,R}-\sigma_R)u_{2,R}+p_{12,R}, \\
&a_2^R:=E_{12,R}(u_{1,R}-\sigma_R)+\frac{1}{2}(p_{11,R}u_{2,R}+p_{12,R}u_{1,R}).
\end{align*}
Due to
\[
\det(\mathbf{A}^R)=\frac{p_{11,\ast}(2\rho_R-\rho_{\ast R})+p_{11,R}\rho_R}{2\rho_{\ast R}}>0,
\]
the solution of \eqref{eq:u2_p12_astR} exists and is unique.

By  the  Rankine-Hugoniot jump conditions for the variable $E_{22}$, one has
\[
E_{22,\ast R}u_{1,\ast}+p_{12,\ast R}u_{2,\ast R}-(E_{22,R}u_{1,R}+p_{12,R}u_{2,R})=\sigma_R(E_{22,\ast R}-E_{22,R}),
\]
which implies that
\[
E_{22,\ast R}=\frac{E_{22,R}u_{1,R}+p_{12,R}u_{2,R}-\sigma_RE_{22,R}-p_{12,\ast R}u_{2,\ast R}}{u_{1,\ast}-\sigma_R},
\]
and then using the ``equation of state'' yields
\begin{equation}\label{p22_astR}
p_{22,\ast R}=2E_{22,\ast R}-\rho_{\ast R}u_{2,\ast R}^2.
\end{equation}

\begin{remark}
For the 1-shock wave with the speed
\[
\sigma_L=u_{1,L}-c_L\sqrt{\frac{2p_{11,\ast}}{3p_{11,L}}+\frac{1}{3}}= {u_{1,\ast}-\frac{1}{\rho_{\ast L}}\sqrt{(2p_{11,\ast}+p_{11,L})\rho_L}},
\]
which is identified by the condition $p_{11,\ast}>p_{11,L}$, one can obtain
\[
\rho_{\ast L}=\rho_L\frac{2p_{11,\ast}+p_{11,L}}{p_{11,\ast}+2p_{11,L}}.
\]
The variables $u_{2,\ast L}$ and $p_{12,\ast L}$ are obtained by solving the following system
\[
\mathbf{A}^L\left(\begin{array}{c}
              u_{2,\ast L} \\
              p_{12,\ast L}
            \end{array}\right)
=\left(\begin{array}{c}
   a_1^L \\
   a_2^L
 \end{array}\right),
\]
where
\begin{align*}
&\mathbf{A}^L:=\begin{pmatrix}
                \rho_{\ast L}(u_{1,\ast}-\sigma_L) & 1 \\
                E_{11,\ast L}-\frac{1}{2}\rho_{\ast L}u_{1,\ast}\sigma_L & u_{1,\ast}-\frac{1}{2}\sigma_L
              \end{pmatrix}, \\
&a_1^L:=\rho_L(u_{1,L}-\sigma_L)u_{2,L}+p_{12,L}, \\
&a_2^L:=E_{12,L}(u_{1,L}-\sigma_L)+\frac{1}{2}(p_{11,L}u_{2,L}+p_{12,L}u_{1,L}).
\end{align*}
The variable $p_{22,\ast L}$ is given by
\[
p_{22,\ast L}=2E_{22,\ast L}-\rho_{\ast L}u_{2,\ast L}^2,
\]
where
\[
E_{22,\ast L}=\frac{E_{22,L}u_{1,L}+p_{12,L}u_{2,L}-\sigma_LE_{22,L}-p_{12,\ast L}u_{2,\ast L}}{u_{1,\ast}-\sigma_L}.
\]
\end{remark}

\subsection{Computing $\mathbf{V}_{\ast\ast L}$ and $\mathbf{V}_{\ast\ast R}$}
The solution state $\mathbf{V}_{\ast\ast L}$ is between the 2-shear wave and the contact discontinuity, while $\mathbf{V}_{\ast\ast R}$ is between the 5-shear wave and the contact discontinuity.
Across the 2-shear wave, utilizing the generalized Riemann invariant $u_2+\frac{p_{12}}{\sqrt{\rho p_{11}}}$
gives
\begin{equation}\label{eq1:u2_p12_ast2}
u_{2,\ast\ast}+\frac{p_{12,\ast\ast}}{\sqrt{\rho_{\ast L}p_{11,\ast}}}=a_3^L,
\end{equation}
with $a_3^L:=u_{2,\ast L}+\frac{p_{12,\ast L}}{\sqrt{\rho_{\ast L}p_{11,\ast}}}$.
For the 5-shear wave, using the generalized Riemann invariant $u_2-\frac{p_{12}}{\sqrt{\rho p_{11}}}$ yields
\begin{equation}\label{eq2:u2_p12_ast2}
u_{2,\ast\ast}-\frac{p_{12,\ast\ast}}{\sqrt{\rho_{\ast R}p_{11,\ast}}}=a_3^R,
\end{equation}
with $a_3^R:=u_{2,\ast R}-\frac{p_{12,\ast R}}{\sqrt{\rho_{\ast R}p_{11,\ast}}}$.
Solving \eqref{eq1:u2_p12_ast2} and \eqref{eq2:u2_p12_ast2} gives
\begin{equation}
\begin{split}
&u_{2,\ast\ast}=\frac{a_3^L\sqrt{\rho_{\ast L}}+a_3^R\sqrt{\rho_{\ast R}}}{\sqrt{\rho_{\ast L}}+\sqrt{\rho_{\ast R}}}, \\
&p_{12,\ast\ast}=\frac{(a_3^L-a_3^R)\sqrt{\rho_{\ast L}\rho_{\ast R}p_{11,\ast}}}{\sqrt{\rho_{\ast L}}+\sqrt{\rho_{\ast R}}}.
\end{split}
\label{u2_p12_2ast}
\end{equation}
Moreover, the fact that $\det(\mathbf{p})$ is the generalized Riemann invariant of both the 2-shear wave and the 5-shear wave implies that
\begin{equation}\label{p22_2astK}
p_{22,\ast\ast K}=\frac{\det(\mathbf{p}_{\ast K})+p_{12,\ast\ast}^2}{p_{11,\ast}}, \quad K=L,R.
\end{equation}

\section{Resolution of the GRP}\label{GRP solver}
This section resolves the following GRP problem
\begin{equation}\label{GRP2}
\begin{cases}
\eqref{ten-moment2}, \quad t>0, \\
\mathbf{V}(x,0)=\begin{cases}
                    \mathbf{V}_L+x\mathbf{V}_L', \quad x<0, \\
                    \mathbf{V}_R+x\mathbf{V}_R', \quad x>0,
                \end{cases}
\end{cases}
\end{equation}
to derive  the limiting value of $\frac{\partial\mathbf{V}}{\partial t}$ at $x=0$, as $t\rightarrow0+$,
where $\mathbf{V}_K':=(\rho_K',u_{1,K}',u_{2,K}',p_{11,K}',p_{12,K}',p_{22,K}')^\top$, $K=L,R$,  are the constant slope vectors. The initial structure of the solution $\mathbf{V}^{{\rm GRP}}(x,t)$ of \eqref{GRP2} is determined by the exact solution $\bm{\omega}\left(\frac{x}{t};\mathbf{V}_L,\mathbf{V}_R\right)$ of the Riemann problem \eqref{local RP2}  and
\[
\lim_{t\rightarrow0+}\mathbf{V}^{{\rm GRP}}(\lambda t,t)=\bm{\omega}\left(0;\mathbf{V}_L,\mathbf{V}_R\right)=:\mathbf{V}^{{\rm RP}}, \quad x=\lambda t.
\]
The local wave pattern around the singularity point $(x,t)=(0,0)$ of the GRP \eqref{GRP2} typically exhibits piecewise smoothness and comprises elementary wave types, including the rarefaction wave, the shock wave, the shear wave, and the contact discontinuity, as depicted schematically in Figure \ref{localwave_GRP}.

\begin{figure}[!htbp]
\centering
\includegraphics[width=8cm]{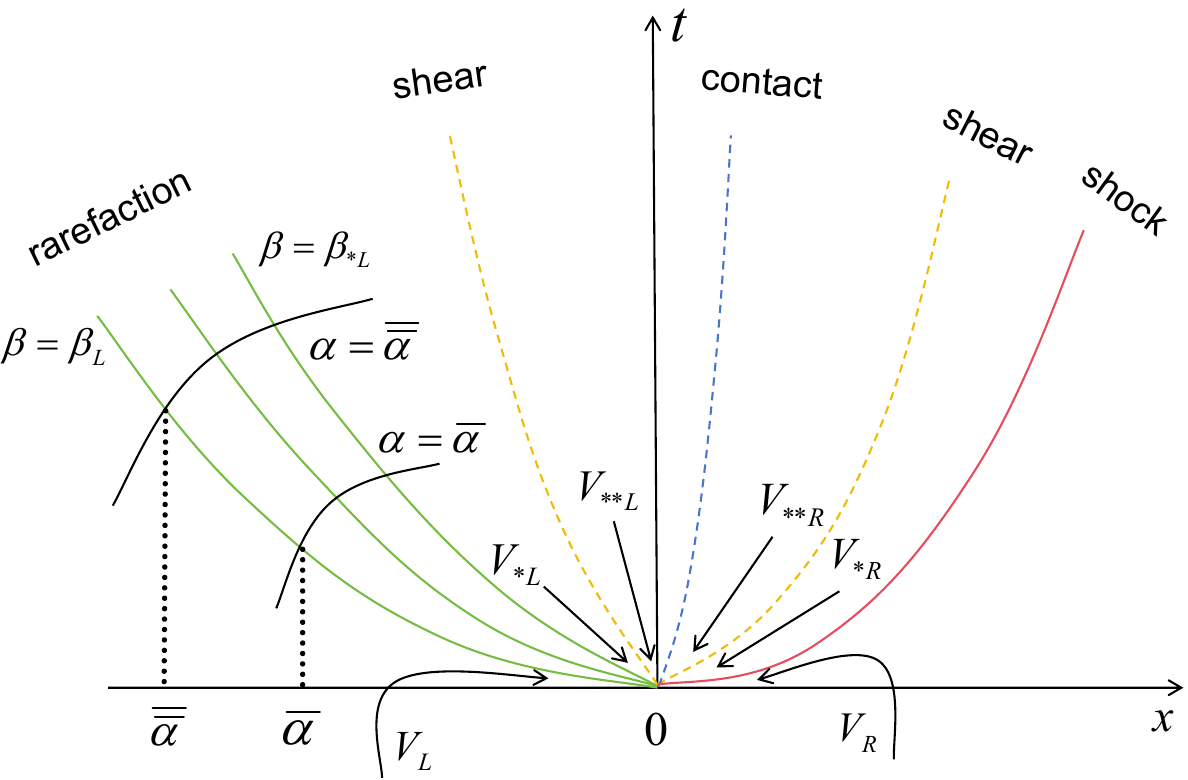}
\caption{The schematic description of a local wave configuration of the GRP \eqref{GRP2} with $0\leq t\ll1$.}
\label{localwave_GRP}
\end{figure}

The rarefaction waves in the solution of the Riemann problem \eqref{local RP2} exhibit isentropic flow properties, making the generalized Riemann invariants constant with vanishing derivatives within $i$th-rarefaction wave ($i=1,6$), but unfortunately, those properties do not generally hold for the generalized Riemann problem \eqref{GRP2} due to the need to consider the curved rarefaction waves.
Nonetheless, for a short time period  $0\leq t\ll1$, the solution of \eqref{GRP2} may be considered as a perturbation of the solution of \eqref{local RP2}, allowing us to expect that the generalized Riemann invariants remain regular within the $i$th-rarefaction waves ($i=1,6$) of the solution $\mathbf{V}^{{\rm GRP}}(x,t)$ near the singularity point $(x,t) = (0,0)$. Consequently, the generalized Riemann invariants are still utilized to resolve the rarefaction waves around this singularity point.

From this, as an example, we will continue to focus our attention on the specific local wave configuration depicted in Figure \ref{localwave_GRP}, corresponding to Figure \ref{localwave_RP}. In this configuration, a rarefaction wave propagates to the left, while a shock wave moves to the right. The intermediate region between them is separated by two shear waves and a contact discontinuity. It is worth noting that other local wave configurations can be analyzed in a similar manner. In the subsequent subsections,   the nonsonic case (see Subsection \ref{nonsonic case}), the sonic case (see Subsection \ref{sonic case}), and the acoustic case (see Subsection \ref{acoustic case})  will be separately discussed in detail to compute the limiting value of $\frac{\partial\mathbf{V}}{\partial t}(0,t)$ as $t\rightarrow0+$.

\subsection{Nonsonic case}\label{nonsonic case}
If the $t$-axis is located between the 1-wave and the 6-wave, the nonsonic case happens. Denote the limiting values of $\frac{\partial\mathbf{V}}{\partial t}(0,t)$ when $t\rightarrow0+$ in the four middle domains by $\left(\frac{\partial\mathbf{V}}{\partial t}\right)_{\ast L}$, $\left(\frac{\partial\mathbf{V}}{\partial t}\right)_{\ast\ast L}$, $\left(\frac{\partial\mathbf{V}}{\partial t}\right)_{\ast\ast R}$ and $\left(\frac{\partial\mathbf{V}}{\partial t}\right)_{\ast R}$, respectively.
The derivations of those limiting values  are very long-winded and tedious, and for the sake of  convenience,
the main steps in this subsection are outlined in Figure \ref{main results of nonsonic case}.

\tikzstyle{format}=[rectangle,draw,very thick,fill=white] 
\tikzstyle{arrow}=[thin,->,>=stealth]

\begin{figure}[!htbp]
\begin{tikzpicture}[node distance=5.5cm, text width=5cm]
  \node[format] (start) {Compute $(\frac{\partial u_1}{\partial t})_{\ast K}$ (=$(\frac{\partial u_1}{\partial t})_{\ast\ast K}$), $(\frac{\partial p_{11}}{\partial t})_{\ast K}$ (=$(\frac{\partial p_{11}}{\partial t})_{\ast\ast K}$) and $(\frac{\partial\rho}{\partial t})_{\ast K}$ (=$(\frac{\partial\rho}{\partial t})_{\ast\ast K}$) {\rm($K=L,R$)}, see Theorem \ref{theorem:u1_p11_rho_t_astK}.};

  \node[format] (input) [right of=start, yshift=1.5cm] {Compute $(\frac{\partial u_2}{\partial t})_{\ast L}$, $(\frac{\partial p_{12}}{\partial t})_{\ast L}$ and $(\frac{\partial p_{22}}{\partial t})_{\ast L}$, see Theorem \ref{theorem for omega_astL}.};

  \node[format] (process)[right of=start, yshift=-1.5cm] {Compute $(\frac{\partial u_2}{\partial t})_{\ast R}$, $(\frac{\partial p_{12}}{\partial t})_{\ast R}$ and $(\frac{\partial p_{22}}{\partial t})_{\ast R}$, see Theorem \ref{theorem:u2_p12_p22_t_astR}.};

  \node[format] (end) [right of=process, yshift=1.5cm] {Compute $\left(\frac{\partial u_2}{\partial t}\right)_{\ast\ast K}$, $\left(\frac{\partial p_{12}}{\partial t}\right)_{\ast\ast K}$ and $\left(\frac{\partial p_{22}}{\partial t}\right)_{\ast\ast K}$ $(K=L,R)$, see Theorem \ref{theorem:u2_p12_t_2astK} and Theorem \ref{theorem:p22_t_2astK}.};

  \draw [arrow] (start) -- (input);
  \draw [arrow] (start) -- (process);
  \draw [arrow] (input) -- (end);
  \draw [arrow] (process) -- (end);
\end{tikzpicture}
\centering
\caption{The main steps for the nonsonic case in Subsection \ref{nonsonic case}.}
\label{main results of nonsonic case}
\end{figure}
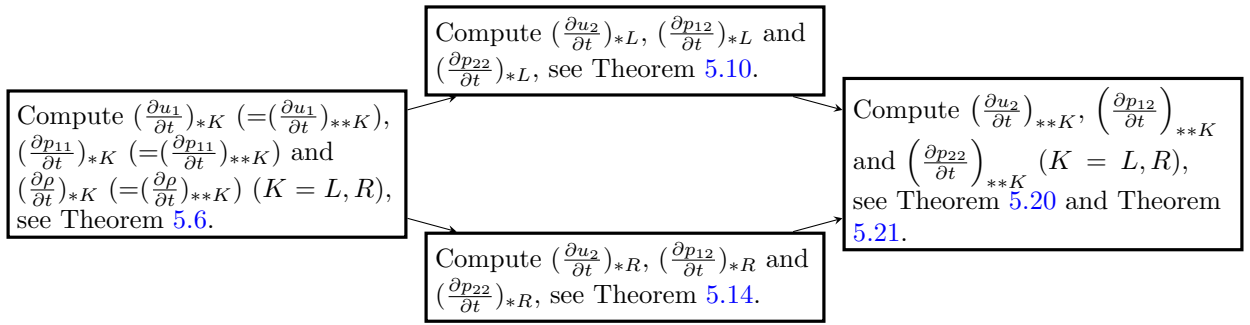

\subsubsection{Computing $(\frac{\partial u_1}{\partial t})_{\ast K}$, $(\frac{\partial p_{11}}{\partial t})_{\ast K}$ and $(\frac{\partial\rho}{\partial t})_{\ast K}$ {\rm($K=L,R$)}}\label{computation_u1t_p11t_rhot_astK}

Similar to the second-order accurate GRP scheme in \cite{ben2006direct,yang2011direct,yang2012direct,kuang2019second}, the region of the 1-rarefaction wave associated with $\lambda_1=u_1-c$ can be described by the characteristic coordinates $(\alpha,\beta)$, where $\alpha\in[-\infty,0]$, $\beta\in[\beta_L,\beta_{\ast L}]$ with $\beta_L=\lambda_1(\mathbf{V}_L)$ and $\beta_{\ast L}=\lambda_1(\mathbf{V}_{\ast L})$, see Figure \ref{localwave_GRP}, and  $\beta=\beta(x,t)$ and $\alpha=\alpha(x,t)$ are the integral curves of the following two equations
\begin{equation}\label{integral curves}
 \frac{\text{d}x}{\text{d}t}=u_1-c, \quad
 \frac{\text{d}x}{\text{d}t}=u_1+c,
\end{equation}
respectively. Here, $\beta$ represents the initial value of the slope $\lambda_1=u_1-c$ at the singularity point (0,0), and $\alpha$ denotes the $x$-coordinate of the intersection point between the transversal characteristic curve and the leading $\beta$-curve, as illustrated in Figure \ref{localwave_GRP}. Two equations in \eqref{integral curves} may give
\begin{equation}
\frac{\partial x}{\partial\alpha}=(u_1-c)\frac{\partial t}{\partial\alpha}, \quad
\frac{\partial x}{\partial\beta}=(u_1+c)\frac{\partial t}{\partial\beta}, \label{x_alpha_beta}
\end{equation}
which further imply that
\begin{equation}\label{t_alpha_beta}
\frac{\partial^2t}{\partial\alpha\partial\beta}=\frac{1}{2c}\left(\frac{\partial(u_1-c)}{\partial\beta}\frac{\partial t}{\partial\alpha}-\frac{\partial(u_1+c)}{\partial\alpha}\frac{\partial t}{\partial\beta}\right).
\end{equation}
Similarly, for the associated Riemann problem, there exists a local coordinate transformation $x_{\text{ass}}=x(\alpha,\beta)$, $t_{\text{ass}}=t(\alpha,\beta)$. Both local coordinate transformations satisfy the properties \cite{ben2006direct}
\begin{equation}\label{transformation properties}
\frac{\partial\lambda_1}{\partial t}(0,\beta)=1, ~ \frac{\partial t}{\partial\alpha}(0,\beta)=\frac{\partial t_{\text{ass}}}{\partial\alpha}(0,\beta), ~ \frac{\partial t}{\partial\beta}(0,\beta)=0, ~ \beta_L\leq\beta\leq\beta_{\ast L}.
\end{equation}
It follows from \eqref{t_alpha_beta} that
\begin{equation}\label{t_alpha_beta_0_beta}
\frac{\partial^2t}{\partial\alpha\partial\beta}(0,\beta)=\frac{1}{2c(0,\beta)}\frac{\partial t}{\partial\alpha}(0,\beta).
\end{equation}

Let $\psi_1:=u_1+c$ and $S_1:=\frac{p_{11}}{\rho^3}$. Following the same derivation in \cite{ben2006direct}, one has the following explicit expressions of the local
coordinate transformation in the 1-rarefaction wave of the associated Riemann problem
\begin{equation}\label{explicit express of tass and xass}
t_{\text{ass}}(\alpha,\beta)=\frac{\alpha}{\psi_{1,L}-\beta}, \ \
x_{\text{ass}}(\alpha,\beta)=\frac{\alpha\beta}{\psi_{1,L}-\beta}.
\end{equation}
Because  $\psi_1=u_1+c$, $\beta=u_1(0,\beta)-c(0,\beta)$ and $\psi_1$ is a generalized Riemann invariant across the 1-rarefaction wave, one has
\begin{equation}\label{c_0beta and u1_0beta}
c(0,\beta)=\frac{1}{2}(\psi_{1,L}-\beta), ~ u_1(0,\beta)=\frac{1}{2}(\psi_{1,L}+\beta).
\end{equation}
Combining \eqref{transformation properties}, \eqref{explicit express of tass and xass} with \eqref{c_0beta and u1_0beta} gives
\begin{equation}\label{t_alpha_0beta}
\frac{\partial t}{\partial\alpha}(0,\beta)=\frac{1}{2c(0,\beta)}.
\end{equation}
Due to $\text{d}\psi_1=\text{d}u_1-\frac{c}{2\rho}\text{d}\rho+\frac{3}{2\rho c}\text{d}p_{11}$ and $\text{d}S_1=\frac{-3p_{11}}{\rho^4}\text{d}\rho+\frac{1}{\rho^3}\text{d}p_{11}$,  using \eqref{eq:rho}, \eqref{eq:u1} and \eqref{eq:p11} yields
\begin{align}
&\frac{\partial\psi_1}{\partial t}+(u_1+c)\frac{\partial\psi_1}{\partial x}=\Pi_1, \label{eq:psi1} \\
&\frac{\partial S_1}{\partial t}+u_1\frac{\partial S_1}{\partial x}=0, \label{eq:S1}
\end{align}
where $\Pi_1:=-\frac{c^2}{2\rho}\frac{\partial\rho}{\partial x}+\frac{1}{2\rho}\frac{\partial p_{11}}{\partial x}-\frac{1}{2}W_x$. Besides, one has
\begin{equation}\label{dp11}
\text{d}p_{11}=c^2\text{d}\rho+\rho^3\text{d}S_1.
\end{equation}
It follows that
\begin{equation}\label{Pi1}
\Pi_1=\frac{1}{2\rho}\left(\frac{\partial p_{11}}{\partial x}-c^2\frac{\partial\rho}{\partial x}\right)-\frac{1}{2}W_x=\frac{1}{2}\rho^2\frac{\partial S_1}{\partial x}-\frac{1}{2}W_x.
\end{equation}
For $\psi_1$, by the total differentials of $\psi_1$ and $S_1$, one has
\begin{equation}\label{dpsi1}
\text{d}\psi_1=\text{d}u_1+\frac{1}{\rho c}\text{d}p_{11}+\frac{\rho^2}{2c}\text{d}S_1.
\end{equation}
The following lemma gives the expressions of $\left(\rho^2\frac{\partial S_1}{\partial t}\right)(0,\beta)$ and $\frac{\partial\psi_1}{\partial t}(0,\beta)$.

\begin{lemma}\label{S1_t_psi1_t_0beta}
If assuming that the 1-rarefaction wave associated with $u_1-c$ moves to the left, and considering
the generalized Riemann invariants $\psi_1, S_1$ and their time derivatives $\frac{\partial\psi_1}{\partial t}$, $\frac{\partial S_1}{\partial t}$ as continuous functions of $\alpha, \beta$ with $-\alpha_0\leq\alpha\leq0$ and $\beta_L\leq\beta\leq\beta_{\ast L}$, then one has
\begin{align}
&\left(\rho^2\frac{\partial S_1}{\partial t}\right)(0,\beta)=-\frac{\rho_L^2S_{1,L}'}{c_L^3}[\beta+c(0,\beta)]\cdot c^3(0,\beta), \label{rho^2_S1_t_0_beta} \\
&\frac{\partial\psi_1}{\partial t}(0,\beta)=\frac{\rho_L^2S_{1,L}'}{8c_L^3}[\psi_{1,L}(3c_L^2-c^2)-2\beta c^2](0,\beta)-\psi_{1,L}'\psi_1(0,\beta)-\frac{1}{2}W_x(0), \label{psi1_t_0_beta}
\end{align}
where
\begin{align}
S_{1,L}'=\frac{1}{\rho_L^3}(p_{11,L}'-c_L^2\rho_L'), 
\quad \psi_{1,L}'=u_{1,L}'+\frac{1}{\rho_Lc_L}p_{11,L}'+\frac{\rho_L^2}{2c_L}S_{1,L}',
\label{psi1_L'}
\end{align}
obtained by \eqref{dp11} and \eqref{dpsi1}.
\end{lemma}

\begin{proof}
(i) Computing $\left(\rho^2\frac{\partial S_1}{\partial t}\right)(0,\beta)$.
Using \eqref{x_alpha_beta} and \eqref{eq:S1} gives
\begin{align}
&\frac{\partial S_1}{\partial\beta}=\frac{\partial t}{\partial\beta}\left[\frac{\partial S_1}{\partial t}+(u_1+c)\frac{\partial S_1}{\partial x}\right]=\frac{\partial t}{\partial\beta}\cdot c\frac{\partial S_1}{\partial x}, \label{S1_beta} \\
&\frac{\partial S_1}{\partial\alpha}=\frac{\partial t}{\partial\alpha}\left[\frac{\partial S_1}{\partial t}+(u_1-c)\frac{\partial S_1}{\partial x}\right]=-\frac{\partial t}{\partial\alpha}\cdot c\frac{\partial S_1}{\partial x}. \label{S1_alpha}
\end{align}
Differentiating \eqref{S1_beta} with respect to $\alpha$ and noting that $\frac{\partial t}{\partial\beta}(0,\beta)=0$, one gets
\begin{equation}\label{S1_alpha_beta_0beta}
\frac{\partial}{\partial\beta}\left(\frac{\partial S_1}{\partial\alpha}(0,\beta)\right)=\frac{1}{2}\frac{\partial t}{\partial\alpha}(0,\beta)\frac{\partial S_1}{\partial x}(0,\beta)=-\frac{1}{2c(0,\beta)}\frac{\partial S_1}{\partial\alpha}(0,\beta),
\end{equation}
where \eqref{t_alpha_beta_0_beta} and \eqref{S1_alpha} have been used in the first and second equality, respectively. Integrating \eqref{S1_alpha_beta_0beta} from $\beta_L$ to $\beta$ yields
\begin{equation}\label{S1_alpha_0beta}
\frac{\partial S_1}{\partial\alpha}(0,\beta)=\frac{\partial S_1}{\partial\alpha}(0,\beta_L)\cdot\exp\left(-\int_{\beta_L}^{\beta}\frac{1}{2c(0,\eta)}\text{d}\eta\right)=\frac{\partial S_1}{\partial\alpha}(0,\beta_L)\cdot\frac{c(0,\beta)}{c_L}.
\end{equation}
Similarly, from \eqref{t_alpha_beta_0_beta}, one has
\begin{equation}\label{t_alpha_0beta1}
\frac{\partial t}{\partial\alpha}(0,\beta)=\frac{\partial t}{\partial\alpha}(0,\beta_L)\cdot\frac{c_L}{c(0,\beta)}.
\end{equation}
By \eqref{S1_alpha}, \eqref{S1_alpha_0beta} and \eqref{t_alpha_0beta1}, one has
\begin{align}
\frac{\partial S_1}{\partial x}(0,\beta)&=-\frac{\partial S_1}{\partial\alpha}(0,\beta)\cdot\left(c\frac{\partial t}{\partial\alpha}\right)^{-1}(0,\beta) \notag \\
&=-\frac{\partial S_1}{\partial\alpha}(0,\beta_L)\cdot\frac{c(0,\beta)}{c_L}\frac{1}{c(0,\beta)}\left(\frac{\partial t}{\partial\alpha}\right)^{-1}(0,\beta_L)\cdot\frac{c(0,\beta)}{c_L} \notag \\
&=-\frac{\partial S_1}{\partial\alpha}(0,\beta_L)\cdot\left(\frac{\partial t}{\partial\alpha}\right)^{-1}(0,\beta_L)\cdot\frac{c(0,\beta)}{c_L^2} \notag \\
&=c_L\frac{\partial S_1}{\partial x}(0,\beta_L)\cdot\frac{c(0,\beta)}{c_L^2} 
 =\frac{S_{1,L}'}{c_L}c(0,\beta), \label{S1_x_0beta}
\end{align}
It follows that
\begin{equation}\label{S1_t_0beta}
\frac{\partial S_1}{\partial t}(0,\beta)=-u_1(0,\beta)\cdot\frac{\partial S_1}{\partial x}(0,\beta)=-\frac{S_{1,L}'}{c_L}[\beta+c(0,\beta)]\cdot c(0,\beta).
\end{equation}
Because $S_1=p_{11}/\rho^3$ is a generalized Riemann invariant for the 1-rarefaction wave, it holds that
\begin{equation}\label{rho^2/rhoL^2}
\frac{\rho^2(0,\beta)}{\rho_L^2}=\frac{(3p_{11}/\rho)(0,\beta)}{3p_{11,L}/\rho_L}=\frac{c^2(0,\beta)}{c_L^2},
\end{equation}
and then combining \eqref{S1_t_0beta} and \eqref{rho^2/rhoL^2} gives \eqref{rho^2_S1_t_0_beta}.

(ii) Computing $\frac{\partial\psi_1}{\partial t}(0,\beta)$. Using \eqref{eq:psi1} gives
\begin{align}
&\frac{\partial\psi_1}{\partial\beta}=\frac{\partial t}{\partial\beta}\Pi_1(\alpha,\beta), \label{psi1_beta} \\
&\frac{\partial\psi_1}{\partial\alpha}=\frac{\partial t}{\partial\alpha}\left[\frac{\partial\psi_1}{\partial t}+(u_1-c)\frac{\partial\psi_1}{\partial x}\right]=\frac{\partial t}{\partial\alpha}\left[\Pi_1(\alpha,\beta)-2c\frac{\partial\psi_1}{\partial x}\right]. \label{psi1_alpha}
\end{align}
Differentiating \eqref{psi1_beta} with respect to $\alpha$ and using \eqref{transformation properties}, \eqref{t_alpha_beta_0_beta} gets
\begin{equation*}
\frac{\partial}{\partial\beta}\left(\frac{\partial\psi_1}{\partial\alpha}(0,\beta)\right)=\frac{1}{2c(0,\beta)}\frac{\partial t}{\partial\alpha}(0,\beta)\cdot\Pi_1(0,\beta).
\end{equation*}
Integrating the above equation from $\beta_L$ to $\beta$ yields
\begin{equation}\label{psi1_alpha_0beta}
\frac{\partial\psi_1}{\partial\alpha}(0,\beta)=\frac{\partial\psi_1}{\partial\alpha}(0,\beta_L)+\int_{\beta_L}^{\beta}\frac{1}{2c(0,\eta)}\frac{\partial t}{\partial\alpha}(0,\eta)\cdot\Pi_1(0,\eta)\text{d}\eta.
\end{equation}
Due to \eqref{Pi1} and \eqref{psi1_alpha}, the initial value is computed as
\begin{equation}\label{psi1_alpha_0betaL}
\frac{\partial\psi_1}{\partial\alpha}(0,\beta_L)=\frac{\partial t}{\partial\alpha}(0,\beta_L)\cdot\left[\frac{1}{2}\rho_L^2S_{1,L}'-\frac{1}{2}W_x(0)-2c_L\psi_{1,L}'\right].
\end{equation}
Moreover, by \eqref{c_0beta and u1_0beta}, \eqref{Pi1}, \eqref{S1_x_0beta} and \eqref{rho^2/rhoL^2}, one has
\begin{equation}\label{Pi1_0beta}
\Pi_1(0,\beta)=\frac{\rho_L^2S_{1,L}'}{16c_L^3}(\psi_{1,L}-\beta)^3-\frac{1}{2}W_x(0).
\end{equation}
Thus the integral in \eqref{psi1_alpha_0beta} can be exactly obtained and
\begin{equation}\label{integral 1}
\int_{\beta_L}^{\beta}\frac{1}{2c(0,\eta)}\frac{\partial t}{\partial\alpha}(0,\eta)\cdot\Pi_1(0,\eta)\text{d}\eta=
-\frac{\rho_L^2S_{1,L}'}{8c_L^3}[c^2(0,\beta)-c_L^2]-\frac{1}{4}W_x(0)\left(\frac{1}{c(0,\beta)}-\frac{1}{c_L}\right),
\end{equation}
where \eqref{t_alpha_0beta} has been used.
Substituting \eqref{psi1_alpha_0betaL} and \eqref{integral 1} into \eqref{psi1_alpha_0beta} yields the expression of $\frac{\partial\psi_1}{\partial\alpha}(0,\beta)$. The second equality in \eqref{psi1_alpha} implies that
\begin{equation}\label{2c_psi1_x_0beta}
2c(0,\beta)\cdot\frac{\partial\psi_1}{\partial x}(0,\beta)=\Pi_1(0,\beta)-\left(\frac{\partial t}{\partial\alpha}\right)^{-1}(0,\beta)\cdot\frac{\partial\psi_1}{\partial\alpha}(0,\beta).
\end{equation}
Combining \eqref{2c_psi1_x_0beta} with the first equality in \eqref{psi1_alpha}, one gets
\begin{equation}\label{psi1_t_0beta}
\frac{\partial\psi_1}{\partial t}(0,\beta)=-\left(\frac{u_1-c}{2c}\right)(0,\beta)\cdot\Pi_1(0,\beta)+\left(\frac{u_1+c}{2c}\right)(0,\beta)\cdot\left(\frac{\partial t}{\partial\alpha}\right)^{-1}(0,\beta)\cdot\frac{\partial\psi_1}{\partial\alpha}(0,\beta).
\end{equation}
Substituting the expression of $\frac{\partial\psi_1}{\partial\alpha}(0,\beta)$ and \eqref{Pi1_0beta} into \eqref{psi1_t_0beta} gives \eqref{psi1_t_0_beta}.
\end{proof}

\begin{lemma}[Resolution of the 1-rarefaction wave]
Assuming that the 1-rarefaction wave associated with $u_1-c$ moves to the left,   the limiting values $\frac{\mathcal{D}u_1}{\mathcal{D}t}(0,\beta)$ and $\frac{\mathcal{D}p_{11}}{\mathcal{D}t}(0,\beta)$ satisfy
\begin{equation}\label{resolution of 1-rarefaction}
\widetilde{a}_1(0,\beta)\cdot\frac{\mathcal{D}u_1}{\mathcal{D}t}(0,\beta)+\widetilde{b}_1(0,\beta)\cdot\frac{\mathcal{D}p_{11}}{\mathcal{D}t}(0,\beta)=\widetilde{d}_1(0,\beta),
\end{equation}
where
\begin{align*}
&\widetilde{a}_1(0,\beta)=1+\frac{u_1(0,\beta)}{c(0,\beta)}, \quad \widetilde{b}_1(0,\beta)=\left(\frac{u_1}{3p_{11}}+\frac{1}{\rho c}\right)(0,\beta) ,\\
&\widetilde{d}_1(0,\beta)=\frac{\rho_L^2S_{1,L}'\psi_{1,L}}{8c_L^3}[3c_L^2+c^2(0,\beta)]-\psi_{1,L}'\psi_{1,L}-\frac{1}{2}\left(1+\frac{u_1(0,\beta)}{c(0,\beta)}\right)W_x(0).
\end{align*}
\end{lemma}

\begin{proof}
Using \eqref{dpsi1}  gets
\begin{equation}\label{eq:du1_dp11_dt}
\frac{\partial u_1}{\partial t}+\frac{1}{\rho c}\frac{\partial p_{11}}{\partial t}=\frac{\partial\psi_1}{\partial t}-\frac{\rho^2}{2c}\frac{\partial S_1}{\partial t}.
\end{equation}
Substituting \eqref{u1_p11_t} into above equation obtains
\[
\left(1+\frac{u_1}{c}\right)\frac{\mathcal{D}u_1}{\mathcal{D}t}+\left(\frac{u_1}{3p_{11}}+\frac{1}{\rho c}\right)\frac{\mathcal{D}p_{11}}{\mathcal{D}t}=\frac{\partial\psi_1}{\partial t}-\frac{\rho^2}{2c}\frac{\partial S_1}{\partial t}-\frac{u_1}{2c}W_x.
\]
Combining it with \eqref{rho^2_S1_t_0_beta} and \eqref{psi1_t_0_beta} gives \eqref{resolution of 1-rarefaction}.
\end{proof}

Taking $\beta=\beta_{\ast L}$ in \eqref{resolution of 1-rarefaction} and using \eqref{u1_p11_t} obtains the first equation for $\left(\frac{\mathcal{D}u_1}{\mathcal{D}t}\right)_{\ast}$ and $\left(\frac{\mathcal{D}p_{11}}{\mathcal{D}t}\right)_{\ast}$ as follows
\begin{equation}\label{eq1:D_u1_p11_Dtast}
a_1\left(\frac{\mathcal{D}u_1}{\mathcal{D}t}\right)_\ast+b_1\left(\frac{\mathcal{D}p_{11}}{\mathcal{D}t}\right)_\ast=d_1,
\end{equation}
where
\begin{equation}
\begin{cases}
a_1=\widetilde{a}_1(0,\beta_{\ast L})=1+\frac{u_{1,\ast}}{c_{\ast L}}, \\
b_1=\widetilde{b}_1(0,\beta_{\ast L})=\frac{u_{1,\ast}}{3p_{11,\ast}}+\frac{1}{\rho_{\ast L}c_{\ast L}}, \\
d_1=\widetilde{d}_1(0,\beta_{\ast L})=\left[\frac{\rho_L^2S_{1,L}'}{8c_L^3}(3c_L^2+c_{\ast L}^2)-\psi_{1,L}'\right]\psi_{1,L}-\frac{1}{2}\left(1+\frac{u_{1,\ast}}{c_{\ast L}}\right)W_x(0).
\end{cases}
\label{a1_b1_d1}
\end{equation}

\begin{remark}
It is necessary to prove that both the limiting values of $\frac{\mathcal{D}u_1}{\mathcal{D}t}$ and $\frac{\mathcal{D}p_{11}}{\mathcal{D}t}$ do not change in the whole domain between the 1-wave and the 6-wave, which are denoted by $\left(\frac{\mathcal{D}u_1}{\mathcal{D}t}\right)_{\ast}$ and $\left(\frac{\mathcal{D}p_{11}}{\mathcal{D}t}\right)_{\ast}$, respectively.

To this end, we firstly prove that
\begin{equation}\label{Du1_p11_astK=ast2K}
\left(\frac{\mathcal{D}u_1}{\mathcal{D}t}\right)_{\ast K}=\left(\frac{\mathcal{D}u_1}{\mathcal{D}t}\right)_{\ast\ast K}, ~
\left(\frac{\mathcal{D}p_{11}}{\mathcal{D}t}\right)_{\ast K}=\left(\frac{\mathcal{D}p_{11}}{\mathcal{D}t}\right)_{\ast\ast K}, ~
K=L,R.
\end{equation}
For $K=L$, because both $u_1$ and $p_{11}$ are the generalized Riemann invariants for the 2-shear wave, one has
\[
\left(\frac{\mathcal{D}_2u_1}{\mathcal{D}t}\right)_{\ast L}=\left(\frac{\mathcal{D}_2u_1}{\mathcal{D}t}\right)_{\ast\ast L}, ~
\left(\frac{\mathcal{D}_2p_{11}}{\mathcal{D}t}\right)_{\ast L}=\left(\frac{\mathcal{D}_2p_{11}}{\mathcal{D}t}\right)_{\ast\ast L},
\]
where $\frac{\mathcal{D}_2}{\mathcal{D}t}:=\frac{\partial}{\partial t}+\lambda_2\frac{\partial}{\partial x}$
and $\lambda_2=u_{1,\ast}-\frac{c_{\ast L}}{\sqrt{3}}$.
By \eqref{u1_x} and \eqref{p11_x}, it follows that
\begin{align*}
&\left(\frac{\mathcal{D}u_1}{\mathcal{D}t}\right)_{\ast L}-\frac{\lambda_2-u_{1,\ast}}{3p_{11,\ast}}\left(\frac{\mathcal{D}p_{11}}{\mathcal{D}t}\right)_{\ast L}
=\left(\frac{\mathcal{D}u_1}{\mathcal{D}t}\right)_{\ast\ast L}-\frac{\lambda_2-u_{1,\ast}}{3p_{11,\ast}}\left(\frac{\mathcal{D}p_{11}}{\mathcal{D}t}\right)_{\ast\ast L}, \\
&\left(\frac{\mathcal{D}p_{11}}{\mathcal{D}t}\right)_{\ast L}-\rho_{\ast L}(\lambda_2-u_{1,\ast})\left[\left(\frac{\mathcal{D}u_1}{\mathcal{D}t}\right)_{\ast L}+\frac{1}{2}W_x(0)\right]
=\left(\frac{\mathcal{D}p_{11}}{\mathcal{D}t}\right)_{\ast\ast L}-\rho_{\ast L}(\lambda_2-u_{1,\ast})\left[\left(\frac{\mathcal{D}u_1}{\mathcal{D}t}\right)_{\ast\ast L}+\frac{1}{2}W_x(0)\right].
\end{align*}
The above two equations can imply \eqref{Du1_p11_astK=ast2K} for $K=L$, and the proof for $K=R$ is similar. Besides, both $u_1$ and $p_{11}$ are also the generalized Riemann invariants for the contact discontinuity, thus
\begin{equation}\label{Du1_p11_ast2L=ast2R}
\left(\frac{\mathcal{D}u_1}{\mathcal{D}t}\right)_{\ast\ast L}=\left(\frac{\mathcal{D}u_1}{\mathcal{D}t}\right)_{\ast\ast R}, ~
\left(\frac{\mathcal{D}p_{11}}{\mathcal{D}t}\right)_{\ast\ast L}=\left(\frac{\mathcal{D}p_{11}}{\mathcal{D}t}\right)_{\ast\ast R}.
\end{equation}
In virtue of \eqref{Du1_p11_astK=ast2K} and \eqref{Du1_p11_ast2L=ast2R}, one finds that both the limiting values of $\frac{\mathcal{D}u_1}{\mathcal{D}t}$ and $\frac{\mathcal{D}p_{11}}{\mathcal{D}t}$ do not change in the whole domain between the 1-wave and the 6-wave.
Furthermore, combining \eqref{Du1_p11_astK=ast2K} with \eqref{u1_p11_t}, one obtains 
\begin{equation*}\label{u1_p11_t_astK=ast2K}
\left(\frac{\partial u_1}{\partial t}\right)_{\ast K}=\left(\frac{\partial u_1}{\partial t}\right)_{\ast\ast K}, ~
\left(\frac{\partial p_{11}}{\partial t}\right)_{\ast K}=\left(\frac{\partial p_{11}}{\partial t}\right)_{\ast\ast K},~
K=L,R,
\end{equation*}
which means that both the limiting values of $\frac{\partial u_1}{\partial t}$ and $\frac{\partial p_{11}}{\partial t}$ do not change across the 2-shear wave and the 5-shear wave.
\end{remark}

%
%
%

Up to now, we have established the first equation \eqref{eq1:D_u1_p11_Dtast} that the limiting values $\left(\frac{\mathcal{D}u_1}{\mathcal{D}t}\right)_\ast$ and $\left(\frac{\mathcal{D}p_{11}}{\mathcal{D}t}\right)_\ast$ satisfy. To obtain their values, another equation is necessary by resolving the 6-shock wave. For the shock wavs, one has
\[
\sigma=\frac{\rho u_1-\bar{\rho}\bar{u}_1}{\rho-\overline{\rho}}, \quad u_1=\overline{u}_1\pm\Phi(p_{11};\overline{p}_{11},\overline{\rho}), \quad \rho=H(p_{11};\overline{p}_{11},\overline{\rho}),
\]
where $(\rho,u_1,p_{11})$ and $(\overline{\rho},\overline{u}_1,\overline{p}_{11})$ are the states ahead and behind the shock wave with the speed $\sigma$, and in the second equality, the "$+$" is for 6-shock wave and "$-$" is for 1-shock wave. According to the discussion on the exact Riemann solver, one knows that
\begin{align*}
&\Phi(p_{11};\overline{p}_{11},\overline{\rho})=(p_{11}-\overline{p}_{11})\left[\frac{1}{\overline{\rho}(2p_{11}+\overline{p}_{11})}\right]^{\frac{1}{2}}, \\
&H(p_{11};\overline{p}_{11},\overline{\rho})=\frac{2p_{11}+\overline{p}_{11}}{p_{11}+2\overline{p}_{11}}\overline{\rho}.
\end{align*}
Moreover, along the shock waves, there holds
\begin{equation}\label{D_sigema_gamma}
\frac{\mathcal{D}_{\sigma}\Gamma}{\mathcal{D}t}=0,
\end{equation}
where $\frac{\mathcal{D}_{\sigma}}{\mathcal{D}t}:=\frac{\partial}{\partial t}+\sigma\frac{\partial}{\partial x}$ and
$
\Gamma=u_1-[\overline{u}_1\pm\Phi(p_{11};\overline{p}_{11},\overline{\rho})]
$ or
$
\Gamma=\rho-H(p_{11};\overline{p}_{11},\overline{\rho}).
$
Taking $\Gamma=u_1-[\overline{u}_1\pm\Phi(p_{11};\overline{p}_{11},\overline{\rho})]$ in \eqref{D_sigema_gamma} and utilizing \eqref{eq:rho}, \eqref{eq:u1}, \eqref{eq:p11}, \eqref{u1_x}, \eqref{p11_x} and \eqref{u1_p11_t} yields
\begin{align*}
&\left[1\mp\rho(u_1-\sigma)\Phi_1\right]\frac{\mathcal{D}u_1}{\mathcal{D}t}+\left[\frac{u_1-\sigma}{3p_{11}}\mp\Phi_1\right]\frac{\mathcal{D}p_{11}}{\mathcal{D}t} \\
=&\pm\Phi_3(\sigma-\overline{u}_1)\overline{\rho}'+[\sigma-\overline{u}_1\mp3\overline{p}_{11}\Phi_2\mp\overline{\rho}\Phi_3]\overline{u}_1'+
\left[-\frac{1}{\overline{\rho}}\pm(\sigma-\overline{u}_1)\Phi_2\right]\overline{p}_{11}'-
\frac{1}{2}\left[\mp\rho(u_1-\sigma)\Phi_1+1\right]W_x,
\end{align*}
where
\begin{equation}\label{Phi_123}
\begin{cases}
\Phi_1=\frac{\partial\Phi}{\partial p_{11}}=\frac{p_{11}+2\overline{p}_{11}}{2p_{11}+\overline{p}_{11}}\left[\frac{1}{\overline{\rho}(2p_{11}+\overline{p}_{11})}\right]^{\frac{1}{2}}, \\
\Phi_2=\frac{\partial\Phi}{\partial\overline{p}_{11}}=-\frac{5p_{11}+\overline{p}_{11}}{4p_{11}+2\overline{p}_{11}}\left[\frac{1}{\overline{\rho}(2p_{11}+\overline{p}_{11})}\right]^{\frac{1}{2}}, \\
\Phi_3=\frac{\partial\Phi}{\partial\overline{\rho}}=\frac{\overline{p}_{11}-p_{11}}{2\overline{\rho}}\left[\frac{1}{\overline{\rho}(2p_{11}+\overline{p}_{11})}\right]^{\frac{1}{2}}.
\end{cases}
\end{equation}
Specifically, for the 6-shock wave, one has the following result.
\begin{lemma}[Resolution of the 6-shock wave]
Assume that the 6-shock wave associated with $u_1+c$ moves to the right. The limiting values $\left(\frac{\mathcal{D}u_1}{\mathcal{D}t}\right)_\ast$ and $\left(\frac{\mathcal{D}p_{11}}{\mathcal{D}t}\right)_\ast$ satisfy
\begin{equation}\label{eq2:D_u1_p11_Dtast}
a_2\left(\frac{\mathcal{D}u_1}{\mathcal{D}t}\right)_\ast+b_2\left(\frac{\mathcal{D}p_{11}}{\mathcal{D}t}\right)_\ast=d_2,
\end{equation}
where
\begin{equation}\label{a2_b2_d2}
\begin{cases}
a_2=1-\rho_{\ast R}(u_{1,\ast}-\sigma_R)\cdot\Phi_1^R, \\
b_2=\frac{u_{1,\ast}-\sigma_R}{3p_{11,\ast}}-\Phi_1^R, \\
d_2=L_{\rho}^R\cdot\rho_R'+L_{u_1}^R\cdot u_{1,R}'+L_{p_{11}}^R\cdot p_{11,R}'-\frac{1}{2}W_x(0)\cdot L_{W}^R,
\end{cases}
\end{equation}
and
\begin{align*}
&L_{\rho}^R=(\sigma_R-u_{1,R})\cdot\Phi_3^R, ~ L_{u_1}^R=\sigma_R-u_{1,R}-3p_{11,R}\cdot\Phi_2^R-\rho_R\cdot\Phi_3^R, \\
&L_{p_{11}}^R=-\frac{1}{\rho_R}+(\sigma_R-u_{1,R})\cdot\Phi_2^R, ~ L_{W}^R=-\rho_{\ast R}(u_{1,\ast}-\sigma_R)\cdot\Phi_1^R+1
\end{align*}
with $\Phi_i^R:=\Phi_i(p_{11,\ast};p_{11,R},\rho_R)$ {\rm($i=1,2,3$)}.
\end{lemma}


\begin{remark}
We now prove that
\begin{equation}\label{rho_t_astK=ast2K}
\left(\frac{\partial\rho}{\partial t}\right)_{\ast K}=\left(\frac{\partial\rho}{\partial t}\right)_{\ast\ast K}, ~ K=L,R.
\end{equation}

(i) If $u_{1,\ast}=0$, then by \eqref{eq:rho} and \eqref{u1_x}, one has $\frac{\partial\rho}{\partial t}=\frac{\rho}{3p_{11}}\frac{\mathcal{D}p_{11}}{\mathcal{D}t}$, and thus \eqref{rho_t_astK=ast2K} may be concluded from \eqref{Du1_p11_astK=ast2K}.

(ii) If $u_{1,\ast}\neq0$, then since $\rho$ is a generalized Riemann invariant for the 2-shear wave, one has
\[
u_{1,\ast}\left(\frac{\mathcal{D}_2\rho}{\mathcal{D}t}\right)_{\ast L}=u_{1,\ast}\left(\frac{\mathcal{D}_2\rho}{\mathcal{D}t}\right)_{\ast\ast L},
\]
thus by using \eqref{eq:rho}, \eqref{u1_x} and \eqref{Du1_p11_astK=ast2K}, one can have  \eqref{rho_t_astK=ast2K} for $K=L$.
The proof for $K=R$ is similar, because $\rho$ is also a generalized  Riemann invariant for the 5-shear wave.

In virtue of \eqref{rho_t_astK=ast2K}, henceforth, we will not distinguish $\left(\frac{\partial\rho}{\partial t}\right)_{\ast\ast K}$ from $\left(\frac{\partial\rho}{\partial t}\right)_{\ast K}$ ($K=L,R$) any more.
\end{remark}

\begin{theorem}[Computing $(\frac{\partial u_1}{\partial t})_{\ast K}$, $(\frac{\partial p_{11}}{\partial t})_{\ast K}$ and $(\frac{\partial\rho}{\partial t})_{\ast K}$ ($K=L,R$)]\label{theorem:u1_p11_rho_t_astK}
In the domain between the 1-rarefaction wave and the 6-shock wave, one has
\begin{equation}
\left(\frac{\mathcal{D}u_1}{\mathcal{D}t}\right)_\ast=\frac{d_1b_2-d_2b_1}{a_1b_2-a_2b_1}, \ \
 \left(\frac{\mathcal{D}p_{11}}{\mathcal{D}t}\right)_\ast=\frac{d_1a_2-d_2a_1}{a_2b_1-a_1b_2},
\label{Du1_Dp11_Dt_ast}
\end{equation}
where $a_1,b_1,d_1$ are given in \eqref{a1_b1_d1}  and $a_2,b_2,d_2$ are given in \eqref{a2_b2_d2}.
Then by \eqref{u1_p11_t}, one gets
\begin{equation}
\begin{split}
&\left(\frac{\partial u_1}{\partial t}\right)_{\ast K}=\frac{u_{1,\ast}}{3p_{11,\ast}}\left(\frac{\mathcal{D}p_{11}}{\mathcal{D}t}\right)_\ast+\left(\frac{\mathcal{D}u_1}{\mathcal{D}t}\right)_\ast, \\
&\left(\frac{\partial p_{11}}{\partial t}\right)_{\ast K}=\left(\frac{\mathcal{D}p_{11}}{\mathcal{D}t}\right)_\ast+\rho_{\ast K}u_{1,\ast}\left(\frac{\mathcal{D}u_1}{\mathcal{D}t}\right)_\ast+\frac{1}{2}\rho_{\ast K}u_{1,\ast}W_x(0)
\end{split}
\label{u1_p11_t_astK}
\end{equation}
with $K=L,R$. Moreover,
\begin{equation}\label{rho_t_astL}
\left(\frac{\partial\rho}{\partial t}\right)_{\ast L}=\frac{1}{c_{\ast L}^2}\left[\left(\frac{\partial p_{11}}{\partial t}\right)_{\ast L}+\rho_L^2S_{1,L}'\rho_{\ast L}u_{1,\ast}\frac{c_{\ast L}^3}{c_L^3}\right],
\end{equation}
and ${\rm \left(\frac{\partial\rho}{\partial t}\right)_{\ast R}}$ satisfies
\begin{equation}\label{rho_t_astR}
g_{\rho}^R\cdot\left(\frac{\partial\rho}{\partial t}\right)_{\ast R}+g_{u_1}^R\cdot\left(\frac{\mathcal{D}u_1}{\mathcal{D}t}\right)_\ast+g_{p_{11}}^R\cdot\left(\frac{\mathcal{D}p_{11}}{\mathcal{D}t}\right)_\ast=u_{1,\ast}\cdot f_R,
\end{equation}
where
\begin{align*}
&g_{\rho}^R=u_{1,\ast}-\sigma_R, ~ g_{u_1}^R=\rho_{\ast R}u_{1,\ast}(\sigma_R-u_{1,\ast})\cdot H_1^R, ~ g_{p_{11}}^R=\frac{\sigma_R}{c_{\ast R}^2}-u_{1,\ast}\cdot H_1^R, \\
&f_R=(\sigma_R-u_{1,R})\cdot H_2^R\cdot p_{11,R}'+(\sigma_R-u_{1,R})\cdot H_3^R\cdot\rho_{R}'-\rho_R(H_3^R+c_R^2\cdot H_2^R)\cdot u_{1,R}'-\frac{1}{2}\rho_{\ast R}(\sigma_R-u_{1,\ast})W_x(0)\cdot H_1^R
\end{align*}
with $H_i^R=H_i(p_{11,\ast};p_{11,R},\rho_R)$ {\rm ($i=1,2,3$)} and
\begin{equation}\label{H_123}
H_1=\frac{\partial H}{\partial p_{11}}=\frac{3\bar{\rho}\bar{p}_{11}}{(p_{11}+2\overline{p}_{11})^2}, \quad H_2=\frac{\partial H}{\partial\overline{p}_{11}}=-\frac{3\overline{\rho}p_{11}}{(p_{11}+2\overline{p}_{11})^2}, \quad H_3=\frac{\partial H}{\partial\overline{\rho}}=\frac{2p_{11}+\overline{p}_{11}}{p_{11}+2\overline{p}_{11}}.
\end{equation}
\end{theorem}

\begin{proof}
In virtue of \eqref{eq1:D_u1_p11_Dtast} and \eqref{eq2:D_u1_p11_Dtast}, the proofs of \eqref{Du1_Dp11_Dt_ast} and \eqref{u1_p11_t_astK} are direct.
At the 1-rarefaction side, due to \eqref{dp11}, one has
\[
\left(\frac{\partial p_{11}}{\partial t}\right)_{\ast L}=c_{\ast L}^2\left(\frac{\partial\rho}{\partial t}\right)_{\ast L}+\rho_{\ast L}^3\left(\frac{\partial S_1}{\partial t}\right)_{\ast L}.
\]
Combing it with the result obtained by taking $\beta=\beta_{\ast L}$ in \eqref{rho^2_S1_t_0_beta} implies \eqref{rho_t_astL}.

At the shock wave side, by taking $\Gamma=\rho-H(p_{11};\overline{p}_{11},\overline{\rho})$ in \eqref{D_sigema_gamma}, one gets
\begin{equation}\label{D_sigma_rho}
\frac{\partial\rho}{\partial t}+\sigma\frac{\partial\rho}{\partial x}=H_1\cdot\left(\frac{\partial p_{11}}{\partial t}+\sigma\frac{\partial p_{11}}{\partial x}\right)+H_2\cdot\left(\frac{\partial\overline{p}_{11}}{\partial t}+\sigma\frac{\partial\overline{p}_{11}}{\partial x}\right)+H_3\cdot\left(\frac{\partial\overline{\rho}}{\partial t}+\sigma\frac{\partial\overline{\rho}}{\partial x}\right).
\end{equation}
Then multiplying both sides of \eqref{D_sigma_rho} by $u_1$, and utilizing \eqref{eq:rho}, \eqref{eq:p11}, \eqref{u1_x} and \eqref{p11_x}, one obtains
\begin{align}
&(u_1-\sigma)\frac{\partial\rho}{\partial t}+\rho u_1(\sigma-u_1)H_1\frac{\mathcal{D}u_1}{\mathcal{D}t}+\left(\frac{\sigma}{c^2}-u_1 H_1\right)\frac{\mathcal{D}p_{11}}{\mathcal{D}t} \notag \\
=&u_1\cdot\left[(\sigma-\overline{u}_1)H_2\cdot\overline{p}_{11}'+(\sigma-\overline{u}_1)H_3\cdot\overline{\rho}'-\overline{\rho}(H_3+\overline{c}^2H_2)\cdot
\overline{u}_1'-\frac{1}{2}\rho(\sigma-u_1)W_xH_1\right], \label{shock_rho_t}
\end{align}
where $\overline{c}=\sqrt{\frac{3\overline{p}_{11}}{\overline{\rho}}}$. Specifically, for the 6-shock wave, \eqref{shock_rho_t} implies \eqref{rho_t_astR}.
\end{proof}

\begin{remark}
For other possible cases, computing $(\frac{\partial u_1}{\partial t})_{\ast K}$, $(\frac{\partial p_{11}}{\partial t})_{\ast K}$ and $(\frac{\partial\rho}{\partial t})_{\ast K}$ {\rm($K=L,R$)} is presented in \ref{u1_p11_rho_t_all_cases}.
\end{remark}

\subsubsection{Computing $(\frac{\partial u_2}{\partial t})_{\ast L}$, $(\frac{\partial p_{12}}{\partial t})_{\ast L}$ and $(\frac{\partial p_{22}}{\partial t})_{\ast L}$}\label{computing omega_astL}
This section gives the values of $(\frac{\partial u_2}{\partial t})_{\ast L}$, $(\frac{\partial p_{12}}{\partial t})_{\ast L}$ and $(\frac{\partial p_{22}}{\partial t})_{\ast L}$ by utilizing
the remaining three generalized Riemann invariants associated with the 1-rarefaction wave
\[
\psi_2:=u_2+\frac{\sqrt{3}p_{12}}{\sqrt{\rho p_{11}}}, \quad \psi_3:=\frac{p_{12}}{\rho^3}, \quad S_2:=\frac{\det(\mathbf{p})}{\rho^4}.
\]
From those,  one has the following three total differentials
\begin{align}
&\text{d}\psi_2=\text{d}u_2+\frac{\sqrt{3}}{\sqrt{\rho p_{11}}}\text{d}p_{12}-\frac{\sqrt{3}p_{12}}{2\rho\sqrt{\rho p_{11}}}\text{d}\rho-\frac{\sqrt{3}p_{12}}{2p_{11}\sqrt{\rho p_{11}}}\text{d}p_{11}, \label{dpsi2} \\
&\text{d}\psi_3=\frac{1}{\rho^3}\text{d}p_{12}-\frac{3p_{12}}{\rho^4}\text{d}\rho, \label{dpsi3} \\
&\text{d}S_2=\frac{p_{11}}{\rho^4}\text{d}p_{22}+\frac{p_{22}}{\rho^4}\text{d}p_{11}-\frac{2p_{12}}{\rho^4}\text{d}p_{12}-\frac{4\det(\mathbf{p})}{\rho^5}\text{d}\rho. \label{dS2}
\end{align}
Combing them with \eqref{eq:rho} and \eqref{eq:u2}-\eqref{eq:p22} gives
\[
\frac{\partial\psi_2}{\partial t}+(u_1+c)\frac{\partial\psi_2}{\partial x}=\Pi_2, \quad
\frac{\partial\psi_3}{\partial t}+(u_1+c)\frac{\partial\psi_3}{\partial x}=\Pi_3, \quad
\frac{\partial S_2}{\partial t}+u_1\frac{\partial S_2}{\partial x}=0,
\]
where
\begin{align}
&\Pi_2:=\frac{2}{\rho}\frac{\partial p_{12}}{\partial x}-\frac{3p_{12}}{2\rho^2}\frac{\partial\rho}{\partial x}-\frac{3p_{12}}{2\rho p_{11}}\frac{\partial p_{11}}{\partial x}, \label{Pi2} \\
&\Pi_3:=\frac{c}{\rho^3}\frac{\partial p_{12}}{\partial x}+\frac{p_{12}}{\rho^3}\frac{\partial u_1}{\partial x}-\frac{p_{11}}{\rho^3}\frac{\partial u_2}{\partial x}-\frac{3p_{12}c}{\rho^4}\frac{\partial\rho}{\partial x}. \label{Pi3}
\end{align}
Similar to the derivation of \eqref{psi1_alpha_0beta}, for $\psi_2$, one can also obtain
\begin{equation}\label{psi2_alpha_0beta}
\frac{\partial\psi_2}{\partial\alpha}(0,\beta)=\frac{\partial\psi_2}{\partial\alpha}(0,\beta_L)+\int_{\beta_L}^{\beta}\frac{1}{2c(0,\eta)}\frac{\partial t}{\partial\alpha}(0,\eta)\cdot\Pi_2(0,\eta)\text{d}\eta,
\end{equation}
where
\begin{equation*}\label{psi2_alpha_0betaL}
\frac{\partial\psi_2}{\partial\alpha}(0,\beta_L)=\frac{\partial t}{\partial\alpha}(0,\beta_L)\cdot(\Pi_{2,L}-2c_L\psi_{2,L}'),
\end{equation*}
with $\Pi_{2,L}$ and $\psi_{2,L}'$ being given by \eqref{Pi2} and \eqref{dpsi2}, respectively.
Similar to \eqref{psi1_t_0beta}, one has
\begin{equation}\label{psi2_t_0beta}
\frac{\partial\psi_2}{\partial t}(0,\beta)=-\left(\frac{u_1-c}{2c}\right)(0,\beta)\cdot\Pi_2(0,\beta)+\left(\frac{u_1+c}{2c}\right)(0,\beta)\cdot\left(\frac{\partial t}{\partial\alpha}\right)^{-1}(0,\beta)\cdot\frac{\partial\psi_2}{\partial\alpha}(0,\beta).
\end{equation}
Note that if one wants to get the explicit expression of $\frac{\partial\psi_2}{\partial t}(0,\beta)$ in a way similar to the derivation of \eqref{psi1_t_0_beta} in Lemma \ref{S1_t_psi1_t_0beta}, one has to first obtain the explicit expression of $\Pi_2(0,\beta)$ so as to exactly obtain the integral in \eqref{psi2_alpha_0beta}.
Such task seems to be difficultly completed since it is hard to explicitly represent the spatial derivatives of the primitive variables with the characteristic coordinate $(\alpha,\beta)$.
In practice, it may not be necessary to exactly get the integral in \eqref{psi2_alpha_0beta}
and a second-order numerical approximation to the integral is enough.
Specifically, applying the trapezoidal rule in \eqref{psi2_alpha_0beta}  obtains
\[
\frac{\partial\psi_2}{\partial\alpha}(0,\beta)=\frac{\partial\psi_2}{\partial\alpha}(0,\beta_L)+
\frac{\beta-\beta_L}{2}\left[\frac{1}{4c(0,\beta)}\Pi_2(0,\beta)+\frac{1}{4c_L}\Pi_{2,L}\right].
\]
Substituting it into \eqref{psi2_t_0beta}  gets
\begin{equation}\label{psi2_t_0beta_2}
\frac{\partial\psi_2}{\partial t}(0,\beta)=\left[\frac{\partial\psi_2}{\partial\alpha}(0,\beta_L)+\frac{\Pi_{2,L}}{8c_L}(\beta-\beta_L)\right]\psi_{1,L}
+\frac{1}{2c(0,\beta)}\left(\frac{\beta-\beta_L}{4}\psi_{1,L}-\beta\right)\cdot\Pi_2(0,\beta).
\end{equation}
Moreover, by using \eqref{dp11}, \eqref{Pi2} is rewritten into
\[
\Pi_2=\frac{2}{\rho}\frac{\partial p_{12}}{\partial x}-\frac{2p_{12}}{\rho p_{11}}\frac{\partial p_{11}}{\partial x}+\frac{\rho^2p_{12}}{2p_{11}}\frac{\partial S_1}{\partial x}.
\]
Substituting \eqref{p11_x} and \eqref{p12_x} into above expression yields
\begin{equation}\label{Pi2_2}
\Pi_2=-2\frac{\mathcal{D}u_2}{\mathcal{D}t}+\frac{2p_{12}}{p_{11}}\left(\frac{\mathcal{D}u_1}{\mathcal{D}t}+\frac{1}{2}W_x\right)+\frac{\rho^2p_{12}}{2p_{11}}\frac{\partial S_1}{\partial x}.
\end{equation}

\begin{lemma}[Resolution of the 1-rarefaction wave]
Assume that the 1-rarefaction wave associated with $u_1-c$ moves to the left. Then $\frac{\mathcal{D}u_2}{\mathcal{D}t}(0,\beta)$ and $\frac{\mathcal{D}p_{12}}{\mathcal{D}t}(0,\beta)$ satisfy
\begin{equation}\label{eq1:u2_t_p12_t}
\widetilde{a}_{3,L}(0,\beta)\cdot\frac{\mathcal{D}u_2}{\mathcal{D}t}(0,\beta)+\widetilde{b}_{3,L}(0,\beta)\cdot
\frac{\mathcal{D}p_{12}}{\mathcal{D}t}(0,\beta)=\widetilde{d}_{3,L}(0,\beta),
\end{equation}
where
\begin{align*}
&\widetilde{a}_{3,L}(0,\beta)=1+\left(\frac{\sqrt{3}\rho u_1}{\sqrt{\rho p_{11}}}\right)(0,\beta)+\frac{1}{c(0,\beta)}\left(\frac{\beta-\beta_L}{4}\psi_{1,L}-\beta\right), \\
&\widetilde{b}_{3,L}(0,\beta)=\left(\frac{u_1}{p_{11}}+\frac{\sqrt{3}}{\sqrt{\rho p_{11}}}\right)(0,\beta), \\
&\widetilde{d}_{3,L}(0,\beta)=\left[\frac{\partial\psi_2}{\partial\alpha}(0,\beta_L)+\frac{\Pi_{2,L}}{8c_L}(\beta-\beta_L)\right]\psi_{1,L} +\frac{1}{2c(0,\beta)}\left(\frac{\beta-\beta_L}{4}\psi_{1,L}-\beta\right)\left\{\left(\frac{2p_{12}}{p_{11}}\right)(0,\beta)
\right. \\
&
\quad\quad\quad\left.
\cdot\left[\left(\frac{\mathcal{D}u_1}{\mathcal{D}t}\right)(0,\beta)+\frac{1}{2}W_x(0)\right]
+\frac{S_{1,L}'}{c_L}\left(\frac{c\rho^2p_{12}}{2p_{11}}\right)(0,\beta)\right\}
+\left(\frac{\sqrt{3}p_{12}}{2\rho\sqrt{\rho p_{11}}}\frac{\partial\rho}{\partial t}\right)(0,\beta)
\\
& \quad\quad\quad+\left(\frac{\sqrt{3}p_{12}}{2p_{11}\sqrt{\rho p_{11}}}\frac{\partial p_{11}}{\partial t}\right)(0,\beta)
+\left(\frac{2u_1p_{12}}{3p_{11}^2}\frac{\mathcal{D}p_{11}}{\mathcal{D}t}\right)(0,\beta).
\end{align*}
\end{lemma}

\begin{proof}
Due to \eqref{dpsi2}, one has
\[
\frac{\partial u_2}{\partial t}+\frac{\sqrt{3}}{\sqrt{\rho p_{11}}}\frac{\partial p_{12}}{\partial t}
=\frac{\partial\psi_2}{\partial t}+\frac{\sqrt{3}p_{12}}{2\rho\sqrt{\rho p_{11}}}\frac{\partial\rho}{\partial t}
+\frac{\sqrt{3}p_{12}}{2p_{11}\sqrt{\rho p_{11}}}\frac{\partial p_{11}}{\partial t}.
\]
Substituting \eqref{u2_p12_t} into the above expression gives
\begin{equation}
\left(1+\frac{\sqrt{3}\rho u_1}{\sqrt{\rho p_{11}}}\right)\frac{\mathcal{D}u_2}{\mathcal{D}t}+\left(\frac{u_1}{p_{11}}+\frac{\sqrt{3}}{\sqrt{\rho p_{11}}}\right)\frac{\mathcal{D}p_{12}}{\mathcal{D}t}-\frac{2u_1p_{12}}{3p_{11}^2}\frac{\mathcal{D}p_{11}}{\mathcal{D}t}
=\frac{\partial\psi_2}{\partial t}+\frac{\sqrt{3}p_{12}}{2\rho\sqrt{\rho p_{11}}}\frac{\partial\rho}{\partial t}
+\frac{\sqrt{3}p_{12}}{2p_{11}\sqrt{\rho p_{11}}}\frac{\partial p_{11}}{\partial t}. \label{wjf:tmp1}
\end{equation}
Taking the limiting value at $(0,\beta)$ in \eqref{wjf:tmp1}, and combining it with \eqref{psi2_t_0beta_2} and \eqref{Pi2_2}, one obtains \eqref{eq1:u2_t_p12_t}.  Note that in the expression of $\widetilde{d}_{3,L}(0,\beta)$, \eqref{S1_x_0beta} has been used.
\end{proof}

Taking $\beta=\beta_{\ast L}$ in \eqref{eq1:u2_t_p12_t}, one obtains the first equation that $\left(\frac{\mathcal{D}u_2}{\mathcal{D}t}\right)_{\ast L}$ and $\left(\frac{\mathcal{D}p_{12}}{\mathcal{D}t}\right)_{\ast L}$ satisfy
\begin{equation}\label{eq1:u2_t_p12_t_astL}
a_{3,L}\left(\frac{\mathcal{D}u_2}{\mathcal{D}t}\right)_{\ast L}+b_{3,L}\left(\frac{\mathcal{D}p_{12}}{\mathcal{D}t}\right)_{\ast L}=d_{3,L},
\end{equation}
with $a_{3,L}=\widetilde{a}_{3,L}(0,\beta_{\ast L})$, $b_{3,L}=\widetilde{b}_{3,L}(0,\beta_{\ast L})$ and $d_{3,L}=\widetilde{d}_{3,L}(0,\beta_{\ast L})$.

For $\psi_3$, similar to $\psi_2$, one can get the following results
\begin{align}
&\frac{\partial\psi_3}{\partial\alpha}(0,\beta)=\frac{\partial\psi_3}{\partial\alpha}(0,\beta_L)+\int_{\beta_L}^{\beta}\frac{1}{2c(0,\eta)}\frac{\partial t}{\partial\alpha}(0,\eta)\cdot\Pi_3(0,\eta)\text{d}\eta, \label{psi3_alpha_0beta} \\
&\frac{\partial\psi_3}{\partial t}(0,\beta)=-\left(\frac{u_1-c}{2c}\right)(0,\beta)\cdot\Pi_3(0,\beta)+\left(\frac{u_1+c}{2c}\right)(0,\beta)\cdot\left(\frac{\partial t}{\partial\alpha}\right)^{-1}(0,\beta)\cdot\frac{\partial\psi_3}{\partial\alpha}(0,\beta), \label{psi3_t_0beta}
\end{align}
where
\[
\frac{\partial\psi_3}{\partial\alpha}(0,\beta_L)=\frac{\partial t}{\partial\alpha}(0,\beta_L)\cdot(\Pi_{3,L}-2c_L\psi_{3,L}'),
\]
with $\Pi_{3,L}$ and $\psi_{3,L}'$ being given by \eqref{Pi3} and \eqref{dpsi3}, respectively. Applying the trapezoidal rule in \eqref{psi3_alpha_0beta}  obtains
\[
\frac{\partial\psi_3}{\partial\alpha}(0,\beta)=\frac{\partial\psi_3}{\partial\alpha}(0,\beta_L)+
\frac{\beta-\beta_L}{2}\left[\frac{1}{4c(0,\beta)}\Pi_3(0,\beta)+\frac{1}{4c_L}\Pi_{3,L}\right].
\]
Substituting it into \eqref{psi3_t_0beta}  gets
\begin{equation}\label{psi3_t_0beta_2}
\frac{\partial\psi_3}{\partial t}(0,\beta)=\left[\frac{\partial\psi_3}{\partial\alpha}(0,\beta_L)+\frac{\Pi_{3,L}}{8c_L}(\beta-\beta_L)\right]\psi_{1,L}
+\frac{1}{2c(0,\beta)}\left(\frac{\beta-\beta_L}{4}\psi_{1,L}-\beta\right)\cdot\Pi_3(0,\beta).
\end{equation}
Besides, by using \eqref{dp11}, \eqref{p12_x}, \eqref{u1_x}, \eqref{u2_x} and \eqref{p11_x}, one can rewrite \eqref{Pi3} into
\begin{equation}
\Pi_3=\frac{1}{\rho^3}\left[-\rho c\frac{\mathcal{D}u_2}{\mathcal{D}t}-\frac{p_{12}}{p_{11}}\frac{\mathcal{D}p_{11}}{\mathcal{D}t}
+\frac{\mathcal{D}p_{12}}{\mathcal{D}t}+\frac{3p_{12}}{c}\left(\frac{\mathcal{D}u_1}{\mathcal{D}t}+\frac{1}{2}W_x\right)\right]+\frac{3p_{12}}{\rho c}\frac{\partial S_1}{\partial x}. \label{Pi3_2}
\end{equation}

\begin{lemma}[Resolution of the 1-rarefaction wave]
Assume that the 1-rarefaction wave associated with $u_1-c$ moves to the left. Then $\frac{\mathcal{D}u_2}{\mathcal{D}t}(0,\beta)$ and $\frac{\mathcal{D}p_{12}}{\mathcal{D}t}(0,\beta)$ satisfy
\begin{equation}\label{eq2:u2_t_p12_t}
\widetilde{a}_{4,L}(0,\beta)\cdot\frac{\mathcal{D}u_2}{\mathcal{D}t}(0,\beta)+
\widetilde{b}_{4,L}(0,\beta)\cdot\frac{\mathcal{D}p_{12}}{\mathcal{D}t}(0,\beta)=\widetilde{d}_{4,L}(0,\beta),
\end{equation}
where
\begin{align*}
&\widetilde{a}_{4,L}(0,\beta)=\frac{1}{\rho^2(0,\beta)}\left[u_1(0,\beta)+\frac{1}{2}\left(\frac{\beta-\beta_L}{4}\psi_{1,L}-\beta\right)\right], \\
&\widetilde{b}_{4,L}(0,\beta)=\frac{1}{\rho^3(0,\beta)}\left[1-\frac{1}{2c(0,\beta)}\left(\frac{\beta-\beta_L}{4}\psi_{1,L}-\beta\right)\right], \\
&\widetilde{d}_{4,L}(0,\beta)=\left[\frac{\partial\psi_3}{\partial\alpha}(0,\beta_L)+\frac{\Pi_{3,L}}{8c_L}(\beta-\beta_L)\right]\psi_{1,L}+\left(\frac{3p_{12}}{\rho^4}\frac{\partial\rho}{\partial t}\right)(0,\beta)
+\frac{1}{2c(0,\beta)}\left(\frac{\beta-\beta_L}{4}\psi_{1,L}-\beta\right)
\\ &
\qquad \quad \cdot \left\{\frac{1}{\rho^3(0,\beta)}\left[-\frac{p_{12}}{p_{11}}\frac{\mathcal{D}p_{11}}{\mathcal{D}t}\right.
\left.+\frac{3p_{12}}{c}\left(\frac{\mathcal{D}u_1}{\mathcal{D}t}+\frac{1}{2}W_x(0)\right)\right](0,\beta)
+\frac{S_{1,L}'}{c_L}\left(\frac{3p_{12}}{\rho}\right)(0,\beta)\right\}.
\end{align*}
\end{lemma}
\begin{proof}
By \eqref{dpsi3} and \eqref{u2_p12_t}, one has
\[
\frac{\partial\psi_3}{\partial t}=\frac{1}{\rho^3}\left(\frac{\mathcal{D}p_{12}}{\mathcal{D}t}+\rho u_1\frac{\mathcal{D}u_2}{\mathcal{D}t}\right)-\frac{3p_{12}}{\rho^4}\frac{\partial\rho}{\partial t}.
\]
Combining it with \eqref{Pi3_2} and \eqref{psi3_t_0beta_2} gives \eqref{eq2:u2_t_p12_t}. Similarly,
in the expression of $\widetilde{d}_{4,L}(0,\beta)$, \eqref{S1_x_0beta} has been used.
\end{proof}

Taking $\beta=\beta_{\ast L}$ in \eqref{eq2:u2_t_p12_t}, one obtains the second equation that $\left(\frac{\mathcal{D}u_2}{\mathcal{D}t}\right)_{\ast L}$ and $\left(\frac{\mathcal{D}p_{12}}{\mathcal{D}t}\right)_{\ast L}$ satisfy
\begin{equation}\label{eq2:u2_t_p12_t_astL}
a_{4,L}\left(\frac{\mathcal{D}u_2}{\mathcal{D}t}\right)_{\ast L}+b_{4,L}\left(\frac{\mathcal{D}p_{12}}{\mathcal{D}t}\right)_{\ast L}=d_{4,L},
\end{equation}
with $a_{4,L}=\widetilde{a}_{4,L}(0,\beta_{\ast L})$, $b_{4,L}=\widetilde{b}_{4,L}(0,\beta_{\ast L})$ and $d_{4,L}=\widetilde{d}_{4,L}(0,\beta_{\ast L})$.

For $S_2$, with the similar derivation of \eqref{S1_t_0beta}, one can obtain
\begin{equation}\label{S2_t_0beta}
\frac{\partial S_2}{\partial t}(0,\beta)=-\frac{S_{2,L}'}{c_L}(cu_1)(0,\beta),
\end{equation}
with $S_{2,L}'$ being given by \eqref{dS2}.
Moreover, by \eqref{dS2}, one has
\begin{equation}\label{p22_t_0beta}
\left(\frac{p_{11}}{\rho^4}\frac{\partial p_{22}}{\partial t}\right)(0,\beta)=\frac{\partial S_2}{\partial t}(0,\beta)-\left(\frac{p_{22}}{\rho^4}\frac{\partial p_{11}}{\partial t}\right)(0,\beta)+\left(\frac{2p_{12}}{\rho^4}\frac{\partial p_{12}}{\partial t}\right)(0,\beta)+\left(\frac{4\det(\mathbf{p})}{\rho^5}\frac{\partial\rho}{\partial t}\right)(0,\beta).
\end{equation}

\begin{theorem}[Computing $(\frac{\partial u_2}{\partial t})_{\ast L}$, $(\frac{\partial p_{12}}{\partial t})_{\ast L}$ and $(\frac{\partial p_{22}}{\partial t})_{\ast L}$]\label{theorem for omega_astL}
Assume that the 1-rarefaction wave associated with $u_1-c$ move to the left. Then one has
\begin{equation}
\left(\frac{\mathcal{D}u_2}{\mathcal{D}t}\right)_{\ast L}=\frac{d_{3,L}b_{4,L}-d_{4,L}b_{3,L}}{a_{3,L}b_{4,L}-a_{4,L}b_{3,L}}, \ \
\left(\frac{\mathcal{D}p_{12}}{\mathcal{D}t}\right)_{\ast L}=\frac{d_{3,L}a_{4,L}-d_{4,L}a_{3,L}}{b_{3,L}a_{4,L}-b_{4,L}a_{3,L}},
\label{Du2_Dp12_t_astL}
\end{equation}
and
\begin{equation}
\begin{split}
&\left(\frac{\partial u_2}{\partial t}\right)_{\ast L}=\left(\frac{\mathcal{D}u_2}{\mathcal{D}t}\right)_{\ast L}+\frac{u_{1,\ast}}{p_{11,\ast}}\left[\left(\frac{\mathcal{D}p_{12}}{\mathcal{D}t}\right)_{\ast L}-\frac{2p_{12,\ast L}}{3p_{11,\ast}}\left(\frac{\mathcal{D}p_{11}}{\mathcal{D}t}\right)_{\ast}\right], \\
&\left(\frac{\partial p_{12}}{\partial t}\right)_{\ast L}=\left(\frac{\mathcal{D}p_{12}}{\mathcal{D}t}\right)_{\ast L}+\rho_{\ast L}u_{1,\ast}\left(\frac{\mathcal{D}u_2}{\mathcal{D}t}\right)_{\ast L}.
\end{split}
\label{u2_p12_t_astL}
\end{equation}
The limiting value $\left(\frac{\partial p_{22}}{\partial t}\right)_{\ast L}$ can be obtained by
\begin{equation}\label{p22_t_astL}
\left(\frac{p_{11}}{\rho^4}\frac{\partial p_{22}}{\partial t}\right)_{\ast L}=-\frac{S_{2,L}'}{c_L}c_{\ast L}u_{1,\ast}-
\left(\frac{p_{22}}{\rho^4}\frac{\partial p_{11}}{\partial t}\right)_{\ast L}+\left(\frac{2p_{12}}{\rho^4}\frac{\partial p_{12}}{\partial t}\right)_{\ast L}+\left(\frac{4\det(\mathbf{p})}{\rho^5}\frac{\partial\rho}{\partial t}\right)_{\ast L}.
\end{equation}
\end{theorem}

\begin{proof}
Combining \eqref{eq1:u2_t_p12_t_astL} and \eqref{eq2:u2_t_p12_t_astL} yields \eqref{Du2_Dp12_t_astL},
and then  can get \eqref{u2_p12_t_astL} by \eqref{u2_p12_t}. Substituting \eqref{S2_t_0beta} into \eqref{p22_t_0beta} and taking $\beta=\beta_{\ast L}$  obtain  \eqref{p22_t_astL}.
\end{proof}

\begin{remark}
For the 1-shock wave, computing $(\frac{\partial u_2}{\partial t})_{\ast L}$, $(\frac{\partial p_{12}}{\partial t})_{\ast L}$ and $(\frac{\partial p_{22}}{\partial t})_{\ast L}$ is presented in \ref{u2_p12_p22_astL_1shock}.
\end{remark}

\subsubsection{Computing $(\frac{\partial u_2}{\partial t})_{\ast R}$, $(\frac{\partial p_{12}}{\partial t})_{\ast R}$ and $(\frac{\partial p_{22}}{\partial t})_{\ast R}$}\label{compute omega_astR}

At the 6-shock wave side, the Rankine-Hugoniot jump conditions will be used to derive the limiting values $(\frac{\partial u_2}{\partial t})_{\ast R}$, $(\frac{\partial p_{12}}{\partial t})_{\ast R}$ and $(\frac{\partial p_{22}}{\partial t})_{\ast R}$.

Across the shock wave with speed $\sigma$, the  Rankine-Hugoniot jump condition for $m_2$ implies
\[
\rho u_2(u_1-\sigma)+p_{12}=\bar{\rho}\bar{u}_2(\overline{u}_1-\sigma)+\overline{p}_{12}.
\]
If denoting $\Gamma_{m_2}:=\rho u_2(u_1-\sigma)+p_{12}$, one has $\Gamma_{m_2}=\Gamma_{\overline{m}_2}$ and
\begin{equation*}\label{d_gamma_m2}
\text{d}\Gamma_{m_2}=u_2(u_1-\sigma)\text{d}\rho+\rho u_2\text{d}u_1+\rho(u_1-\sigma)\text{d}u_2+\text{d}p_{12}-\rho u_2\text{d}\sigma.
\end{equation*}
It follows that
\begin{equation*}\label{Dsigma_gamma_m2}
\frac{\mathcal{D}_{\sigma}\Gamma_{m_2}}{\mathcal{D}t}=u_2(u_1-\sigma)\frac{\mathcal{D}_{\sigma}\rho}{\mathcal{D}t}+\rho u_2\frac{\mathcal{D}_{\sigma}u_1}{\mathcal{D}t}+\rho(u_1-\sigma)\frac{\mathcal{D}_{\sigma}u_2}{\mathcal{D}t}+\frac{\mathcal{D}_{\sigma}p_{12}}{\mathcal{D}t}
-\rho u_2\frac{\mathcal{D}_{\sigma}\sigma}{\mathcal{D}t},
\end{equation*}
and then one can utilize the relation along the shock wave, i.e., $\frac{\mathcal{D}_{\sigma}\Gamma_{m_2}}{\mathcal{D}t}=\frac{\mathcal{D}_{\sigma}\Gamma_{\overline{m}_2}}{\mathcal{D}t}$, to get the first equation that the limiting values $\left(\frac{\mathcal{D}u_2}{\mathcal{D}t}\right)_{\ast K}$ and $\left(\frac{\mathcal{D}p_{12}}{\mathcal{D}t}\right)_{\ast K}$ ($K=L,R$) satisfy.

Specifically, for the 6-shock wave, one has
\begin{equation}\label{DsigmaR_gamma_m2astR_m2R}
\frac{\mathcal{D}_{\sigma_R}\Gamma_{m_{2,\ast R}}}{\mathcal{D}t}=\frac{\mathcal{D}_{\sigma_R}\Gamma_{m_{2,R}}}{\mathcal{D}t},
\end{equation}
where
\begin{align}
&\sigma_R=u_{1,R}+\left(\frac{2p_{11,\ast}+p_{11,R}}{\rho_R}\right)^{\frac{1}{2}}, \label{sigmaR} \\
&\frac{\mathcal{D}_{\sigma_R}\Gamma_{m_{2,\ast R}}}{\mathcal{D}t}=u_{2,\ast R}(u_{1,\ast}-\sigma_R)\left(\frac{\mathcal{D}_{\sigma_R}\rho}{\mathcal{D}t}\right)_{\ast R}+\rho_{\ast R}u_{2,\ast R}\left(\frac{\mathcal{D}_{\sigma_R}u_1}{\mathcal{D}t}\right)_{\ast R} \notag \\
&\quad\quad+\rho_{\ast R}(u_{1,\ast}-\sigma_R)\left(\frac{\mathcal{D}_{\sigma_R}u_2}{\mathcal{D}t}\right)_{\ast R}+\left(\frac{\mathcal{D}_{\sigma_R}p_{12}}{\mathcal{D}t}\right)_{\ast R}-\rho_{\ast R}u_{2,\ast R}\frac{\mathcal{D}_{\sigma_R}\sigma_R}{\mathcal{D}t}, \label{DsigmaR_gamma_m2astR} \\
&\frac{\mathcal{D}_{\sigma_R}\Gamma_{m_{2,R}}}{\mathcal{D}t}=u_{2, R}(u_{1,R}-\sigma_R)\frac{\mathcal{D}_{\sigma_R}\rho_R}{\mathcal{D}t}+\rho_{R}u_{2, R}\frac{\mathcal{D}_{\sigma_R}u_{1,R}}{\mathcal{D}t} \notag \\
&\quad\quad+\rho_{R}(u_{1,R}-\sigma_R)\frac{\mathcal{D}_{\sigma_R}u_{2,R}}{\mathcal{D}t}+\frac{\mathcal{D}_{\sigma_R}p_{12,R}}{\mathcal{D}t}-\rho_Ru_{2, R}\frac{\mathcal{D}_{\sigma_R}\sigma_R}{\mathcal{D}t}. \label{DsigmaR_gamma_m2R}
\end{align}
By \eqref{u1_x}, one has
\begin{equation}\label{DsigmaR_u1_astR}
\left(\frac{\mathcal{D}_{\sigma_R}u_1}{\mathcal{D}t}\right)_{\ast R}=\left(\frac{\partial u_1}{\partial t}\right)_{\ast R}-\frac{\sigma_R}{3p_{11,\ast}}\left(\frac{\mathcal{D}p_{11}}{\mathcal{D}t}\right)_{\ast}.
\end{equation}
By \eqref{u2_x}, one gets
\begin{equation}\label{DsigmaR_u2_astR}
\left(\frac{\mathcal{D}_{\sigma_R}u_2}{\mathcal{D}t}\right)_{\ast R}=\left(\frac{\mathcal{D}u_2}{\mathcal{D}t}\right)_{\ast R}-\frac{\sigma_R-u_{1,\ast}}{p_{11,\ast}}\left[\left(\frac{\mathcal{D}p_{12}}{\mathcal{D}t}\right)_{\ast R}-\frac{2p_{12,\ast R}}{3p_{11,\ast}}\left(\frac{\mathcal{D}p_{11}}{\mathcal{D}t}\right)_{\ast}\right].
\end{equation}
By \eqref{p12_x}, one obtains
\begin{equation}\label{DsigmaR_p12_astR}
\left(\frac{\mathcal{D}_{\sigma_R}p_{12}}{\mathcal{D}t}\right)_{\ast R}=\left(\frac{\mathcal{D}p_{12}}{\mathcal{D}t}\right)_{\ast R}-\rho_{\ast R}(\sigma_R-u_{1,\ast})\left(\frac{\mathcal{D}u_2}{\mathcal{D}t}\right)_{\ast R}.
\end{equation}
By \eqref{quasilinear_form} and \eqref{sigmaR}, one has
\begin{subequations}\label{DsigmaR_VR}
\begin{align}
&\frac{\mathcal{D}_{\sigma_R}\rho_R}{\mathcal{D}t}=(\sigma_R-u_{1,R})\rho_R'-\rho_Ru_{1,R}', \label{DsigmaR_rhoR} \\
&\frac{\mathcal{D}_{\sigma_R}u_{1,R}}{\mathcal{D}t}=(\sigma_R-u_{1,R})u_{1,R}'-\frac{1}{\rho_R}p_{11,R}'-\frac{1}{2}W_x(0), \label{DsigmaR_u1R} \\
&\frac{\mathcal{D}_{\sigma_R}u_{2,R}}{\mathcal{D}t}=(\sigma_R-u_{1,R})u_{2,R}'-\frac{1}{\rho_R}p_{12,R}', \label{DsigmaR_u2R} \\
&\frac{\mathcal{D}_{\sigma_R}p_{11,R}}{\mathcal{D}t}=(\sigma_R-u_{1,R})p_{11,R}'-3p_{11,R}u_{1,R}', \label{DsigmaR_p11R} \\
&\frac{\mathcal{D}_{\sigma_R}p_{12,R}}{\mathcal{D}t}=(\sigma_R-u_{1,R})p_{12,R}'-2p_{12,R}u_{1,R}'-p_{11,R}u_{2,R}', \label{DsigmaR_p12R} \\
& \frac{\mathcal{D}_{\sigma_R}\sigma_R}{\mathcal{D}t}=\frac{\mathcal{D}_{\sigma_R}u_{1,R}}{\mathcal{D}t}-\frac{\sqrt{2p_{11,\ast}+p_{11,R}}}{2\rho_R\sqrt{\rho_R}}
\frac{\mathcal{D}_{\sigma_R}\rho_R}{\mathcal{D}t}+\frac{1}{2\sqrt{\rho_R(2p_{11,\ast}+p_{11,R})}}\frac{\mathcal{D}_{\sigma_R}p_{11,R}}{\mathcal{D}t}
+\frac{1}{\sqrt{\rho_R(2p_{11,\ast}+p_{11,R})}}\frac{\mathcal{D}_{\sigma_R}p_{11,\ast}}{\mathcal{D}t}. \label{DsigmaR_sigmaR}
\end{align}
\end{subequations}
By \eqref{p11_x}, one obtains
\begin{equation}\label{DsigmaR_p11ast}
\frac{\mathcal{D}_{\sigma_R}p_{11,\ast}}{\mathcal{D}t}=\left(\frac{\mathcal{D}p_{11}}{\mathcal{D}t}\right)_{\ast}-\rho_{\ast R}(\sigma_R-u_{1,\ast})\left[\left(\frac{\mathcal{D}u_1}{\mathcal{D}t}\right)_{\ast}+\frac{1}{2}W_x(0)\right].
\end{equation}
Substituting \eqref{DsigmaR_rhoR}, \eqref{DsigmaR_u1R}, \eqref{DsigmaR_p11R} and \eqref{DsigmaR_p11ast} into \eqref{DsigmaR_sigmaR}  can obtain the value of $\frac{\mathcal{D}_{\sigma_R}\sigma_R}{\mathcal{D}t}$.   Similarly, substituting \eqref{DsigmaR_rhoR}, \eqref{DsigmaR_u1R}, \eqref{DsigmaR_u2R}, \eqref{DsigmaR_p12R} and \eqref{DsigmaR_sigmaR} into \eqref{DsigmaR_gamma_m2R} may get the value of $\frac{\mathcal{D}_{\sigma_R}\Gamma_{m_{2,R}}}{\mathcal{D}t}$.
As for the computation of $\left(\frac{\mathcal{D}_{\sigma_R}\rho}{\mathcal{D}t}\right)_{\ast R}$, one uses the relation across the shock wave, i.e., $\rho=H(p_{11};\overline{p}_{11},\overline{\rho})$, to obtain
\[
\frac{\mathcal{D}_{\sigma}\rho}{\mathcal{D}t}=H_1\cdot\frac{\mathcal{D}_{\sigma}p_{11}}{\mathcal{D}t}
+H_2\cdot\frac{\mathcal{D}_{\sigma}\overline{p}_{11}}{\mathcal{D}t}+H_3\cdot\frac{\mathcal{D}_{\sigma}\overline{\rho}}{\mathcal{D}t}.
\]
Applying it to the 6-shock wave  gets
\begin{equation}\label{DsigmaR_rhoastR}
\left(\frac{\mathcal{D}_{\sigma_R}\rho}{\mathcal{D}t}\right)_{\ast R}=H_1^R\cdot\frac{\mathcal{D}_{\sigma_R}p_{11,\ast}}{\mathcal{D}t}+H_2^R\cdot\frac{\mathcal{D}_{\sigma_R}p_{11,R}}{\mathcal{D}t}
+H_3^R\cdot\frac{\mathcal{D}_{\sigma_R}\rho_R}{\mathcal{D}t}.
\end{equation}

\begin{lemma}[Resolution of the 6-shock wave]
Assume that the 6-shock wave associated with $u_1+c$ moves to the right. Then the limiting values $\left(\frac{\mathcal{D}u_2}{\mathcal{D}t}\right)_{\ast R}$ and $\left(\frac{\mathcal{D}p_{12}}{\mathcal{D}t}\right)_{\ast R}$ satisfy
\begin{equation}\label{eq1:Du2_Dp12_astR}
a_{3,R}\left(\frac{\mathcal{D}u_2}{\mathcal{D}t}\right)_{\ast R}+b_{3,R}\left(\frac{\mathcal{D}p_{12}}{\mathcal{D}t}\right)_{\ast R}=d_{3,R},
\end{equation}
where
\begin{align*}
&a_{3,R}=2\rho_{\ast R}(u_{1,\ast}-\sigma_R), \quad b_{3,R}=1+\frac{\rho_{\ast R}(u_{1,\ast}-\sigma_R)^2}{p_{11,\ast}}, \\
&d_{3,R}=\frac{\mathcal{D}_{\sigma_R}\Gamma_{m_{2,R}}}{\mathcal{D}t}+\rho_{\ast R}u_{2,\ast R}\frac{\mathcal{D}_{\sigma_R}\sigma_R}{\mathcal{D}t}-u_{2,\ast R}(u_{1,\ast}-\sigma_R)\left(\frac{\mathcal{D}_{\sigma_R}\rho}{\mathcal{D}t}\right)_{\ast R} \\
&\quad-\rho_{\ast R}u_{2,\ast R}\left(\frac{\mathcal{D}_{\sigma_R}u_1}{\mathcal{D}t}\right)_{\ast R}+\frac{2\rho_{\ast R}p_{12,\ast R}(u_{1,\ast}-\sigma_R)^2}{3p_{11,\ast}^2}\left(\frac{\mathcal{D}p_{11}}{\mathcal{D}t}\right)_{\ast}.
\end{align*}
Here the values of $\frac{\mathcal{D}_{\sigma_R}\Gamma_{m_{2,R}}}{\mathcal{D}t}$, $\frac{\mathcal{D}_{\sigma_R}\sigma_R}{\mathcal{D}t}$, $\left(\frac{\mathcal{D}_{\sigma_R}\rho}{\mathcal{D}t}\right)_{\ast R}$, $\left(\frac{\mathcal{D}_{\sigma_R}u_1}{\mathcal{D}t}\right)_{\ast R}$ and $\left(\frac{\mathcal{D}p_{11}}{\mathcal{D}t}\right)_{\ast}$ are given by \eqref{DsigmaR_gamma_m2R}, \eqref{DsigmaR_sigmaR}, \eqref{DsigmaR_rhoastR}, \eqref{DsigmaR_u1_astR} and \eqref{Du1_Dp11_Dt_ast}, respectively.
\end{lemma}

\begin{proof}
One can first substitute \eqref{DsigmaR_u2_astR} and \eqref{DsigmaR_p12_astR} into \eqref{DsigmaR_gamma_m2astR}, and then combine \eqref{DsigmaR_gamma_m2astR} with \eqref{DsigmaR_gamma_m2astR_m2R} to yield \eqref{eq1:Du2_Dp12_astR}.
\end{proof}

Across the shock wave with speed $\sigma$, the Rankine-Hugoniot jump condition for $E_{12}$ implies that $\Gamma_{E_{12}}=\Gamma_{\overline{E}_{12}}$ with
\[
\Gamma_{E_{12}}:=\frac{1}{2}(\rho u_1u_2+p_{12})(u_1-\sigma)+\frac{1}{2}(p_{11}u_2+p_{12}u_1).
\]
It follows that
\begin{align*}
\frac{\mathcal{D}_{\sigma}\Gamma_{E_{12}}}{\mathcal{D}t}=&\frac{1}{2}u_1u_2(u_1-\sigma)\frac{\mathcal{D}_{\sigma}\rho}{\mathcal{D}t}+\left(2E_{12}-\frac{1}{2}\rho u_2\sigma\right)\frac{\mathcal{D}_{\sigma}u_1}{\mathcal{D}t}+\left(E_{11}-\frac{1}{2}\rho u_1\sigma\right)\frac{\mathcal{D}_{\sigma}u_2}{\mathcal{D}t} \\
&+\frac{1}{2}u_2\frac{\mathcal{D}_{\sigma}p_{11}}{\mathcal{D}t}
+\left(u_1-\frac{1}{2}\sigma\right)\frac{\mathcal{D}_{\sigma}p_{12}}{\mathcal{D}t}-E_{12}\frac{\mathcal{D}_{\sigma}\sigma}{\mathcal{D}t}.
\end{align*}
Then one can use the relation along the shock wave, i.e., $\frac{\mathcal{D}_{\sigma}\Gamma_{E_{12}}}{\mathcal{D}t}=\frac{\mathcal{D}_{\sigma}\Gamma_{\overline{E}_{12}}}{\mathcal{D}t}$, to build the second equation that the limiting values $\left(\frac{\mathcal{D}u_2}{\mathcal{D}t}\right)_{\ast K}$ and $\left(\frac{\mathcal{D}p_{12}}{\mathcal{D}t}\right)_{\ast K}$ ($K=L,R$) satisfy.

Specifically, for the 6-shock wave, one has
\begin{equation}\label{DsigmaR_gamma_E12astR_E12R}
\frac{\mathcal{D}_{\sigma_R}\Gamma_{E_{12,\ast R}}}{\mathcal{D}t}=\frac{\mathcal{D}_{\sigma_R}\Gamma_{E_{12,R}}}{\mathcal{D}t},
\end{equation}
where
\begin{align}
&\frac{\mathcal{D}_{\sigma_R}\Gamma_{E_{12,\ast R}}}{\mathcal{D}t}=\frac{1}{2}u_{1,\ast}u_{2,\ast R}(u_{1,\ast}-\sigma_R)\left(\frac{\mathcal{D}_{\sigma_R}\rho}{\mathcal{D}t}\right)_{\ast R}+\left(2E_{12,\ast R}-\frac{1}{2}\rho_{\ast R}u_{2,\ast R}\sigma_R\right)\left(\frac{\mathcal{D}_{\sigma_R}u_1}{\mathcal{D}t}\right)_{\ast R} \notag \\
&\qquad+\left(E_{11,\ast R}-\frac{1}{2}\rho_{\ast R}u_{1,\ast}\sigma_R\right)\left(\frac{\mathcal{D}_{\sigma_R}u_2}{\mathcal{D}t}\right)_{\ast R}
+\frac{1}{2}u_{2,\ast R}\frac{\mathcal{D}_{\sigma_R}p_{11,\ast}}{\mathcal{D}t} \notag \\
&\qquad+\left(u_{1,\ast}-\frac{1}{2}\sigma_R\right)\left(\frac{\mathcal{D}_{\sigma_R}p_{12}}{\mathcal{D}t}\right)_{\ast R}-E_{12,\ast R}\frac{\mathcal{D}_{\sigma_R}\sigma_R}{\mathcal{D}t}, \label{DsigmaR_gamma_E12astR} \\
&\frac{\mathcal{D}_{\sigma_R}\Gamma_{E_{12,R}}}{\mathcal{D}t}=\frac{1}{2}u_{1,R}u_{2, R}(u_{1,R}-\sigma_R)\frac{\mathcal{D}_{\sigma_R}\rho_R}{\mathcal{D}t}+\left(2E_{12,R}-\frac{1}{2}\rho_{R}u_{2, R}\sigma_R\right)\frac{\mathcal{D}_{\sigma_R}u_{1,R}}{\mathcal{D}t} \notag \\
&\quad+\left(E_{11,R}-\frac{1}{2}\rho_{R}u_{1,R}\sigma_R\right)\frac{\mathcal{D}_{\sigma_R}u_{2,R}}{\mathcal{D}t}
+\frac{1}{2}u_{2,R}\frac{\mathcal{D}_{\sigma_R}p_{11,R}}{\mathcal{D}t}
+\left(u_{1,R}-\frac{1}{2}\sigma_R\right)\frac{\mathcal{D}_{\sigma_R}p_{12,R}}{\mathcal{D}t}-E_{12, R}\frac{\mathcal{D}_{\sigma_R}\sigma_R}{\mathcal{D}t}. \label{DsigmaR_gamma_E12R}
\end{align}
Substituting \eqref{DsigmaR_VR} into \eqref{DsigmaR_gamma_E12R}  can obtain the value of $\frac{\mathcal{D}_{\sigma_R}\Gamma_{E_{12,R}}}{\mathcal{D}t}$.

\begin{lemma}[Resolution of the 6-shock wave]
Assume that the 6-shock wave  associated with   $u_1+c$ moves to the right. Then the limiting values $\left(\frac{\mathcal{D}u_2}{\mathcal{D}t}\right)_{\ast R}$ and $\left(\frac{\mathcal{D}p_{12}}{\mathcal{D}t}\right)_{\ast R}$ satisfy
\begin{equation}\label{eq2:Du2_Dp12_astR}
a_{4,R}\left(\frac{\mathcal{D}u_2}{\mathcal{D}t}\right)_{\ast R}+b_{4,R}\left(\frac{\mathcal{D}p_{12}}{\mathcal{D}t}\right)_{\ast R}=d_{4,R},
\end{equation}
where
\begin{align*}
&a_{4,R}=E_{11,\ast R}+\rho_{\ast R}(u_{1,\ast}-\sigma_R)^2-\frac{1}{2}\rho_{\ast R}\sigma_R^2, \quad b_{4,R}=\left(E_{11,\ast R}-\frac{1}{2}\rho_{\ast R}u_{1,\ast}\sigma_R\right)\frac{u_{1,\ast}-\sigma_R}{p_{11,\ast}}+u_{1,\ast}-\frac{1}{2}\sigma_R, \\
&d_{4,R}=\frac{\mathcal{D}_{\sigma_R}\Gamma_{E_{12,R}}}{\mathcal{D}t}+E_{12,\ast R}\frac{\mathcal{D}_{\sigma_R}\sigma_R}{\mathcal{D}t}-\frac{1}{2}u_{1,\ast}u_{2,\ast R}(u_{1,\ast}-\sigma_R)\left(\frac{\mathcal{D}_{\sigma_R}\rho}{\mathcal{D}t}\right)_{\ast R}-\frac{1}{2}u_{2,\ast R}\frac{\mathcal{D}_{\sigma_R}p_{11,\ast}}{\mathcal{D}t} \\
&\qquad-\left(2E_{12,\ast R}-\frac{1}{2}\rho_{\ast R}u_{2,\ast R}\sigma_R\right)\left(\frac{\mathcal{D}_{\sigma_R}u_1}{\mathcal{D}t}\right)_{\ast R}
+\left(E_{11,\ast R}-\frac{1}{2}\rho_{\ast R}u_{1,\ast}\sigma_R\right)\frac{2p_{12,\ast R}(u_{1,\ast}-\sigma_R)}{3p_{11,\ast}^2}\left(\frac{\mathcal{D}p_{11}}{\mathcal{D}t}\right)_{\ast}.
\end{align*}
Here the value of $d_{4,R}$ can be obtained by \eqref{DsigmaR_gamma_E12R}, \eqref{DsigmaR_sigmaR}, \eqref{DsigmaR_rhoastR}, \eqref{DsigmaR_p11ast}, \eqref{DsigmaR_u1_astR} and \eqref{Du1_Dp11_Dt_ast}.
\end{lemma}

\begin{proof}
The proof is easily completed by combining \eqref{DsigmaR_u2_astR}, \eqref{DsigmaR_p12_astR}, \eqref{DsigmaR_gamma_E12astR} and \eqref{DsigmaR_gamma_E12astR_E12R}.
\end{proof}

Across the shock wave with speed $\sigma$, the  Rankine-Hugoniot jump condition for $E_{22}$ implies that $\Gamma_{E_{22}}=\Gamma_{\overline{E}_{22}}$ with
\[
\Gamma_{E_{22}}:=\frac{1}{2}(p_{22}+\rho u_2^2)(u_1-\sigma)+p_{12}u_2.
\]
It follows that
\begin{align*}
\frac{\mathcal{D}_{\sigma}\Gamma_{E_{22}}}{\mathcal{D}t}=&\frac{1}{2}u_2^2(u_1-\sigma)\frac{\mathcal{D}_{\sigma}\rho}{\mathcal{D}t}+E_{22}\frac{\mathcal{D}_{\sigma}u_1}{\mathcal{D}t}+(2E_{12}-\rho u_2\sigma)\frac{\mathcal{D}_{\sigma}u_2}{\mathcal{D}t} \\
&+u_2\frac{\mathcal{D}_{\sigma}p_{12}}{\mathcal{D}t}+\frac{1}{2}(u_1-\sigma)\frac{\mathcal{D}_{\sigma}p_{22}}{\mathcal{D}t}-E_{22}\frac{\mathcal{D}_{\sigma}\sigma}{\mathcal{D}t}.
\end{align*}
Then one can utilize the relation along the shock wave, i.e., $\frac{\mathcal{D}_{\sigma}\Gamma_{E_{22}}}{\mathcal{D}t}=\frac{\mathcal{D}_{\sigma}\Gamma_{\overline{E}_{22}}}{\mathcal{D}t}$, to compute the limiting value $\left(\frac{\partial p_{22}}{\partial t}\right)_{\ast K}$, $K=L,R$.

Specifically, for the 6-shock wave, one has
\begin{equation}\label{DsigmaR_E22astR_E22R}
\frac{\mathcal{D}_{\sigma_R}\Gamma_{E_{22,\ast R}}}{\mathcal{D}t}=\frac{\mathcal{D}_{\sigma_R}\Gamma_{E_{22,R}}}{\mathcal{D}t},
\end{equation}
where
\begin{align}
&\frac{\mathcal{D}_{\sigma_R}\Gamma_{E_{22,\ast R}}}{\mathcal{D}t}=\frac{1}{2}u_{2,\ast R}^2(u_{1,\ast}-\sigma_R)\left(\frac{\mathcal{D}_{\sigma_R}\rho}{\mathcal{D}t}\right)_{\ast R}+E_{22,\ast R}\left(\frac{\mathcal{D}_{\sigma_R}u_1}{\mathcal{D}t}\right)_{\ast R}+(2E_{12,\ast R}-\rho_{\ast R}u_{2,\ast R}\sigma_R)\left(\frac{\mathcal{D}_{\sigma_R}u_2}{\mathcal{D}t}\right)_{\ast R} \notag \\
&\qquad+u_{2,\ast R}\left(\frac{\mathcal{D}_{\sigma_R}p_{12}}{\mathcal{D}t}\right)_{\ast R}+\frac{1}{2}(u_{1,\ast}-\sigma_R)\left(\frac{\mathcal{D}_{\sigma_R}p_{22}}{\mathcal{D}t}\right)_{\ast R}-E_{22,\ast R}\frac{\mathcal{D}_{\sigma_R}\sigma_R}{\mathcal{D}t}, \label{DsigmaR_gamma_E22astR} \\
&\frac{\mathcal{D}_{\sigma_R}\Gamma_{E_{22,R}}}{\mathcal{D}t}=\frac{1}{2}u_{2, R}^2(u_{1,R}-\sigma_R)\frac{\mathcal{D}_{\sigma_R}\rho_R}{\mathcal{D}t}+E_{22, R}\frac{\mathcal{D}_{\sigma_R}u_{1,R}}{\mathcal{D}t}+(2E_{12,R}-\rho_{R}u_{2, R}\sigma_R)\frac{\mathcal{D}_{\sigma_R}u_{2,R}}{\mathcal{D}t} \notag \\
&\qquad+u_{2,R}\frac{\mathcal{D}_{\sigma_R}p_{12,R}}{\mathcal{D}t}+\frac{1}{2}(u_{1,R}-\sigma_R)\frac{\mathcal{D}_{\sigma_R}p_{22,R}}{\mathcal{D}t}-E_{22, R}\frac{\mathcal{D}_{\sigma_R}\sigma_R}{\mathcal{D}t}. \label{DsigmaR_gamma_E22R}
\end{align}
By \eqref{eq:p22}, one has
\begin{equation}\label{DsigmaR_p22R}
\frac{\mathcal{D}_{\sigma_R}p_{22,R}}{\mathcal{D}t}=(\sigma_R-u_{1,R})p_{22,R}'-p_{22,R}u_{1,R}'-2p_{12,R}u_{2,R}'.
\end{equation}
Substituting \eqref{DsigmaR_rhoR}, \eqref{DsigmaR_u1R}, \eqref{DsigmaR_u2R}, \eqref{DsigmaR_p12R} and \eqref{DsigmaR_p22R} into \eqref{DsigmaR_gamma_E22R} may obtain the value of $\frac{\mathcal{D}_{\sigma_R}\Gamma_{E_{22,R}}}{\mathcal{D}t}$.

\begin{theorem}[Computing $(\frac{\partial u_2}{\partial t})_{\ast R}$, $(\frac{\partial p_{12}}{\partial t})_{\ast R}$ and $(\frac{\partial p_{22}}{\partial t})_{\ast R}$]\label{theorem:u2_p12_p22_t_astR}
Assume that the 6-shock wave associated with $u_1+c$ moves to the right. Then the limiting values $\left(\frac{\mathcal{D}u_2}{\mathcal{D}t}\right)_{\ast R}$ and $\left(\frac{\mathcal{D}p_{12}}{\mathcal{D}t}\right)_{\ast R}$ are
\begin{equation}
\left(\frac{\mathcal{D}u_2}{\mathcal{D}t}\right)_{\ast R}=\frac{d_{3,R}b_{4,R}-d_{4,R}b_{3,R}}{a_{3,R}b_{4,R}-a_{4,R}b_{3,R}}, \quad
\left(\frac{\mathcal{D}p_{12}}{\mathcal{D}t}\right)_{\ast R}=\frac{d_{3,R}a_{4,R}-d_{4,R}a_{3,R}}{b_{3,R}a_{4,R}-b_{4,R}a_{3,R}}.
\label{Du2_Dp12_t_astR}
\end{equation}
It follows that
\begin{equation}
\begin{split}
&\left(\frac{\partial u_2}{\partial t}\right)_{\ast R}=\left(\frac{\mathcal{D}u_2}{\mathcal{D}t}\right)_{\ast R}+\frac{u_{1,\ast}}{p_{11,\ast}}\left[\left(\frac{\mathcal{D}p_{12}}{\mathcal{D}t}\right)_{\ast R}-\frac{2p_{12,\ast R}}{3p_{11,\ast}}\left(\frac{\mathcal{D}p_{11}}{\mathcal{D}t}\right)_{\ast}\right],  \\
&\left(\frac{\partial p_{12}}{\partial t}\right)_{\ast R}=\left(\frac{\mathcal{D}p_{12}}{\mathcal{D}t}\right)_{\ast R}+\rho_{\ast R}u_{1,\ast}\left(\frac{\mathcal{D}u_2}{\mathcal{D}t}\right)_{\ast R}.
\end{split}
\label{u2_p12_t_astR}
\end{equation}
Moreover, the limiting value $\left(\frac{\partial p_{22}}{\partial t}\right)_{\ast R}$ is computed by
\begin{equation}\label{p22_t_astR}
g_{p_{22}}^{\ast R}\left(\frac{\partial p_{22}}{\partial t}\right)_{\ast R}=f_{p_{22}}^{\ast R},
\end{equation}
where
\begin{align*}
&g_{p_{22}}^{\ast R}=\frac{1}{2}(u_{1,\ast}-\sigma_R)^2, \\
&f_{p_{22}}^{\ast R}=u_{1,\ast}\cdot\widetilde{f}_{p_{22}}^{\ast R}-\frac{(p_{11,\ast}p_{22,\ast R}-4p_{12,\ast R}^2)(u_{1,\ast}-\sigma_R)\sigma_R}{6p_{11,\ast}^2}\left(\frac{\mathcal{D}p_{11}}{\mathcal{D}t}\right)_{\ast}
-\frac{p_{12,\ast R}\sigma_R(u_{1,\ast}-\sigma_R)}{p_{11,\ast}}\left(\frac{\mathcal{D}p_{12}}{\mathcal{D}t}\right)_{\ast R},
\end{align*}
with
\begin{align*}
&\widetilde{f}_{p_{22}}^{\ast R}=\frac{\mathcal{D}_{\sigma_R}\Gamma_{E_{22,R}}}{\mathcal{D}t}+E_{22,\ast R}\frac{\mathcal{D}_{\sigma_R}\sigma_R}{\mathcal{D}t}-\frac{1}{2}u_{2,\ast R}^2(u_{1,\ast}-\sigma_R)\left(\frac{\mathcal{D}_{\sigma_R}\rho}{\mathcal{D}t}\right)_{\ast R} \\
&\qquad-E_{22,\ast R}\left(\frac{\mathcal{D}_{\sigma_R}u_1}{\mathcal{D}t}\right)_{\ast R}-(2E_{12,\ast R}-\rho_{\ast R}u_{2,\ast R}\sigma_R)\left(\frac{\mathcal{D}_{\sigma_R}u_2}{\mathcal{D}t}\right)_{\ast R}-u_{2,\ast R}\left(\frac{\mathcal{D}_{\sigma_R}p_{12}}{\mathcal{D}t}\right)_{\ast R}.
\end{align*}
\end{theorem}

\begin{proof}
One can deduce \eqref{Du2_Dp12_t_astR} from \eqref{eq1:Du2_Dp12_astR} and \eqref{eq2:Du2_Dp12_astR},
and then by \eqref{u2_p12_t},   get  \eqref{u2_p12_t_astR}. To derive \eqref{p22_t_astR}, one can first utilize \eqref{eq:p22}, \eqref{u1_x} and \eqref{u2_x} to obtain
\begin{equation}\label{u1ast_DsigmaR_p22astR}
u_{1,\ast}\left(\frac{\mathcal{D}_{\sigma_R}p_{22}}{\mathcal{D}t}\right)_{\ast R}=(u_{1,\ast}-\sigma_R)\left(\frac{\partial p_{22}}{\partial t}\right)_{\ast R}+\frac{(p_{11,\ast}p_{22,\ast R}-4p_{12,\ast R}^2)\sigma_R}{3p_{11,\ast}^2}\left(\frac{\mathcal{D}p_{11}}{\mathcal{D}t}\right)_{\ast}+\frac{2p_{12,\ast R}\sigma_R}{p_{11,\ast}}\left(\frac{\mathcal{D}p_{12}}{\mathcal{D}t}\right)_{\ast R}.
\end{equation}
Multiplying both sides of \eqref{DsigmaR_E22astR_E22R} by $u_{1,\ast}$ yields
\begin{equation}\label{u1ast_DsigmaR_E22astR_E22R}
u_{1,\ast}\frac{\mathcal{D}_{\sigma_R}\Gamma_{E_{22,\ast R}}}{\mathcal{D}t}=u_{1,\ast}\frac{\mathcal{D}_{\sigma_R}\Gamma_{E_{22,R}}}{\mathcal{D}t}.
\end{equation}
Then   substituting \eqref{DsigmaR_gamma_E22astR} and \eqref{DsigmaR_gamma_E22R} into \eqref{u1ast_DsigmaR_E22astR_E22R} and applying \eqref{u1ast_DsigmaR_p22astR}   gets \eqref{p22_t_astR}.
\end{proof}

\begin{remark}
For the 6-rarefaction wave, computing $(\frac{\partial u_2}{\partial t})_{\ast R}$, $(\frac{\partial p_{12}}{\partial t})_{\ast R}$ and $(\frac{\partial p_{22}}{\partial t})_{\ast R}$ is  in \ref{u2_p12_p22_astR_6rarefaction}.
\end{remark}

\subsubsection{Computing $\left(\frac{\partial u_2}{\partial t}\right)_{\ast\ast K}$, $\left(\frac{\partial p_{12}}{\partial t}\right)_{\ast\ast K}$ and $\left(\frac{\partial p_{22}}{\partial t}\right)_{\ast\ast K}$ $(K=L,R)$}

In this subsection, the generalized Riemann invariants associated with the 2- and 5-shear waves are used to derive the limiting values $\left(\frac{\partial u_2}{\partial t}\right)_{\ast\ast K}$, $\left(\frac{\partial p_{12}}{\partial t}\right)_{\ast\ast K}$ and $\left(\frac{\partial p_{22}}{\partial t}\right)_{\ast\ast K}$ $(K=L,R)$.

If denoting $\Theta_2:=u_2+\frac{p_{12}}{\sqrt{\rho p_{11}}}$ and recalling that $\frac{\mathcal{D}_2}{\mathcal{D}t}:=\frac{\partial}{\partial t}+\lambda_2\frac{\partial}{\partial x}$ with $\lambda_2=u_{1,\ast}-\frac{c_{\ast L}}{\sqrt{3}}$, then
\[
\frac{\mathcal{D}_2\Theta_2}{\mathcal{D}t}=-\frac{p_{12}}{2\rho\sqrt{\rho p_{11}}}\frac{\mathcal{D}_2\rho}{\mathcal{D}t}+\frac{\mathcal{D}_2u_2}{\mathcal{D}t}-\frac{p_{12}}{2p_{11}\sqrt{\rho p_{11}}}\frac{\mathcal{D}_2p_{11}}{\mathcal{D}t}+\frac{1}{\sqrt{\rho p_{11}}}\frac{\mathcal{D}_2p_{12}}{\mathcal{D}t}.
\]
Because $\Theta_2$ is a generalized Riemann invariant associated with the 2-shear wave, one has
\begin{equation}\label{D2_Theta2ast2L_Theta2astL}
\frac{\mathcal{D}_2\Theta_{2,\ast\ast L}}{\mathcal{D}t}=\frac{\mathcal{D}_2\Theta_{2,\ast L}}{\mathcal{D}t},
\end{equation}
where
{\small\begin{align}
&\frac{\mathcal{D}_2\Theta_{2,\ast\ast L}}{\mathcal{D}t}=-\frac{p_{12,\ast\ast}}{2\rho_{\ast L}\sqrt{\rho_{\ast L} p_{11,\ast}}}\left(\frac{\mathcal{D}_2\rho}{\mathcal{D}t}\right)_{\ast\ast L}+\left(\frac{\mathcal{D}_2u_2}{\mathcal{D}t}\right)_{\ast\ast L}-\frac{p_{12,\ast\ast}}{2p_{11,\ast}\sqrt{\rho_{\ast L} p_{11,\ast}}}\left(\frac{\mathcal{D}_2p_{11}}{\mathcal{D}t}\right)_{\ast\ast L}+\frac{1}{\sqrt{\rho_{\ast L} p_{11,\ast}}}\left(\frac{\mathcal{D}_2p_{12}}{\mathcal{D}t}\right)_{\ast\ast L}, \label{D2_Theta2ast2L} \\
&\frac{\mathcal{D}_2\Theta_{2,\ast L}}{\mathcal{D}t}=-\frac{p_{12,\ast L}}{2\rho_{\ast L}\sqrt{\rho_{\ast L} p_{11,\ast}}}\left(\frac{\mathcal{D}_2\rho}{\mathcal{D}t}\right)_{\ast L}+\left(\frac{\mathcal{D}_2u_2}{\mathcal{D}t}\right)_{\ast L} -\frac{p_{12,\ast L}}{2p_{11,\ast}\sqrt{\rho_{\ast L} p_{11,\ast}}}\left(\frac{\mathcal{D}_2p_{11}}{\mathcal{D}t}\right)_{\ast L}+\frac{1}{\sqrt{\rho_{\ast L} p_{11,\ast}}}\left(\frac{\mathcal{D}_2p_{12}}{\mathcal{D}t}\right)_{\ast L}. \label{D2_Theta2astL}
\end{align}}%
Because  both $\rho$ and $p_{11}$ are the generalized Riemann invariants associated with the 2-shear wave, one actually has $\left(\frac{\mathcal{D}_2\rho}{\mathcal{D}t}\right)_{\ast\ast L}=\left(\frac{\mathcal{D}_2\rho}{\mathcal{D}t}\right)_{\ast L}$ and $\left(\frac{\mathcal{D}_2p_{11}}{\mathcal{D}t}\right)_{\ast\ast L}=\left(\frac{\mathcal{D}_2p_{11}}{\mathcal{D}t}\right)_{\ast L}$. By \eqref{p11_x}, one gets
\begin{equation*}\label{D2_p11ast2L_p11astL}
\left(\frac{\mathcal{D}_2p_{11}}{\mathcal{D}t}\right)_{\ast\ast L}=\left(\frac{\mathcal{D}_2p_{11}}{\mathcal{D}t}\right)_{\ast L}
=\left(\frac{\mathcal{D}p_{11}}{\mathcal{D}t}\right)_{\ast}-\rho_{\ast L}(\lambda_2-u_{1,\ast})\left[\left(\frac{\mathcal{D}u_1}{\mathcal{D}t}\right)_{\ast}+\frac{1}{2}W_x(0)\right].
\end{equation*}
As for the value of $\left(\frac{\mathcal{D}_2\rho}{\mathcal{D}t}\right)_{\ast L}$, recalling the equation \eqref{dp11}, that is $\text{d}p_{11}=c^2\text{d}\rho+\rho^3\text{d}S_{1}$, one has
\[
\left(\frac{\partial\rho}{\partial x}\right)_{\ast L}=\frac{1}{c_{\ast L}^2}\left(\frac{\partial p_{11}}{\partial x}\right)_{\ast L}-\frac{\rho_{\ast L}^3}{c_{\ast L}^2}\left(\frac{\partial S_1}{\partial x}\right)_{\ast L}
=-\frac{\rho_{\ast L}}{c_{\ast L}^2}\left[\left(\frac{\mathcal{D}u_1}{\mathcal{D}t}\right)_{\ast}+\frac{1}{2}W_x(0)\right]-\frac{\rho_{\ast L}^3S_{1,L}'}{c_{\ast L}c_L},
\]
where \eqref{p11_x} and \eqref{S1_x_0beta} have been used in the second equality, and $S_{1,L}'$ is given in \eqref{psi1_L'},
then $\left(\frac{\mathcal{D}_2\rho}{\mathcal{D}t}\right)_{\ast L}=\left(\frac{\partial\rho}{\partial t}\right)_{\ast L}+\lambda_2\left(\frac{\partial\rho}{\partial x}\right)_{\ast L}$ can be obtained.


\begin{remark}\label{computation of D2_rhoastL for 1-shock}
For the 1-shock wave, one can get the value of $\left(\frac{\mathcal{D}_{\sigma_L}\rho}{\mathcal{D}t}\right)_{\ast L}$ in a similar way to compute $\left(\frac{\mathcal{D}_{\sigma_R}\rho}{\mathcal{D}t}\right)_{\ast R}$ in Subsection {\rm\ref{compute omega_astR}}, {\rm}see \eqref{DsigmaR_rhoastR}{\rm}. Thus
\[
\left(\frac{\mathcal{D}_2\rho}{\mathcal{D}t}\right)_{\ast L}=\left(1-\frac{\lambda_2}{\sigma_L}\right)\left(\frac{\partial \rho}{\partial t}\right)_{\ast L}+\frac{\lambda_2}{\sigma_L}\left(\frac{\mathcal{D}_{\sigma_L}\rho}{\mathcal{D}t}\right)_{\ast L}.
\]
\end{remark}

By \eqref{u2_x}, one has
\begin{align}
&\left(\frac{\mathcal{D}_2u_2}{\mathcal{D}t}\right)_{\ast\ast L}=\left(\frac{\mathcal{D}u_2}{\mathcal{D}t}\right)_{\ast\ast}-
\frac{\lambda_2-u_{1,\ast}}{p_{11,\ast}}\left[\left(\frac{\mathcal{D}p_{12}}{\mathcal{D}t}\right)_{\ast\ast}-
\frac{2p_{12,\ast\ast}}{3p_{11,\ast}}\left(\frac{\mathcal{D}p_{11}}{\mathcal{D}t}\right)_{\ast}\right], \label{D2_u2ast2L} \\
&\left(\frac{\mathcal{D}_2u_2}{\mathcal{D}t}\right)_{\ast L}=\left(\frac{\mathcal{D}u_2}{\mathcal{D}t}\right)_{\ast L}-
\frac{\lambda_2-u_{1,\ast}}{p_{11,\ast}}\left[\left(\frac{\mathcal{D}p_{12}}{\mathcal{D}t}\right)_{\ast L}-
\frac{2p_{12,\ast L}}{3p_{11,\ast}}\left(\frac{\mathcal{D}p_{11}}{\mathcal{D}t}\right)_{\ast}\right], \notag 
\end{align}
where the notations $\left(\frac{\mathcal{D}u_2}{\mathcal{D}t}\right)_{\ast\ast}:=\left(\frac{\mathcal{D}u_2}{\mathcal{D}t}\right)_{\ast\ast L}=\left(\frac{\mathcal{D}u_2}{\mathcal{D}t}\right)_{\ast\ast R}$ and $\left(\frac{\mathcal{D}p_{12}}{\mathcal{D}t}\right)_{\ast\ast}:=\left(\frac{\mathcal{D}p_{12}}{\mathcal{D}t}\right)_{\ast\ast L}=\left(\frac{\mathcal{D}p_{12}}{\mathcal{D}t}\right)_{\ast\ast R}$ have been used
because both $u_2$ and $p_{12}$ are generalized Riemann invariants associated with the contact discontinuity.
By \eqref{p12_x}, one has
\begin{align}
&\left(\frac{\mathcal{D}_2p_{12}}{\mathcal{D}t}\right)_{\ast\ast L}=\left(\frac{\mathcal{D}p_{12}}{\mathcal{D}t}\right)_{\ast\ast}-\rho_{\ast L}(\lambda_2-u_{1,\ast})\left(\frac{\mathcal{D}u_2}{\mathcal{D}t}\right)_{\ast\ast}, \label{D2_p12ast2L} \\
&\left(\frac{\mathcal{D}_2p_{12}}{\mathcal{D}t}\right)_{\ast L}=\left(\frac{\mathcal{D}p_{12}}{\mathcal{D}t}\right)_{\ast L}-\rho_{\ast L}(\lambda_2-u_{1,\ast})\left(\frac{\mathcal{D}u_2}{\mathcal{D}t}\right)_{\ast L}. \notag 
\end{align}
Note that the limiting values $\left(\frac{\mathcal{D}u_2}{\mathcal{D}t}\right)_{\ast L}$ and $\left(\frac{\mathcal{D}p_{12}}{\mathcal{D}t}\right)_{\ast L}$ have been obtained in Subsection \ref{computing omega_astL}, see Theorem \ref{theorem for omega_astL}.

\begin{lemma}
Assume that the 2-shear wave associated with $u_1-\frac{c}{\sqrt{3}}$ moves to the left. Then the limiting values $\left(\frac{\mathcal{D}u_2}{\mathcal{D}t}\right)_{\ast\ast}$ and $\left(\frac{\mathcal{D}p_{12}}{\mathcal{D}t}\right)_{\ast\ast}$ satisfy
\begin{equation}\label{eq1:Du2_Dp12_ast2}
a_{5,L}\left(\frac{\mathcal{D}u_2}{\mathcal{D}t}\right)_{\ast\ast}+b_{5,L}\left(\frac{\mathcal{D}p_{12}}{\mathcal{D}t}\right)_{\ast\ast}=d_{5,L},
\end{equation}
where
\begin{align*}
&a_{5,L}=2, \quad b_{5,L}=\frac{2}{\sqrt{\rho_{\ast L}p_{11,\ast}}}, \\
&d_{5,L}=\frac{p_{12,\ast\ast}-p_{12,\ast L}}{2\rho_{\ast L}\sqrt{\rho_{\ast L}p_{11,\ast}}}\left(\frac{\mathcal{D}_2\rho}{\mathcal{D}t}\right)_{\ast L}+\left(\frac{\mathcal{D}_2u_2}{\mathcal{D}t}\right)_{\ast L}
+\frac{p_{12,\ast\ast}-p_{12,\ast L}}{2p_{11,\ast}\sqrt{\rho_{\ast L}p_{11,\ast}}}\left(\frac{\mathcal{D}_2p_{11}}{\mathcal{D}t}\right)_{\ast L} \\
&\qquad+\frac{1}{\sqrt{\rho_{\ast L}p_{11,\ast}}}\left(\frac{\mathcal{D}_2p_{12}}{\mathcal{D}t}\right)_{\ast L}
-\frac{2p_{12,\ast\ast}(\lambda_2-u_{1,\ast})}{3p_{11,\ast}^2}\left(\frac{\mathcal{D}p_{11}}{\mathcal{D}t}\right)_{\ast}.
\end{align*}
\end{lemma}

\begin{proof}
Substituting \eqref{D2_Theta2ast2L} and \eqref{D2_Theta2astL} into \eqref{D2_Theta2ast2L_Theta2astL} and utilizing \eqref{D2_u2ast2L} and \eqref{D2_p12ast2L} may obtain \eqref{eq1:Du2_Dp12_ast2}.
\end{proof}

If denoting $\Theta_5:=u_2-\frac{p_{12}}{\sqrt{\rho p_{11}}}$ and $\frac{\mathcal{D}_5}{\mathcal{D}t}:=\frac{\partial}{\partial t}+\lambda_5\frac{\partial}{\partial x}$ with $\lambda_5=u_{1,\ast}+\frac{c_{\ast R}}{\sqrt{3}}$, then
\[
\text{d}\Theta_5=\frac{p_{12}}{2\rho\sqrt{\rho p_{11}}}\text{d}\rho+\text{d}u_2+\frac{p_{12}}{2p_{11}\sqrt{\rho p_{11}}}\text{d}p_{11}-\frac{1}{\sqrt{\rho p_{11}}}\text{d}p_{12}.
\]
Because $\Theta_5$ is a generalized Riemann invariant associated with the 5-shear wave, one has
\begin{equation*}
\frac{\mathcal{D}_5\Theta_{5,\ast\ast R}}{\mathcal{D}t}=\frac{\mathcal{D}_5\Theta_{5,\ast R}}{\mathcal{D}t}.
\end{equation*}
With similar derivation to build \eqref{eq1:Du2_Dp12_ast2}, one can obtain the following result.

\begin{lemma}
Assume that the 5-shear wave associated with $u_1+\frac{c}{\sqrt{3}}$ moves to the right. Then the limiting values $\left(\frac{\mathcal{D}u_2}{\mathcal{D}t}\right)_{\ast\ast}$ and $\left(\frac{\mathcal{D}p_{12}}{\mathcal{D}t}\right)_{\ast\ast}$ satisfy
\begin{equation}\label{eq2:Du2_Dp12_ast2}
a_{5,R}\left(\frac{\mathcal{D}u_2}{\mathcal{D}t}\right)_{\ast\ast}+b_{5,R}\left(\frac{\mathcal{D}p_{12}}{\mathcal{D}t}\right)_{\ast\ast}=d_{5,R},
\end{equation}
where
\begin{align*}
&a_{5,R}=2, \quad b_{5,R}=-\frac{2}{\sqrt{\rho_{\ast R}p_{11,\ast}}}, \\
&d_{5,R}=\frac{p_{12,\ast R}-p_{12,\ast\ast}}{2\rho_{\ast R}\sqrt{\rho_{\ast R}p_{11,\ast}}}\left(\frac{\mathcal{D}_5\rho}{\mathcal{D}t}\right)_{\ast R}+\left(\frac{\mathcal{D}_5u_2}{\mathcal{D}t}\right)_{\ast R}
+\frac{p_{12,\ast R}-p_{12,\ast\ast}}{2p_{11,\ast}\sqrt{\rho_{\ast R}p_{11,\ast}}}\left(\frac{\mathcal{D}_5p_{11}}{\mathcal{D}t}\right)_{\ast R} \\
&\qquad-\frac{1}{\sqrt{\rho_{\ast R}p_{11,\ast}}}\left(\frac{\mathcal{D}_5p_{12}}{\mathcal{D}t}\right)_{\ast R}
-\frac{2p_{12,\ast\ast}(\lambda_5-u_{1,\ast})}{3p_{11,\ast}^2}\left(\frac{\mathcal{D}p_{11}}{\mathcal{D}t}\right)_{\ast}
\end{align*}
with the computation of $\left(\frac{\mathcal{D}_5\rho}{\mathcal{D}t}\right)_{\ast R}$ being similar to $\left(\frac{\mathcal{D}_2\rho}{\mathcal{D}t}\right)_{\ast L}$ for the 1-shock wave case, see {\rm Remark \ref{computation of D2_rhoastL for 1-shock}}, and
\begin{align*}
&\left(\frac{\mathcal{D}_5p_{11}}{\mathcal{D}t}\right)_{\ast R}=\left(\frac{\mathcal{D}p_{11}}{\mathcal{D}t}\right)_{\ast}-\rho_{\ast R}(\lambda_5-u_{1,\ast})\left[\left(\frac{\mathcal{D}u_1}{\mathcal{D}t}\right)_{\ast}+\frac{1}{2}W_x(0)\right], \\
&\left(\frac{\mathcal{D}_5u_2}{\mathcal{D}t}\right)_{\ast R}=\left(\frac{\mathcal{D}u_2}{\mathcal{D}t}\right)_{\ast R}-
\frac{\lambda_5-u_{1,\ast}}{p_{11,\ast}}\left[\left(\frac{\mathcal{D}p_{12}}{\mathcal{D}t}\right)_{\ast R}-
\frac{2p_{12,\ast R}}{3p_{11,\ast}}\left(\frac{\mathcal{D}p_{11}}{\mathcal{D}t}\right)_{\ast}\right], \\
&\left(\frac{\mathcal{D}_5p_{12}}{\mathcal{D}t}\right)_{\ast R}=\left(\frac{\mathcal{D}p_{12}}{\mathcal{D}t}\right)_{\ast R}-\rho_{\ast R}(\lambda_5-u_{1,\ast})\left(\frac{\mathcal{D}u_2}{\mathcal{D}t}\right)_{\ast R}.
\end{align*}
\end{lemma}

\begin{remark}
For the 6-rarefaction wave case,  {similar to the 1-rarefaction wave case, one has
\[
\left(\frac{\partial\rho}{\partial x}\right)_{\ast R}=
-\frac{\rho_{\ast R}}{c_{\ast R}^2}\left[\left(\frac{\mathcal{D}u_1}{\mathcal{D}t}\right)_{\ast}+\frac{1}{2}W_x(0)\right]-\frac{\rho_{\ast R}^3S_{1,R}'}{c_{\ast R}c_R},
\]
where $S_{1,R}'=\frac{1}{\rho_R^3}(p_{11,R}'-c_R^2\rho_R')$.
}
Thus the value of $\left(\frac{\mathcal{D}_5\rho}{\mathcal{D}t}\right)_{\ast R}=\left(\frac{\partial\rho}{\partial t}\right)_{\ast R}+\lambda_5\left(\frac{\partial\rho}{\partial x}\right)_{\ast R}$ can be obtained.
\end{remark}

\begin{theorem}[Computing  $\left(\frac{\partial u_2}{\partial t}\right)_{\ast\ast K}$ and $\left(\frac{\partial p_{12}}{\partial t}\right)_{\ast\ast K}$ with $K=L,R$]\label{theorem:u2_p12_t_2astK}
Assume that the 2-shear wave associated with $u_1-\frac{c}{\sqrt{3}}$ moves to the left and the 5-shear wave associated with $u_1+\frac{c}{\sqrt{3}}$ moves to the right. Then one has
\begin{equation}
\begin{split}
&\left(\frac{\mathcal{D}u_2}{\mathcal{D}t}\right)_{\ast\ast}=\frac{d_{5,L}b_{5,R}-d_{5,R}b_{5,L}}{a_{5,L}b_{5,R}-a_{5,R}b_{5,L}}, \\
&\left(\frac{\mathcal{D}p_{12}}{\mathcal{D}t}\right)_{\ast\ast}=\frac{d_{5,L}a_{5,R}-d_{5,R}a_{5,L}}{a_{5,R}b_{5,L}-a_{5,L}b_{5,R}}.
\end{split}
\label{Du2_Dp12_ast2}
\end{equation}
It follows that
\begin{equation}
\begin{split}
&\left(\frac{\partial u_2}{\partial t}\right)_{\ast\ast K}=\left(\frac{\mathcal{D}u_2}{\mathcal{D}t}\right)_{\ast\ast}+
\frac{u_{1,\ast}}{p_{11,\ast}}\left[\left(\frac{\mathcal{D}p_{12}}{\mathcal{D}t}\right)_{\ast\ast}-\frac{2p_{12,\ast\ast}}{3p_{11,\ast}}
\left(\frac{\mathcal{D}p_{11}}{\mathcal{D}t}\right)_{\ast}\right], \\
&\left(\frac{\partial p_{12}}{\partial t}\right)_{\ast\ast K}=\left(\frac{\mathcal{D}p_{12}}{\mathcal{D}t}\right)_{\ast\ast}+\rho_{\ast K}u_{1,\ast}\left(\frac{\mathcal{D}u_2}{\mathcal{D}t}\right)_{\ast\ast},
\end{split}
\label{u2_p12_t_ast2K}
\end{equation}
with $K=L,R$. 
\end{theorem}

\begin{proof}
\eqref{Du2_Dp12_ast2} is the direct result  from \eqref{eq1:Du2_Dp12_ast2} and \eqref{eq2:Du2_Dp12_ast2}. Applying \eqref{u2_p12_t} may obtain \eqref{u2_p12_t_ast2K}.
\end{proof}

Because $\det(\mathbf{p})$ is a generalized Riemann invariant for the 2-shear wave, one has
\begin{equation}\label{D2_past2L_pastL}
\frac{\mathcal{D}_2\det(\mathbf{p}_{\ast\ast L})}{\mathcal{D}t}=\frac{\mathcal{D}_2\det(\mathbf{p}_{\ast L})}{\mathcal{D}t}.
\end{equation}
Due to $\text{d}(\det(\mathbf{p}))=p_{22}\text{d}p_{11}-2p_{12}\text{d}p_{12}+p_{11}\text{d}p_{22}$, one has
\begin{equation}
\begin{split}
&\frac{\mathcal{D}_2\det(\mathbf{p}_{\ast\ast L})}{\mathcal{D}t}=p_{22,\ast\ast L}\left(\frac{\mathcal{D}_2p_{11}}{\mathcal{D}t}\right)_{\ast\ast L}-2p_{12,\ast\ast}\left(\frac{\mathcal{D}_2p_{12}}{\mathcal{D}t}\right)_{\ast\ast L}+p_{11,\ast}\left(\frac{\mathcal{D}_2p_{22}}{\mathcal{D}t}\right)_{\ast\ast L}, \\
&\frac{\mathcal{D}_2\det(\mathbf{p}_{\ast L})}{\mathcal{D}t}=p_{22,\ast L}\left(\frac{\mathcal{D}_2p_{11}}{\mathcal{D}t}\right)_{\ast L}-2p_{12,\ast L}\left(\frac{\mathcal{D}_2p_{12}}{\mathcal{D}t}\right)_{\ast L}+p_{11,\ast}\left(\frac{\mathcal{D}_2p_{22}}{\mathcal{D}t}\right)_{\ast L}.
\end{split}
\label{D2_detp_ast2L_astL}
\end{equation}
Using \eqref{eq:p22}, \eqref{u1_x} and \eqref{u2_x} yields
\begin{equation}
\begin{split}
&u_{1,\ast}\left(\frac{\mathcal{D}_2p_{22}}{\mathcal{D}t}\right)_{\ast\ast L}=(u_{1,\ast}-\lambda_2)\left(\frac{\partial p_{22}}{\partial t}\right)_{\ast\ast L}+\frac{\lambda_2(p_{11,\ast}p_{22,\ast\ast L}-4p_{12,\ast\ast}^2)}{3p_{11,\ast}^2}\left(\frac{\mathcal{D}p_{11}}{\mathcal{D}t}\right)_{\ast}+\frac{2\lambda_2p_{12,\ast\ast}}{p_{11,\ast}}
\left(\frac{\mathcal{D}p_{12}}{\mathcal{D}t}\right)_{\ast\ast}, \\
&u_{1,\ast}\left(\frac{\mathcal{D}_2p_{22}}{\mathcal{D}t}\right)_{\ast L}=(u_{1,\ast}-\lambda_2)\left(\frac{\partial p_{22}}{\partial t}\right)_{\ast L}+\frac{\lambda_2(p_{11,\ast}p_{22,\ast L}-4p_{12,\ast L}^2)}{3p_{11,\ast}^2}\left(\frac{\mathcal{D}p_{11}}{\mathcal{D}t}\right)_{\ast}+\frac{2\lambda_2p_{12,\ast L}}{p_{11,\ast}}
\left(\frac{\mathcal{D}p_{12}}{\mathcal{D}t}\right)_{\ast L}.
\end{split}
\label{u1ast_D2_p22_ast2L_astL}
\end{equation}
Similarly, one has
\begin{align*}
&u_{1,\ast}\left(\frac{\mathcal{D}_5p_{22}}{\mathcal{D}t}\right)_{\ast\ast R}=(u_{1,\ast}-\lambda_5)\left(\frac{\partial p_{22}}{\partial t}\right)_{\ast\ast R}+\frac{\lambda_5(p_{11,\ast}p_{22,\ast\ast R}-4p_{12,\ast\ast}^2)}{3p_{11,\ast}^2}\left(\frac{\mathcal{D}p_{11}}{\mathcal{D}t}\right)_{\ast}+\frac{2\lambda_5p_{12,\ast\ast}}{p_{11,\ast}}
\left(\frac{\mathcal{D}p_{12}}{\mathcal{D}t}\right)_{\ast\ast}, \\ 
&u_{1,\ast}\left(\frac{\mathcal{D}_5p_{22}}{\mathcal{D}t}\right)_{\ast R}=(u_{1,\ast}-\lambda_5)\left(\frac{\partial p_{22}}{\partial t}\right)_{\ast R}+\frac{\lambda_5(p_{11,\ast}p_{22,\ast R}-4p_{12,\ast R}^2)}{3p_{11,\ast}^2}\left(\frac{\mathcal{D}p_{11}}{\mathcal{D}t}\right)_{\ast}+\frac{2\lambda_5p_{12,\ast R}}{p_{11,\ast}}
\left(\frac{\mathcal{D}p_{12}}{\mathcal{D}t}\right)_{\ast R}. 
\end{align*}

\begin{theorem}[Computing $\left(\frac{\partial p_{22}}{\partial t}\right)_{\ast\ast K}$ $(K=L,R)$]\label{theorem:p22_t_2astK}
Assume that the 2-shear wave associated with $u_1-\frac{c}{\sqrt{3}}$ moves to the left and the 5-shear wave associated with $u_1+\frac{c}{\sqrt{3}}$ moves to the right. Then $\left(\frac{\partial p_{22}}{\partial t}\right)_{\ast\ast K}$  satisfies
\begin{equation}\label{p22_t_ast2K}
g_{p_{22}}^{\ast\ast K}\left(\frac{\partial p_{22}}{\partial t}\right)_{\ast\ast K}=f_{p_{22}}^{\ast\ast K},
\ K=L,R,
\end{equation}
where
\begin{align*}
&g_{p_{22}}^{\ast\ast L}=p_{11,\ast}(u_{1,\ast}-\lambda_2), \quad g_{p_{22}}^{\ast\ast R}=p_{11,\ast}(u_{1,\ast}-\lambda_5), \\
&f_{p_{22}}^{\ast\ast K}=u_{1,\ast}\cdot\widetilde{f}_{p_{22}}^{\ast\ast K}+p_{11,\ast}\cdot\widehat{f}_{p_{22}}^{\ast\ast K}, \quad K=L,R,
\end{align*}
and
\begin{align*}
&\widetilde{f}_{p_{22}}^{\ast\ast L}=(p_{22,\ast L}-p_{22,\ast\ast L})\left(\frac{\mathcal{D}_2p_{11}}{\mathcal{D}t}\right)_{\ast L}-2p_{12,\ast L}\left(\frac{\mathcal{D}_2p_{12}}{\mathcal{D}t}\right)_{\ast L}+2p_{12,\ast\ast}\left(\frac{\mathcal{D}_2p_{12}}{\mathcal{D}t}\right)_{\ast\ast L}, \\
&\widetilde{f}_{p_{22}}^{\ast\ast R}=(p_{22,\ast R}-p_{22,\ast\ast R})\left(\frac{\mathcal{D}_5p_{11}}{\mathcal{D}t}\right)_{\ast R}-2p_{12,\ast R}\left(\frac{\mathcal{D}_5p_{12}}{\mathcal{D}t}\right)_{\ast R}+2p_{12,\ast\ast}\left(\frac{\mathcal{D}_5p_{12}}{\mathcal{D}t}\right)_{\ast\ast R}, \\
&\widehat{f}_{p_{22}}^{\ast\ast L}=(u_{1,\ast}-\lambda_2)\left(\frac{\partial p_{22}}{\partial t}\right)_{\ast L}+
\frac{\lambda_2}{3p_{11,\ast}^2}(p_{11,\ast}p_{22,\ast L}-4p_{12,\ast L}^2-p_{11,\ast}p_{22,\ast\ast L}+4p_{12,\ast\ast}^2)\left(\frac{\mathcal{D}p_{11}}{\mathcal{D}t}\right)_{\ast} \\
&\qquad+\frac{2\lambda_2}{p_{11,\ast}}\left[p_{12,\ast L}\left(\frac{\mathcal{D}p_{12}}{\mathcal{D}t}\right)_{\ast L}-p_{12,\ast\ast}\left(\frac{\mathcal{D}p_{12}}{\mathcal{D}t}\right)_{\ast\ast}\right], \\
&\widehat{f}_{p_{22}}^{\ast\ast R}=(u_{1,\ast}-\lambda_5)\left(\frac{\partial p_{22}}{\partial t}\right)_{\ast R}+
\frac{\lambda_5}{3p_{11,\ast}^2}(p_{11,\ast}p_{22,\ast R}-4p_{12,\ast R}^2-p_{11,\ast}p_{22,\ast\ast R}+4p_{12,\ast\ast}^2)\left(\frac{\mathcal{D}p_{11}}{\mathcal{D}t}\right)_{\ast} \\
&\qquad+\frac{2\lambda_5}{p_{11,\ast}}\left[p_{12,\ast R}\left(\frac{\mathcal{D}p_{12}}{\mathcal{D}t}\right)_{\ast R}-p_{12,\ast\ast}\left(\frac{\mathcal{D}p_{12}}{\mathcal{D}t}\right)_{\ast\ast}\right],
\end{align*}
with
\[
\left(\frac{\mathcal{D}_5p_{12}}{\mathcal{D}t}\right)_{\ast\ast R}=\left(\frac{\mathcal{D}p_{12}}{\mathcal{D}t}\right)_{\ast\ast}-\rho_{\ast R}(\lambda_5-u_{1,\ast})\left(\frac{\mathcal{D}u_2}{\mathcal{D}t}\right)_{\ast\ast}.
\]
\end{theorem}

\begin{proof}
Multiplying both sides of \eqref{D2_past2L_pastL} by $u_{1,\ast}$ gets
\begin{equation}\label{u1ast_D2_past2L_pastL}
u_{1,\ast}\frac{\mathcal{D}_2\det(\mathbf{p}_{\ast\ast L})}{\mathcal{D}t}=u_{1,\ast}\frac{\mathcal{D}_2\det(\mathbf{p}_{\ast L})}{\mathcal{D}t}.
\end{equation}
Then substituting \eqref{D2_detp_ast2L_astL} into \eqref{u1ast_D2_past2L_pastL}, and utilizing \eqref{u1ast_D2_p22_ast2L_astL} yields \eqref{p22_t_ast2K} for $K=L$. Because $\det(\mathbf{p})$ is also a generalized Riemann invariant for the 5-shear wave, one can also complete the proof for $K=R$ with similar derivation.
\end{proof}

\subsection{Sonic case}\label{sonic case}

When the $t$-axis is located inside a rarefaction fan, the sonic case happens, and the results for the nonsonic case can not be applied directly. Without loss of generality, consider the local wave configuration where $t$-axis is within the 1-rarefaction wave. The local characteristic
coordinates $(\alpha,\beta)$ within the 1-rarefaction wave are also introduced similar to those in Subsection \ref{computation_u1t_p11t_rhot_astK}.

Denote $\phi:=u_1-c$, then
\begin{equation}\label{dphi}
\text{d}\phi=\text{d}u_1+\frac{c}{2\rho}\text{d}\rho-\frac{3}{2\rho c}\text{d}p_{11}=\text{d}u_1-\frac{1}{\rho c}\text{d}p_{11}-\frac{\rho^2}{2c}\text{d}S_1,
\end{equation}
where \eqref{dp11} has been used. By \eqref{eq:rho}, \eqref{eq:u1} and \eqref{eq:p11}, one has
\begin{equation*}\label{eq:phi}
\frac{\partial\phi}{\partial t}+(u_1-c)\frac{\partial\phi}{\partial x}=\Pi_1,
\end{equation*}
where   $\Pi_1=\frac{\rho^2}{2}\frac{\partial S_1}{\partial x}-\frac{1}{2}W_x$.
Denote $\mathbf{V}_0$ is the limiting value of $\mathbf{V}$ when $t\rightarrow0+$ along the $t$-axis. Then along this characteristic curve which is tangential to $t$-axis, one has $\beta_0=u_{1,0}-c_0=0$. It follows that
\begin{equation*}
\left(\frac{\partial\phi}{\partial t}\right)_0=\left(\frac{\partial\phi}{\partial t}\right)_0+(u_{1,0}-c_0)\left(\frac{\partial\phi}{\partial x}\right)_0=\left(\frac{\rho^2}{2}\frac{\partial S_1}{\partial x}\right)_0-\frac{1}{2}W_x(0).
\end{equation*}
Hence, by \eqref{dphi}, one has
\begin{equation}\label{eq1:u1_p11_t_0}
\left(\frac{\partial u_1}{\partial t}\right)_0-\frac{1}{\rho_0c_0}\left(\frac{\partial p_{11}}{\partial t}\right)_0=-\frac{1}{2}W_x(0),
\end{equation}
where the equality $\left(\frac{\partial S_1}{\partial t}\right)_0=-u_{1,0}\left(\frac{\partial S_1}{\partial x}\right)_0=-c_{0}\left(\frac{\partial S_1}{\partial x}\right)_0$ has been used.
Taking the limit as $t\rightarrow0$ in \eqref{eq:du1_dp11_dt}, and let $\beta=\beta_0=0$ in \eqref{rho^2_S1_t_0_beta} and \eqref{psi1_t_0_beta} obtains
\begin{equation}\label{eq2:u1_p11_t_0}
\left(\frac{\partial u_1}{\partial t}\right)_0+\frac{1}{\rho_0c_0}\left(\frac{\partial p_{11}}{\partial t}\right)_0=\widetilde{d}_{L,0},
\end{equation}
where
\[
\widetilde{d}_{L,0}=\frac{\rho_L^2S_{1,L}'}{4c_L^3}u_{1,0}(3c_L^2+u_{1,0}^2)-2\psi_{1,L}'u_{1,0}-\frac{1}{2}W_x(0),
\]
with the values of $S_{1,L}'$ and $\psi_{1,L}'$ being given in \eqref{psi1_L'}.

\begin{theorem}[Computing $\left(\frac{\partial u_1}{\partial t}\right)_0$, $\left(\frac{\partial p_{11}}{\partial t}\right)_0$ and $\left(\frac{\partial\rho}{\partial t}\right)_0$]
Assume that the $t$-axis is located inside the 1-rarefaction wave associated with the $u_1-c$ characteristic family. Then
\begin{eqnarray}
&\begin{split}
&\left(\frac{\partial u_1}{\partial t}\right)_0=\frac{1}{2}\left[\widetilde{d}_{L,0}-\frac{1}{2}W_x(0)\right], \\
&\left(\frac{\partial p_{11}}{\partial t}\right)_0=\frac{\rho_0c_0}{2}\left[\widetilde{d}_{L,0}+\frac{1}{2}W_x(0)\right],
\end{split} \label{u1_p11_t_0} \\
&~~\left(\frac{\partial\rho}{\partial t}\right)_0=\frac{1}{c_0^2}\left[\left(\frac{\partial p_{11}}{\partial t}\right)_0+\frac{S_{1,L}'}{c_L}\rho_0^3u_{1,0}^2\right]. \label{rho_t_0}
\end{eqnarray}
\end{theorem}

\begin{proof}
By \eqref{eq1:u1_p11_t_0} and \eqref{eq2:u1_p11_t_0},
\eqref{u1_p11_t_0} may be obtained. One takes $\beta=\beta_0=0$ in \eqref{S1_t_0beta} to get
\[
\left(\frac{\partial S_1}{\partial t}\right)_0=-\frac{S_{1,L}'}{c_L}u_{1,0}^2.
\]
Substituting it into
\[
\left(\frac{\partial\rho}{\partial t}\right)_0=\frac{1}{c_0^2}\left[\left(\frac{\partial p_{11}}{\partial t}\right)_0-\rho_0^3\left(\frac{\partial S_1}{\partial t}\right)_0\right]
\]
which is derived from \eqref{dp11}, one obtains \eqref{rho_t_0}.
\end{proof}

It is noted that for the sonic case, one has the relation $p_{11,0}=\frac13(\rho u_1^2)_0$ which is derived from $u_{1,0}-c_0=0$; and
\eqref{u2_p12_t} induces that if $p_{11}\neq\rho u_1^2$, one has
\begin{equation}
\begin{split}
&\frac{\mathcal{D}u_2}{\mathcal{D}t}=\frac{p_{11}}{p_{11}-\rho u_1^2}\left(\frac{\partial u_2}{\partial t}+\frac{2u_1p_{12}}{3p_{11}^2}\frac{\mathcal{D}p_{11}}{\mathcal{D}t}-\frac{u_1}{p_{11}}\frac{\partial p_{12}}{\partial t}\right), \\
&\frac{\mathcal{D}p_{12}}{\mathcal{D}t}=\frac{-p_{11}}{p_{11}-\rho u_1^2}\left[\rho u_1\left(\frac{\partial u_2}{\partial t}+\frac{2u_1p_{12}}{3p_{11}^2}\frac{\mathcal{D}p_{11}}{\mathcal{D}t}\right)-\frac{\partial p_{12}}{\partial t}\right].
\end{split}
\label{Du2_Dp12_t}
\end{equation}
 Hence, substituting \eqref{Du2_Dp12_t} into \eqref{eq1:u2_t_p12_t} and \eqref{eq2:u2_t_p12_t}, one can acquire the two equations which $\frac{\partial u_2}{\partial t}(0,\beta)$ and $\frac{\partial p_{12}}{\partial t}(0,\beta)$ satisfy inside the 1-rarefaction wave, and taking $\beta=\beta_0=0$ yields the following system
\begin{equation}
\begin{cases}
\widetilde{a}_{3,L,0}\left(\frac{\partial u_2}{\partial t}\right)_0+\widetilde{b}_{3,L,0}\left(\frac{\partial p_{12}}{\partial t}\right)_0
=\widetilde{d}_{3,L,0},  \\
\widetilde{a}_{4,L,0}\left(\frac{\partial u_2}{\partial t}\right)_0+\widetilde{b}_{4,L,0}\left(\frac{\partial p_{12}}{\partial t}\right)_0
=\widetilde{d}_{4,L,0},
\end{cases}
\label{eq12:u2_p12_t_0}
\end{equation}
where
\begin{align*}
&\widetilde{a}_{3,L,0}=1+\frac{\beta_L}{4}, ~ \widetilde{a}_{4,L,0}=\frac{\beta_Lc_0}{4\rho_0^2}, ~ \widetilde{b}_{3,L,0}=\frac{u_{1,0}}{p_{11,0}}\left(1-\frac{\beta_L}{4}\right), ~ \widetilde{b}_{4,L,0}=\frac{1}{\rho_0^3}\left(1-\frac{\beta_L}{2}\right), \\
&\widetilde{d}_{3,L,0}=\left[\frac{\partial\psi_2}{\partial\alpha}(0,\beta_L)-\frac{\Pi_{2,L}\beta_L}{8c_L}\right]\psi_{1,L}
-\frac{\beta_L}{4}\cdot\left\{\left(\frac{2p_{12}}{p_{11}}\right)_0
\cdot\left[\left(\frac{\partial u_1}{\partial t}\right)_0+\frac{1}{2}W_x(0)\right]
+\frac{S_{1,L}'}{c_L}\left(\frac{c\rho^2p_{12}}{2p_{11}}\right)_0\right\} \\
&\qquad+\left(\frac{\sqrt{3}p_{12}}{2\rho\sqrt{\rho p_{11}}}\frac{\partial\rho}{\partial t}\right)_0
+\left(\frac{\sqrt{3}p_{12}}{2p_{11}\sqrt{\rho p_{11}}}\frac{\partial p_{11}}{\partial t}\right)_0, \\
&\widetilde{d}_{4,L,0}=\left[\frac{\partial\psi_3}{\partial\alpha}(0,\beta_L)-\frac{\Pi_{3,L}\beta_L}{8c_L}\right]\psi_{1,L}
-\frac{\beta_L}{4}\cdot\left\{\left(\frac{3p_{12}}{c\rho^3}\right)_0\left[\left(\frac{\partial u_1}{\partial t}\right)_0+\frac{1}{2}W_x(0)\right]
+\frac{S_{1,L}'}{c_L}\left(\frac{3p_{12}}{\rho}\right)_0\right\}
+\left(\frac{3p_{12}}{\rho^4}\frac{\partial\rho}{\partial t}\right)_0.
\end{align*}
It is noted that \eqref{u1_p11_t} has been used during deriving the above coefficients.

\begin{theorem}[Computing $\left(\frac{\partial u_2}{\partial t}\right)_0$, $\left(\frac{\partial p_{12}}{\partial t}\right)_0$ and $\left(\frac{\partial p_{22}}{\partial t}\right)_0$]
Assume that the $t$-axis is located inside the 1-rarefaction wave associated with the $u_1-c$ characteristic family. Then
\begin{equation}
\left(\frac{\partial u_2}{\partial t}\right)_0=\frac{\widetilde{d}_{3,L,0}\widetilde{b}_{4,L,0}-\widetilde{d}_{4,L,0}\widetilde{b}_{3,L,0}}{\widetilde{a}_{3,L,0}\widetilde{b}_{4,L,0}-\widetilde{a}_{4,L,0}\widetilde{b}_{3,L,0}},
\quad
\left(\frac{\partial p_{12}}{\partial t}\right)_0=\frac{\widetilde{d}_{3,L,0}\widetilde{a}_{4,L,0}-\widetilde{d}_{4,L,0}\widetilde{a}_{3,L,0}}{\widetilde{a}_{4,L,0}\widetilde{b}_{3,L,0}-\widetilde{a}_{3,L,0}\widetilde{b}_{4,L,0}}.
\label{u2_p12_t_0}
\end{equation}
The limiting value $\left(\frac{\partial p_{22}}{\partial t}\right)_0$ is computed by
\begin{equation}\label{p22_t_0}
\left(\frac{p_{11}}{\rho^4}\frac{\partial p_{22}}{\partial t}\right)_{0}=-\frac{S_{2,L}'}{c_L}u_{1,0}^2-
\left(\frac{p_{22}}{\rho^4}\frac{\partial p_{11}}{\partial t}\right)_0+\left(\frac{2p_{12}}{\rho^4}\frac{\partial p_{12}}{\partial t}\right)_0+\left(\frac{4\det(\mathbf{p})}{\rho^5}\frac{\partial\rho}{\partial t}\right)_0.
\end{equation}
\end{theorem}

\begin{proof}
\eqref{u2_p12_t_0} is directly obtained by solving the system \eqref{eq12:u2_p12_t_0}. Taking $\beta=\beta_0=0$ in \eqref{p22_t_0beta} and using \eqref{S2_t_0beta}  yields \eqref{p22_t_0}.
\end{proof}

\begin{remark}
For the case that the $t$-axis is located inside the 6-rarefaction wave, computing $\left(\frac{\partial\mathbf{V}}{\partial t}\right)_0$ is presented in \ref{sonic_case_6rarefaction}.
\end{remark}

\subsection{Acoustic case}\label{acoustic case}

When $\mathbf{U}_L=\mathbf{U}_{\ast}=\mathbf{U}_R$ but $\mathbf{U}_L'\neq\mathbf{U}_R'$, the acoustic case happens. Only the linear waves emanate from the origin $(x,t)=(0,0)$. This case will be simpler than the general case discussed before.

\subsubsection{Computing $\left(\frac{\partial u_1}{\partial t}\right)_{\ast}$, $\left(\frac{\partial p_{11}}{\partial t}\right)_{\ast}$ and $\left(\frac{\partial\rho}{\partial t}\right)_{\ast K}$ $(K=L,R)$}

For the 1-acoustic wave in the left associated with $u_1-c$, since the solution is continuous across the 1-acoustic wave, taking the differentiation along it, one has
\[
\left(\frac{\partial u_1}{\partial t}\right)_{\ast}+(u_{1,\ast}-c_{\ast})\left(\frac{\partial u_1}{\partial x}\right)_{\ast}
=\frac{\partial u_{1,L}}{\partial t}+(u_{1,\ast}-c_{\ast})\frac{\partial u_{1,L}}{\partial x}.
\]
By \eqref{u1_x} and \eqref{p11_x}, it follows that
\begin{equation}\label{eq1:u1_t_Dp11_t_ast}
\left(\frac{\partial u_1}{\partial t}\right)_{\ast}-\frac{u_{1,\ast}-c_{\ast}}{3p_{11,\ast}}\left(\frac{\mathcal{D}p_{11}}{\mathcal{D}t}\right)_{\ast}
=-\frac{p_{11,L}'}{\rho_{\ast}}-c_{\ast}u_{1,L}'-\frac{1}{2}W_x(0).
\end{equation}
By resolving the 6-acoustic wave in the right associated with $u_1+c$, one has
\begin{equation}\label{eq2:u1_t_Dp11_t_ast}
\left(\frac{\partial u_1}{\partial t}\right)_{\ast}-\frac{u_{1,\ast}+c_{\ast}}{3p_{11,\ast}}\left(\frac{\mathcal{D}p_{11}}{\mathcal{D}t}\right)_{\ast}
=-\frac{p_{11,R}'}{\rho_{\ast}}+c_{\ast}u_{1,R}'-\frac{1}{2}W_x(0).
\end{equation}

\begin{theorem}[Computing $\left(\frac{\partial u_1}{\partial t}\right)_{\ast}$, $\left(\frac{\partial p_{11}}{\partial t}\right)_{\ast}$ and $\left(\frac{\partial\rho}{\partial t}\right)_{\ast K}$ $(K=L,R)$]
For the acoustic case, one has
\begin{align}
&\left(\frac{\partial u_1}{\partial t}\right)_{\ast}=\frac{1}{2c_{\ast}}\left[\left(-\frac{p_{11,L}'}{\rho_{\ast}}-c_{\ast}u_{1,L}'-\frac{1}{2}W_x(0)\right)(u_{1,\ast}+c_{\ast})+
\left(\frac{p_{11,R}'}{\rho_{\ast}}-c_{\ast}u_{1,R}'+\frac{1}{2}W_x(0)\right)(u_{1,\ast}-c_{\ast})\right], \label{u1_t_ast} \\
&\left(\frac{\partial p_{11}}{\partial t}\right)_{\ast}=-\frac{\rho_{\ast}c_{\ast}}{2}\left[\left(u_{1,L}'+\frac{p_{11,L}'}{\rho_{\ast}c_{\ast}}\right)(u_{1,\ast}+c_{\ast})
-\left(u_{1,R}'-\frac{p_{11,R}'}{\rho_{\ast}c_{\ast}}\right)(u_{1,\ast}-c_{\ast})\right], \label{p11_t_ast}
\end{align}
and
\begin{equation}
\left(\frac{\partial\rho}{\partial t}\right)_{\ast K}=\frac{1}{c_{\ast}^2}\left[\left(\frac{\partial p_{11}}{\partial t}\right)_{\ast}+u_{1,\ast}(p_{11,K}'-c_{\ast}^2\rho_K')\right], ~ K=L,R. \label{rho_t_astK_acoustic}
\end{equation}
\end{theorem}

\begin{proof}
Combining \eqref{eq1:u1_t_Dp11_t_ast} with \eqref{eq2:u1_t_Dp11_t_ast}   yields \eqref{u1_t_ast} and
\begin{equation}\label{Dp11_Dt_ast_acoustic}
\left(\frac{\mathcal{D}p_{11}}{\mathcal{D}t}\right)_{\ast}=\frac{3p_{11,\ast}}{2c_{\ast}}\left[\frac{1}{\rho_{\ast}}(p_{11,R}'-p_{11,L}')-c_{\ast}(u_{1,L}'+u_{1,R}')\right].
\end{equation}
By utilizing \eqref{u1_p11_t} and
\[
\left(\frac{\mathcal{D}u_1}{\mathcal{D}t}\right)_{\ast}=\left(\frac{\partial u_1}{\partial t}\right)_{\ast}-\frac{u_{1,\ast}}{3p_{11,\ast}}\left(\frac{\mathcal{D}p_{11}}{\mathcal{D}t}\right)_{\ast},
\]
where \eqref{u1_x} has been used,  one obtains \eqref{p11_t_ast}. Applying \eqref{dp11} and \eqref{S1_t_0beta} gives \eqref{rho_t_astK_acoustic}.
\end{proof}

\subsubsection{Computing $\left(\frac{\partial u_2}{\partial t}\right)_{\ast K}$, $\left(\frac{\partial p_{12}}{\partial t}\right)_{\ast K}$ and $\left(\frac{\partial p_{22}}{\partial t}\right)_{\ast K}$ $(K=L,R)$}

Taking the differentiation along the 1-acoustic wave for $u_2$ and $p_{12}$ gives
\begin{align*}
&\left(\frac{\partial u_2}{\partial t}\right)_{\ast L}+(u_{1,\ast}-c_{\ast})\left(\frac{\partial u_2}{\partial x}\right)_{\ast L}
=\frac{\partial u_{2,L}}{\partial t}+(u_{1,\ast}-c_{\ast})u_{2,L}', \\
&\left(\frac{\partial p_{12}}{\partial t}\right)_{\ast L}+(u_{1,\ast}-c_{\ast})\left(\frac{\partial p_{12}}{\partial x}\right)_{\ast L}
=\frac{\partial p_{12,L}}{\partial t}+(u_{1,\ast}-c_{\ast})p_{12,L}'.
\end{align*}
Applying \eqref{eq:u2}, \eqref{eq:p12}, \eqref{u2_x} and \eqref{p12_x} to above two equations obtains
\begin{align*}
&\left(\frac{\mathcal{D}u_2}{\mathcal{D}t}\right)_{\ast L}+\frac{c_{\ast}}{p_{11,\ast}}\left(\frac{\mathcal{D}p_{12}}{\mathcal{D}t}\right)_{\ast L}=\frac{2p_{12,\ast}c_{\ast}}{3p_{11,\ast}^2}\left(\frac{\mathcal{D}p_{11}}{\mathcal{D}t}\right)_{\ast}-\frac{p_{12,L}'}{\rho_L}-c_{\ast}u_{2,L}',
 \\
&\rho_{\ast}c_{\ast}\left(\frac{\mathcal{D}u_2}{\mathcal{D}t}\right)_{\ast L}+\left(\frac{\mathcal{D}p_{12}}{\mathcal{D}t}\right)_{\ast L}=-c_{\ast}p_{12,L}'-2p_{12,L}u_{1,L}'-p_{11,L}u_{2,L}'.
\end{align*}
It follows that
\begin{equation}
\begin{split}
&\left(\frac{\mathcal{D}u_2}{\mathcal{D}t}\right)_{\ast L}=-\frac{p_{12,L}'}{\rho_L}-\frac{c_{\ast}p_{12,L}}{p_{11,\ast}}u_{1,L}'-\frac{p_{12,\ast}c_{\ast}}{3p_{11,\ast}^2}\left(\frac{\mathcal{D}p_{11}}{\mathcal{D}t}\right)_{\ast},
\\
&\left(\frac{\mathcal{D}p_{12}}{\mathcal{D}t}\right)_{\ast L}=\frac{p_{12,\ast}}{p_{11,\ast}}\left(\frac{\mathcal{D}p_{11}}{\mathcal{D}t}\right)_{\ast}+p_{12,L}u_{1,L}'-p_{11,L}u_{2,L}'.
\end{split}
\label{Du2_Dp12_Dt_astL_acoustic}
\end{equation}
Taking the differentiation along the 1-acoustic wave for $p_{22}$ yields
\begin{equation}\label{p22_x_astL}
\left(\frac{\partial p_{22}}{\partial x}\right)_{\ast L}=\frac{1}{c_{\ast}}\left[\left(\frac{\mathcal{D}p_{22}}{\mathcal{D}t}\right)_{\ast L}+p_{22,L}u_{1,L}'+2p_{12,L}u_{2,L}'+c_{\ast}p_{22,L}'\right],
\end{equation}
where \eqref{eq:p22} has been used.
Moreover,  resolving the 6-acoustic wave  has
\begin{equation}\label{Du2_Dp12_Dt_astR_acoustic and p22_x_astR}
\begin{cases}
\left(\frac{\mathcal{D}u_2}{\mathcal{D}t}\right)_{\ast R}=-\frac{p_{12,R}'}{\rho_R}+\frac{c_{\ast}p_{12,R}}{p_{11,\ast}}u_{1,R}'+\frac{p_{12,\ast}c_{\ast}}{3p_{11,\ast}^2}\left(\frac{\mathcal{D}p_{11}}{\mathcal{D}t}\right)_{\ast},
\\ 
\left(\frac{\mathcal{D}p_{12}}{\mathcal{D}t}\right)_{\ast R}=\frac{p_{12,\ast}}{p_{11,\ast}}\left(\frac{\mathcal{D}p_{11}}{\mathcal{D}t}\right)_{\ast}+p_{12,R}u_{1,R}'-p_{11,R}u_{2,R}',
\\ 
\left(\frac{\partial p_{22}}{\partial x}\right)_{\ast R}=\frac{1}{c_{\ast}}\left[-p_{22,R}u_{1,R}'-2p_{12,R}u_{2,R}'+c_{\ast}p_{22,R}'-\left(\frac{\mathcal{D}p_{22}}{\mathcal{D}t}\right)_{\ast R}\right]. 
\end{cases}
\end{equation}

\begin{theorem}[Computing $\left(\frac{\partial u_2}{\partial t}\right)_{\ast K}$, $\left(\frac{\partial p_{12}}{\partial t}\right)_{\ast K}$ and $\left(\frac{\partial p_{22}}{\partial t}\right)_{\ast K}$ $K=L,R$]
For the acoustic case, one has
\begin{equation}
\begin{split}\label{u2_p12_t_astK_acoustic}
&\left(\frac{\partial u_2}{\partial t}\right)_{\ast K}=\left(\frac{\mathcal{D}u_2}{\mathcal{D}t}\right)_{\ast K}+\frac{u_{1,\ast}}{p_{11,\ast}}\left[\left(\frac{\mathcal{D}p_{12}}{\mathcal{D}t}\right)_{\ast K}-\frac{2p_{12,\ast}}{3p_{11,\ast}}\left(\frac{\mathcal{D}p_{11}}{\mathcal{D}t}\right)_{\ast}\right], \\ 
&\left(\frac{\partial p_{12}}{\partial t}\right)_{\ast K}=\left(\frac{\mathcal{D}p_{12}}{\mathcal{D}t}\right)_{\ast K}+\rho_{\ast}u_{1,\ast}\left(\frac{\mathcal{D}u_2}{\mathcal{D}t}\right)_{\ast K}, 
\end{split}
\end{equation}
and
\begin{equation}
\left(\frac{\partial p_{22}}{\partial t}\right)_{\ast K}=\left(\frac{\mathcal{D}p_{22}}{\mathcal{D}t}\right)_{\ast K}-u_{1,\ast}\left(\frac{\partial p_{22}}{\partial x}\right)_{\ast K}, \label{p22_t_astK_acoustic}
\end{equation}
where
\begin{equation}\label{Dp22_Dt_astK_acoustic}
\left(\frac{\mathcal{D}p_{22}}{\mathcal{D}t}\right)_{\ast K}=\left(\frac{p_{11}p_{22}-4p_{12}^2}{3p_{11}^2}\frac{\mathcal{D}p_{11}}{\mathcal{D}t}\right)_{\ast}+\frac{2p_{12,\ast}}{p_{11,\ast}}\left(\frac{\mathcal{D}p_{12}}{\mathcal{D}t}\right)_{\ast K},
\end{equation}
and $\left(\frac{\mathcal{D}p_{11}}{\mathcal{D}t}\right)_{\ast}$, $\left(\frac{\mathcal{D}u_2}{\mathcal{D}t}\right)_{\ast K}$, $\left(\frac{\mathcal{D}p_{12}}{\mathcal{D}t}\right)_{\ast K}$ and $\left(\frac{\partial p_{22}}{\partial x}\right)_{\ast K}$ are given by
\eqref{Dp11_Dt_ast_acoustic}-\eqref{Du2_Dp12_Dt_astR_acoustic and p22_x_astR}, $K=L,R$.
\end{theorem}

\begin{proof}
Utilizing \eqref{u2_p12_t} yields \eqref{u2_p12_t_astK_acoustic}. \eqref{p22_t_astK_acoustic} is direct, and applying \eqref{Dp22_t} gives \eqref{Dp22_Dt_astK_acoustic}.
\end{proof}

\subsubsection{Computing $\left(\frac{\partial u_2}{\partial t}\right)_{\ast\ast}$, $\left(\frac{\partial p_{12}}{\partial t}\right)_{\ast\ast}$ and $\left(\frac{\partial p_{22}}{\partial t}\right)_{\ast\ast K}$ $(K=L,R)$}

Since the solution is continuous across the 2-acoustic wave, taking the differentiation along the 2-acoustic wave for $u_2$, one has
\[
\left(\frac{\partial u_2}{\partial t}\right)_{\ast\ast}+(u_{1,\ast}-\frac{c_{\ast}}{\sqrt{3}})\left(\frac{\partial u_2}{\partial x}\right)_{\ast\ast}=\left(\frac{\partial u_2}{\partial t}\right)_{\ast L}+(u_{1,\ast}-\frac{c_{\ast}}{\sqrt{3}})\left(\frac{\partial u_2}{\partial x}\right)_{\ast L}.
\]
Applying \eqref{u2_x} to the above formula gives
\begin{equation*}\label{eq1:Du2_Dp12_Dt_ast2_acoustic}
\left(\frac{\mathcal{D}u_2}{\mathcal{D}t}\right)_{\ast\ast}+\frac{c_{\ast}}{\sqrt{3}p_{11,\ast}}\left(\frac{\mathcal{D}p_{12}}{\mathcal{D}t}\right)_{\ast\ast}
=\left(\frac{\mathcal{D}u_2}{\mathcal{D}t}\right)_{\ast L}+\frac{c_{\ast}}{\sqrt{3}p_{11,\ast}}\left(\frac{\mathcal{D}p_{12}}{\mathcal{D}t}\right)_{\ast L}.
\end{equation*}
Resolving the 5-acoustic wave yields
\begin{equation*}\label{eq2:Du2_Dp12_Dt_ast2_acoustic}
\left(\frac{\mathcal{D}u_2}{\mathcal{D}t}\right)_{\ast\ast}-\frac{c_{\ast}}{\sqrt{3}p_{11,\ast}}\left(\frac{\mathcal{D}p_{12}}{\mathcal{D}t}\right)_{\ast\ast}
=\left(\frac{\mathcal{D}u_2}{\mathcal{D}t}\right)_{\ast R}-\frac{c_{\ast}}{\sqrt{3}p_{11,\ast}}\left(\frac{\mathcal{D}p_{12}}{\mathcal{D}t}\right)_{\ast R}.
\end{equation*}
Then it follows from the above two equations that
\begin{equation}
\begin{split}\label{Du2_Dp12_Dt_ast2_acoustic}
&\left(\frac{\mathcal{D}u_2}{\mathcal{D}t}\right)_{\ast\ast}=\frac{1}{2}\left[\left(\frac{\mathcal{D}u_2}{\mathcal{D}t}\right)_{\ast L}+\left(\frac{\mathcal{D}u_2}{\mathcal{D}t}\right)_{\ast R}+\frac{c_{\ast}}{\sqrt{3}p_{11,\ast}}\left(\left(\frac{\mathcal{D}p_{12}}{\mathcal{D}t}\right)_{\ast L}-\left(\frac{\mathcal{D}p_{12}}{\mathcal{D}t}\right)_{\ast R}\right)\right], \\
&\left(\frac{\mathcal{D}p_{12}}{\mathcal{D}t}\right)_{\ast\ast}=\frac{\sqrt{3}p_{11,\ast}}{2c_{\ast}}\left[\left(\frac{\mathcal{D}u_2}{\mathcal{D}t}\right)_{\ast L}-\left(\frac{\mathcal{D}u_2}{\mathcal{D}t}\right)_{\ast R}+\frac{c_{\ast}}{\sqrt{3}p_{11,\ast}}\left(\left(\frac{\mathcal{D}p_{12}}{\mathcal{D}t}\right)_{\ast L}+\left(\frac{\mathcal{D}p_{12}}{\mathcal{D}t}\right)_{\ast R}\right)\right].
\end{split}
\end{equation}
Taking the differentiation along the 2-acoustic wave or the 5-acoustic wave for $p_{22}$, one has
\begin{equation}
\begin{split}\label{p22_x_ast2L_ast2R}
&\left(\frac{\partial p_{22}}{\partial x}\right)_{\ast\ast L}=\frac{\sqrt{3}}{c_{\ast}}\left[\left(\frac{\mathcal{D}p_{22}}{\mathcal{D}t}\right)_{\ast\ast}-\left(\frac{\mathcal{D}p_{22}}{\mathcal{D}t}\right)_{\ast L}\right]-\left(\frac{\partial p_{22}}{\partial x}\right)_{\ast L}, \\
&\left(\frac{\partial p_{22}}{\partial x}\right)_{\ast\ast R}=\left(\frac{\partial p_{22}}{\partial x}\right)_{\ast R}+\frac{\sqrt{3}}{c_{\ast}}\left[\left(\frac{\mathcal{D}p_{22}}{\mathcal{D}t}\right)_{\ast R}-\left(\frac{\mathcal{D}p_{22}}{\mathcal{D}t}\right)_{\ast\ast}\right],
\end{split}
\end{equation}
where $\left(\frac{\partial p_{22}}{\partial x}\right)_{\ast L}$ and $\left(\frac{\partial p_{22}}{\partial x}\right)_{\ast R}$ are given in \eqref{p22_x_astL} and \eqref{Du2_Dp12_Dt_astR_acoustic and p22_x_astR}, respectively, and $\left(\frac{\mathcal{D}p_{22}}{\mathcal{D}t}\right)_{\ast K}$ ($K=L,R$) are given by \eqref{Dp22_Dt_astK_acoustic}.

\begin{theorem}[Computing $\left(\frac{\partial u_2}{\partial t}\right)_{\ast\ast}$, $\left(\frac{\partial p_{12}}{\partial t}\right)_{\ast\ast}$ and $\left(\frac{\partial p_{22}}{\partial t}\right)_{\ast\ast K}$ $(K=L,R)$]
For the acoustic case, one has
\begin{equation}
\begin{split}\label{u2_p12_t_ast2_acoustic}
&\left(\frac{\partial u_2}{\partial t}\right)_{\ast\ast}=\left(\frac{\mathcal{D}u_2}{\mathcal{D}t}\right)_{\ast\ast}+
\frac{u_{1,\ast}}{p_{11,\ast}}\left[\left(\frac{\mathcal{D}p_{12}}{\mathcal{D}t}\right)_{\ast\ast}
-\frac{2p_{12,\ast}}{3p_{11,\ast}}\left(\frac{\mathcal{D}p_{11}}{\mathcal{D}t}\right)_{\ast}\right], \\
&\left(\frac{\partial p_{12}}{\partial t}\right)_{\ast\ast}=\left(\frac{\mathcal{D}p_{12}}{\mathcal{D}t}\right)_{\ast\ast}
+\rho_{\ast}u_{1,\ast}\left(\frac{\mathcal{D}u_2}{\mathcal{D}t}\right)_{\ast\ast},
\end{split}
\end{equation}
and
\begin{equation}
\left(\frac{\partial p_{22}}{\partial t}\right)_{\ast\ast K}=\left(\frac{\mathcal{D}p_{22}}{\mathcal{D}t}\right)_{\ast\ast}
-u_{1,\ast}\left(\frac{\partial p_{22}}{\partial x}\right)_{\ast\ast K}, \label{p22_t_ast2K_acoustic}
\end{equation}
where
\begin{equation}\label{Dp22_Dt_ast2_acoustic}
\left(\frac{\mathcal{D}p_{22}}{\mathcal{D}t}\right)_{\ast\ast}=\left(\frac{p_{11}p_{22}-4p_{12}^2}{3p_{11}^2}\frac{\mathcal{D}p_{11}}{\mathcal{D}t}\right)_{\ast}
+\frac{2p_{12,\ast}}{p_{11,\ast}}\left(\frac{\mathcal{D}p_{12}}{\mathcal{D}t}\right)_{\ast\ast},
\end{equation}
and $\left(\frac{\mathcal{D}p_{11}}{\mathcal{D}t}\right)_{\ast}$, $\left(\frac{\mathcal{D}u_2}{\mathcal{D}t}\right)_{\ast\ast}$, $\left(\frac{\mathcal{D}p_{12}}{\mathcal{D}t}\right)_{\ast\ast}$ and $\left(\frac{\partial p_{22}}{\partial x}\right)_{\ast\ast K}$ are given by \eqref{Dp11_Dt_ast_acoustic}, \eqref{Du2_Dp12_Dt_ast2_acoustic} and \eqref{p22_x_ast2L_ast2R},  $K=L,R$.
\end{theorem}

\begin{proof}
Utilizing \eqref{u2_p12_t} yields \eqref{u2_p12_t_ast2_acoustic}. \eqref{p22_t_ast2K_acoustic} is direct, and applying \eqref{Dp22_t} gives \eqref{Dp22_Dt_ast2_acoustic}.
\end{proof}

\section{Numerical experiments}\label{numerical experiments}
In this section, some initial value and initial-boundary-value problems of the 1D and 2D ten-moment equations will be solved to verify the accuracy and performance of the proposed GRP schemes. It should be emphasized that several examples of 2D Riemann problems are constructed for the first time.
For the 1D case, the time step size is computed by
\[
\Delta t=C_{\text{cfl}}\frac{\Delta x}{\max\limits_j\{\lambda_{\max}^x(\mathbf{V}_j)\}},
\]
where $\lambda_{\max}^x(\mathbf{V}_j):=|u_{1,j}|+\sqrt{\frac{3p_{11,j}}{\rho_j}}$, $j=1,\cdots,N_x$ with $N_x$ denoting the number of cells in the $x$-direction. For the 2D case, the time step size is computed by
\[
\Delta t=\frac{C_{\text{cfl}}}{\frac{\max\limits_{j,k}\{\lambda_{\max}^x(\mathbf{V}_{j,k})\}}{\Delta x}+\frac{\max\limits_{j,k}\{\lambda_{\max}^y(\mathbf{V}_{j,k})
\}}{\Delta y}},
\]
where $\lambda_{\max}^x(\mathbf{V}_{j,k}):=|u_{1,j,k}|+\sqrt{\frac{3p_{11,j,k}}{\rho_{j,k}}}$ and $\lambda_{\max}^y(\mathbf{V}_{j,k}):=|u_{2,j,k}|+\sqrt{\frac{3p_{22,j,k}}{\rho_{j,k}}}$, $j=1,\cdots,N_x$, $k=1,\cdots,N_y$ with $N_x$, $N_y$ denoting respectively the numbers of cells in the $x$- and $y$-directions. The CFL number $C_{\text{cfl}}$ is taken as 0.45.
Without special station, the van Leer slope limiter \eqref{vanLeer limiter} with characteristic decomposition is utilized.

\begin{example}{Accuracy test 1 without source term}\label{accuracy_test1}
\rm

The first example is a smooth problem within the interval $[-0.5, 0.5]$ without source term, and the exact solution is
\[
\rho = 2+ \sin(2\pi(x - t)), ~ u_1 = 1, ~ u_2 = 0, ~ p_{11} = 1, ~ p_{12} = 0, ~ p_{22} = 1.
\]
The periodic boundary conditions are imposed. The proposed GRP scheme is applied to solve this problem up to the final time $t=0.5$. Table \ref{sin_nr1} shows the errors and  corresponding orders of the density $\rho$. One can see that almost second-order convergence rates are obtained for the $l^1$ and $l^2$ errors, but the $l^\infty$ convergence rate is lower than 1.5, which is coincided with those convergent results in \cite{kuang2019second}.

\end{example}

\begin{table}[!htbp]
\centering
\caption{Example \ref{accuracy_test1}: The convergent results of $\rho$ at $t=0.5$ for the accuracy test 1.}
\label{sin_nr1}
\begin{center}
\small
\begin{tabular}{c|c|cc|cc|cc}
\hline
& $N_x$ & $l^1$ error & order & $l^2$ error & order & $l^\infty$ error & order \\
\hline
\multirow{7}{*}{$\rho$} & 10 &   8.8236e-02 & -- &      1.1507e-01 & -- &     2.1657e-01 & -- \\
\multirow{7}{*}{} & 20 &   3.4526e-02 & 1.35 &      4.1664e-02 & 1.47 &     7.8594e-02 & 1.46 \\
\multirow{7}{*}{} & 40 &   9.6788e-03 & 1.83 &      1.2793e-02 & 1.70 &     3.1038e-02 & 1.34 \\
\multirow{7}{*}{} & 80 &   2.4275e-03 & 2.00 &      3.8036e-03 & 1.75 &     1.1914e-02 & 1.38 \\
\multirow{7}{*}{} & 160 &   5.7060e-04 & 2.09 &      1.1138e-03 & 1.77 &     4.4951e-03 & 1.41 \\
\multirow{7}{*}{} & 320 &   1.3677e-04 & 2.06 &      3.2356e-04 & 1.78 &     1.6742e-03 & 1.42 \\
\multirow{7}{*}{} & 640 &   3.2039e-05 & 2.09 &      9.3616e-05 & 1.79 &     6.1817e-04 & 1.44 \\
\hline
\end{tabular}
\end{center}
\end{table}

\begin{example}{Accuracy test 2 with nontrivial source term}\label{accuracy_test2}
\rm

This example examines a smooth problem within the interval $[-0.25, 0.25]$ with the nontrivial potential $W(x) = x$, and the exact solution is
\[
\rho = \epsilon + \sin^2(2\pi(x - t)), ~ u_1 = 1, ~ u_2 = 0, ~ p_{11} = 1 + (t - x)\left(\frac{\epsilon}{2} + \frac{1}{4}\right) + \frac{\sin(4\pi(x - t))}{16\pi}, ~ p_{12} = 0, ~ p_{22} = 1.
\]
The parameter $\epsilon$ is taken as $10^{-2}$. For the density $\rho$, the periodic boundary conditions are imposed, and for the pressure component $p_{11}$, the exact boundary conditions are imposed. The proposed GRP scheme is applied to solve this problem up to the final time $t=0.1$. Table \ref{sin_nr2} shows the errors and  corresponding orders of the density $\rho$ and the pressure component $p_{11}$. One can see that for the density $\rho$, there are similar convergent rates with those in accuracy test 1, and for the pressure component $p_{11}$, almost second-order convergent rates are obtained on the fine meshes, which may be due to the $p_{11}$ being close to a linear function.

\end{example}

\begin{table}[!htbp]
\centering
\caption{Example \ref{accuracy_test2}: The convergent results of $\rho$ and $p_{11} $at $t=0.1$ for the accuracy test 2 with $\epsilon=10^{-2}$.}
\label{sin_nr2}
\begin{center}
\small
\begin{tabular}{c|c|cc|cc|cc}
\hline
& $N_x$ & $l^1$ error & order & $l^2$ error & order & $l^\infty$ error & order \\
\hline
\multirow{7}{*}{$\rho$} & 10 &   2.6529e-02 & -- &      3.2438e-02 & -- &     5.5835e-02 & -- \\
\multirow{7}{*}{} & 20 &   1.0243e-02 & 1.37 &      1.1660e-02 & 1.48 &     2.3462e-02 & 1.25 \\
\multirow{7}{*}{} & 40 &   2.6514e-03 & 1.95 &      3.6216e-03 & 1.69 &     9.4498e-03 & 1.31 \\
\multirow{7}{*}{} & 80 &   6.6426e-04 & 2.00 &      1.0997e-03 & 1.72 &     3.6870e-03 & 1.36 \\
\multirow{7}{*}{} & 160 &   1.6390e-04 & 2.02 &      3.2991e-04 & 1.74 &     1.4114e-03 & 1.39 \\
\multirow{7}{*}{} & 320 &   4.0720e-05 & 2.01 &      9.7789e-05 & 1.75 &     5.3239e-04 & 1.41 \\
\multirow{7}{*}{} & 640 &   1.0855e-05 & 1.91 &      2.9285e-05 & 1.74 &     1.9890e-04 & 1.42 \\
\hline
\multirow{7}{*}{$p_{11}$} & 10 &   3.9221e-03 & -- &      4.8356e-03 & -- &     8.5007e-03 & -- \\
\multirow{7}{*}{} & 20 &   2.3719e-03 & 0.73 &      3.5839e-03 & 0.43 &     1.2636e-02 & -0.57 \\
\multirow{7}{*}{} & 40 &   9.8731e-04 & 1.26 &      1.6739e-03 & 1.10 &     8.0796e-03 & 0.65 \\
\multirow{7}{*}{} & 80 &   2.9479e-04 & 1.74 &      5.3125e-04 & 1.66 &     2.3384e-03 & 1.79 \\
\multirow{7}{*}{} & 160 &   8.1947e-05 & 1.85 &      1.6025e-04 & 1.73 &     1.0013e-03 & 1.22 \\
\multirow{7}{*}{} & 320 &   2.2038e-05 & 1.89 &      4.1142e-05 & 1.96 &     1.7405e-04 & 2.52 \\
\multirow{7}{*}{} & 640 &   6.5996e-06 & 1.74 &      1.0550e-05 & 1.96 &     3.9586e-05 & 2.14 \\
\hline
\end{tabular}
\end{center}
\end{table}

\begin{example}{1D Riemann problems}\label{1D_RP}
\rm

Three 1D Riemann problems are considered here to test the wave-capturing ability of the proposed GRP scheme. Their initial conditions and setups are as follows:

(i) 1D Riemann problem 1:
\begin{equation}\label{1D_RP1}
(\rho,u_1,u_2,p_{11},p_{12},p_{22})=\begin{cases}
                                      (1,0,0,2,0.05,0.6), &  x\leq0.5, \\
                                      (0.125,0,0,0.2,0.1,0.2), &  x>0.5.
                                    \end{cases}
\end{equation}
The discontinuity is located at $x=0.5$ in the interval $[0,1]$. The final time is $t=0.125$. This is the Sod shock-tube problem, and the solutions at the final time contains five waves: the left rarefaction wave, the left shear wave, the middle contact discontinuity, the right shear wave, and the right shock wave, see Figure \ref{RP1_nr}.

(ii) 1D Riemann problem 2:
\begin{equation}\label{1D_RP2}
(\rho,u_1,u_2,p_{11},p_{12},p_{22})=\begin{cases}
                                      (1,1,1,1,0,1), &  x\leq0, \\
                                      (1,-1,-1,1,0,1), &  x>0.
                                    \end{cases}
\end{equation}
The discontinuity is located at $x=0$ in the interval $[-0.5,0.5]$. The final time is $t=0.125$. The solutions at the final time contains four waves: the left shock wave, the left shear wave, the right shear wave, and the right shock wave, see Figure \ref{RP3_nr}.

(iii) 1D Riemann problem 3:
\begin{equation}\label{1D_RP3}
(\rho,u_1,u_2,p_{11},p_{12},p_{22})=\begin{cases}
                                      (2,-0.5,-0.5,1.5,0.5,1.5), &  x\leq0, \\
                                      (1,1,1,1,0,1), &  x>0.
                                    \end{cases}
\end{equation}
The discontinuity is located at $x=0$ in the interval $[-0.5,0.5]$. The final time is $t=0.15$. The solutions at the final time contains five waves: the left rarefaction wave, the left shear wave, the middle contact discontinuity, the right shear wave, and the right rarefaction wave, see Figure \ref{RP4_nr}.

For the above three problems, the outflow boundary conditions are imposed. They are simulated by the GRP scheme on a mesh consisting of 400 cells. The numerical results are respectively presented in Figures \ref{RP1_nr}-\ref{RP4_nr}. One can find that the GRP scheme captures all possible waves well and has high resolution for the shock waves.

\end{example}

\begin{figure}[!htbp]
\subfigure[$\rho$]{
\begin{minipage}[c]{0.3\linewidth}
\centering
\includegraphics[width=5cm]{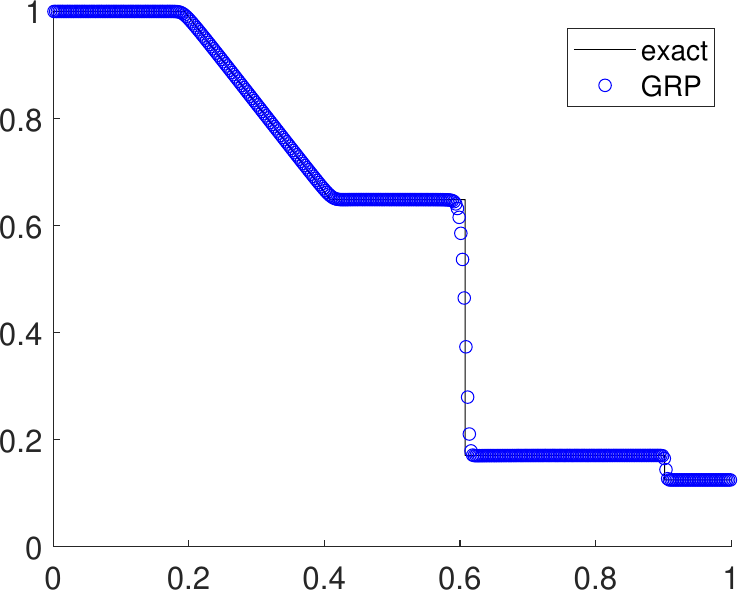}
\end{minipage}
}
\subfigure[$u_1$]{
\begin{minipage}[c]{0.3\linewidth}
\centering
\includegraphics[width=5cm]{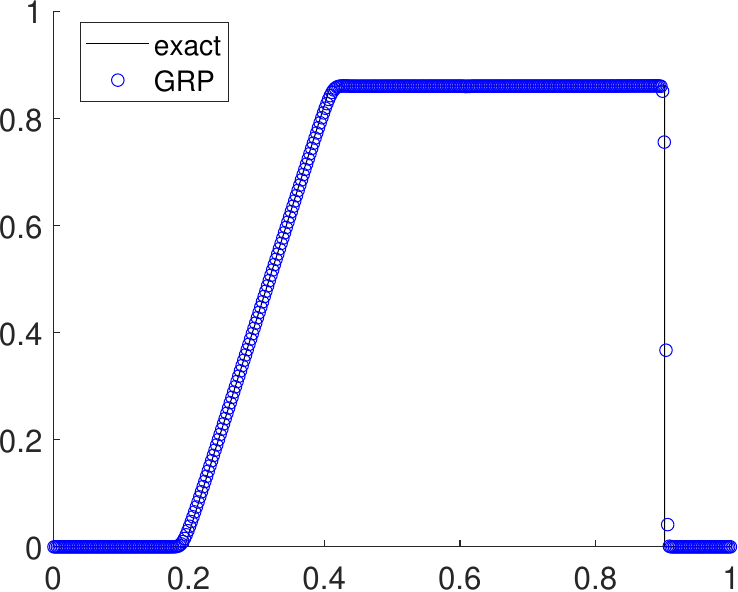}
\end{minipage}
}
\subfigure[$u_2$]{
\begin{minipage}[c]{0.3\linewidth}
\centering
\includegraphics[width=5cm]{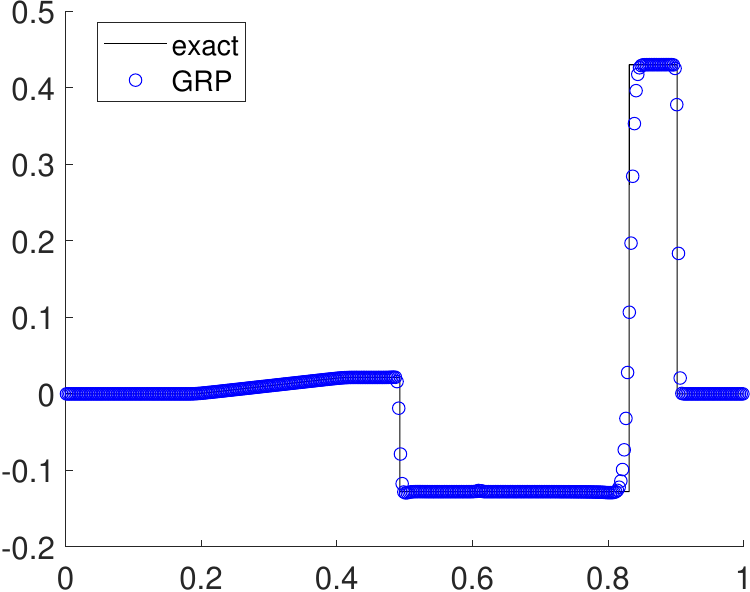}
\end{minipage}
}
\subfigure[$p_{11}$]{
\begin{minipage}[c]{0.3\linewidth}
\centering
\includegraphics[width=5cm]{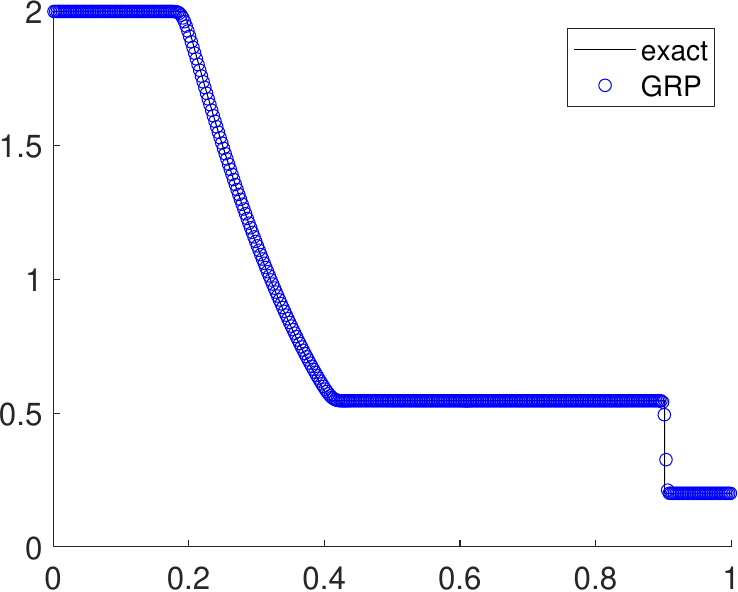}
\end{minipage}
}
\subfigure[$p_{12}$]{
\begin{minipage}[c]{0.3\linewidth}
\centering
\includegraphics[width=5cm]{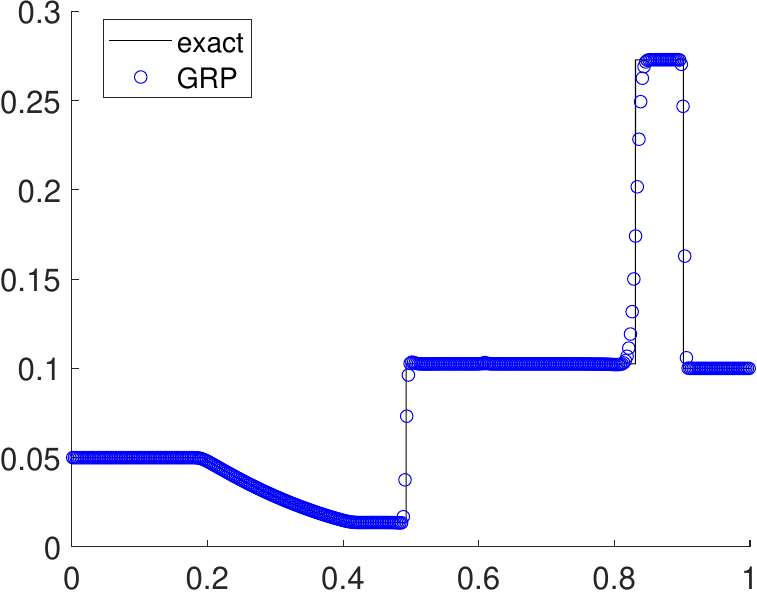}
\end{minipage}
}
\subfigure[$p_{22}$]{
\begin{minipage}[c]{0.3\linewidth}
\centering
\includegraphics[width=5cm]{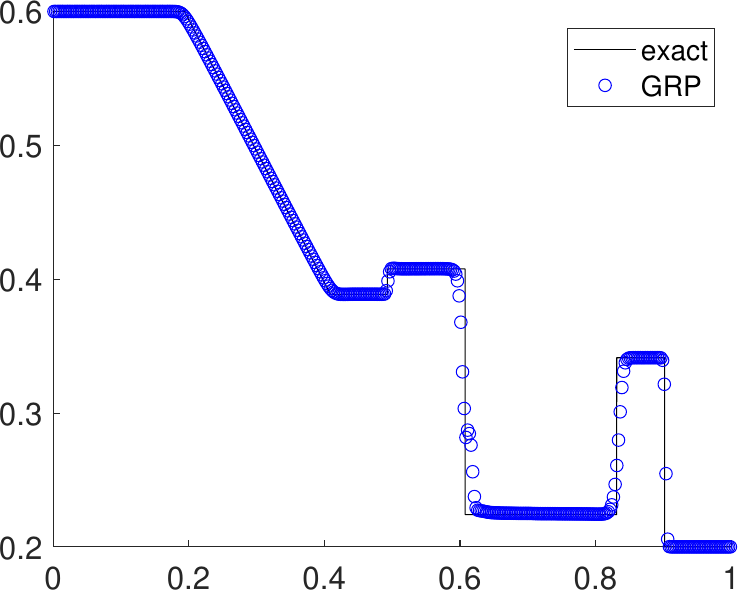}
\end{minipage}
}
\centering
\caption{Example \ref{1D_RP}: The numerical solutions of the GRP scheme on 400 uniform cells for the 1D Riemann problem \eqref{1D_RP1}.}
\label{RP1_nr}
\end{figure}

\begin{figure}[!htbp]
\subfigure[$\rho$]{
\begin{minipage}[c]{0.3\linewidth}
\centering
\includegraphics[width=5cm]{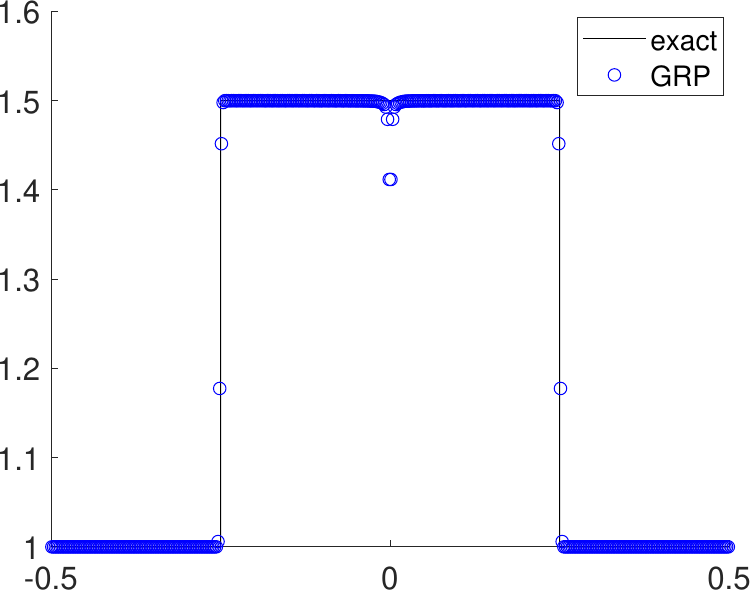}
\end{minipage}
}
\subfigure[$u_1$]{
\begin{minipage}[c]{0.3\linewidth}
\centering
\includegraphics[width=5cm]{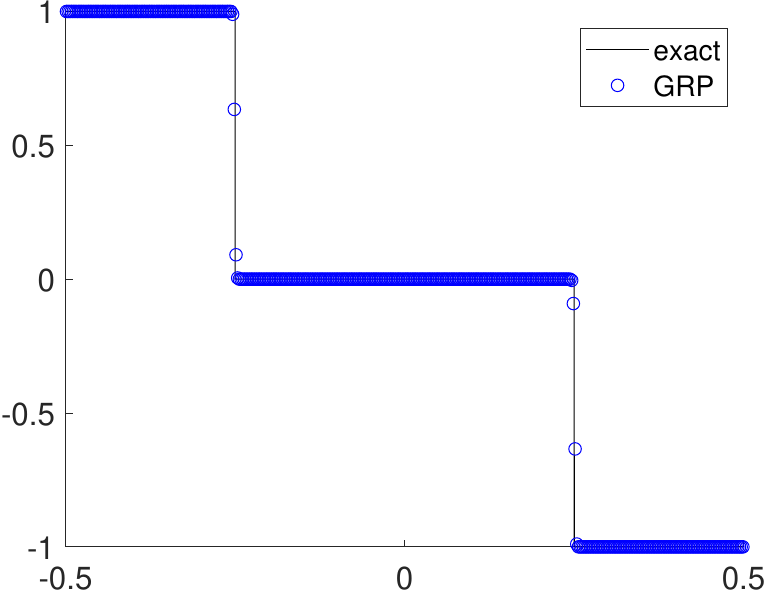}
\end{minipage}
}
\subfigure[$u_2$]{
\begin{minipage}[c]{0.3\linewidth}
\centering
\includegraphics[width=5cm]{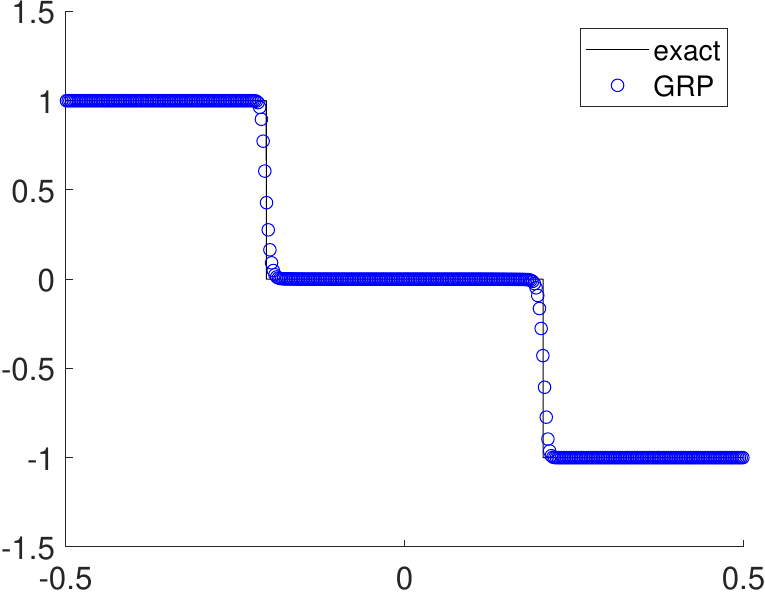}
\end{minipage}
}
\subfigure[$p_{11}$]{
\begin{minipage}[c]{0.3\linewidth}
\centering
\includegraphics[width=5cm]{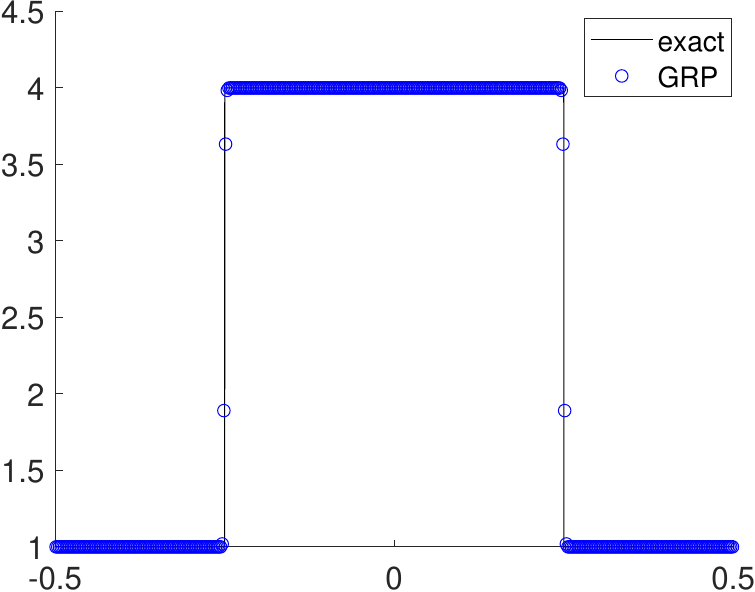}
\end{minipage}
}
\subfigure[$p_{12}$]{
\begin{minipage}[c]{0.3\linewidth}
\centering
\includegraphics[width=5cm]{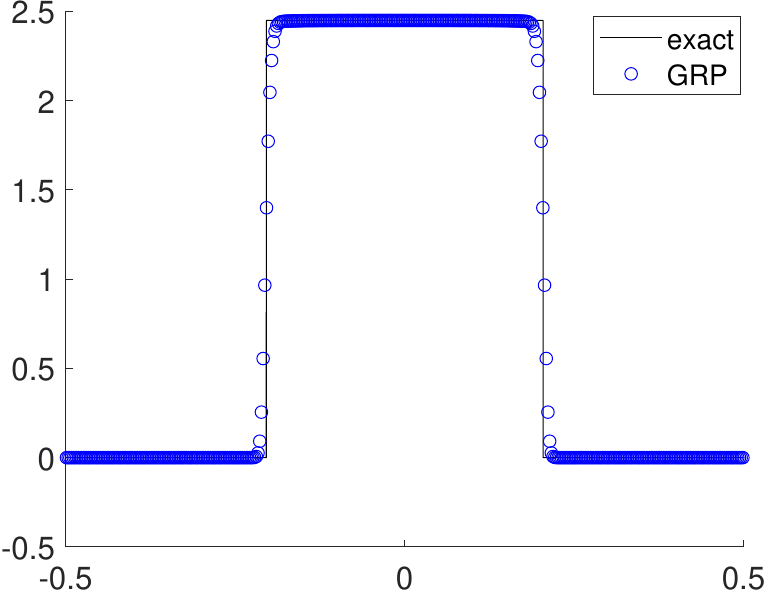}
\end{minipage}
}
\subfigure[$p_{22}$]{
\begin{minipage}[c]{0.3\linewidth}
\centering
\includegraphics[width=5cm]{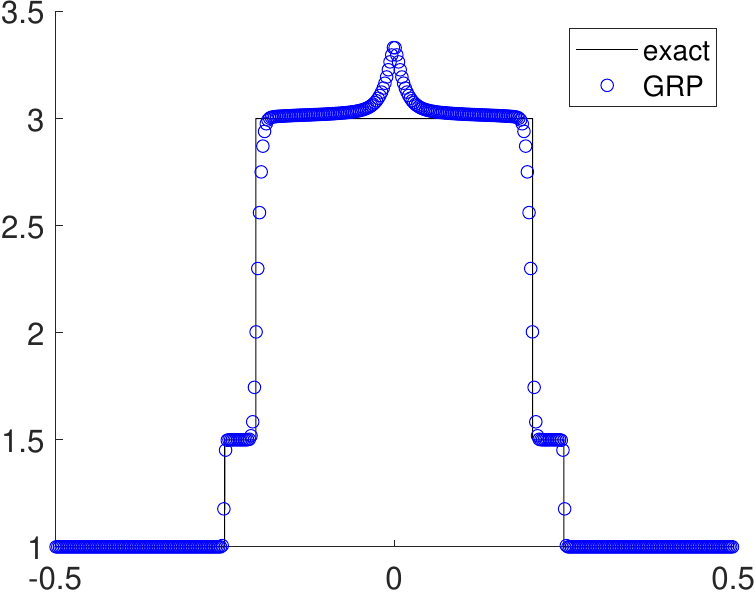}
\end{minipage}
}
\centering
\caption{Example \ref{1D_RP}: The numerical solutions of the GRP scheme on 400 uniform cells for the 1D Riemann problem \eqref{1D_RP2}.}
\label{RP3_nr}
\end{figure}

\begin{figure}[!htbp]
\subfigure[$\rho$]{
\begin{minipage}[c]{0.3\linewidth}
\centering
\includegraphics[width=5cm]{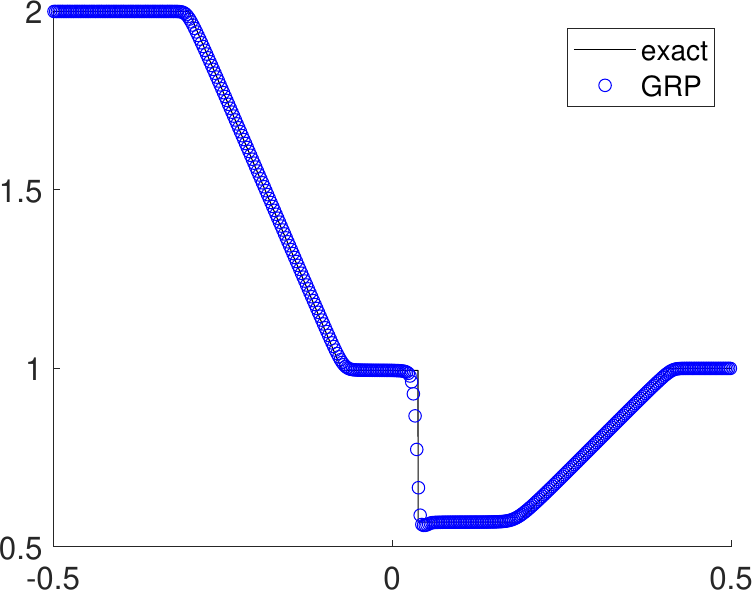}
\end{minipage}
}
\subfigure[$u_1$]{
\begin{minipage}[c]{0.3\linewidth}
\centering
\includegraphics[width=5cm]{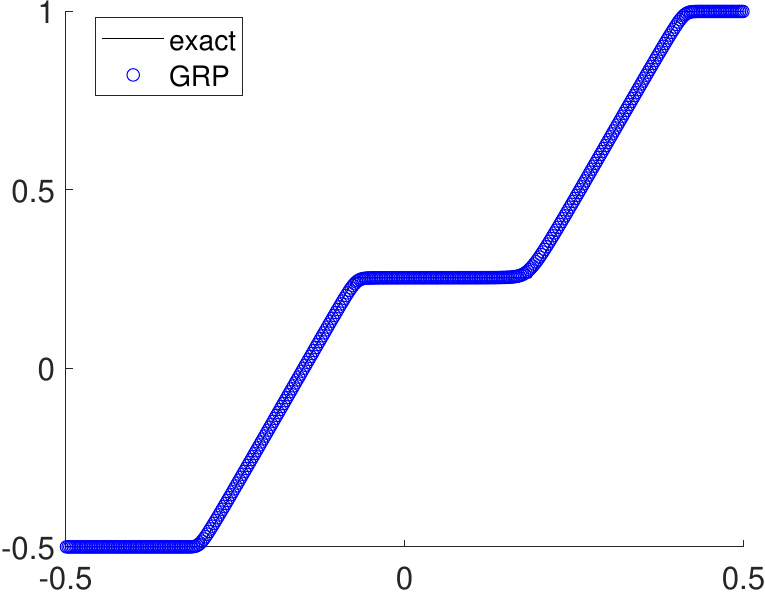}
\end{minipage}
}
\subfigure[$u_2$]{
\begin{minipage}[c]{0.3\linewidth}
\centering
\includegraphics[width=5cm]{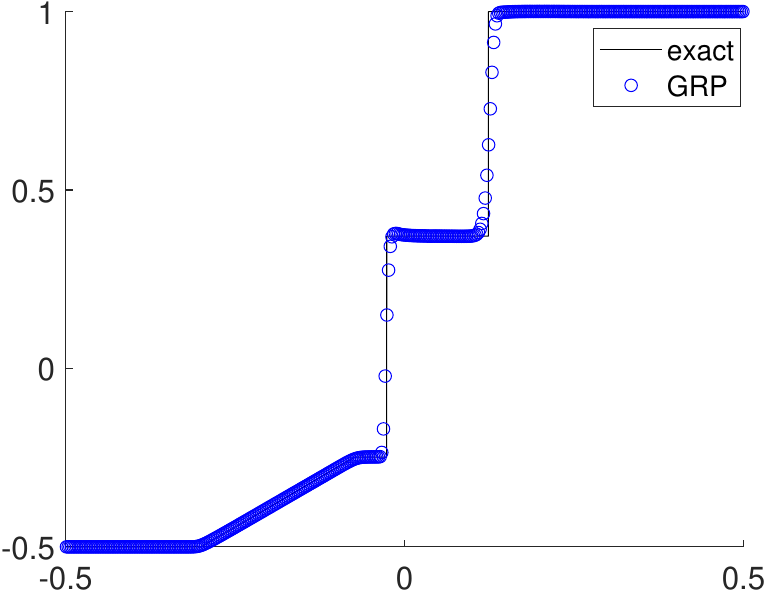}
\end{minipage}
}
\subfigure[$p_{11}$]{
\begin{minipage}[c]{0.3\linewidth}
\centering
\includegraphics[width=5cm]{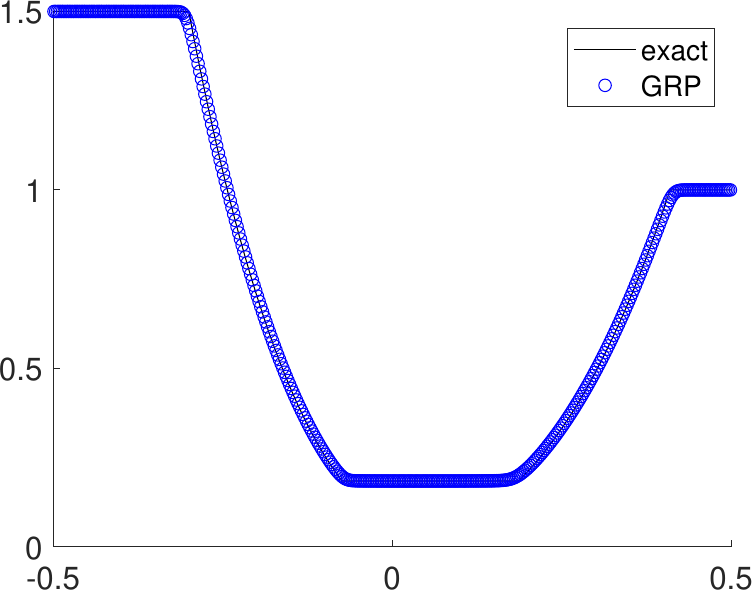}
\end{minipage}
}
\subfigure[$p_{12}$]{
\begin{minipage}[c]{0.3\linewidth}
\centering
\includegraphics[width=5cm]{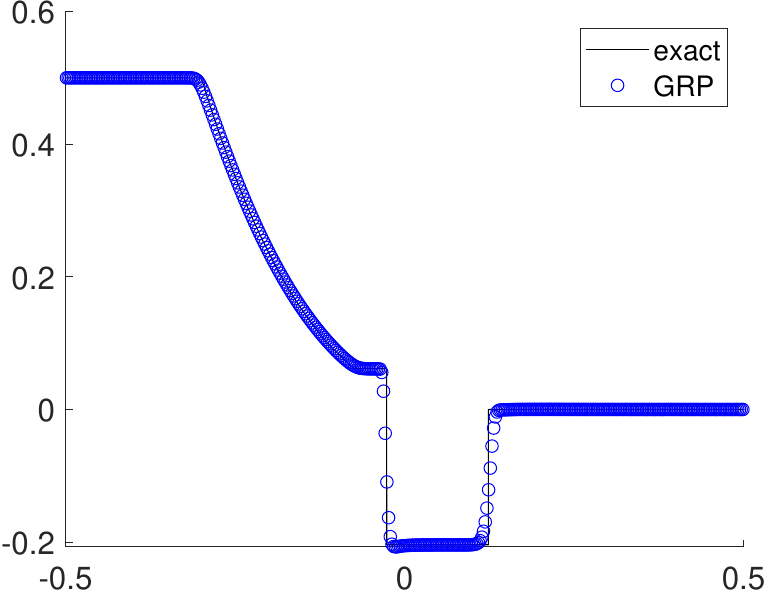}
\end{minipage}
}
\subfigure[$p_{22}$]{
\begin{minipage}[c]{0.3\linewidth}
\centering
\includegraphics[width=5cm]{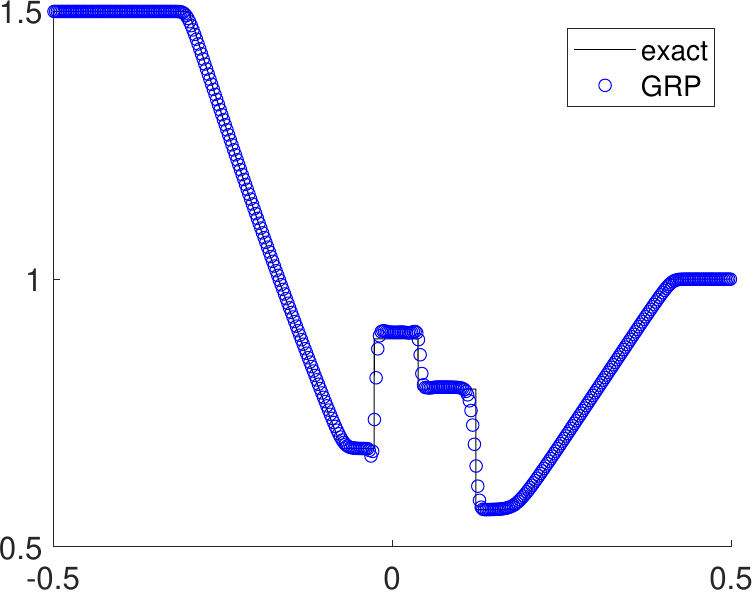}
\end{minipage}
}
\centering
\caption{Example \ref{1D_RP}: The numerical solutions of the GRP scheme on 400 uniform cells for the 1D Riemann problem \eqref{1D_RP3}.}
\label{RP4_nr}
\end{figure}

%

\begin{example}{Leblanc problem}\label{Leblanc_problem}
\rm

The Leblanc problem is considered in this test, which is an extension of the Leblanc problem of the Euler equations. The initial solution is given by
\[
(\rho,u_1,u_2,p_{11},p_{12},p_{22})=\begin{cases}
                                      (2,0,0,10^9,0,10^9), &  x\leq5, \\
                                      (0.001,0,0,1,0,1), &  x>5.
                                    \end{cases}
\]
This problem is highly challenging due to the presence of the strong jumps in the initial density and pressure. The computational domain is taken as $[0,10]$ with the outflow boundary conditions. This problem is simulated until the time $t=0.00004$. Figure \ref{leblanc_nr} displays the numerical solutions obtained on a mesh with 800 cells. For this problem, the minmod slope limiter \eqref{minmod limiter} with $\theta=1.8$ is utilized. One can see that the proposed GRP scheme is able to capture the strong discontinuities with high resolution.

\end{example}

\begin{figure}[!htbp]
\subfigure[$\log(\rho)$]{
\begin{minipage}[c]{0.3\linewidth}
\centering
\includegraphics[width=5cm]{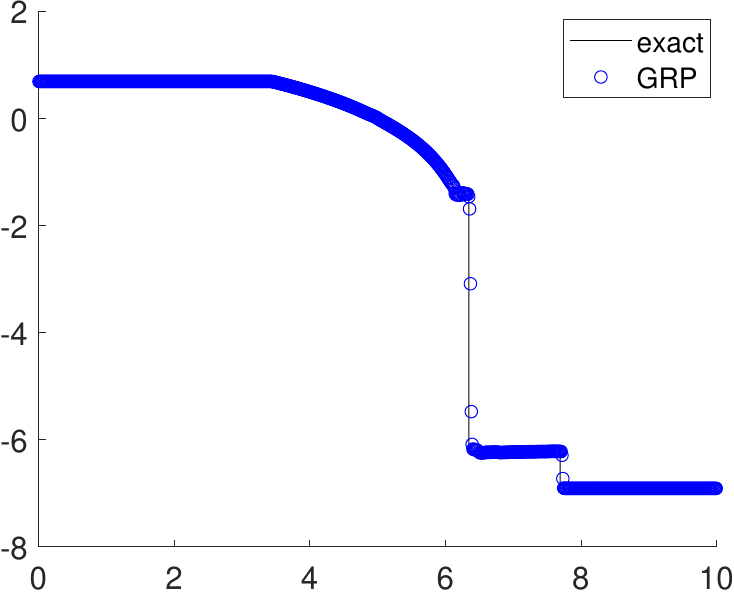}
\end{minipage}
}
\subfigure[$\log(p_{11})$]{
\begin{minipage}[c]{0.3\linewidth}
\centering
\includegraphics[width=5cm]{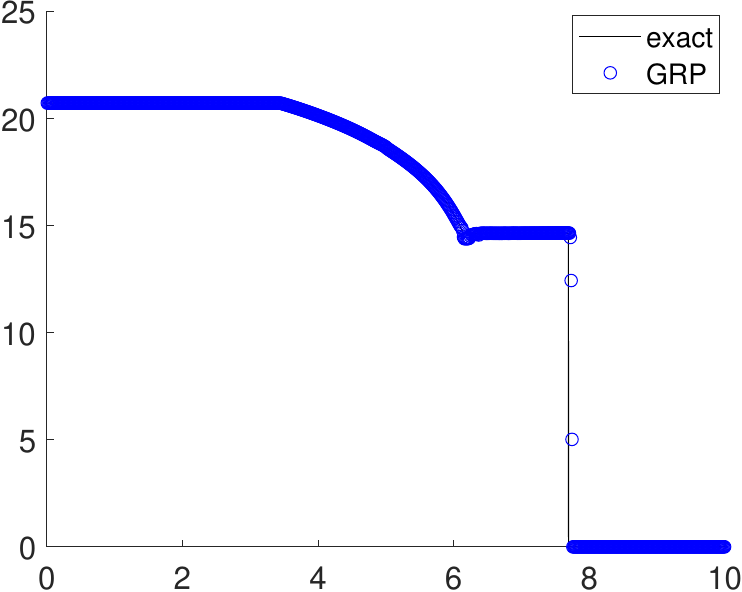}
\end{minipage}
}
\subfigure[$\log(p_{22})$]{
\begin{minipage}[c]{0.3\linewidth}
\centering
\includegraphics[width=5cm]{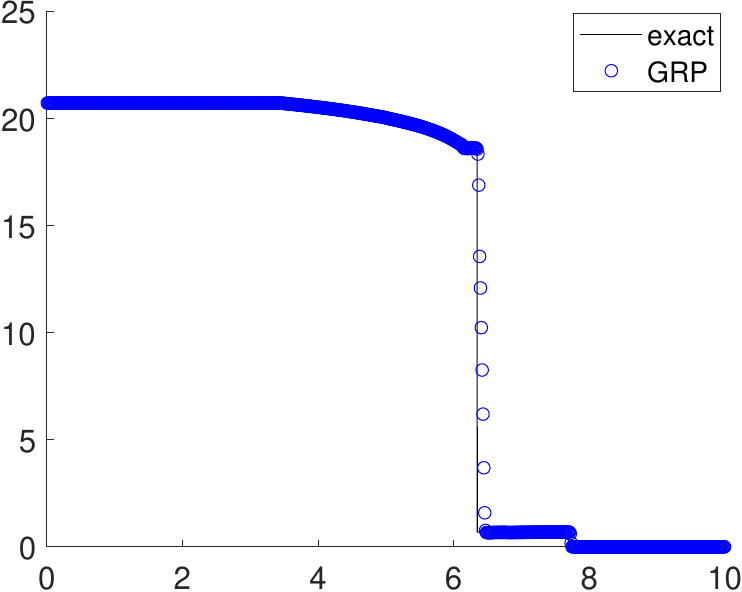}
\end{minipage}
}
\centering
\caption{Example \ref{Leblanc_problem}: The log plots of density $\rho$ (left), the pressure components $p_{11}$ (middle) and $p_{22}$ (right) for the Leblanc problem obtained by the GRP scheme with the minmod slope limiter and $\theta=1.8$ on 800 uniform cells.}
\label{leblanc_nr}
\end{figure}

\begin{example}{Shu-Osher problem}\label{shu-osher}
\rm

The Shu-Osher problem of the ten-moment equations is considered, and it is an extension of the Shu-Osher problem of the Euler equations \cite{shu1989efficient}. The initial solution is given by
\[
(\rho,u_1,u_2,p_{11},p_{12},p_{22})=\begin{cases}
                                      (3.857143,2.629369,0,10.33333,0,10.33333), &  x\leq-4, \\
                                      (1+0.2\sin(5x),0,0,1,0,1), &  x>-4.
                                    \end{cases}
\]
The computational interval is taken as $[-5,5]$. The outflow boundary conditions are imposed. This problem is simulated until the time $t=1.4$ on a mesh with 800 cells. Figure \ref{shu_osher_nr} displays the numerical solutions. The reference solutions are obtained by the GRP scheme with 10000 cells.

 To observe all the waves clearly, the plots of the reference solutions $\rho$, $u_1$, $p_{11}$ and $p_{22}$ are presented together in Figure \ref{shu_osher_nr2}, where the plots of $p_{11}$ and $p_{22}$ translate downward by 10 and 8 units, respectively. One can observe that there are three right-going shock waves near $x=-4.68$, $x=-3.61$ and $x=3.40$, respectively, and a contact discontinuity near $x=-0.69$. In the regions between these locations, the solutions are smooth. From Figure \ref{shu_osher_nr}, one can see that the GRP scheme captures these waves well and resolves the shock waves with high resolution on the mesh with 800 cells.

 Moreover, one can find that the pressure component $p_{11}$ is continuous across the contact discontinuity, but its slope changes. This phenomenon is reasonable. At this moment, the material derivative $\frac{\mathcal{D}u_1}{\mathcal{D}t}$ is non-trivial near the contact discontinuity, and the density has a jump across the contact discontinuity, so by \eqref{p11_x}, one knows that $\frac{\partial p_{11}}{\partial x}$ takes different values at the left and right side of the contact discontinuity.

\end{example}

\begin{figure}[!htbp]
\subfigure[$\rho$]{
\begin{minipage}[c]{0.4\linewidth}
\centering
\includegraphics[width=5cm]{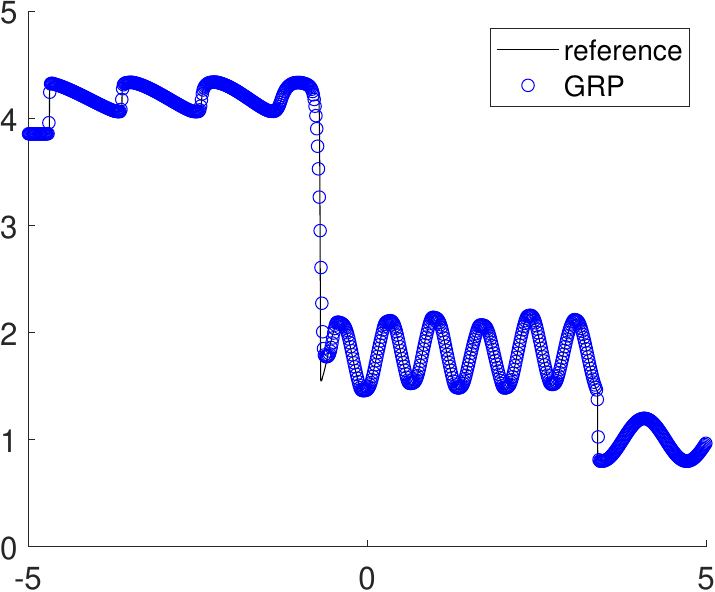}
\end{minipage}
}
\subfigure[$u_1$]{
\begin{minipage}[c]{0.4\linewidth}
\centering
\includegraphics[width=5cm]{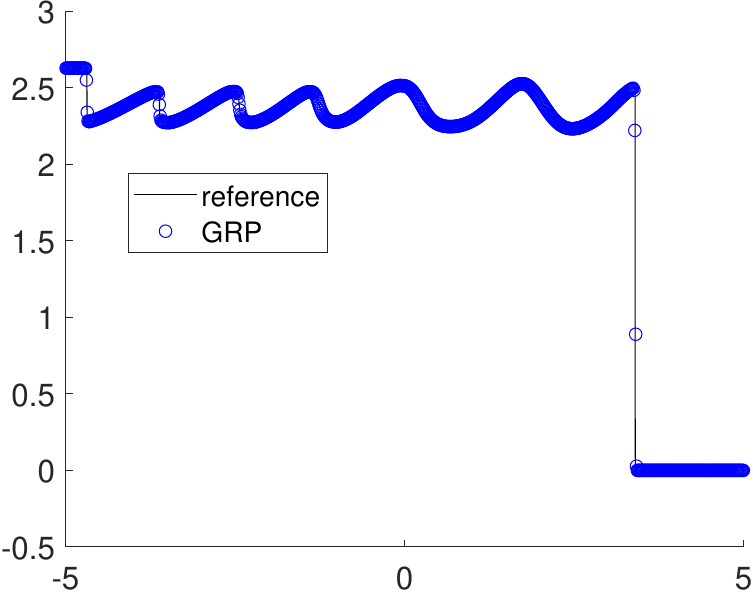}
\end{minipage}
}
\subfigure[$p_{11}$]{
\begin{minipage}[c]{0.4\linewidth}
\centering
\includegraphics[width=5cm]{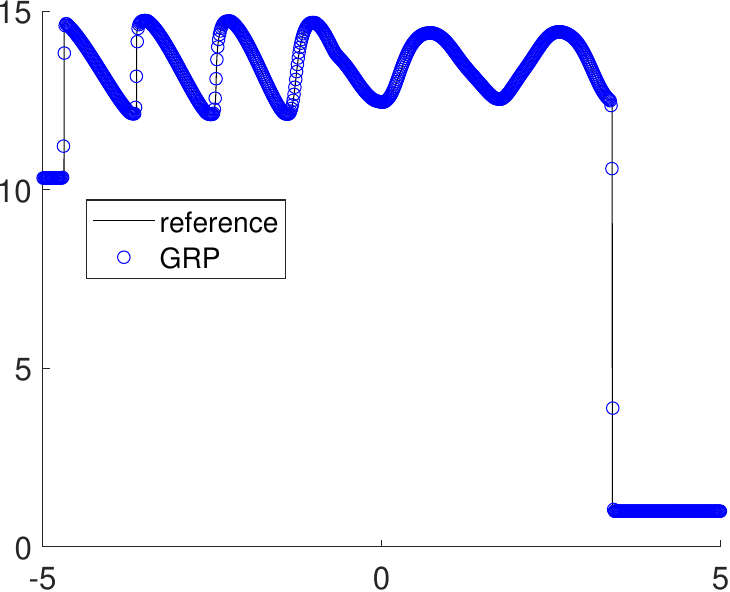}
\end{minipage}
}
\subfigure[$p_{22}$]{
\begin{minipage}[c]{0.4\linewidth}
\centering
\includegraphics[width=5cm]{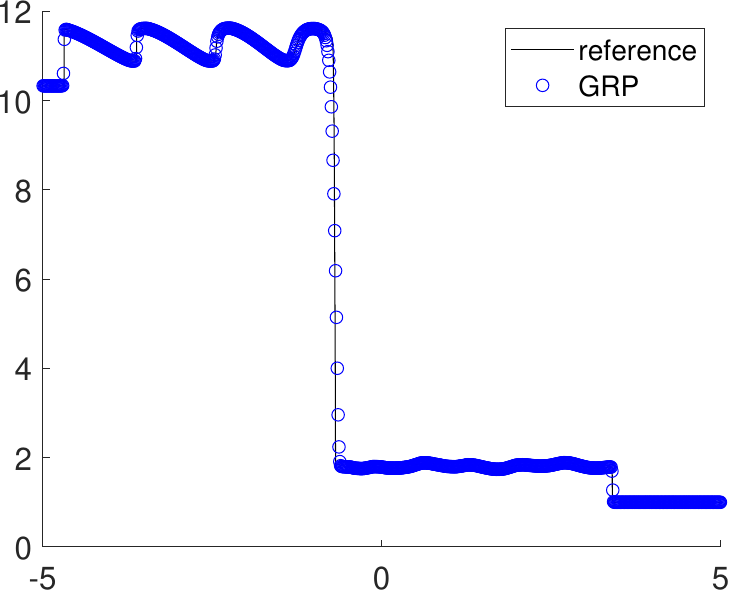}
\end{minipage}
}
\centering
\caption{Example \ref{shu-osher}: The numerical solutions of the GRP scheme on 800 uniform cells for Shu-Osher problem.}
\label{shu_osher_nr}
\end{figure}

\begin{figure}[!htbp]
\centering
\includegraphics[width=6cm]{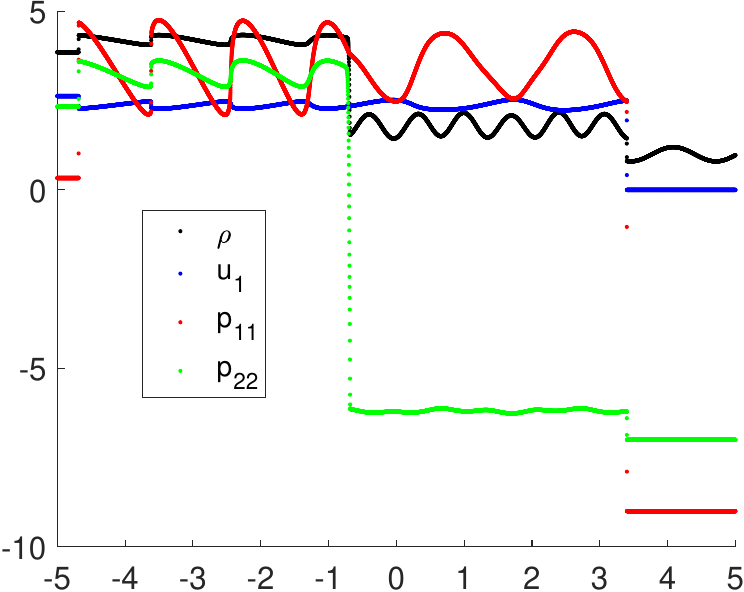}
\centering
\caption{Example \ref{shu-osher}: The comparison of the reference solutions $\rho$, $u_1$, $p_{11}$ and $p_{22}$ for Shu-Osher problem.}
\label{shu_osher_nr2}
\end{figure}

%

\begin{example}{2D Riemann problems}\label{2D_RP}
\rm

In this example, four 2D Riemann problems  are constructed for the first time to verify the performance of
the proposed GRP schemes.
The domain is taken as $[0,1]^2$ with the outflow boundary conditions.
The initial discontinuities are located on the line $x=0.5$ for the $x$-direction and on the line $y=0.5$ for the $y$-direction. Hence, the domain is divided into four subdomains, which are counterclockwise denoted by $\rm\uppercase\expandafter{\romannumeral1}$, $\rm\uppercase\expandafter{\romannumeral2}$, $\rm\uppercase\expandafter{\romannumeral3}$ and $\rm\uppercase\expandafter{\romannumeral4}$ as depicted in Figure \ref{2d_domain}. The proposed GRP scheme will be applied to simulate these four problems on a mesh of $200\times200$ cells up to different final time.

\begin{figure}[!htbp]
\includegraphics[width=6cm]{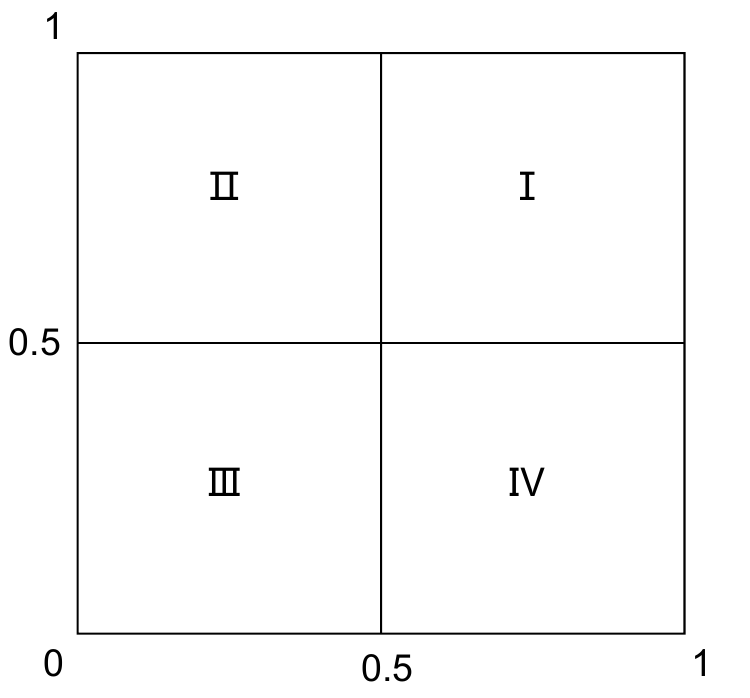}
\centering
\caption{The partition of the domain for the 2D Riemann problems in Example \ref{2D_RP}.}
\label{2d_domain}
\end{figure}

(i) 2D Riemann problem 1:

\begin{equation}\label{2D_RP1}
(\rho,u_1,u_2,p_{11},p_{12},p_{22})=\begin{cases}
                                      (1,0.8939,0.8939,1.0541337,0,1.0541337), & 0.5<x<1, ~ 0.5<y<1,  \\
                                      (1.0541337,0.8939,0.8,1.0541337,0,1.1716956),  & 0<x<0.5, ~ 0.5<y<1, \\
                                      (1,0.8939,0.8939,1,0,1), & 0<x<0.5, ~ 0<y<0.5, \\
                                      (1.0541337,0.8,0.8939,1.1716956,0,1.0541337),  & 0.5<x<1, ~ 0<y<0.5.
                                    \end{cases}
\end{equation}
For this problem, all the possible local wave patterns are elaborated as follows:
\begin{itemize}
\item Between $\rm\uppercase\expandafter{\romannumeral2}$ and $\rm\uppercase\expandafter{\romannumeral1}$: The wave shown by $\rho$ is a right-going contact discontinuity. $u_1$ and $p_{11}$ keep constant. The waves  shown by $u_2$ and $p_{12}$ both contain a left-going shock wave and a right-going shock wave. The waves  shown by $p_{22}$ include a left-going shock wave, a right-going contact discontinuity and a right-going shock wave. However, the two shock waves jumps observed in $p_{22}$ are both small and approximately equal to 0.0022628.

\item Between $\rm\uppercase\expandafter{\romannumeral3}$ and $\rm\uppercase\expandafter{\romannumeral2}$: $u_1$ and $p_{12}$ are constants. The waves  shown by $\rho$, $u_2$, $p_{11}$ and $p_{22}$ are all down-going shock waves.

\item Between $\rm\uppercase\expandafter{\romannumeral3}$ and $\rm\uppercase\expandafter{\romannumeral4}$: $u_2$ and $p_{12}$ are constants. The waves  shown by $\rho$, $u_1$, $p_{11}$ and $p_{22}$ are all left-going shock waves.

\item Between $\rm\uppercase\expandafter{\romannumeral4}$ and $\rm\uppercase\expandafter{\romannumeral1}$: The wave  shown by $\rho$ is an up-going contact discontinuity. $u_2$ and $p_{22}$ keep constant. The waves  shown by $u_1$ and $p_{12}$ both contain a down-going shock wave and an up-going shock wave. The waves shown by $p_{11}$ include a down-going shock wave, an up-going contact discontinuity and an up-going shock wave. However, the jumps across the two shock waves  shown by $p_{11}$ are both small and approximately equal to 0.0022628.
\end{itemize}
This problem is simulated up to $t=0.15$. The plots with 40 equally spaced contour lines of the primitive variables are presented in Figure \ref{2d_RP1_nr}. From Figures \ref{2d_RP1_nr_p11} and \ref{2d_RP1_nr_p22}, it seems that the shock waves  shown by $p_{11}$ and $p_{22}$ are lost. It is due to their small shock jumps and only 40 contour lines being displayed.  Actually, if setting many more equally spaced contour lines, the shock waves in $p_{11}$ and $p_{22}$ can be viewed, see Figure \ref{2d_RP1_nr2}.

\begin{figure}[!htbp]
\subfigure[$\rho$]{
\begin{minipage}[c]{0.3\linewidth}
\centering
\includegraphics[width=5cm]{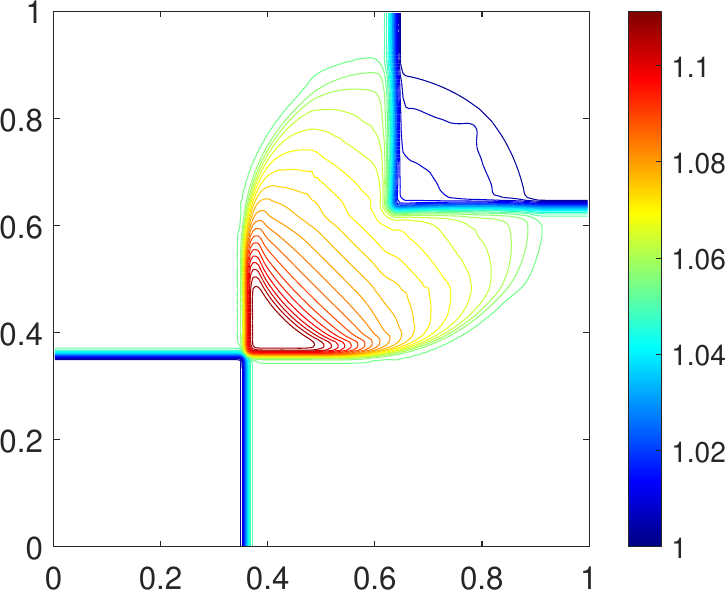}
\end{minipage}
}
\subfigure[$u_1$]{
\begin{minipage}[c]{0.3\linewidth}
\centering
\includegraphics[width=5cm]{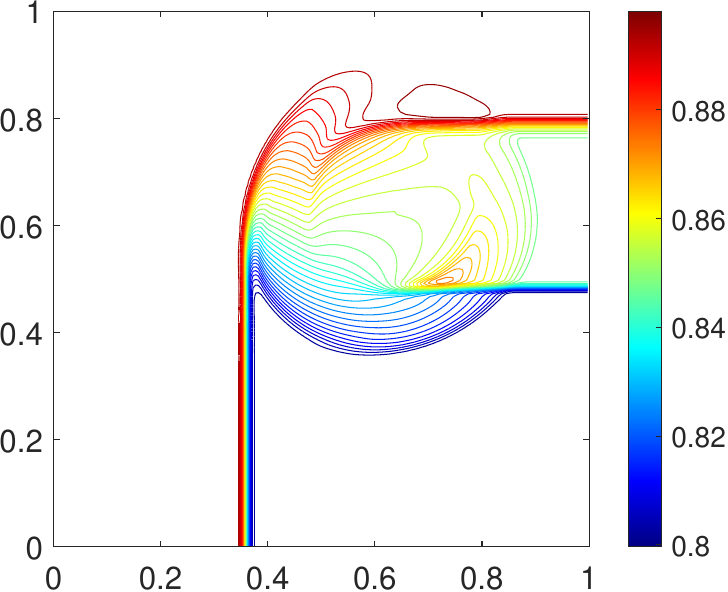}
\end{minipage}
}
\subfigure[$u_2$]{
\begin{minipage}[c]{0.3\linewidth}
\centering
\includegraphics[width=5cm]{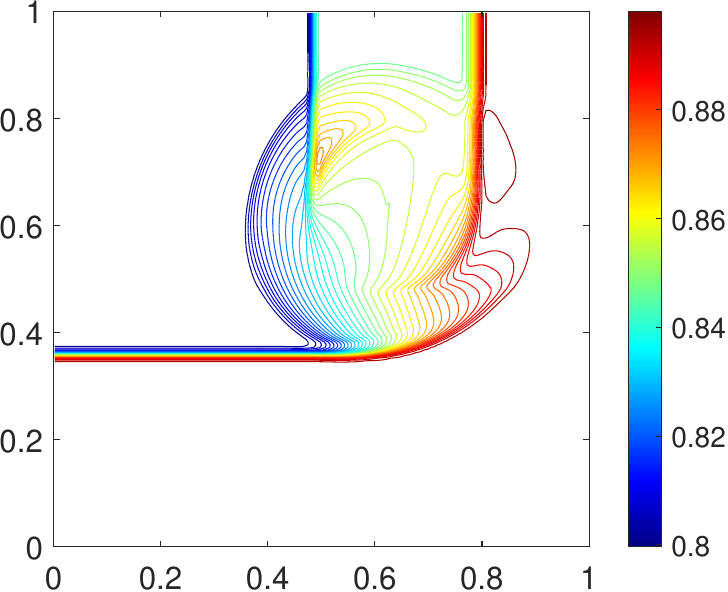}
\end{minipage}
}
\subfigure[$p_{11}$]{
\begin{minipage}[c]{0.3\linewidth}
\centering
\includegraphics[width=5cm]{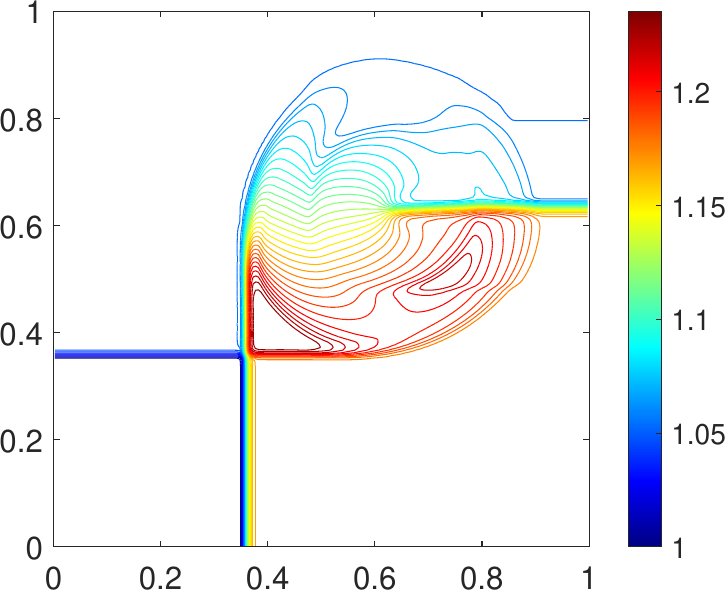}
\end{minipage}
\label{2d_RP1_nr_p11}
}
\subfigure[$p_{12}$]{
\begin{minipage}[c]{0.3\linewidth}
\centering
\includegraphics[width=5cm]{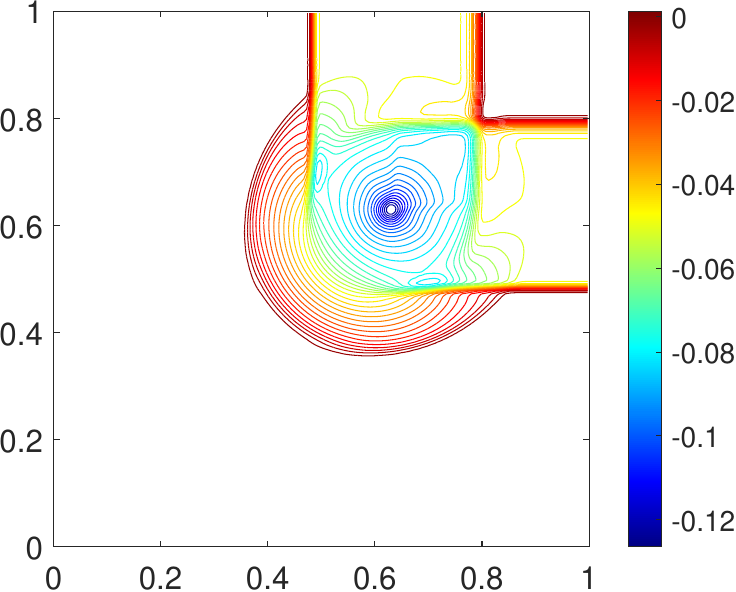}
\end{minipage}
}
\subfigure[$p_{22}$]{
\begin{minipage}[c]{0.3\linewidth}
\centering
\includegraphics[width=5cm]{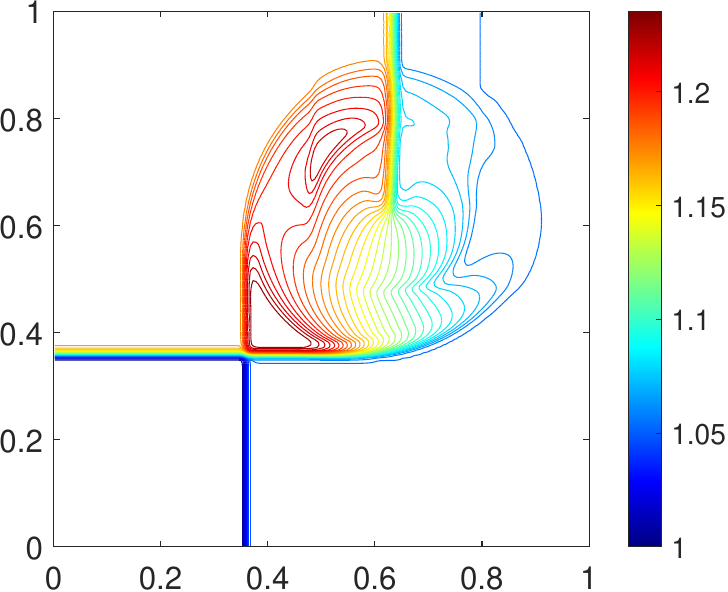}
\end{minipage}
\label{2d_RP1_nr_p22}
}
\centering
\caption{Example \ref{2D_RP}: The contour plots of solutions obtained by the GRP scheme for the 2D Riemann problem \eqref{2D_RP1} on a mesh of $200\times200$ cells. 40 equally spaced contour lines are displayed.}
\label{2d_RP1_nr}
\end{figure}

\begin{figure}[!htbp]
\subfigure[$p_{11}$]{
\begin{minipage}[c]{0.4\linewidth}
\centering
\includegraphics[width=6cm]{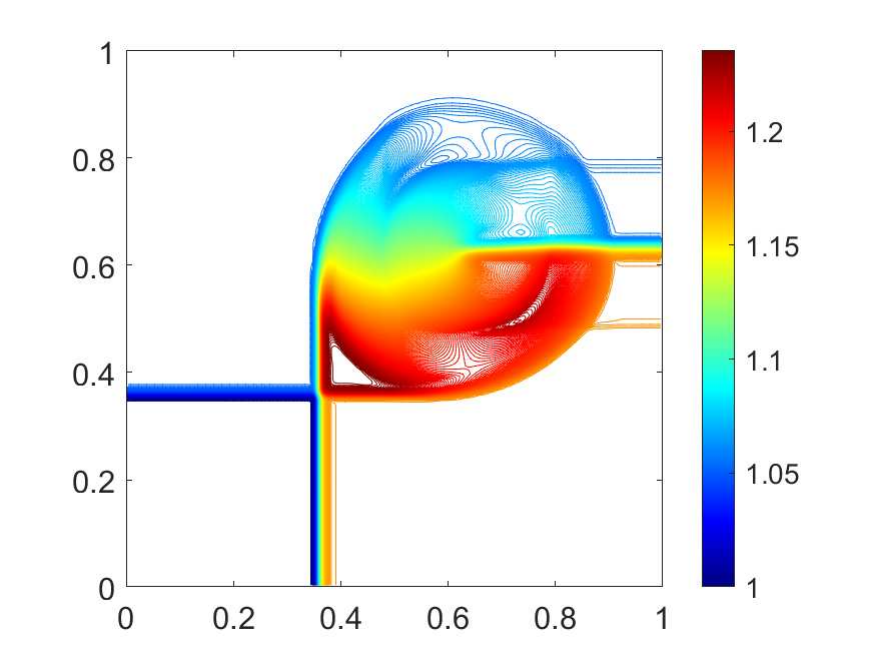}
\end{minipage}
}
\subfigure[$p_{22}$]{
\begin{minipage}[c]{0.4\linewidth}
\centering
\includegraphics[width=6cm]{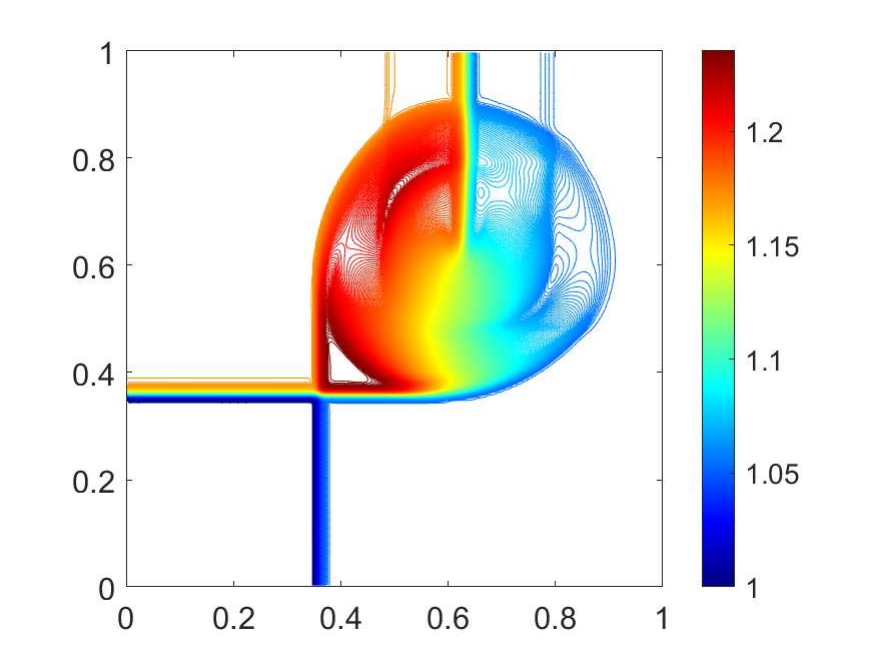}
\end{minipage}
}
\centering
\caption{Example \ref{2D_RP}: The contour plots of $p_{11}$ and $p_{22}$ obtained by the GRP scheme for the 2D Riemann problem \eqref{2D_RP1} on a mesh of $200\times200$ cells. 400 equally spaced contour lines are displayed.}
\label{2d_RP1_nr2}
\end{figure}

(ii) 2D Riemann problem 2:
\begin{equation}\label{2D_RP2}
(\rho,u_1,u_2,p_{11},p_{12},p_{22})=\begin{cases}
                                      (1,0.75,-0.5,1,0.5,1), & 0.5<x<1, ~ 0.5<y<1,  \\
                                      (1,0.75,0.5,1,-0.5,1),  & 0<x<0.5, ~ 0.5<y<1, \\
                                      (1,-0.25,0.5,1,0.5,1), & 0<x<0.5, ~ 0<y<0.5, \\
                                      (1,-0.25,-0.5,1,-0.5,1),  & 0.5<x<1, ~ 0<y<0.5.
                                    \end{cases}
\end{equation}
For this problem, the local wave patterns are demonstrated as follows:
\begin{itemize}
\item Between $\rm\uppercase\expandafter{\romannumeral2}$ and $\rm\uppercase\expandafter{\romannumeral1}$: $\rho$, $u_1$, $p_{11}$ and $p_{22}$ are all constants. Two waves in $u_2$ and $p_{12}$ are the left-going shear waves.

\item Between $\rm\uppercase\expandafter{\romannumeral3}$ and $\rm\uppercase\expandafter{\romannumeral2}$: $\rho$, $u_2$, $p_{11}$ and $p_{22}$ are all constants. Two  waves in $u_1$ and $p_{12}$ are the down-going shear waves.

\item Between $\rm\uppercase\expandafter{\romannumeral3}$ and $\rm\uppercase\expandafter{\romannumeral4}$: $\rho$, $u_1$, $p_{11}$ and $p_{22}$ are all constants. Two  waves in $u_2$ and $p_{12}$ are the right-going shear waves.

\item Between $\rm\uppercase\expandafter{\romannumeral4}$ and $\rm\uppercase\expandafter{\romannumeral1}$: $\rho$, $u_2$, $p_{11}$ and $p_{22}$ are all constants. Two  waves in $u_1$ and $p_{12}$ are the up-going shear waves.
\end{itemize}
The final time is $t=0.15$. The contour plots with 40 equally spaced contour lines are illustrated in Figure \ref{2d_RP2_nr}.

\begin{figure}[!htbp]
\subfigure[$\rho$]{
\begin{minipage}[c]{0.3\linewidth}
\centering
\includegraphics[width=5cm]{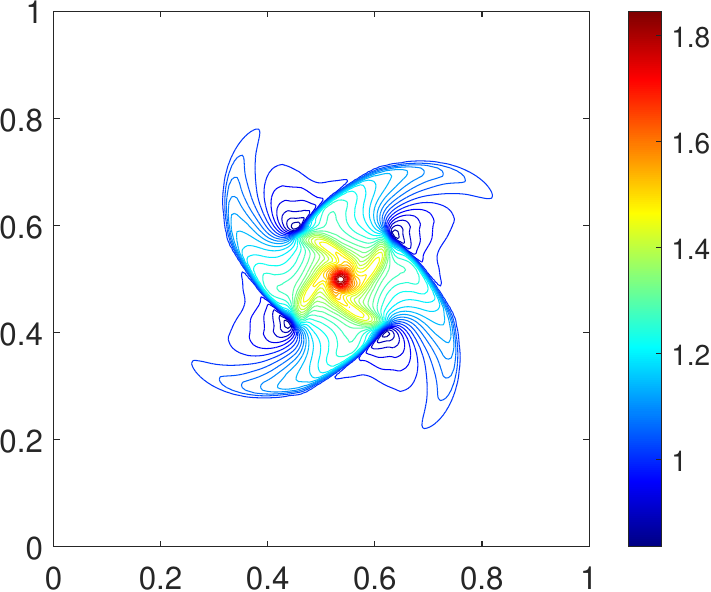}
\end{minipage}
}
\subfigure[$u_1$]{
\begin{minipage}[c]{0.3\linewidth}
\centering
\includegraphics[width=5cm]{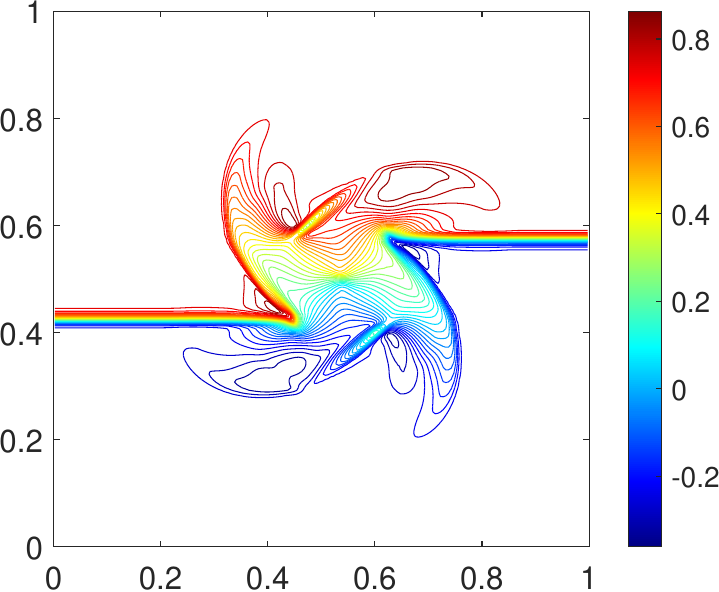}
\end{minipage}
}
\subfigure[$u_2$]{
\begin{minipage}[c]{0.3\linewidth}
\centering
\includegraphics[width=5cm]{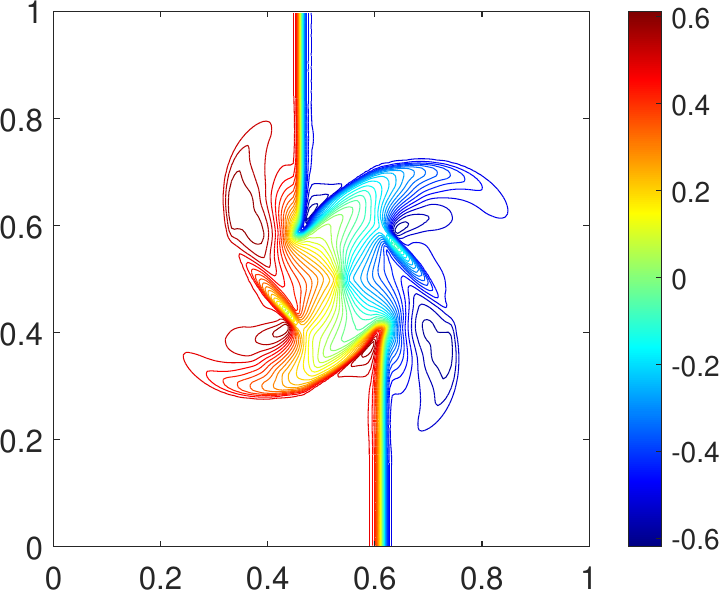}
\end{minipage}
}
\subfigure[$p_{11}$]{
\begin{minipage}[c]{0.3\linewidth}
\centering
\includegraphics[width=5cm]{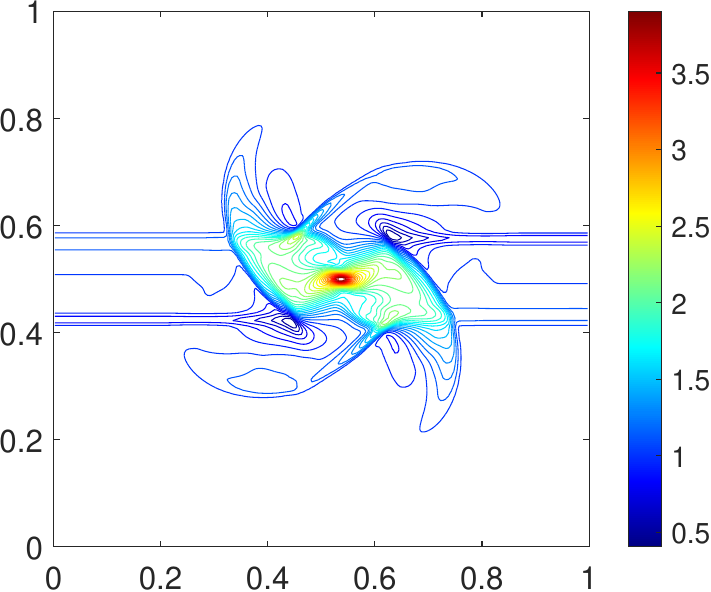}
\end{minipage}
}
\subfigure[$p_{12}$]{
\begin{minipage}[c]{0.3\linewidth}
\centering
\includegraphics[width=5cm]{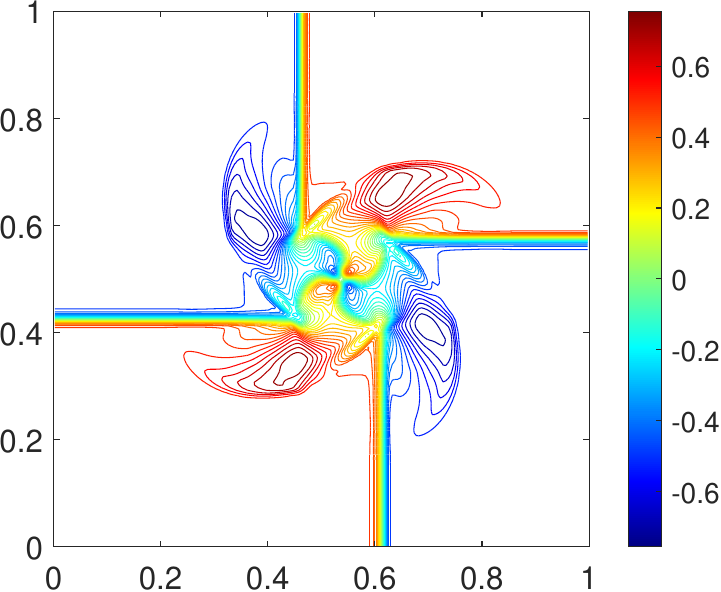}
\end{minipage}
}
\subfigure[$p_{22}$]{
\begin{minipage}[c]{0.3\linewidth}
\centering
\includegraphics[width=5cm]{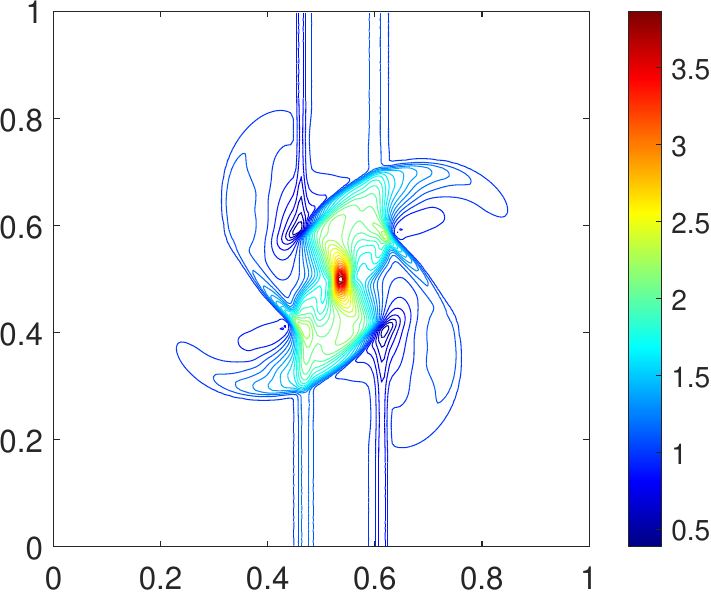}
\end{minipage}
}
\centering
\caption{Example \ref{2D_RP}: The contour plots of solutions obtained by the GRP scheme for the 2D Riemann problem \eqref{2D_RP2} on a mesh of $200\times200$ cells. 40 equally spaced contour lines are displayed.}
\label{2d_RP2_nr}
\end{figure}

(iii) 2D Riemann problem 3:

\begin{equation}\label{2D_RP3}
(\rho,u_1,u_2,p_{11},p_{12},p_{22})=\begin{cases}
                                      (1,-0.5,-0.5,1,0,1), & 0.5<x<1, ~ 0.5<y<1,  \\
                                      (0.9422650,-0.6,-0.5,0.8366024,0,0.9422650),  & 0<x<0.5, ~ 0.5<y<1, \\
                                      (1,-0.5,-0.5,0.9422650,0,0.9422650), & 0<x<0.5, ~ 0<y<0.5, \\
                                      (0.9422650,-0.5,-0.6,0.9422650,0,0.8366024),  & 0.5<x<1, ~ 0<y<0.5.
                                    \end{cases}
\end{equation}
For this problem, all the possible local wave patterns are detailed as follows:
\begin{itemize}
\item Between $\rm\uppercase\expandafter{\romannumeral2}$ and $\rm\uppercase\expandafter{\romannumeral1}$: $u_2$ and $p_{12}$ are constants. All   waves shown in $\rho$, $u_1$, $p_{11}$ and $p_{22}$ are the right-going rarefaction waves.

\item Between $\rm\uppercase\expandafter{\romannumeral3}$ and $\rm\uppercase\expandafter{\romannumeral2}$: The wave shown in $\rho$ is a down-going contact discontinuity. $u_2$ and $p_{22}$ keep constant. The waves shown in $u_1$ and $p_{12}$ both contain a down-going shock wave and an up-going shock wave. The waves in $p_{11}$ include a down-going shock wave, an up-going contact discontinuity and an up-going shock wave. However, the two shock wave jumps in  $p_{11}$ are both small and approximately equal to 0.0024262.

\item Between $\rm\uppercase\expandafter{\romannumeral3}$ and $\rm\uppercase\expandafter{\romannumeral4}$: The wave in $\rho$ is a left-going contact discontinuity. $u_1$ and $p_{11}$ keep constant. The waves in both $u_2$ and $p_{12}$  contain a left-going shock wave and a right-going shock wave. The waves in $p_{22}$ include a left-going shock wave, a right-going contact discontinuity and a right-going shock wave. However,  both two shock wave jumps in $p_{22}$ are small and approximately equal to 0.0024262.

\item Between $\rm\uppercase\expandafter{\romannumeral4}$ and $\rm\uppercase\expandafter{\romannumeral1}$: $u_1$ and $p_{12}$ are constants. All  waves in $\rho$, $u_2$, $p_{11}$ and $p_{22}$ are the up-going rarefaction waves.
\end{itemize}
This problem is simulated up to $t=0.2$. The contour plots with 40 equally spaced lines are presented in Figure \ref{2d_RP3_nr}. Similar with the 2D Riemann problem 1, due to the small shock jumps and only 40 contour lines being displayed, the shock waves in $p_{11}$ and $p_{22}$ seem to miss in Figures \ref{2d_RP3_nr_p11} and   \ref{2d_RP3_nr_p22}, which can be viewed by spacing many more contour lines as depicted in Figure \ref{2d_RP3_nr2}.

\begin{figure}[!htbp]
\subfigure[$\rho$]{
\begin{minipage}[c]{0.3\linewidth}
\centering
\includegraphics[width=5cm]{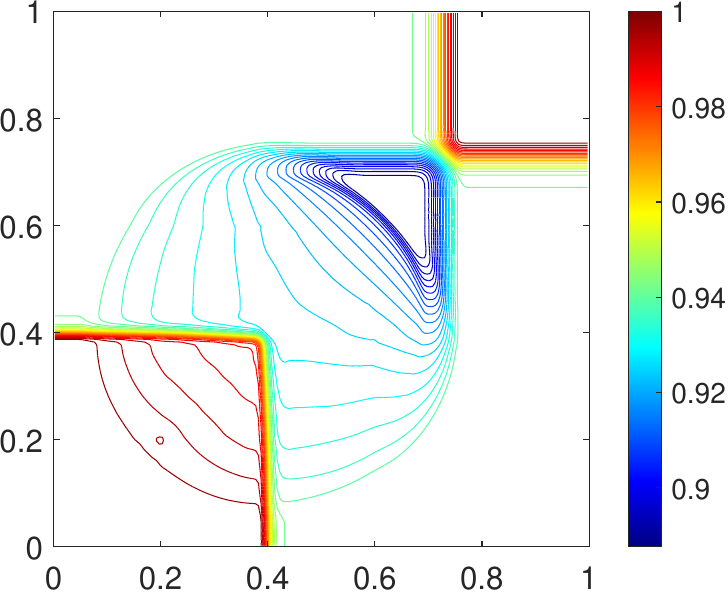}
\end{minipage}
}
\subfigure[$u_1$]{
\begin{minipage}[c]{0.3\linewidth}
\centering
\includegraphics[width=5cm]{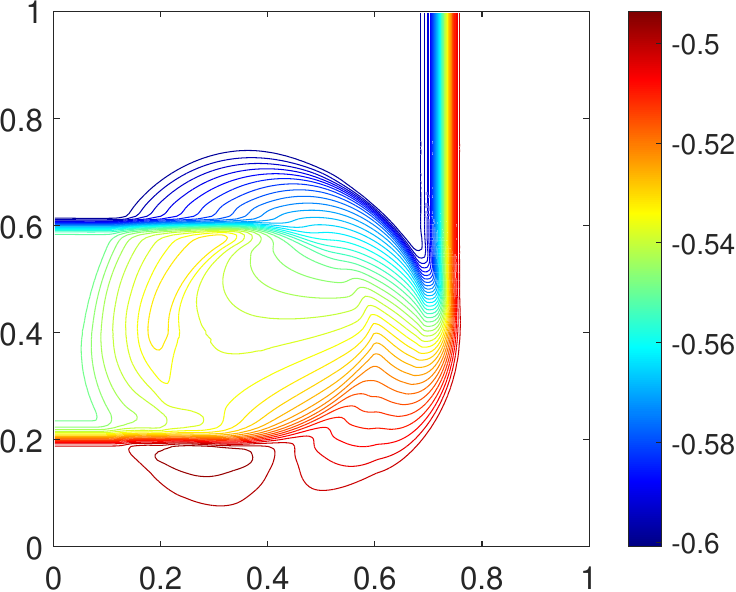}
\end{minipage}
}
\subfigure[$u_2$]{
\begin{minipage}[c]{0.3\linewidth}
\centering
\includegraphics[width=5cm]{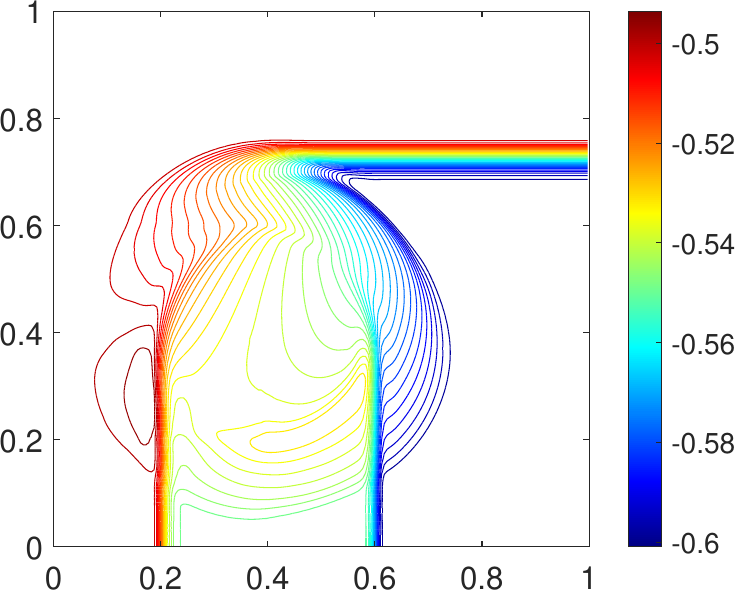}
\end{minipage}
}
\subfigure[$p_{11}$]{
\begin{minipage}[c]{0.3\linewidth}
\centering
\includegraphics[width=5cm]{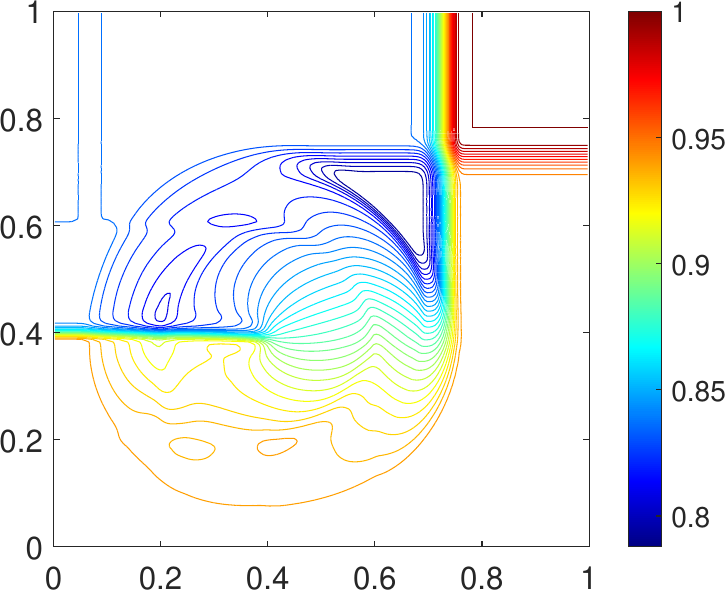}
\end{minipage}
\label{2d_RP3_nr_p11}
}
\subfigure[$p_{12}$]{
\begin{minipage}[c]{0.3\linewidth}
\centering
\includegraphics[width=5cm]{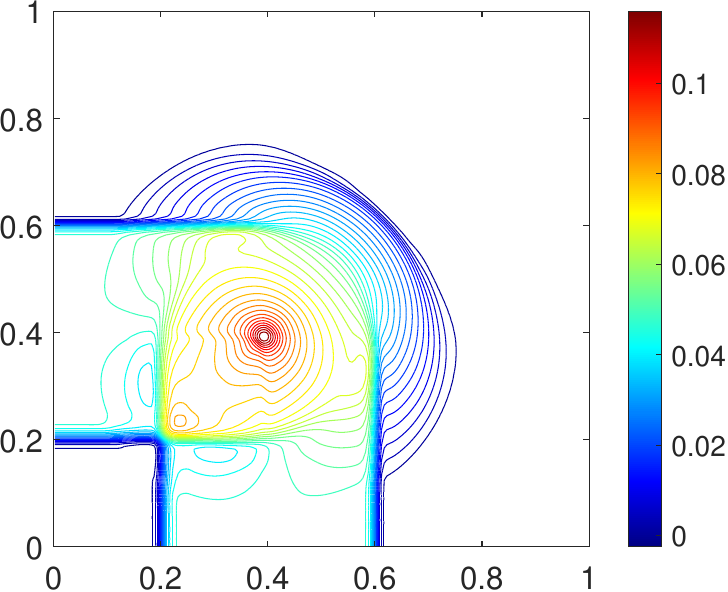}
\end{minipage}
}
\subfigure[$p_{22}$]{
\begin{minipage}[c]{0.3\linewidth}
\centering
\includegraphics[width=5cm]{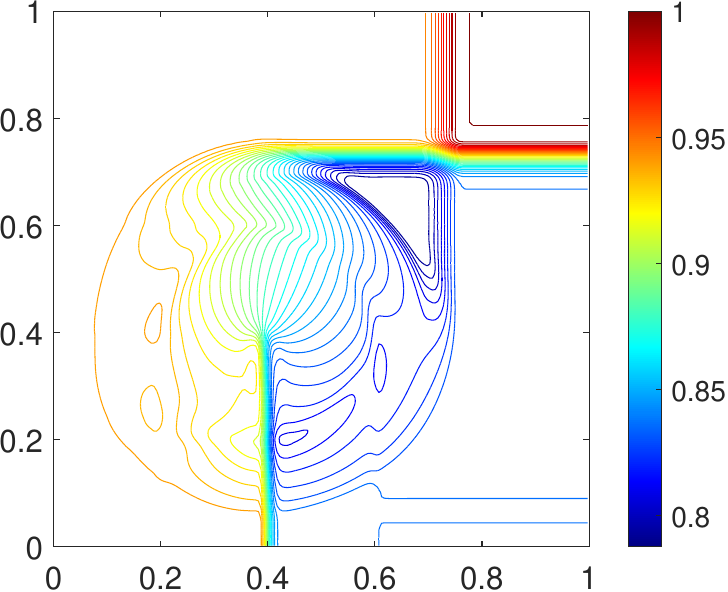}
\end{minipage}
\label{2d_RP3_nr_p22}
}
\centering
\caption{Example \ref{2D_RP}: The contour plots of solutions obtained by the GRP scheme for the 2D Riemann problem \eqref{2D_RP3} on a mesh of $200\times200$ cells. 40 equally spaced contour lines are displayed.}
\label{2d_RP3_nr}
\end{figure}

\begin{figure}[!htbp]
\subfigure[$p_{11}$]{
\begin{minipage}[c]{0.45\linewidth}
\centering
\includegraphics[width=6cm]{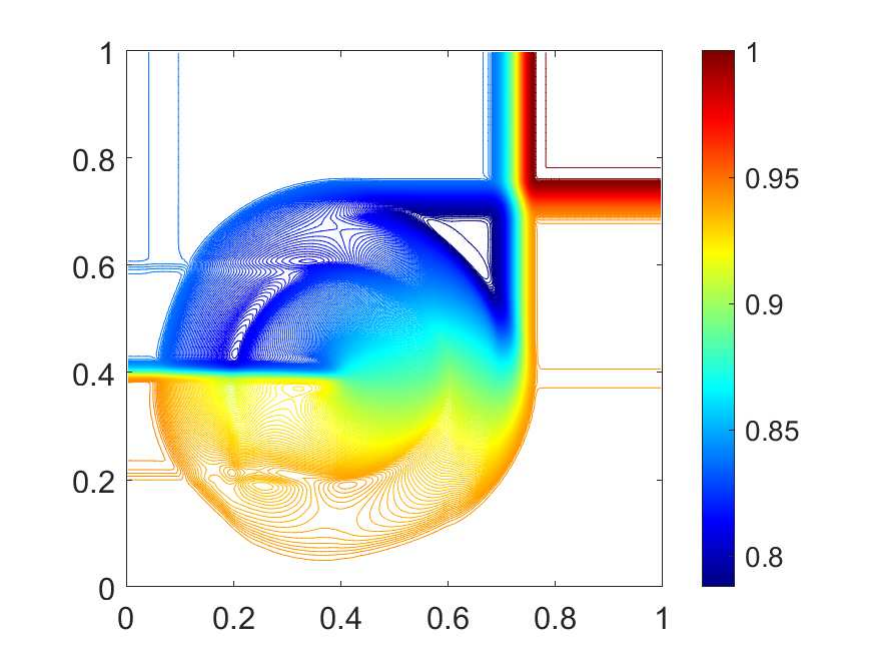}
\end{minipage}
}
\subfigure[$p_{22}$]{
\begin{minipage}[c]{0.45\linewidth}
\centering
\includegraphics[width=6cm]{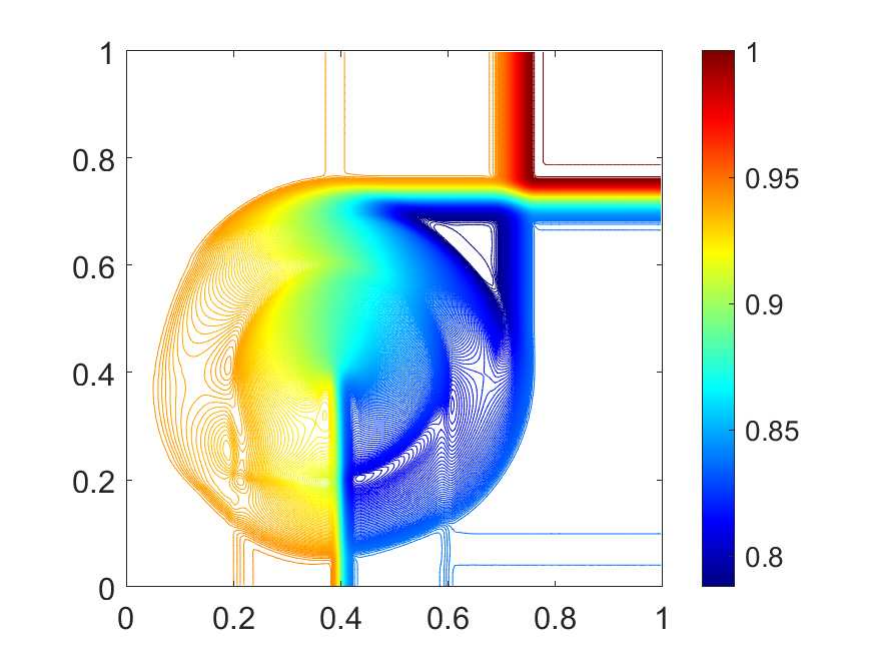}
\end{minipage}
}
\centering
\caption{Example \ref{2D_RP}: The contour plots of $p_{11}$ and $p_{22}$ obtained by the GRP scheme for the 2D Riemann problem \eqref{2D_RP3} on a mesh of $200\times200$ cells. 400 equally spaced contour lines are displayed.}
\label{2d_RP3_nr2}
\end{figure}

(iv) 2D Riemann problem 4:
\begin{equation}\label{2D_RP4}
(\rho,u_1,u_2,p_{11},p_{12},p_{22})=\begin{cases}
                                      (1.0909,0,0,1.0909,0,1.0909), & 0.5<x<1, ~ 0.5<y<1,  \\
                                      (0.5065,1.2024,0,0.3499,0,0.3499),  & 0<x<0.5, ~ 0.5<y<1, \\
                                      (1.0909,1.2024,1.2024,1.0909,0,1.0909), & 0<x<0.5, ~ 0<y<0.5, \\
                                      (0.5065,0,1.2024,0.3499,0,0.3499),  & 0.5<x<1, ~ 0<y<0.5.
                                    \end{cases}
\end{equation}
For this problem, all the possible local wave patterns are elaborated as follows:
\begin{itemize}
\item Between $\rm\uppercase\expandafter{\romannumeral2}$ and $\rm\uppercase\expandafter{\romannumeral1}$: The waves in $\rho$, $u_1$, $p_{11}$ and $p_{22}$ all contain a left-going shock wave and a right-going shock wave. Besides, the waves in $\rho$ and $p_{22}$ both contain a right-going contact discontinuity. $u_2$ and $p_{12}$ are constants.

\item Between $\rm\uppercase\expandafter{\romannumeral3}$ and $\rm\uppercase\expandafter{\romannumeral2}$: All waves in $\rho$, $u_2$, $p_{11}$ and $p_{22}$  contain a down-going shock wave and an up-going shock wave. Besides, the waves in both $\rho$ and $p_{11}$  contain an up-going contact discontinuity. $u_1$ and $p_{12}$ are constants.

\item Between $\rm\uppercase\expandafter{\romannumeral3}$ and $\rm\uppercase\expandafter{\romannumeral4}$: The wave patterns are similar with those between $\rm\uppercase\expandafter{\romannumeral2}$ and $\rm\uppercase\expandafter{\romannumeral1}$.

\item Between $\rm\uppercase\expandafter{\romannumeral4}$ and $\rm\uppercase\expandafter{\romannumeral1}$: The wave patterns are similar with those between $\rm\uppercase\expandafter{\romannumeral3}$ and $\rm\uppercase\expandafter{\romannumeral2}$.
\end{itemize}
The final time is $t=0.15$. The plots with 40 equally spaced contour lines of the primitive variables are presented in Figure \ref{2d_RP4_nr}.

From the above numerical results for the four 2D Riemann problems, one can find that the proposed GRP scheme obtains satisfactory results and captures these waves well, which include the shock wave, the rarefaction wave, the shear wave, the contact discontinuity and their interactions.

\begin{figure}[!htbp]
\subfigure[$\rho$]{
\begin{minipage}[c]{0.3\linewidth}
\centering
\includegraphics[width=5cm]{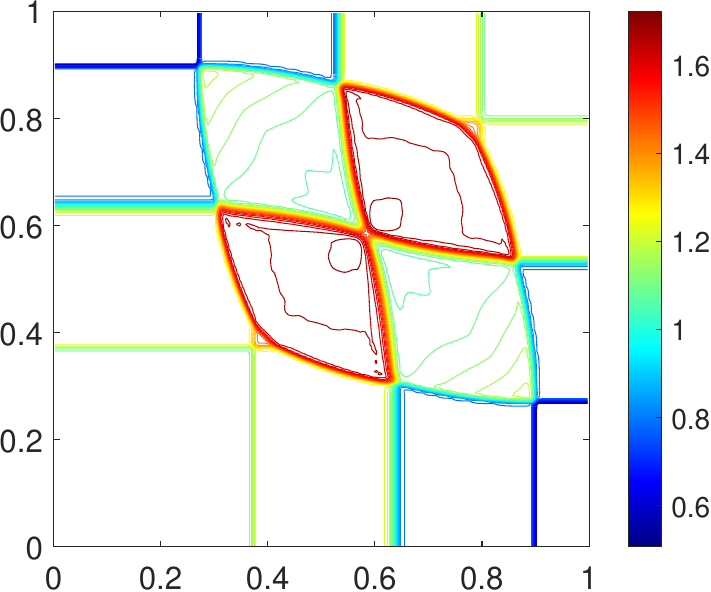}
\end{minipage}
}
\subfigure[$u_1$]{
\begin{minipage}[c]{0.3\linewidth}
\centering
\includegraphics[width=5cm]{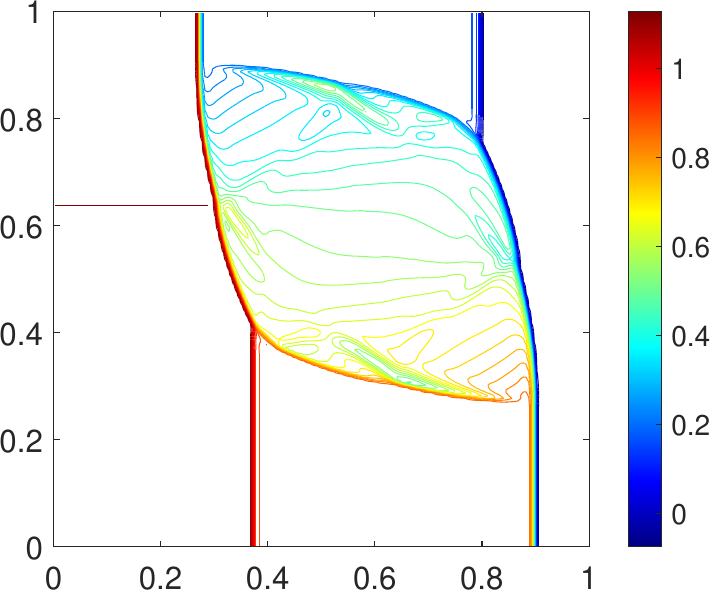}
\end{minipage}
}
\subfigure[$u_2$]{
\begin{minipage}[c]{0.3\linewidth}
\centering
\includegraphics[width=5cm]{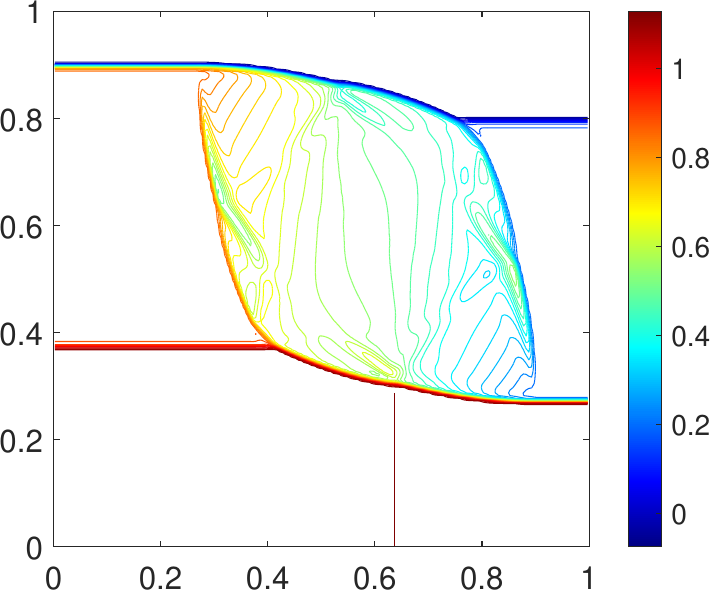}
\end{minipage}
}
\subfigure[$p_{11}$]{
\begin{minipage}[c]{0.3\linewidth}
\centering
\includegraphics[width=5cm]{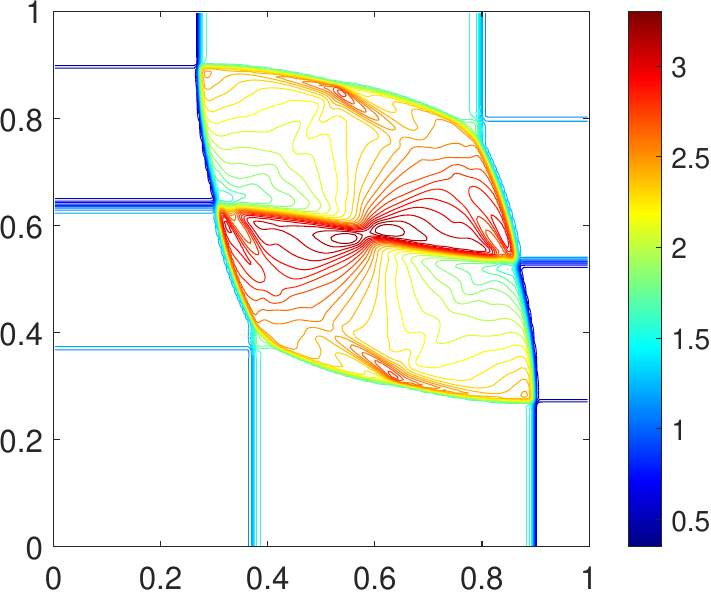}
\end{minipage}
}
\subfigure[$p_{12}$]{
\begin{minipage}[c]{0.3\linewidth}
\centering
\includegraphics[width=5cm]{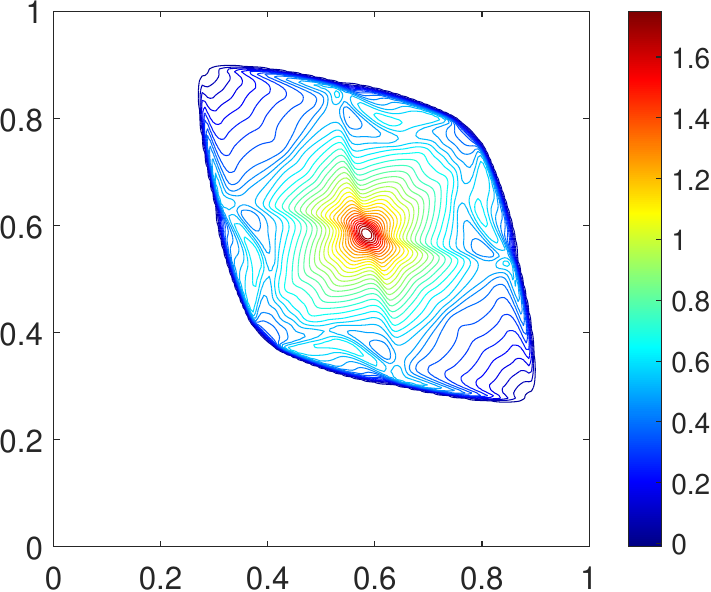}
\end{minipage}
}
\subfigure[$p_{22}$]{
\begin{minipage}[c]{0.3\linewidth}
\centering
\includegraphics[width=5cm]{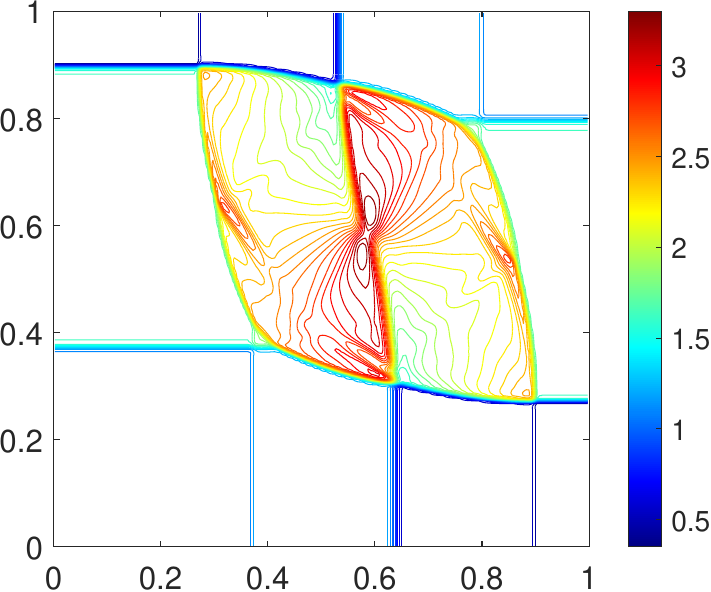}
\end{minipage}
}
\centering
\caption{Example \ref{2D_RP}: The contour plots of solutions obtained by the GRP scheme for the 2D Riemann problem \eqref{2D_RP4} on a mesh of $200\times200$ cells. 40 equally spaced contour lines are displayed.}
\label{2d_RP4_nr}
\end{figure}

\end{example}

\begin{example}{Uniform plasma state with 2D Gaussian source}\label{uniform_plasma}
\rm

This test is applied to evaluate the influence of a Gaussian source term on a 2D plasma model \cite{meena2017positivity,sen2018entropy,meena2020positivity}. The plasma is initially in a uniform state given by
\begin{equation*}
(\rho, u_1, u_2, p_{11}, p_{12}, p_{22}) = (0.1, 0, 0, 9, 7, 9),
\end{equation*}
with the potential
\begin{equation*}
W(x,y) = 25\exp\left(-200\left((x - 2)^2 + (y - 2)^2\right)\right)
\end{equation*}
over the spatial domain $[0, 4]^2$. The outflow boundary conditions are imposed. Figure \ref{uniform_plasma_nr} presents the numerical results at $t = 0.1$, obtained by the GRP scheme on a mesh consisting of $200 \times 200$ cells. One can observe the anisotropic changes in the density due to the Gaussian source's influence.

\end{example}

\begin{figure}[!htbp]
\subfigure[$\rho$]{
\begin{minipage}[c]{0.3\linewidth}
\centering
\includegraphics[width=5.1cm]{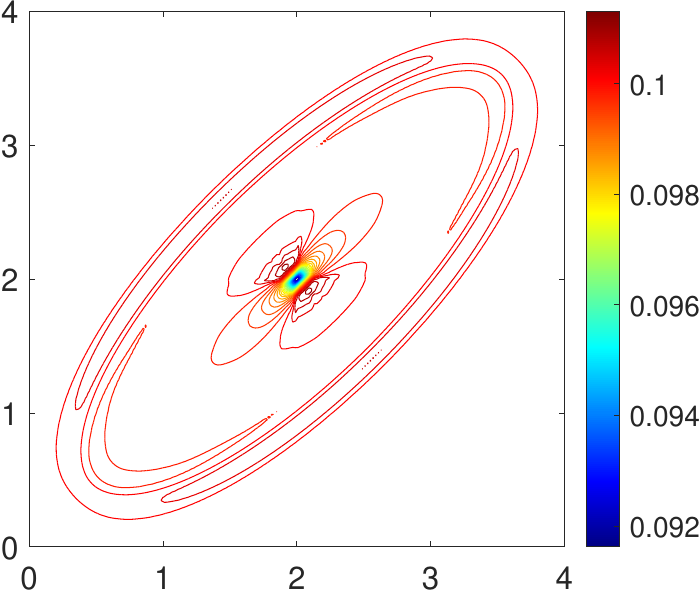}
\end{minipage}
\label{p2_rho}
}
\subfigure[${\rm trace}(\mathbf{p})$]{
\begin{minipage}[c]{0.3\linewidth}
\centering
\includegraphics[width=5cm]{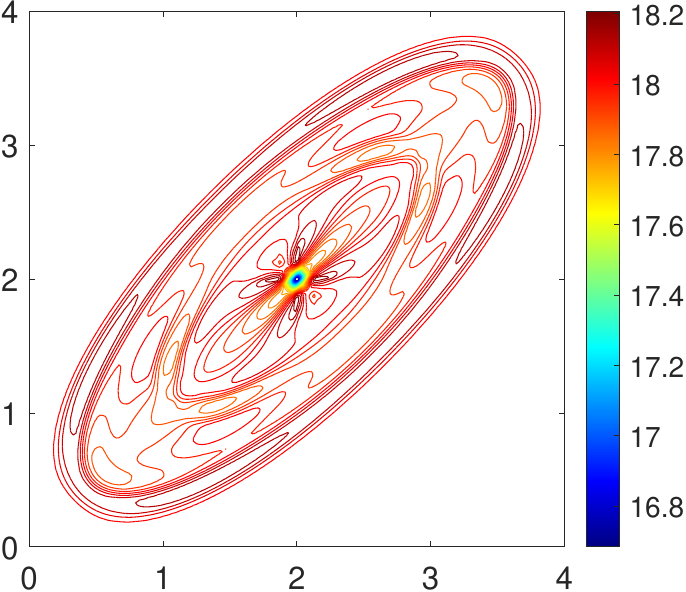}
\end{minipage}
\label{p2_tracep}
}
\subfigure[$\det(\mathbf{p})$]{
\begin{minipage}[c]{0.3\linewidth}
\centering
\includegraphics[width=4.8cm]{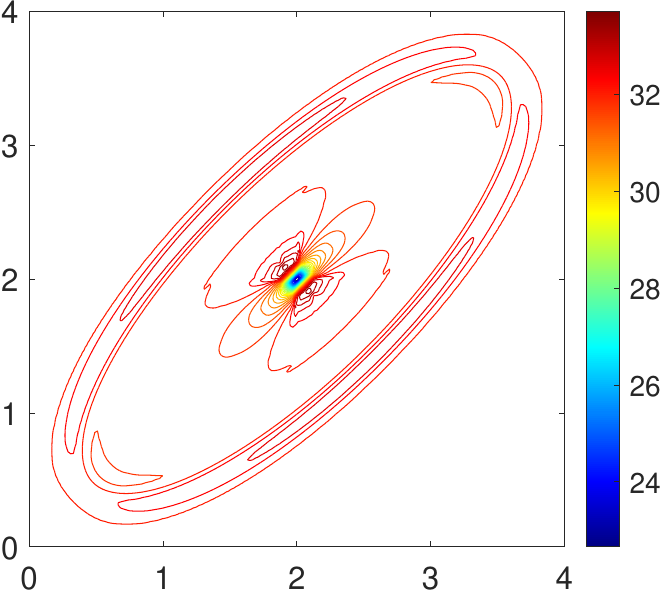}
\end{minipage}
\label{p2_detp}
}
\centering
\caption{Example \ref{uniform_plasma}: The contour plots of the uniform plasma state with Gaussian source at time $t=0.1$ on a mesh of $200\times200$ cells. 40 equally spaced contour lines are displayed.}
\label{uniform_plasma_nr}
\end{figure}

\begin{example}{Realistic simulation in two dimensions}\label{realistic_simulation}
\rm

In this  example \cite{berthon2015entropy,meena2017positivity,meena2020positivity}, a plasma state within the domain $[0,100]^2$, initially defined by
\begin{equation*}
(\rho, u_1, u_2, p_{11}, p_{12}, p_{22}) = (0.109885, 0, 0, 1, 0, 1),
\end{equation*}
is examined.
The potential is taken as
\begin{equation*}
W(x,y) = \exp\left(\frac{-(x-50)^2 - (y-50)^2}{100}\right),
\end{equation*}
which exerts an influence solely in the $x$-direction with $\mathbf{S}^y(\mathbf{U})=0$. Outflow boundary conditions are implemented.

This problem originally aimed to study the effects of inverse Bremsstrahlung absorption (IBA) \cite{berthon2015entropy}. To simulate the IBA in an anisotropic plasma, we augment the energy equation for component $E_{11}$ with an additional source term, $v_T\rho W$, where $v_T$ denotes the absorption coefficient with $v_T$ values of 0, 0.5, and 1.

By employing the GRP scheme, this problem is simulated up to $t = 0.5$ on a mesh with $200 \times 200$ cells.
Figure \ref{realistic_simulation_nr1} shows contour plots of $\rho$, ${\rm trace}(\mathbf{p})$, and $\det(\mathbf{p})$ for $v_T=0.5$. Figure \ref{realistic_simulation_nr2} illustrates the 1D profiles of $\rho$ and $p_{11}$ along the line $y = 50$. An increase in the absorption coefficient, $v_T$, is observed to raise the pressure component $p_{11}$ around the center. This, in turn, drives a more pronounced expulsion of particles from the region, leading to a reduction in density near the center. These observations are consistent with the results documented in  \cite{meena2017positivity,sen2018entropy,meena2020positivity}.

\end{example}

\begin{figure}[!htbp]
\subfigure[$\rho$]{
\begin{minipage}[c]{0.3\linewidth}
\centering
\includegraphics[width=5.3cm]{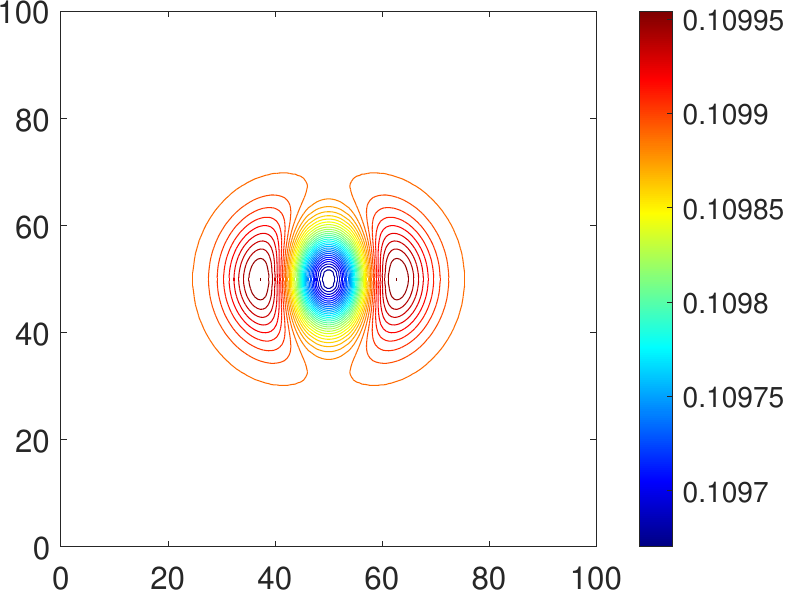}
\end{minipage}
}
\subfigure[${\rm trace}(\mathbf{p})$]{
\begin{minipage}[c]{0.3\linewidth}
\centering
\includegraphics[width=5cm]{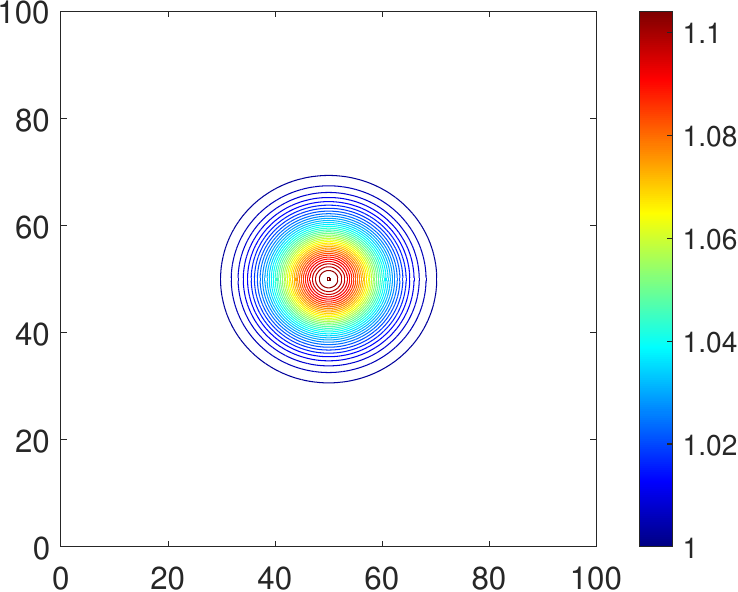}
\end{minipage}
}
\subfigure[$\det(\mathbf{p})$]{
\begin{minipage}[c]{0.3\linewidth}
\centering
\includegraphics[width=5cm]{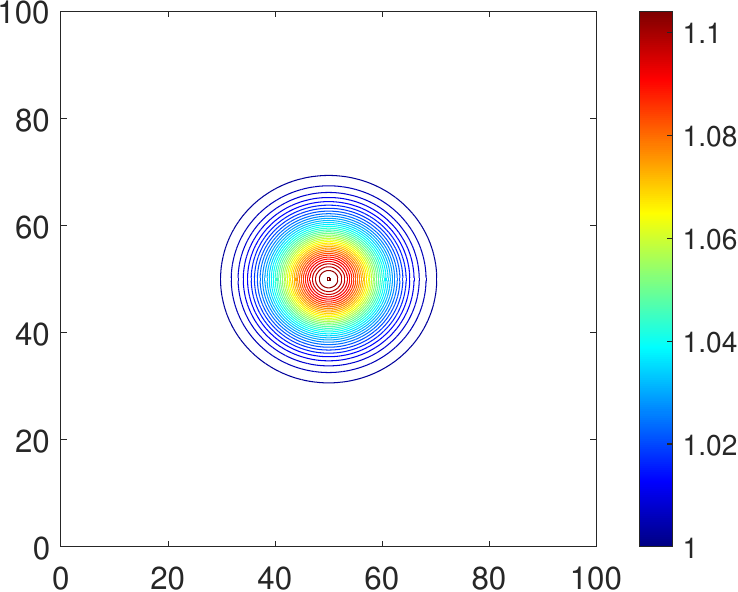}
\end{minipage}
}
\centering
\caption{Example \ref{realistic_simulation}: The contour plots of the realistic simulation at time $t=0.5$ on the meshes of $200\times200$ cells. 40 equally spaced contour lines are displayed.}
\label{realistic_simulation_nr1}
\end{figure}

\begin{figure}[!htbp]
\subfigure[$\rho$]{
\begin{minipage}[c]{0.45\linewidth}
\centering
\includegraphics[width=5cm]{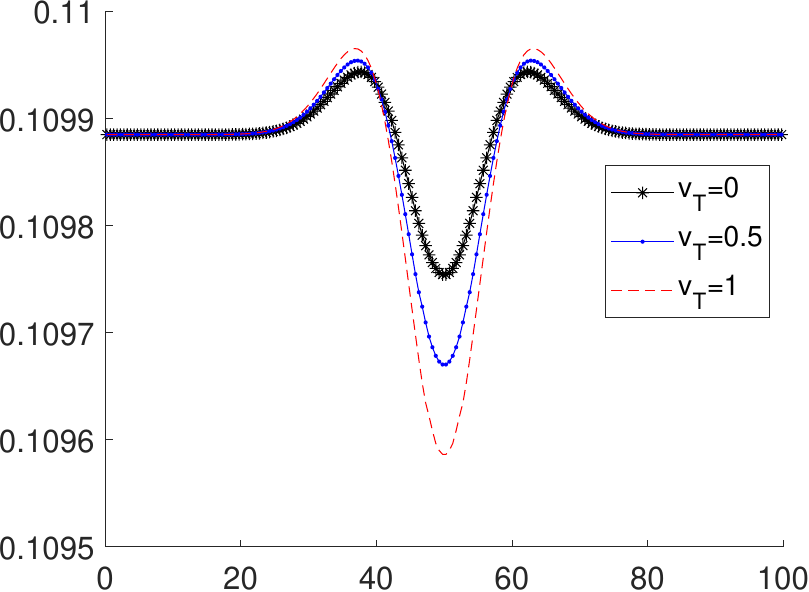}
\end{minipage}
}
\subfigure[$p_{11}$]{
\begin{minipage}[c]{0.45\linewidth}
\centering
\includegraphics[width=5cm]{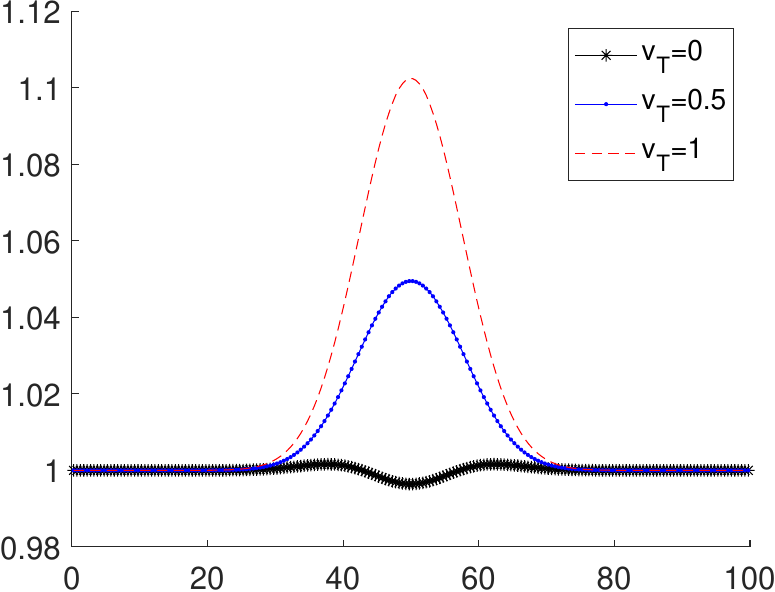}
\end{minipage}
}
\centering
\caption{Example \ref{realistic_simulation}: Comparison of $\rho$ and $p_{11}$ for different absorption coefficient $v_T=0$, $v_T=0.5$ and $v_T=1$ along the line $y=50$.}
\label{realistic_simulation_nr2}
\end{figure}

\section{Conclusion}\label{conclusion}

This paper developed the second-order accurate direct Eulerian generalized Riemann problem (GRP) scheme for the ten-moment Gaussian closure equations with source terms.
The generalized Riemann invariants associated with the rarefaction waves, the contact discontinuity and the shear waves are given, and the 1D  exact Riemann solver was obtained for the homogeneous system based on the characteristic analysis.  By directly using two main ingredients, the generalized Riemann invariants and the Rankine-Hugoniot jump conditions, the left and right nonlinear waves of the local GRP were resolved in the Eulerian formulation, and the limiting values of the time derivatives of the primitive variables along the cell interface were also obtained. It was noted that compared to some other systems, such as the Euler equations \cite{ben2006direct}, the shallow water equations \cite{li2006generalized}, the radiation hydrodynamical equations \cite{kuang2019second} and the blood flow model \cite{sheng2021direct} and so on, there were more physical variables and two more elementary waves (the left and right shear waves) for the ten-moment equations, which made the derivation of the GRP scheme much more complicated and nontrivial.  Four 2D Riemann problems  were constructed for the first time to verify the performance of
the proposed GRP schemes.  Some 1D and 2D numerical experiments were employed to demonstrate the accuracy and high resolution of the proposed GRP scheme.

 \section*{Acknowledgements}
The works were supported by the National Key R\&D Program of China (Project Number
2020YFA0712000 \& 2020YFE0204200),  and the National Natural Science Foundation of China (Nos.~12171227, 12326314, \& 12288101).
 \begin{appendix}

 \section{The right eigenvector matrix.}\label{R(U)}

 The right eigenvector matrix of the Jacobian matrix $\frac{\partial\mathbf{F(U)}}{\partial\mathbf{U}}$ used in this paper is given by
 \[
 \mathbf{R(U)}=\begin{pmatrix}
                 0 & 2 & 0 & 2\rho p_{11} & 0 & 2\rho p_{11} \\
                 0 & 2u_1 & 0 & -2p_{11}\sqrt{\rho}M_1 & 0 & 2p_{11}\sqrt{\rho}M_3 \\
                 0 & 2u_2 & \sqrt{\rho p_{11}} & -2\sqrt{\rho p_{11}}M_2 & \sqrt{\rho p_{11}} & 2\sqrt{\rho p_{11}}M_4 \\
                 0 & u_1^2 & 0 & p_{11}M_1^2 & 0 & p_{11}M_3^2 \\
                 0 & u_1u_2 & \frac{1}{2}(u_1\sqrt{\rho p_{11}}-p_{11}) & \sqrt{p_{11}}M_1M_2 & \frac{1}{2}(u_1\sqrt{\rho p_{11}}+p_{11}) & \sqrt{p_{11}}M_3M_4\\
                 1 & u_2^2 & u_2\sqrt{\rho p_{11}}-p_{12} & M_2^2+\det(\mathbf{p}) & u_2\sqrt{\rho p_{11}}+p_{12} & M_4^2+\det(\mathbf{p})
               \end{pmatrix},
 \]
 where $M_1:=\sqrt{3p_{11}}-u_1\sqrt{\rho}$, $M_2:=\sqrt{3}p_{12}-u_2\sqrt{\rho p_{11}}$, $M_3:=\sqrt{3p_{11}}+u_1\sqrt{\rho}$, and $M_4:=\sqrt{3}p_{12}+u_2\sqrt{\rho p_{11}}$.

 \section{Computing $\left(\frac{\partial u_1}{\partial t}\right)_{\ast K}$, $\left(\frac{\partial p_{11}}{\partial t}\right)_{\ast K}$ and $\left(\frac{\partial\rho}{\partial t}\right)_{\ast K}$ $(K=L,R)$ for all possible cases.}\label{u1_p11_rho_t_all_cases}

 In Subsection \ref{computation_u1t_p11t_rhot_astK}, only the case that the 1-rarefaction wave and the 6-shock wave is considered. Here, the computations of $\left(\frac{\partial u_1}{\partial t}\right)_{\ast K}$, $\left(\frac{\partial p_{11}}{\partial t}\right)_{\ast K}$ and $\left(\frac{\partial\rho}{\partial t}\right)_{\ast K}$ for all possible cases are directly presented.

 One can solve the following system
 \begin{equation*}
 \begin{cases}
 a_1\left(\frac{\mathcal{D}u_1}{\mathcal{D}t}\right)_{\ast}+b_1\left(\frac{\mathcal{D}p_{11}}{\mathcal{D}t}\right)_{\ast}=d_1, \\
 a_2\left(\frac{\mathcal{D}u_1}{\mathcal{D}t}\right)_{\ast}+b_2\left(\frac{\mathcal{D}p_{11}}{\mathcal{D}t}\right)_{\ast}=d_2,
 \end{cases}
 \end{equation*}
to obtain the values of $\left(\frac{\mathcal{D}u_1}{\mathcal{D}t}\right)_{\ast}$ and $\left(\frac{\mathcal{D}p_{11}}{\mathcal{D}t}\right)_{\ast}$,
where
\begin{align}
&(a_1,b_1,d_1)=\begin{cases}
                (a_L^{\text{rare}},b_L^{\text{rare}},d_L^{\text{rare}}), \quad \text{1-rarefaction}, \\
                (a_L^{\text{shock}},b_L^{\text{shock}},d_L^{\text{shock}}), \quad \text{1-shock},
               \end{cases} ~
(a_2,b_2,d_2)=\begin{cases}
                (a_R^{\text{rare}},b_R^{\text{rare}},d_R^{\text{rare}}), \quad \text{6-rarefaction}, \\
                (a_R^{\text{shock}},b_R^{\text{shock}},d_R^{\text{shock}}), \quad \text{6-shock}.
               \end{cases} \label{a12_b12_d12}
\end{align}
Then utilizing \eqref{u1_p11_t} gives the values of $\left(\frac{\partial u_1}{\partial t}\right)_{\ast K}$ and $\left(\frac{\partial p_{11}}{\partial t}\right)_{\ast K}$ with $K=L,R$. The $(a_K^{\text{rare}},b_K^{\text{rare}},d_K^{\text{rare}})$ and $(a_K^{\text{shock}},b_K^{\text{shock}},d_K^{\text{shock}})$ $(K=L,R)$ in \eqref{a12_b12_d12} are taken as
\begin{align*}
&a_K^{\text{rare}}=1+\delta_K\frac{u_{1,\ast}}{c_{\ast K}}, \quad b_K^{\text{rare}}=\frac{u_{1,\ast}}{3p_{11,\ast}}+\frac{\delta_K}{\rho_{\ast K}c_{\ast K}}, \\
&d_K^{\text{rare}}=\left[\delta_K\frac{\rho_K^2S_{1,K}'}{8c_K^3}(3c_K^2+c_{\ast K}^2)-(u_{1,K}+\delta_Kc_K)'\right](u_{1,K}+\delta_Kc_K)-\frac{1}{2}\left(1+\delta_K\frac{u_{1,\ast}}{c_{\ast K}}\right)W_x(0), \\
&a_K^{\text{shock}}=1+\delta_K\rho_{\ast K}(u_{1,\ast}-\sigma_K)\cdot\Phi_1^K, \quad b_K^{\text{shock}}=\frac{u_{1,\ast}-\sigma_K}{3p_{11,\ast}}+\delta_K\cdot\Phi_1^K, \\
&d_K^{\text{shock}}=L_{\rho}^K\cdot\rho_K'+L_{u_1}^K\cdot u_{1,K}'+L_{p_{11}}^K\cdot p_{11,K}'-\frac{1}{2}W_x(0)\cdot L_W^K,
\end{align*}
where
 \begin{align*}
 &\delta_K=\begin{cases}
            1, & \mbox{if } K=L, \\
            -1, & \mbox{if } K=R,
          \end{cases} \\
 &\sigma_K=\frac{\rho_{\ast K}u_{1,\ast}-\rho_Lu_{1,L}}{\rho_{\ast K}-\rho_K}, ~ \Phi_i^K=\Phi_i(p_{11,\ast};p_{11,K},\rho_K), ~ i=1,2,3, \\
 &L_{\rho}^K=-\delta_K(\sigma_K-u_{1,K})\cdot\Phi_3^K, ~ L_{u_1}^K=\sigma_K-u_{1,K}+3\delta_Kp_{11,K}\cdot\Phi_2^K+\delta_K\rho_K\cdot\Phi_3^K, \\
 &L_{p_{11}}^K=-\frac{1}{\rho_K}-\delta_K(\sigma_K-u_{1,K})\cdot\Phi_2^K, ~ L_{W}^K=\delta_K\rho_{\ast K}(u_{1,\ast}-\sigma_K)\cdot\Phi_1^K+1.
 \end{align*}
 The expressions of the functions $\Phi_i$ ($i=1,2,3$) are given in \eqref{Phi_123}.

 For the 1-shock wave, $\left(\frac{\partial\rho}{\partial t}\right)_{\ast L}$ is computed by
 \[
 g_{\rho}^L\cdot\left(\frac{\partial\rho}{\partial t}\right)_{\ast L}+g_{u_1}^L\cdot\left(\frac{\mathcal{D}u_1}{\mathcal{D}t}\right)_\ast+g_{p_{11}}^L\cdot\left(\frac{\mathcal{D}p_{11}}{\mathcal{D}t}\right)_\ast=u_{1,\ast}\cdot f_L,
\]
where
\begin{align*}
&g_{\rho}^L=u_{1,\ast}-\sigma_L, ~ g_{u_1}^L=\rho_{\ast L}u_{1,\ast}(\sigma_L-u_{1,\ast})\cdot H_1^L, ~ g_{p_{11}}^L=\frac{\sigma_L}{c_{\ast L}^2}-u_{1,\ast}\cdot H_1^L, \\
&f_L=(\sigma_L-u_{1,L})\cdot H_2^L\cdot p_{11,L}'+(\sigma_L-u_{1,L})\cdot H_3^L\cdot\rho_{L}'-\rho_L(H_3^L+c_L^2\cdot H_2^L)\cdot u_{1,L}'-\frac{1}{2}\rho_{\ast L}(\sigma_L-u_{1,\ast})W_x(0)\cdot H_1^L
\end{align*}
with $H_i^L=H_i(p_{11,\ast};p_{11,L},\rho_L)$ $(i=1,2,3)$. The expressions of the functions $H_i$ ($i=1,2,3$) are given in \eqref{H_123}.
For the 1-rarefaction wave, $\left(\frac{\partial\rho}{\partial t}\right)_{\ast L}$ is given by \eqref{rho_t_astL}.

For the 6-rarefaction wave, one has
\[
\left(\frac{\partial\rho}{\partial t}\right)_{\ast R}=\frac{1}{c_{\ast R}^2}\left[\left(\frac{\partial p_{11}}{\partial t}\right)_{\ast R}+\rho_R^2S_{1,R}'\rho_{\ast R}u_{1,\ast}\frac{c_{\ast R}^3}{c_R^3}\right].
\]
For the 6-shock wave, $\left(\frac{\partial\rho}{\partial t}\right)_{\ast R}$ is given by \eqref{rho_t_astR}.


 \section{Computing $\left(\frac{\partial u_2}{\partial t}\right)_{\ast L}$, $\left(\frac{\partial p_{12}}{\partial t}\right)_{\ast L}$ and $\left(\frac{\partial p_{22}}{\partial t}\right)_{\ast L}$ for the 1-shock wave.}\label{u2_p12_p22_astL_1shock}

 For the 1-shock wave, one can first solve the system
 \begin{align*}
 \begin{cases}
 a_{3,L}\left(\frac{\mathcal{D}u_2}{\mathcal{D}t}\right)_{\ast L}+b_{3,L}\left(\frac{\mathcal{D}p_{12}}{\mathcal{D}t}\right)_{\ast L}=d_{3,L}, \\
 a_{4,L}\left(\frac{\mathcal{D}u_2}{\mathcal{D}t}\right)_{\ast L}+b_{4,L}\left(\frac{\mathcal{D}p_{12}}{\mathcal{D}t}\right)_{\ast L}=d_{4,L},
 \end{cases}
 \end{align*}
 to get the values of $\left(\frac{\mathcal{D}u_2}{\mathcal{D}t}\right)_{\ast L}$ and $\left(\frac{\mathcal{D}p_{12}}{\mathcal{D}t}\right)_{\ast L}$, then substituting them into \eqref{u2_p12_t} gives the values of $\left(\frac{\partial u_2}{\partial t}\right)_{\ast L}$ and $\left(\frac{\partial p_{12}}{\partial t}\right)_{\ast L}$. In the above system,
 \begin{align*}
 &a_{3,L}=2\rho_{\ast L}(u_{1,\ast}-\sigma_L), \quad a_{4,L}=E_{11,\ast L}+\rho_{\ast L}(u_{1,\ast}-\sigma_L)^2-\frac{1}{2}\rho_{\ast L}\sigma_L^2, \\
 &b_{3,L}=1+\frac{\rho_{\ast L}(u_{1,\ast}-\sigma_L)^2}{p_{11,\ast}}, \quad b_{4,L}=\left(E_{11,\ast L}-\frac{1}{2}\rho_{\ast L}u_{1,\ast}\sigma_L\right)\frac{u_{1,\ast}-\sigma_L}{p_{11,\ast}}+u_{1,\ast}-\frac{1}{2}\sigma_L, \\
 &d_{3,L}=\frac{\mathcal{D}_{\sigma_L}\Gamma_{m_{2,L}}}{\mathcal{D}t}+\rho_{\ast L}u_{2,\ast L}\frac{\mathcal{D}_{\sigma_L}\sigma_L}{\mathcal{D}t}-u_{2,\ast L}(u_{1,\ast}-\sigma_L)\left(\frac{\mathcal{D}_{\sigma_L}\rho}{\mathcal{D}t}\right)_{\ast L}-\rho_{\ast L}u_{2,\ast L}\left(\frac{\mathcal{D}_{\sigma_L}u_1}{\mathcal{D}t}\right)_{\ast L} \\
 &\qquad+\frac{2\rho_{\ast L}p_{12,\ast L}(u_{1,\ast}-\sigma_L)^2}{3p_{11,\ast}^2}\left(\frac{\mathcal{D}p_{11}}{\mathcal{D}t}\right)_{\ast}, \\
 &d_{4,L}=\frac{\mathcal{D}_{\sigma_L}\Gamma_{E_{12,L}}}{\mathcal{D}t}+E_{12,\ast L}\frac{\mathcal{D}_{\sigma_L}\sigma_L}{\mathcal{D}t}-\frac{1}{2}u_{1,\ast}u_{2,\ast L}(u_{1,\ast}-\sigma_L)\left(\frac{\mathcal{D}_{\sigma_L}\rho}{\mathcal{D}t}\right)_{\ast L}
 -\frac{1}{2}u_{2,\ast L}\frac{\mathcal{D}_{\sigma_L}p_{11,\ast}}{\mathcal{D}t} \\
 &\qquad-\left(2E_{12,\ast L}-\frac{1}{2}\rho_{\ast L}u_{2,\ast L}\sigma_L\right)\left(\frac{\mathcal{D}_{\sigma_L}u_1}{\mathcal{D}t}\right)_{\ast L}
+\left(E_{11,\ast L}-\frac{1}{2}\rho_{\ast L}u_{1,\ast}\sigma_L\right)\frac{2p_{12,\ast L}(u_{1,\ast}-\sigma_L)}{3p_{11,\ast}^2}\left(\frac{\mathcal{D}p_{11}}{\mathcal{D}t}\right)_{\ast},
 \end{align*}
 where
 \begin{align*}
 &\frac{\mathcal{D}_{\sigma_L}\Gamma_{m_{2,L}}}{\mathcal{D}t}=u_{2,L}(u_{1,L}-\sigma_L)\frac{\mathcal{D}_{\sigma_L}\rho_L}{\mathcal{D}t}
 +\rho_Lu_{2,L}\frac{\mathcal{D}_{\sigma_L}u_{1,L}}{\mathcal{D}t}+\rho_L(u_{1,L}-\sigma_L)\frac{\mathcal{D}_{\sigma_L}u_{2,L}}{\mathcal{D}t}
 +\frac{\mathcal{D}_{\sigma_L}p_{12,L}}{\mathcal{D}t}-\rho_Lu_{2,L}\frac{\mathcal{D}_{\sigma_L}\sigma_L}{\mathcal{D}t}, \\
 &\frac{\mathcal{D}_{\sigma_L}\Gamma_{E_{12,L}}}{\mathcal{D}t}=\frac{1}{2}u_{1,L}u_{2,L}(u_{1,L}-\sigma_L)\frac{\mathcal{D}_{\sigma_L}\rho_L}{\mathcal{D}t}
 +\left(2E_{12,L}-\frac{1}{2}\rho_Lu_{2,L}\sigma_L\right)\frac{\mathcal{D}_{\sigma_L}u_{1,L}}{\mathcal{D}t}-E_{12,L}\frac{\mathcal{D}_{\sigma_L}\sigma_L}{\mathcal{D}t} \\
 &\qquad+\left(E_{11,L}-\frac{1}{2}\rho_Lu_{1,L}\sigma_L\right)\frac{\mathcal{D}_{\sigma_L}u_{2,L}}{\mathcal{D}t}+\frac{1}{2}u_{2,L}\frac{\mathcal{D}_{\sigma_L}p_{11,L}}{\mathcal{D}t}
 +\left(u_{1,L}-\frac{1}{2}\sigma_L\right)\frac{\mathcal{D}_{\sigma_L}p_{12,L}}{\mathcal{D}t}, \\
 &\frac{\mathcal{D}_{\sigma_L}\sigma_L}{\mathcal{D}t}=\frac{\mathcal{D}_{\sigma_L}u_{1,L}}{\mathcal{D}t}+\frac{\sqrt{2p_{11,\ast}+p_{11,L}}}{2\rho_L\sqrt{\rho_L}}\frac{\mathcal{D}_{\sigma_L}\rho_L}{\mathcal{D}t}
 -\frac{1}{\sqrt{\rho_L(2p_{11,\ast}+p_{11,L})}}\frac{\mathcal{D}_{\sigma_L}p_{11,\ast}}{\mathcal{D}t}-\frac{1}{2\sqrt{\rho_L(2p_{11,\ast}+p_{11,L})}}\frac{\mathcal{D}_{\sigma_L}p_{11,L}}{\mathcal{D}t}, \\
 &\left(\frac{\mathcal{D}_{\sigma_L}\rho}{\mathcal{D}t}\right)_{\ast L}=H_1^L\cdot\frac{\mathcal{D}_{\sigma_L}p_{11,\ast}}{\mathcal{D}t}
 +H_2^L\cdot\frac{\mathcal{D}_{\sigma_L}p_{11,L}}{\mathcal{D}t}+H_3^L\cdot\frac{\mathcal{D}_{\sigma_L}\rho_L}{\mathcal{D}t}, ~ \left(\frac{\mathcal{D}_{\sigma_L}u_1}{\mathcal{D}t}\right)_{\ast L}=\left(\frac{\partial u_1}{\partial t}\right)_{\ast L}-\frac{\sigma_L}{3p_{11,\ast}}\left(\frac{\mathcal{D}p_{11}}{\mathcal{D}t}\right)_{\ast}, \\
 &\frac{\mathcal{D}_{\sigma_L}p_{11,\ast}}{\mathcal{D}t}=\left(\frac{\mathcal{D}p_{11}}{\mathcal{D}t}\right)_{\ast}-\rho_{\ast L}(\sigma_L-u_{1,\ast})\left[\left(\frac{\mathcal{D}u_1}{\mathcal{D}t}\right)_{\ast}+\frac{1}{2}W_x(0)\right],
 \end{align*}
 with
 \begin{align*}
 &\frac{\mathcal{D}_{\sigma_L}\rho_L}{\mathcal{D}t}=(\sigma_L-u_{1,L})\rho_L'-\rho_Lu_{1,L}', \quad
 \frac{\mathcal{D}_{\sigma_L}u_{1,L}}{\mathcal{D}t}=(\sigma_L-u_{1,L})u_{1,L}'-\frac{p_{11,L}'}{\rho_L}-\frac{1}{2}W_x(0), \\
 &\frac{\mathcal{D}_{\sigma_L}u_{2,L}}{\mathcal{D}t}=(\sigma_L-u_{1,L})u_{2,L}'-\frac{p_{12,L}'}{\rho_L}, \quad
 \frac{\mathcal{D}_{\sigma_L}p_{11,L}}{\mathcal{D}t}=(\sigma_L-u_{1,L})p_{11,L}'-3p_{11,L}u_{1,L}', \\
 &\frac{\mathcal{D}_{\sigma_L}p_{12,L}}{\mathcal{D}t}=(\sigma_L-u_{1,L})p_{12,L}'-2p_{12,L}u_{1,L}'-p_{11,L}u_{2,L}'.
 \end{align*}

 The value of $\left(\frac{\partial p_{22}}{\partial t}\right)_{\ast L}$ is computed by
 \[
 g_{p_{22}}^{\ast L}\left(\frac{\partial p_{22}}{\partial t}\right)_{\ast L}=f_{p_{22}}^{\ast L},
 \]
 where
 \begin{align*}
 &g_{p_{22}}^{\ast L}=\frac{1}{2}(u_{1,\ast}-\sigma_L)^2, \\
 &f_{p_{22}}^{\ast L}=u_{1,\ast}\widetilde{f}_{p_{22}}^{\ast L}-\frac{(p_{11,\ast}p_{22,\ast L}-4p_{12,\ast L}^2)(u_{1,\ast}-\sigma_L)\sigma_L}{6p_{11,\ast}^2}\left(\frac{\mathcal{D}p_{11}}{\mathcal{D}t}\right)_{\ast}
 -\frac{p_{12,\ast L}\sigma_L(u_{1,\ast}-\sigma_L)}{p_{11,\ast}}\left(\frac{\mathcal{D}p_{12}}{\mathcal{D}t}\right)_{\ast L}, \\
 &\widetilde{f}_{p_{22}}^{\ast L}=\frac{\mathcal{D}_{\sigma_L}E_{22,L}}{\mathcal{D}t}+E_{22,\ast L}\frac{\mathcal{D}_{\sigma_L}\sigma_L}{\mathcal{D}t}-\frac{1}{2}u_{2,\ast L}^2(u_{1,\ast}-\sigma_L)\left(\frac{\mathcal{D}_{\sigma_L}\rho}{\mathcal{D}t}\right)_{\ast L}-E_{22,\ast L}\left(\frac{\mathcal{D}_{\sigma_L}u_1}{\mathcal{D}t}\right)_{\ast L} \\
 &\qquad-(2E_{12,\ast L}-\rho_{\ast L}u_{2,\ast L}\sigma_L)\left(\frac{\mathcal{D}_{\sigma_L}u_2}{\mathcal{D}t}\right)_{\ast L}
 -u_{2,\ast L}\left(\frac{\mathcal{D}_{\sigma_L}p_{12}}{\mathcal{D}t}\right)_{\ast L},
 \end{align*}
 with
 \begin{align*}
 &\frac{\mathcal{D}_{\sigma_L}E_{22,L}}{\mathcal{D}t}=\frac{1}{2}u_{2,L}^2(u_{1,L}-\sigma_L)\frac{\mathcal{D}_{\sigma_L}\rho_L}{\mathcal{D}t}
 +E_{22,L}\frac{\mathcal{D}_{\sigma_L}u_{1,L}}{\mathcal{D}t}+(2E_{12,L}-\rho_Lu_{2,L}\sigma_L)\frac{\mathcal{D}_{\sigma_L}u_{2,L}}{\mathcal{D}t} \\
 &\qquad+u_{2,L}\frac{\mathcal{D}_{\sigma_L}p_{12,L}}{\mathcal{D}t}+\frac{1}{2}(u_{1,L}-\sigma_L)\frac{\mathcal{D}_{\sigma_L}p_{22,L}}{\mathcal{D}t}
 -E_{22,L}\frac{\mathcal{D}_{\sigma_L}\sigma_L}{\mathcal{D}t}, \\
 &\left(\frac{\mathcal{D}_{\sigma_L}u_2}{\mathcal{D}t}\right)_{\ast L}=\left(\frac{\mathcal{D}u_2}{\mathcal{D}t}\right)_{\ast L}-\frac{\sigma_L-u_{1,\ast}}{p_{11,\ast}}\left[\left(\frac{\mathcal{D}p_{12}}{\mathcal{D}t}\right)_{\ast L}-\frac{2p_{12,\ast L}}{3p_{11,\ast}}\left(\frac{\mathcal{D}p_{11}}{\mathcal{D}t}\right)_{\ast}\right], \\
 &\left(\frac{\mathcal{D}_{\sigma_L}p_{12}}{\mathcal{D}t}\right)_{\ast L}=\left(\frac{\mathcal{D}p_{12}}{\mathcal{D}t}\right)_{\ast L}
 -\rho_{\ast L}(\sigma_L-u_{1,\ast})\left(\frac{\mathcal{D}u_2}{\mathcal{D}t}\right)_{\ast L}, \\
 &\frac{\mathcal{D}_{\sigma_L}p_{22,L}}{\mathcal{D}t}=(\sigma_L-u_{1,L})p_{22,L}'-p_{22,L}u_{1,L}'-2p_{12,L}u_{2,L}'.
 \end{align*}

 \section{Computing $\left(\frac{\partial u_2}{\partial t}\right)_{\ast R}$, $\left(\frac{\partial p_{12}}{\partial t}\right)_{\ast R}$ and $\left(\frac{\partial p_{22}}{\partial t}\right)_{\ast R}$ for the 6-rarefaction wave.}\label{u2_p12_p22_astR_6rarefaction}

Denote $\widehat{\beta}:=u_1+c$, $\varphi_1:=u_1-c$, $\varphi_2:=u_2-\frac{\sqrt{3}p_{12}}{\sqrt{\rho p_{11}}}$ and $\varphi_3:=\frac{p_{12}}{\rho^3}$. For the 6-rarefaction wave, one can obtain the values of $\left(\frac{\mathcal{D}u_2}{\mathcal{D}t}\right)_{\ast R}$ and $\left(\frac{\mathcal{D}p_{12}}{\mathcal{D}t}\right)_{\ast R}$ by solving the following system
\begin{align*}
\begin{cases}
a_{3,R}\left(\frac{\mathcal{D}u_2}{\mathcal{D}t}\right)_{\ast R}+b_{3,R}\left(\frac{\mathcal{D}p_{12}}{\mathcal{D}t}\right)_{\ast R}=d_{3,R}, \\
a_{4,R}\left(\frac{\mathcal{D}u_2}{\mathcal{D}t}\right)_{\ast R}+b_{4,R}\left(\frac{\mathcal{D}p_{12}}{\mathcal{D}t}\right)_{\ast R}=d_{4,R},
\end{cases}
\end{align*}
where
\begin{align*}
&a_{3,R}=1-\frac{\sqrt{3}\rho_{\ast R}u_{1,\ast}}{\sqrt{\rho_{\ast R}p_{11,\ast}}}+\frac{1}{c_{\ast R}}\left(\widehat{\beta}_{\ast R}+\frac{\widehat{\beta}_{\ast R}-\widehat{\beta}_{R}}{4}\varphi_{1,R}\right), \quad b_{3,R}=\frac{u_{1,\ast}}{p_{11,\ast}}-\frac{\sqrt{3}}{\sqrt{\rho_{\ast R}p_{11,\ast}}}, \\
&d_{3,R}=\left[-\frac{\Pi_{2,R}}{2c_R}-\varphi_{2,R}'+\frac{\Pi_{2,R}}{8c_R}(\widehat{\beta}_{\ast R}-\widehat{\beta}_{R})\right]\varphi_{1,R}+\frac{1}{2c_{\ast R}}\left(\widehat{\beta}_{\ast R}+\frac{\widehat{\beta}_{\ast R}-\widehat{\beta}_{R}}{4}\varphi_{1,R}\right)\left\{\frac{2p_{12,\ast R}}{p_{11,\ast}}\left[\left(\frac{\mathcal{D}u_1}{\mathcal{D}t}\right)_{\ast}+\frac{1}{2}W_x(0)\right]\right.\\
&\left.\qquad+\frac{\rho_{\ast R}^2p_{12,\ast R}S_{1,R}'c_{\ast R}}{2c_Rp_{11,\ast}}\right\}-\frac{\sqrt{3}p_{12,\ast R}}{2\rho_{\ast R}\sqrt{\rho_{\ast R}p_{11,\ast}}}\left(\frac{\partial\rho}{\partial t}\right)_{\ast R}
-\frac{\sqrt{3}p_{12,\ast R}}{2p_{11,\ast}\sqrt{\rho_{\ast R}p_{11,\ast}}}\left(\frac{\partial p_{11}}{\partial t}\right)_{\ast R}
+\frac{2u_{1,\ast}p_{12,\ast R}}{3p_{11,\ast}^2}\left(\frac{\mathcal{D}p_{11}}{\mathcal{D}t}\right)_{\ast}, \\
&a_{4,R}=\frac{1}{\rho_{\ast R}^2}\left[u_{1,\ast}-\frac{1}{2}\left(\widehat{\beta}_{\ast R}+\frac{\widehat{\beta}_{\ast R}+\widehat{\beta}_{R}}{4}\varphi_{1,R}\right)\right], \quad b_{4,R}=\frac{1}{\rho_{\ast R}^3}\left[1-\frac{1}{2c_{\ast R}}\left(\widehat{\beta}_{\ast R}+\frac{\widehat{\beta}_{\ast R}+\widehat{\beta}_{R}}{4}\varphi_{1,R}\right)\right], \\
&d_{4,R}=\left[-\frac{\Pi_{4,R}}{2c_R}-\varphi_{3,R}'+\frac{\Pi_{4,R}}{8c_R}(\widehat{\beta}_{\ast R}-\widehat{\beta}_{R})\right]\varphi_{1,R}+\frac{3p_{12,\ast R}}{\rho_{\ast R}^4}\left(\frac{\partial\rho}{\partial t}\right)_{\ast R}+\frac{1}{2c_{\ast R}}\left(\widehat{\beta}_{\ast R}+\frac{\widehat{\beta}_{\ast R}-\widehat{\beta}_{R}}{4}\varphi_{1,R}\right) \\
&\qquad\left\{-\frac{3p_{12,\ast R}}{\rho_{\ast R}^3c_{\ast R}}\left[\left(\frac{\mathcal{D}u_1}{\mathcal{D}t}\right)_{\ast}+\frac{1}{2}W_x(0)\right]-\frac{p_{12,\ast R}}{p_{11,\ast}\rho_{\ast R}^3}\left(\frac{\mathcal{D}p_{11}}{\mathcal{D}t}\right)_{\ast}-\frac{3p_{12,\ast R}S_{1,R}'}{\rho_{\ast R}c_R}\right\},
\end{align*}
with
\begin{align*}
&\Pi_{2,R}=\frac{2}{\rho_R}p_{12,R}'-\frac{3p_{12,R}}{2\rho_R^2}\rho_R'-\frac{3p_{12,R}}{2\rho_Rp_{11,R}}p_{11,R}', \\
&\Pi_{4,R}=-\frac{c_R}{\rho_R^3}p_{12,R}'+\frac{p_{12,R}}{\rho_R^3}u_{1,R}'-\frac{p_{11,R}}{\rho_R^3}u_{2,R}'+\frac{3p_{12,R}c_R}{\rho_R^4}\rho_R', \\
&\varphi_{2,R}'=u_{2,R}'-\frac{\sqrt{3}}{\sqrt{\rho_Rp_{11,R}}}p_{12,R}'+\frac{\sqrt{3}p_{12,R}}{2\rho_R\sqrt{\rho_Rp_{11,R}}}\rho_R'+\frac{\sqrt{3}p_{12,R}}{2p_{11,R}\sqrt{\rho_Rp_{11,R}}}p_{11,R}', \\
&\varphi_{3,R}'=\frac{1}{\rho_R^3}p_{12,R}'-\frac{3p_{12,R}}{\rho_R^4}\rho_R', ~ S_{1,R}'=\frac{1}{\rho_R^3}(p_{11,R}'-c_R^2\rho_R').
\end{align*}
After obtaining the values of $\left(\frac{\mathcal{D}u_2}{\mathcal{D}t}\right)_{\ast R}$ and $\left(\frac{\mathcal{D}p_{12}}{\mathcal{D}t}\right)_{\ast R}$, one can get the values of $\left(\frac{\partial u_2}{\partial t}\right)_{\ast R}$ and $\left(\frac{\partial p_{12}}{\partial t}\right)_{\ast R}$ by \eqref{u2_p12_t}.

The value of $\left(\frac{\partial p_{22}}{\partial t}\right)_{\ast R}$ can be obtained from
\[
\left(\frac{p_{11}}{\rho^4}\frac{\partial p_{22}}{\partial t}\right)_{\ast R}=-\frac{S_{2,R}'}{c_R}c_{\ast R}u_{1,\ast}-
\left(\frac{p_{22}}{\rho^4}\frac{\partial p_{11}}{\partial t}\right)_{\ast R}+\left(\frac{2p_{12}}{\rho^4}\frac{\partial p_{12}}{\partial t}\right)_{\ast R}+\left(\frac{4\det(\mathbf{p})}{\rho^5}\frac{\partial\rho}{\partial t}\right)_{\ast R},
\]
where
\[
S_{2,R}'=\frac{p_{11,R}}{\rho_R^4}p_{22,R}'+\frac{p_{22,R}}{\rho_R^4}p_{11,R}'-\frac{2p_{12,R}}{\rho_R^4}p_{12,R}'-\frac{4\det(\mathbf{p}_R)}{\rho_R^5}\rho_R'.
\]

\section{The sonic case when the $t$-axis is located inside the 6-rarefaction wave.}\label{sonic_case_6rarefaction}

For the sonic case that the $t$-axis is located inside the 6-rarefaction wave, one has
\begin{align*}
&\left(\frac{\partial u_1}{\partial t}\right)_0=\frac{1}{2}\left[\widetilde{d}_{R,0}-\frac{1}{2}W_x(0)\right], \\
&\left(\frac{\partial p_{11}}{\partial t}\right)_0=-\frac{\rho_0c_0}{2}\left[\widetilde{d}_{R,0}+\frac{1}{2}W_x(0)\right], \\
&\left(\frac{\partial\rho}{\partial t}\right)_0=\frac{1}{c_0^2}\left[\left(\frac{\partial p_{11}}{\partial t}\right)_0-\frac{S_{1,R}'}{c_R}\rho_0^3u_{1,0}^2\right],
\end{align*}
where
\[
\widetilde{d}_{R,0}=-\frac{\rho_R^2S_{1,R}'}{4c_R^3}u_{1,0}(3c_R^2+u_{1,0}^2)-2\varphi_{1,R}'u_{1,0}-\frac{1}{2}W_x(0),
\]
with
\[
\varphi_{1,R}'=u_{1,R}'-\frac{1}{\rho_Rc_R}p_{11,R}'-\frac{\rho_R^2}{2c_R}S_{1,R}'.
\]

The values of $\left(\frac{\partial u_2}{\partial t}\right)_0$ and $\left(\frac{\partial p_{12}}{\partial t}\right)_0$ are obtained by solving the following system
\begin{align*}
\begin{cases}
\widetilde{a}_{3,R,0}\left(\frac{\partial u_2}{\partial t}\right)_0+\widetilde{b}_{3,R,0}\left(\frac{\partial p_{12}}{\partial t}\right)_0=\widetilde{d}_{3,R,0}, \\
\widetilde{a}_{4,R,0}\left(\frac{\partial u_2}{\partial t}\right)_0+\widetilde{b}_{4,R,0}\left(\frac{\partial p_{12}}{\partial t}\right)_0=\widetilde{d}_{4,R,0},
\end{cases}
\end{align*}
where
\begin{align*}
&\widetilde{a}_{3,R,0}=1-\frac{\widehat{\beta}_R}{4}, \quad \widetilde{b}_{3,R,0}=\frac{u_{1,0}}{p_{11,0}}\left(1+\frac{\widehat{\beta}_R}{4}\right), \quad
\widetilde{a}_{4,R,0}=\frac{\widehat{\beta}_Rc_0}{4\rho_0^2}, \quad \widetilde{b}_{4,R,0}=\frac{1}{\rho_0^3}\left(1+\frac{\widehat{\beta}_R}{2}\right),\\
&\widetilde{d}_{3,R,0}=\left(-\frac{\Pi_{2,R}}{2c_R}-\varphi_{2,R}'-\frac{\Pi_{2,R}\widehat{\beta}_R}{8c_R}\right)\varphi_{1,R}
+\frac{\widehat{\beta}_R}{4}\left\{\frac{2p_{12,0}}{p_{11,0}}\left[\left(\frac{\partial u_1}{\partial t}\right)_0+\frac{1}{2}W_x(0)\right]+\frac{S_{1,R}'}{c_R}\left(\frac{c\rho^2p_{12}}{2p_{11}}\right)_0\right\} \\
&\qquad-\left(\frac{\sqrt{3}p_{12}}{2\rho\sqrt{\rho p_{11}}}\frac{\partial\rho}{\partial t}\right)_0-\left(\frac{\sqrt{3}p_{12}}{2p_{11}\sqrt{\rho p_{11}}}\frac{\partial p_{11}}{\partial t}\right)_0, \\
&\widetilde{d}_{4,R,0}=\left(-\frac{\Pi_{4,R}}{2c_R}-\varphi_{3,R}'-\frac{\Pi_{4,R}\widehat{\beta}_R}{8c_R}\right)\varphi_{1,R}+\frac{\widehat{\beta}_R}{4}\left\{\left(\frac{3p_{12}}{\rho^3u_1}\right)_0\left[\left(\frac{\partial u_1}{\partial t}\right)_0+\frac{1}{2}W_x(0)\right]-\frac{3p_{12,0}}{\rho_0}\frac{S_{1,R}'}{c_R}\right\}+\left(\frac{3p_{12}}{\rho^4}\frac{\partial\rho}{\partial t}\right)_0,
\end{align*}
with $\widehat{\beta}_R=u_{1,R}+c_R$.

The value of $\left(\frac{\partial p_{22}}{\partial t}\right)_0$ is computed by
\[
\left(\frac{p_{11}}{\rho^4}\frac{\partial p_{22}}{\partial t}\right)_0=\frac{S_{2,R}'}{c_R}c_0^2-\left(\frac{p_{22}}{\rho^4}\frac{\partial p_{11}}{\partial t}\right)_0+\left(\frac{2p_{12}}{\rho^4}\frac{\partial p_{12}}{\partial t}\right)_0+\left(\frac{4\det(\mathbf{p})}{\rho^5}\frac{\partial\rho}{\partial t}\right)_0.
\]

\end{appendix}

\section*{References}
\bibliographystyle{siamplain}

\bibliography{Ten_Moment_GRP}

\begin{thebibliography}{10}

\bibitem{ben1989generalized}
{\sc M.~Ben-Artzi}, {\em The generalized {R}iemann problem for reactive flows},
  Journal of Computational Physics, 81 (1989), pp.~70--101.

\bibitem{ben1984second}
{\sc M.~Ben-Artzi and J.~Falcovitz}, {\em A second-order {G}odunov-type scheme
  for compressible fluid dynamics}, Journal of Computational Physics, 55
  (1984), pp.~1--32.

\bibitem{ben2003generalized}
{\sc M.~Ben-Artzi and J.~Falcovitz}, {\em Generalized {R}iemann Problems in
 Computational Fluid Dynamics},   Cambridge University Press, 2003.

\bibitem{ben2007hyperbolic}
{\sc M.~Ben-Artzi and J.~Li}, {\em Hyperbolic balance laws: {R}iemann
  invariants and the generalized {R}iemann problem}, Numerische Mathematik, 106
  (2007), pp.~369--425.

\bibitem{ben2006direct}
{\sc M.~Ben-Artzi, J.~Li, and G.~Warnecke}, {\em A direct {E}ulerian {GRP}
  scheme for compressible fluid flows}, Journal of Computational Physics, 218
  (2006), pp.~19--43.

\bibitem{berthon2006numerical}
{\sc C.~Berthon}, {\em Numerical approximations of the 10-moment {G}aussian
  closure}, Mathematics of Computation, 75 (2006), pp.~1809--1831.

\bibitem{berthon2015entropy}
{\sc C.~Berthon, B.~Dubroca, and A.~Sangam}, {\em An entropy preserving
  relaxation scheme for ten-moments equations with source terms},
  Communications in Mathematical Sciences, 13 (2015), pp.~2119--2154.

\bibitem{biswas2021entropy}
{\sc B.~Biswas, H.~Kumar, and A.~Yadav}, {\em Entropy stable discontinuous
  {G}alerkin methods for ten-moment {G}aussian closure equations}, Journal of
  Computational Physics, 431 (2021),  110148.

\bibitem{brown1995numerical}
{\sc S.~L. Brown, P.~L. Roe, and C.~P.~T. Groth}, {\em Numerical solution of a
  10-moment model for nonequilibrium gasdynamics},  AIAA 12th Computational Fluid
  Dynamics Conference, San Diego,  USA, AIAA-95-1677-CP, 1995.

\bibitem{cheng2019two}
{\sc J.~Cheng, Z.~Du, X.~Lei, Y.~Wang, and J.~Li}, {\em A two-stage
  fourth-order discontinuous {G}alerkin method based on the {GRP} solver for
  the compressible {E}uler equations}, Computers \& Fluids, 181 (2019),
  pp.~248--258.

\bibitem{dong2019global}
{\sc C.~Dong, L.~Wang, A.~Hakim, A.~Bhattacharjee, J.~A. Slavin, G.~A.
  DiBraccio, and K.~Germaschewski}, {\em Global ten-moment multifluid
  simulations of the solar wind interaction with mercury: from the planetary
  conducting core to the dynamic magnetosphere}, Geophysical Research Letters,
  46 (2019), pp.~11584--11596.

\bibitem{du2023generalized}
{\sc Z.~Du}, {\em Generalized {R}iemann problem method for the {K}apila model
  of compressible multiphase flows {I}: Temporal-spatial coupling finite volume
  scheme}, arXiv preprint, arXiv:2310.08241,  2023.

\bibitem{dubroca2004magnetic}
{\sc B.~Dubroca, M.~Tchong, P.~Charrier, V.~Tikhonchuk, and J.-P. Morreeuw},
  {\em Magnetic field generation in plasmas due to anisotropic laser heating},
  Physics of Plasmas, 11 (2004), pp.~3830--3839.

\bibitem{han2010adaptive}
{\sc EE~Han, J.Q.~Li, and H.Z.~Tang}, {\em An adaptive {GRP} scheme for
  compressible fluid flows}, Journal of Computational Physics, 229 (2010),
  pp.~1448--1466.

\bibitem{han2011accuracy}
{\sc EE~Han, J.Q.~Li, and H.Z.~Tang}, {\em Accuracy of the adaptive {GRP} scheme
  and the simulation of 2-{D} {R}iemann problems for compressible {E}uler
  equations}, Communications in Computational Physics, 10 (2011), pp.~577--609.

\bibitem{huo2024grp}
{\sc Z.~Huo and Z.~Jia}, {\em A {GRP}-based tangential effects preserving, high
  resolution and efficient ghost fluid method for the simulation of
  two-dimensional multi-medium compressible flows}, Computers \& Fluids,
  (2024),  ~106261.

\bibitem{huo2023grp}
{\sc Z.~Huo and J.~Li}, {\em A {GRP}-based high resolution ghost fluid method
  for compressible multi-medium fluid flows {I}: One-dimensional case}, Applied
  Mathematics and Computation, 437 (2023),  ~127506.

\bibitem{johnson2012ten}
{\sc E.~A. Johnson and J.~A. Rossmanith}, {\em Ten-moment two-fluid plasma
  model agrees well with {PIC/V}lasov in {GEM} problem}, in {\em  Hyperbolic
  Problems: Theory, Numerics and Applications} edited by T.T. Li and S. Jiang , World Scientific,
  2012, pp.~461--468.

\bibitem{kuang2019second}
{\sc Y.Y.~Kuang and H.Z.~Tang}, {\em Second-order direct {E}ulerian {GRP} schemes
  for radiation hydrodynamical equations}, Computers \& Fluids, 179 (2019),
  pp.~163--177.

\bibitem{lei2021staggered}
{\sc X.~Lei and J.~Li}, {\em A staggered-projection {G}odunov-type method for
  the {B}aer-{N}unziato two-phase model}, Journal of Computational Physics, 437
  (2021),  ~110312.

\bibitem{levermore1996moment}
{\sc C.~D. Levermore}, {\em Moment closure hierarchies for kinetic theories},
  Journal of Statistical Physics, 83 (1996), pp.~1021--1065.

\bibitem{levermore1998gaussian}
{\sc C.~D. Levermore and W.~J. Morokoff}, {\em The {G}aussian moment closure
  for gas dynamics}, SIAM Journal on Applied Mathematics, 59 (1998),
  pp.~72--96.

\bibitem{li2006generalized}
{\sc J.~Li and G.~Chen}, {\em The generalized {R}iemann problem method for the
  shallow water equations with bottom topography}, International Journal for
  Numerical Methods in Engineering, 65 (2006), pp.~834--862.

\bibitem{li2016two}
{\sc J.~Li and Z.~Du}, {\em A two-stage fourth order time-accurate
  discretization for {L}ax--{W}endroff type flow solvers {I}. {H}yperbolic
  conservation laws}, SIAM Journal on Scientific Computing, 38 (2016),
  pp.~A3046--A3069.

\bibitem{li2011comparison}
{\sc J.~Li, Q.~Li, and K.~Xu}, {\em Comparison of the generalized {R}iemann
  solver and the gas-kinetic scheme for inviscid compressible flow
  simulations}, Journal of Computational Physics, 230 (2011), pp.~5080--5099.

\bibitem{li2009implementation}
{\sc J.~Li, T.~Liu, and Z.~Sun}, {\em Implementation of the {GRP} scheme for
  computing radially symmetric compressible fluid flows}, Journal of
  Computational Physics, 228 (2009), pp.~5867--5887.

\bibitem{li2013adaptive}
{\sc J.~Li and Y.~Zhang}, {\em The adaptive {GRP} scheme for compressible fluid
  flows over unstructured meshes}, Journal of Computational Physics, 242
  (2013), pp.~367--386.

\bibitem{meena2017robust}
{\sc A.~K. Meena and H.~Kumar}, {\em Robust {MUSCL} schemes for ten-moment
  {G}aussian closure equations with source terms}, International Journal on
  Finite Volumes,  14(2017), pp.~1--34.

\bibitem{meena2018well}
{\sc A.~K. Meena and H.~Kumar}, {\em A well-balanced scheme for ten-moment
  {G}aussian closure equations with source term}, Zeitschrift f{\"u}r
  angewandte Mathematik und Physik, 69 (2018), pp.~1--31.

\bibitem{meena2019robust}
{\sc A.~K. Meena and H.~Kumar}, {\em Robust numerical schemes for two-fluid
  ten-moment plasma flow equations}, Zeitschrift f{\"u}r angewandte Mathematik
  und Physik, 70 (2019), pp.~1--30.

\bibitem{meena2017positivity}
{\sc A.~K. Meena, H.~Kumar, and P.~Chandrashekar}, {\em Positivity-preserving
  high-order discontinuous {G}alerkin schemes for ten-moment {G}aussian closure
  equations}, Journal of Computational Physics, 339 (2017), pp.~370--395.

\bibitem{meena2020positivity}
{\sc A.~K. Meena, R.~Kumar, and P.~Chandrashekar}, {\em Positivity-preserving
  finite difference {WENO} scheme for ten-moment equations with source term},
  Journal of Scientific Computing, 82 (2020), p.~15.

\bibitem{morreeuw2006electron}
{\sc J.-P. Morreeuw, A.~Sangam, B.~Dubroca, P.~Charrier, and V.~Tikhonchuk},
  {\em Electron temperature anisotropy modeling and its effect on
  anisotropy-magnetic field coupling in an underdense laser heated plasma}, 
  Journal de Physique IV (Proceedings), 133(2006),
  pp.~295--300.

\bibitem{nkonga2022exact}
{\sc B.~Nkonga and P.~Chandrashekar}, {\em Exact solution for {R}iemann
  problems of the shear shallow water model}, ESAIM: Mathematical Modelling and
  Numerical Analysis, 56 (2022), pp.~1115--1150.

\bibitem{qian2014generalized}
{\sc J.~Qian, J.~Li, and S.~Wang}, {\em The generalized {R}iemann problems for
  compressible fluid flows: Towards high order}, Journal of Computational
  Physics, 259 (2014), pp.~358--389.

\bibitem{qian2023high}
{\sc J.~Qian and S.~Wang}, {\em High-order accurate solutions of generalized
  {R}iemann problems of nonlinear hyperbolic balance laws}, Science China
  Mathematics, 66 (2023), pp.~1609--1648.

\bibitem{sangam2008hllc}
{\sc A.~Sangam}, {\em An {HLLC} scheme for ten-moments approximation coupled
  with magnetic field}, International Journal of Computing Science and
  Mathematics, 2 (2008), pp.~73--109.

\bibitem{sangam2007anisotropic}
{\sc A.~Sangam, J.-P. Morreeuw, and V.~Tikhonchuk}, {\em Anisotropic
  instability in a laser heated plasma}, Physics of Plasmas, 14 (2007), ~053111.
\bibitem{sen2018entropy}
{\sc C.~Sen and H.~Kumar}, {\em Entropy stable schemes for ten-moment
  {G}aussian closure equations}, Journal of Scientific Computing, 75 (2018),
  pp.~1128--1155.

\bibitem{sheng2021direct}
{\sc W.~Sheng, Q.~Zhang, and Y.~Zheng}, {\em A direct {E}ulerian {GRP} scheme
  for a blood flow model in arteries}, SIAM Journal on Scientific Computing, 43
  (2021), pp.~A1975--A1996.

\bibitem{shu1989efficient}
{\sc C.-W. Shu and S.~Osher}, {\em Efficient implementation of essentially
  non-oscillatory shock-capturing schemes, {II}}, Journal of Computational
  Physics, 83 (1989), pp.~32--78.



\bibitem{ta1992generalized}
{\sc T.-T. Li, D.~Serre, and H. Zhang}, {\em The generalized {R}iemann problem
  for the motion of elastic strings}, SIAM Journal on Mathematical Analysis, 23
  (1992), pp.~1189--1203.

\bibitem{tang2003adaptive}
{\sc H.~Tang and T.~Tang}, {\em Adaptive mesh methods for one-and
  two-dimensional hyperbolic conservation laws}, SIAM Journal on Numerical
  Analysis, 41 (2003), pp.~487--515.

\bibitem{toro2013riemann}
{\sc E.~F. Toro}, {\em Riemann Solvers and Numerical Methods for Fluid
  Dynamics: A Practical Introduction}, Springer Science \& Business Media,
  2013.

\bibitem{van1974towards}
{\sc B.~van~Leer}, {\em Towards the ultimate conservative difference scheme.
  {II}. monotonicity and conservation combined in a second-order scheme},
  Journal of Computational Physics, 14 (1974), pp.~361--370.

\bibitem{wang2024high}
{\sc J.F.~Wang, H.Z.~Tang, and K.L.~Wu}, {\em High-order accurate
  positivity-preserving and well-balanced discontinuous {G}alerkin schemes for
  ten-moment {G}aussian closure equations with source terms}, arXiv preprint,
  arXiv:2402.15446,  2024.

\bibitem{wang2018electron}
{\sc L.~Wang, K.~Germaschewski, A.~Hakim, C.~Dong, J.~Raeder, and
  A.~Bhattacharjee}, {\em Electron physics in 3-{D} two-fluid 10-moment
  modeling of {G}anymede's magnetosphere}, Journal of Geophysical Research:
  Space Physics, 123 (2018), pp.~2815--2830.

\bibitem{wang2023stiffened}
{\sc Y.~Wang and J.~Li}, {\em Stiffened gas approximation and {GRP} resolution
  for compressible fluid flows of real materials}, Journal of Scientific
  Computing, 95 (2023), p.~22.

\bibitem{wang2015arbitrary}
{\sc Y.~Wang and S.~Wang}, {\em Arbitrary high order discontinuous {G}alerkin
  schemes based on the {GRP} method for compressible {E}uler equations},
  Journal of Computational Physics, 298 (2015), pp.~113--124.

\bibitem{wu2021geometric}
{\sc K.~Wu and C.-W. Shu}, {\em Geometric quasilinearization framework for
  analysis and design of bound-preserving schemes}, SIAM Review, 65 (2023),
  pp.~1031--1073.

\bibitem{wu2016direct}
{\sc K.~Wu and H.~Tang}, {\em A direct {E}ulerian {GRP} scheme for spherically
  symmetric general relativistic hydrodynamics}, SIAM Journal on Scientific
  Computing, 38 (2016), pp.~B458--B489.

\bibitem{wu2014third1}
{\sc K.~Wu, Z.~Yang, and H.~Tang}, {\em A third-order accurate direct
  {E}ulerian {GRP} scheme for one-dimensional relativistic hydrodynamics}, East
  Asian Journal on Applied Mathematics, 4 (2014), pp.~95--131.

\bibitem{wu2014third2}
{\sc K.~Wu, Z.~Yang, and H.~Tang}, {\em A third-order accurate direct
  {E}ulerian {GRP} scheme for the {E}uler equations in gas dynamics}, Journal
  of Computational Physics, 264 (2014), pp.~177--208.

\bibitem{xu2013direct}
{\sc J.~Xu, M.~Luo, J.~Hu, S.~Wang, B.~Qi, Z.~Qiao}, {\em A direct
  {E}ulerian {GRP} scheme for the prediction of gas-liquid two-phase flow in
  {HTHP} transient wells}, Abstract and Applied Analysis, 2013 (2013),
~171732.

\bibitem{yang2011direct}
{\sc Z.C.~Yang, P.~He, and H.Z.~Tang}, {\em A direct {E}ulerian {GRP} scheme for
  relativistic hydrodynamics: one-dimensional case}, Journal of Computational
  Physics, 230 (2011), pp.~7964--7987.

\bibitem{yang2012direct}
{\sc Z.C.~Yang and H.Z.~Tang}, {\em A direct {E}ulerian {GRP} scheme for
  relativistic hydrodynamics: two-dimensional case}, Journal of Computational
  Physics, 231 (2012), pp.~2116--2139.

\bibitem{yuan2020two}
{\sc Y.H.~Yuan and H.Z.~Tang}, {\em Two-stage fourth-order accurate time
  discretizations for 1{D} and 2{D} special relativistic hydrodynamics},
  Journal of Computational Mathematics, 38 (2020), pp.~768--796.

\bibitem{zhang2024generalized}
{\sc Q.~Zhang and W.~Sheng}, {\em The generalized {R}iemann problem scheme for
  a laminar two-phase flow model with two-velocities}, Journal of Computational
  Physics, 506 (2024), ~112929.

\bibitem{zhu2023high}
{\sc Z.~Zhu, Q.~Cui, and G.~Ni}, {\em A high-resolution scheme for axisymmetric
  hydrodynamics based on the 2{D} {GRP} solvers}, Computers \& Fluids, 264
  (2023), ~105961.

\bibitem{zou2021understand}
{\sc L.~Zou}, {\em Understand slope limiter--graphically}, arXiv preprint,
  arXiv:2102.04435,  2021.

\end{thebibliography}
\end{document}